\documentclass[10pt]{amsart}
\usepackage[T1]{fontenc}        % T1 font encoding
\usepackage[utf8]{inputenc}     % UTF-8 input encoding
\usepackage{txfonts}
\usepackage{xspace}              % proper spacing after macros
\usepackage{xfrac}               % slashed fractions
\usepackage{fancyhdr}

\usepackage{geometry} 
\usepackage{fancyhdr}            % customizable headers and footers

\usepackage{xcolor} 
\usepackage{hyperref}
\hypersetup{
    colorlinks=true,             % use colored links
    linktoc=all,                 % make sections and subsections linked
    linkcolor=blue,
    citecolor=blue,
    filecolor=blue,
    urlcolor=blue
}

\usepackage{booktabs}            % better tables
\usepackage{array} 
\usepackage{tabularx} 
\usepackage{rotating}            % rotated tables or figures

\usepackage{paralist}            % flexible lists
\usepackage{enumitem}            % control over list spacing

\usepackage{amsmath, amsthm, amssymb, amsfonts, amscd, eucal, mathtools, bm} % core math packages

\usepackage[ruled,vlined]{algorithm2e} % algorithms with ruled lines

\usepackage{tikz} 
\usetikzlibrary{patterns}        % patterns (keep only if needed)
\usetikzlibrary{arrows, arrows.meta, decorations.markings, calc} 
\usepackage{pgf} 
\usepackage{tikz-cd} 
\usepackage{tikzscale} 
\usepackage{pgfplots} 
\pgfplotsset{compat=1.6}        % safe version for arXiv

\usepackage{graphicx} 
\usepackage{subfig}               % subfigures

\usepackage{verbatim}             % comment out blocks of text
\usepackage{doi}                  % DOI links in bibliography
\usepackage{soul}                 % highlighting, strikeout
\usepackage{blindtext}            % dummy text
\usepackage{faktor}               % \faktor{A}{B}

\newcommand{\longleftrightarrowTilde}{
    \mathrel{\tikz \draw[<->, thick] (0,0) -- (2.5,0) node[midway, above] {\(\sim\)};}
}
\newcommand{\pto}[1]{\textbf{#1}}

\newcommand{\C}{\mathbb{C}}

\newcommand\Z{{\mathbb Z}}
\newcommand\ziz{{\Z\oplus i\,\Z}}

\DeclareMathOperator\Mod{Mod}

\DeclareMathOperator{\arccot}{arccot}

\newcommand{\spzz}{\mathrm{Sp}(2,\Z)}

\newcommand{\glplus}{\mathrm{GL}^+(2,\mathbb{R})}
\newcommand{\slz}{{\mathrm{SL}(2, \mathbb{Z})}}

\newcommand{\pslz}{{\mathrm{PSL}(2, \mathbb{Z})}}

\newcommand{\cp}{\mathbb{C}\mathrm{\mathbf{P}}^1}
\newcommand{\rp}{\mathbb{R}\mathrm{\mathbf{P}}^1}
\newcommand{\modul}{\textnormal{Mod}(S_{g,n})}

\newcommand{\homolz}{\text{H}_1(X,\mathbb{Z})}

\newcommand{\homolzn}{\textnormal{H}_1\big(\,X,\,\mathbb{Z}\,\big)}

\newcommand{\shomolzn}{\textnormal{H}_1(S_{g,n},\mathbb{Z})}

\newcommand{\shomolzoo}{\textnormal{H}_1(S_{1,1},\mathbb{Z})}

\newtheorem{thm}{Theorem}[section]
\newtheorem{cor}[thm]{Corollary}
\newtheorem{prop}[thm]{Proposition}
\newtheorem{lem}[thm]{Lemma}
\theoremstyle{definition}
\newtheorem{defn}[thm]{Definition}
\theoremstyle{remark}
\newtheorem{rmk}[thm]{Remark}
\theoremstyle{definition}

\theoremstyle{definition}

\theoremstyle{definition}

\numberwithin{equation}{section}
\title[Isoperiodic foliation of the stratum $\Omega\mathcal{M}_1(1,1,-2)$]{Isoperiodic foliation of the stratum $\Omega\mathcal{M}_1(1,1,-2)$}

\author{Gianluca Faraco}
\address[Gianluca Faraco]{Dipartimento di Matematica e Applicazioni U5, Universita` degli Studi di Milano-Bicocca, Via Cozzi 55, 20125 Milano, Italy}
\email{gianluca.faraco@unimib.it}

\author{Guillaume Tahar}
\address[Guillaume Tahar]{Beijing Institute of Mathematical Sciences and Applications, Huairou District, Beijing, China}
\email{guillaume.tahar@bimsa.cn}

\author{Yongquan Zhang}
\address[Yongquan Zhang]{Department of Mathematics, College of William \& Mary, 200 Ukrop Way, Williamsburg, VA 23185, USA}
\email{yzhang114@wm.edu}

\date{\today}

\begin{document}

\begin{abstract}
This paper describes the geometry and topology of leaves of the isoperiodic foliation of the stratum $\Omega\mathcal{M}_1(1,1,-2)$.  We prove that each leaf is a surface of infinite genus homeomorphic to the Loch Ness monster surface and supports a singular Euclidean structure whose singularities correspond to isoperiodic forms in the lower stratum $\Omega\mathcal{M}_1(2,-2)$. Along the way we also characterize the possible groups arising as the Veech group of a leaf and give a description of the large-scale conformal geometry of the wall-and-chamber decomposition of the leaves.
%On a Riemann surface, periods of a meromorphic differential along closed loops define a period character from the absolute homology group into the additive group of complex numbers. Fixing the period character in strata of meromorphic differentials defines the isoperiodic foliation where the remaining degrees of freedom are the relative periods between the zeros of the differential. In strata of meromorphic differentials with exactly two zeros, leaves have a natural structure of translation surface. In this paper, we give a complete description of the isoperiodic leaves in marked stratum $\widetilde{\mathcal{H}}(1,1,-2)$ of meromorphic $1$-forms with two simple zeros and a double pole on an elliptic curve. For each character, the metric completion of the corresponding leaf is a connected Loch Ness Monster. The translation structures of generic leaves feature a ramified cover of infinite degree over the flat torus defined by the lattice of absolute periods. By comparison, metric completions of isoperiodic leaves of the unmarked stratum $\mathcal{H}(1,1,-2)$ are complex disks endowed with a half-translation structure having infinitely many singular points. Finally, we give a description of the large-scale conformal geometry of the wall-and-chamber decomposition of the leaves.

\end{abstract}
\maketitle
\setcounter{tocdepth}{1}
\tableofcontents

\vspace{-0.5cm}

\section{Introduction}

\subsection{Overview from a complex-analytic perspective}\label{ssec:overview} Let \(S_{g,n}\) be a surface of genus \(g\) with \(n\) punctures. Let \(\mathcal M_{g,n}\) be the moduli space of Riemann surfaces homeomorphic to \(S_{g,n}\), and denote by \(\Omega\mathcal M_{g,n}\) the moduli space of pairs \((X,\omega)\) where \(X\) is such a Riemann surface and \(\omega\) is a holomorphic \(1-\)form on $X$ with poles at the punctures. We refer to these as \textit{(meromorphic) abelian differentials}. The moduli space \(\Omega\mathcal M_{g,n}\) admits a natural stratification given by unordered partitions \(\mu=(m_1,\dots,m_k,-p_1,\dots,-p_n)\) of \(2g-2\) (where negative integers are allowed) as follows. If \(\omega\) is a \(1-\)form with exactly \(k\) zeros of orders \(m_1,\dots,m_k\) and \(n\) poles at the punctures of orders \(p_1,\dots,p_n\), then \(\omega\) is said to be of type \(\mu\). Each partition \(\mu\) then singles out the stratum \(\Omega\mathcal M_g(\,\mu\,)\) of all meromorphic \(1-\)forms of type \(\mu\). We call \(\mu\) the \textit{signature} of the stratum. It is worth mentioning that strata are also often denoted by \(\mathcal H_g(\,\mu\,)\). In the present paper, however, we shall adopt the following notations.

\medskip

\textit{Convention:} For every signature \(\mu\) we denote by \(\mathcal H_g(\,\mu\,)\) the stratum of abelian differentials of type \(\mu\) whose zeros of the same order are marked. As a consequence the forgetful map \(\mathcal H_g(\,\mu\,)\longrightarrow \Omega\mathcal M_g(\,\mu\,)\) defines a covering map of degree at most \(k!\) where \(k\) is the number of positive integers in the partition \(\mu\). The reason why we need to make this distinction shall be clear later on.
%\textit{Convention:} For every signature\cm{Yongquan: type / signature, inconsistent terminology?} \(\mu\) we denote by \(\mathcal H_g(\,\mu\,)\) the strata of abelian differentials of type \(\mu\) with marked zeros. As a consequence the forgetful map \(\mathcal H_g(\,\mu\,)\longrightarrow \Omega\mathcal M_g(\,\mu\,)\) defines a genuine covering map of degree at most \(k!\) where \(k\) is the number of positive integers of the partition \(\mu\). The reason why we need to make this distinction shall be clear later on.

\medskip 

\noindent We now introduce \textit{period characters}, the first key notion of the present paper. It is well-known that every pair \((X,\omega)\) naturally determines a representation in homology, via the mapping 
\begin{equation}
    \chi\colon \homolzn \longrightarrow \mathbb C, \qquad \textnormal{ defined by } \qquad \gamma\longmapsto \int_\gamma \omega.
\end{equation}
Two pairs are said to be \textit{isoperiodic} if they determine the same period character \(\chi\). An \textit{isoperiodic fiber} is defined as the moduli space that parametrizes meromorphic forms that yield the same period character. For a general representation, say \(\rho\colon\shomolzn\longrightarrow \mathbb C\), we can define the isoperiodic fiber as follows
\begin{equation}
    \mathfrak {L}(\,\rho,\,\mu\,)=\left\{\,\,(X,\omega)\in\Omega\mathcal M_{g}(\,\mu\,)\,\,\,\Big|\,\,
    \begin{aligned}
        &\,\,\,\rho=f^*\chi, \textnormal{ for some } f\in\textnormal{Homeo}(S_{g,n},\,X)\\
        &\,\,\textnormal{ where } \chi \textnormal{ is the period character of } (X,\omega) 
    \end{aligned}\,\,
    \right\}.
\end{equation}

\noindent Notice that this space may be empty if \(\rho\) does not arise as the period character of any pair with prescribed signature. We will refer to a connected component of an isoperiodic fiber as an \textit{isoperiodic leaf}. Varying $\rho$ we obtain a holomorphic foliation on the stratum, called the \textit{isoperiodic foliation}. In \S\ref{sec:kernel_foliation} we shall provide more details about period realization and thus these spaces. Our results aim to provide a geometric and topological descriptions of \(\mathfrak L(\rho,\mu)\) when \(\mu=(1,1;-2)\). In what follows we shall adopt the following:

\smallskip 

\textit{Convention:} We will use the letter \(\rho\) to denote a general representation not necessarily associated with any particular structure, while the letter \(\chi\) will be used to indicate the period character of a specific structure.
 
\smallskip

\subsection{The stratum \(\Omega\mathcal{M}_1(1,1,-2)\)}\label{ssec:stratum112} We now focus on the stratum with signature \(\mu=(1,1,-2)\). It is shown in \cite{CFG} that every non-trivial representation can be realized in this stratum. As a consequence \(\mathfrak L(\,\rho,\,\mu\,)\) is not empty whenever \(\rho\) is non-trivial. Clearly, by \textit{trivial representation} we mean the unique representation \(\rho\) with \(\textnormal{Im}(\,\rho\,)=\{\,0\,\}\). Our first main result is the following:

\begin{thm}\label{thm:Unique}
For \(\mu=(1,1,-2)\), every locus \(\mathfrak L(\,\rho,\mu\,)\) is connected. Equivalently, every non-trivial representation \(\rho\) is realized by a unique isoperiodic leaf \(\mathfrak L(\,\rho,\mu\,)\).
\end{thm}

\noindent As a direct consequence of connectedness two isoperiodic structures with same signature can always be deformed continuously from one into another in the same leaf. In this case, the terms ``isoperiodic leaf'' and ``isoperiodic fiber'' can be used interchangeably. 

\smallskip

\noindent Our understanding of these leaves extends beyond what is stated in the theorem above. In fact, we are able to provide a description of the isoperiodic leaves not only from a topological perspective but also from a geometric one. Before stating the other results, we need to recall strata with marked points as in our convention in \S\ref{ssec:overview}. In the stratum \(\Omega\mathcal{M}_1(1,1,-2)\), all structures have two simple zeros that are indistinguishable, as there is no canonical way to single out one over the other. Marking the points enables us to differentiate them, and since there are two possible ways to do so, the stratum with marked zeros naturally forms a double cover. In our case, it is often more convenient to study the isoperiodic foliation on the stratum \(\mathcal{H}_1(1,1,-2)\) with marked zeros, whose topology and geometry then inform those of its unmarked counterpart. In order to distinguish isoperiodic leaves in marked and unmarked strata, we introduce the following notation:

\smallskip

\begin{equation}
    \mathfrak {ML}(\,\rho,\,\mu\,)=\left\{\,\,(X,\omega)\in\mathcal H_{g}(\,\mu\,)\,\,\,\Big|\,\,
    \begin{aligned}
        &\,\,\,\rho=f^*\chi, \textnormal{ for some } f\in\textnormal{Homeo}(S_{g,n},\,X)\\
        &\,\,\textnormal{ where } \chi \textnormal{ is the period character of } (X,\omega) 
    \end{aligned}\,\,
    \right\}.
\end{equation}

\smallskip

\begin{rmk}
    Clearly, for every signature \(\mu\) the covering projection \(\mathcal H_g(\,\mu\,)\longrightarrow \Omega\mathcal M_g(\,\mu\,)\) naturally yields a covering projection having the same degree between leaves \(\mathfrak {ML}(\,\rho,\,\mu\,)\longrightarrow \mathfrak {L}(\,\rho,\,\mu\,)\).
\end{rmk}

\noindent Our second result provides a topological description of the isoperiodic fibers in \(\mathcal H_{g}(\,\mu\,)\).

%Our next result focuses on the geometry and topology of the isoperiodic leaves in \(H_{g}(\,\mu\,)\), and as a direct consequence, we will derive conclusions about the isoperiodic foliation on \(\Omega\mathcal{M}_1(1,1,-2)\). 

\begin{thm}\label{thm:LochNess}
Let \(\rho\) be a non trivial representation and let \(\mu=(1,1,-2)\). In the marked stratum $\mathcal{H}_1(\,\mu\,)$, the metric completion of the isoperiodic leaf $\,\mathfrak{ML}(\,\rho,\,\mu\,)$ is homeomorphic to the Loch Ness monster surface. Moreover $\mathfrak{ML}(\,\rho,\,\mu\,)$ admits a translation structure determined by an abelian differential with infinitely many zeros of order two, \textit{i.e.} conical singularities of angle $6\pi$.
\end{thm}

\smallskip

\noindent We remind the reader of a few notions for convenience. The Loch-Ness monster is an infinite-type surface with infinite genus and only one end. In particular, it is well-known that such a surface is unique up to homeomorphism. A translation structure is a geometric structure on an oriented surface, locally modeled on \((\mathbb{C},\,\mathbb{C})\), which can be seen as the geometric counterpart of an abelian differential on a Riemann surface. In fact, each abelian differential naturally gives rise to such a structure, and conversely, every translation structure is determined by an abelian differential on a Riemann surface, see \S\ref{ssec:transurfa}. In what follows, we shall rely on this equivalence and often use this more geometric perspective to prove our results.

\smallskip

\begin{figure}[htp]
    \centering
    \includegraphics[width=0.9\linewidth]{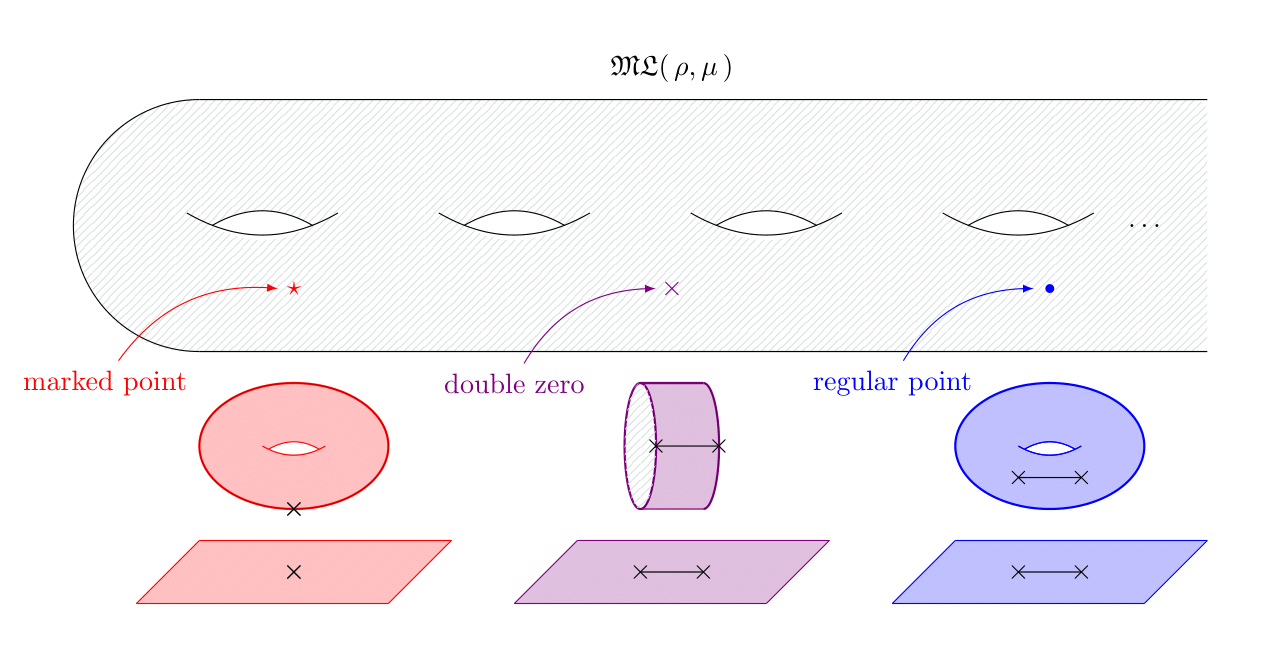}\bigskip
    \caption{Metric completion of an isoperiodic leaf \(\mathfrak{ML}(\,\rho,\mu\,)\).}
    \label{fig:isoleaf}
\end{figure}

\smallskip

\noindent We emphasize that in our theorem the Loch Ness monster arises as the \textit{metric completion} of an isoperiodic leaf in the marked stratum \(\mathcal H_1(1,1,-2)\). The leaf itself is homeomorphic to the Loch Ness monster minus infinitely many points and supports a translation structure whose differential has all zeros and poles (if any) at these punctures. By definition, each point of the leaf parametrizes an abelian differential with two distinct simple zeros and a single double pole on a torus, and small deformation of such a differential yields another one of the same type. At the same time, punctures of the leaf also have a precise geometric meaning: they parametrize degenerated structures along a sequence of abelian differentials represented by points on the leaf tending to a puncture. In our case, there are two possible degenerations: a sequence of abelian differentials of type \((1,1,-2)\) may converge to one of type \((2,-2)\), or the underlying tori may degenerate to a nodal curve. More precisely, the former occurs when the puncture corresponds to a zero of the differential of the metric completion, and the latter occurs otherwise.

%if any sequence in the leaf converges to a zero of the differential of the metric completion then the degenerated differential is of type \((2,-2)\) otherwise it is a nodal curve, see Figure \ref{fig:isoleaf}.

\medskip

\noindent As a consequence of Theorem \ref{thm:LochNess} we derive the following description of isoperiodic leaves in the unmarked stratum \(\Omega\mathcal M_1(1,1,-2)\).

\begin{cor}
    If \(\mu=(1,1,-2)\), then the isoperiodic leaf $\mathfrak{ML}(\,\rho,\,\mu\,)$ naturally comes with a non-trivial isometry of order two whose quotient corresponds to the isoperiodic leaf \(\mathfrak{L}(\,\rho,\,\mu\,)\) in \(\Omega\mathcal M_1(1,1,-2)\), the unmarked stratum. Its metric completion is a half translation surface biholomorphic to the unit disk \(\mathbb D\) endowed with a quadratic differential having infinitely many simple zeros and at most a simple pole. 
\end{cor}

\noindent In fact, for a generic unmarked leaf (when the image of $\rho$ is a lattice in $\mathbb{C}$), we construct an explicit biholomorphic map to the Teichm\"uller space $\mathcal{T}_{1,1}$, which is conformally a disk (see \S\ref{sec:conf_geom}).
%Moreover, if  positive leaf or an arithmetic real leaf, \textit{i.e.} \(\textnormal{Im}(\,\rho\,)\cong \mathbb Z\), the quadratic differential also has a simple pole, namely a conical singularity of angle $\pi$.

\bigskip

\subsection{Veech groups of isoperiodic leaves}  Given a translation structure on a possibly non-compact surface, a homeomorphism is said to be \textit{affine} if it is the restriction of some transformation in \(\textnormal{Aff}(\,\mathbb{R}^{2} \,)\) in local translation charts. Notice that the derivative of an affine homeomorphism is a well-defined element in \(\textnormal{GL}^{+}(2,\mathbb{R})\). For a fixed translation structure, the derivatives of its affine homeomorphisms form the \textit{Veech group}. In fact, since each isoperiodic leaf \(\mathfrak L(\,\rho,\,\mu\,)\) is uniquely determined by the representation, all elements of \(\textnormal{GL}^{+}(2,\mathbb{R})\) that stabilize \(\rho\) yield affine automorphisms of the associated leaf. In the present paper we also determine the Veech groups of isoperiodic leaves for stratum \(\mathcal H_1(1,1,-2)\), which turn out to be very large. More precisely it holds the following:

\begin{thm}\label{thm:Veech}
Let \(\rho\) be a non-trivial representation and let \(\mu=(1,1,-2)\).  The Veech group of an isoperiodic leaf \(\,\,\mathfrak{ML}(\,\rho,\,\mu\,)\) in the marked stratum \(\mathcal{H}_1(1,1,-2)\) is a conjugate of:
\begin{itemize}
    \medskip
    \item \(\textnormal{SL}(2,\,\mathbb Z)\) if \(\textnormal{Im}(\,\rho\,)\) is not contained in a real line;
    \medskip
    \item $\left\{\pm\begin{pmatrix}u&b\\0&a\end{pmatrix}:a,b\in\mathbb{R}, a>0, u\in G\right\}$ if \(\textnormal{Im}(\,\rho\,)\) is proportional to a non-discrete subgroup \(\Gamma\) of a real quadratic field $\mathbb{Q}(\sqrt{D})$ so that $G=\{u>0:u\Gamma=\Gamma\}$ is nontrivial;
    \medskip
    \item $\left\{\pm\begin{pmatrix}1&b\\0&a\end{pmatrix}:a,b\in\mathbb{R}, a>0\right\}$ otherwise.
\end{itemize}
\end{thm}

\medskip

\subsection{Structure of isoperiodic leaves and their conformal geometry}\label{ssec:isostrucandconfgeo} Each isoperiodic leaf is in itself a moduli space and parametrizes isoperiodic abelian differentials of a given type \(\mu\) according to their geometry. Given an abelian differential \((X,\omega)\), its convex core is defined as the convex hull of the zeros (for more details see \S\ref{ssec:core}). For \(\mu=(1,1,-2)\) we can single out three different types of cores: it could be a torus, a cylinder or has empty interior. Although the names are quite intuitive, for now we shall say nothing about how to distinguish these structures and defer the discussion to \S\ref{ssec:core}. Given any structure, the type of core is invariant under sufficiently small deformations but may change if the structure is deformed excessively. In other words, having the same core is an open condition and hence, in a given leaf, the set of structures with the same type of core forms an open subset, possibly disconnected. Each of these open subsets is called \textit{chamber}. A \textit{wall}, instead, is a closed codimension one subset that separates two chambers. Notice that these chambers do not need to parametrize structures with the same kind of core. In fact, if during the deformation of a given structure a wall is crossed, one can observe a drastic change in the core, and the point on the wall is precisely the structure where the transition occurs. 

\smallskip

\noindent As a consequence we thus distinguish three kind of chambers: \textit{torus type}, \textit{cylinder type} and \textit{degenerate type} (that means the core has empty interior). The most prevalent chambers are of those of cylinder type. We show that each such a chamber is conformally equivalent to \(\mathbb D\), the unit disk. For a generic unmarked leaf \(\mathfrak L(\,\rho,\,\mu\,)\), chambers of cylinder type are parametrized by primitive elements in the period group \(\textnormal{Im}(\,\rho\,)\). By using the identification $\mathcal{T}_{1,1}\cong\mathbb{H}$ we can draw these chambers as subsets of the upper half plane \(\mathbb{D}\). Recall that $\textnormal{SL}(2,\,\mathbb Z)$ acts on $\mathcal{T}_{1,1}$ as mapping classes, and correspondingly on the unit disk as hyperbolic isometries.

\begin{thm}\label{thm:confgeoleaves}
Let \(\rho\) be a non-trivial representation, \(\mu=(1,1,-2)\) and let \(\mathfrak{L}(\,\rho,\,\mu\,)\) be a generic unmarked leaf with period group \(\textnormal{Im}(\,\rho\,)\). Under a suitable conformal identification \(\mathfrak{L}(\,\rho,\,\mu)\cong\mathcal{T}_{1,1}\cong\mathbb{H}\), the boundary of each cylinder chamber in \(\mathbb{H}\cup\mathbb{R}\cup\{\infty\}\) is a rational number or \(\infty\), and every element in \(\mathbb{Q}\cup\{\infty\}\) arises this way. Moreover, the actions of $\textnormal{SL}(2,\,\mathbb Z)$ on \(\mathfrak{L}(\,\rho,\,\mu\,)\cong\mathbb{H}\) as the Veech group and hyperbolic isometries respectively agree on the boundary \(\mathbb{R}\cup\{\infty\}\). Finally, for each cylinder chamber, it is not contained in the $K$-neighborhood of a hyperbolic geodesic ray for any $K>0$, nor does it contain any horodisk based at the same point at infinity.
\end{thm}

\noindent We give more precise information about the conformal geometry of cylinder chambers in \S\ref{sec:conf_geom}. Figure~\ref{fig:positive_leaf_schema} shows a schematic picture of an unmarked positive leaf, where some cylinder chambers are shown in red.
\begin{figure}[htp]
    \centering
    \includegraphics[width=0.6\linewidth]{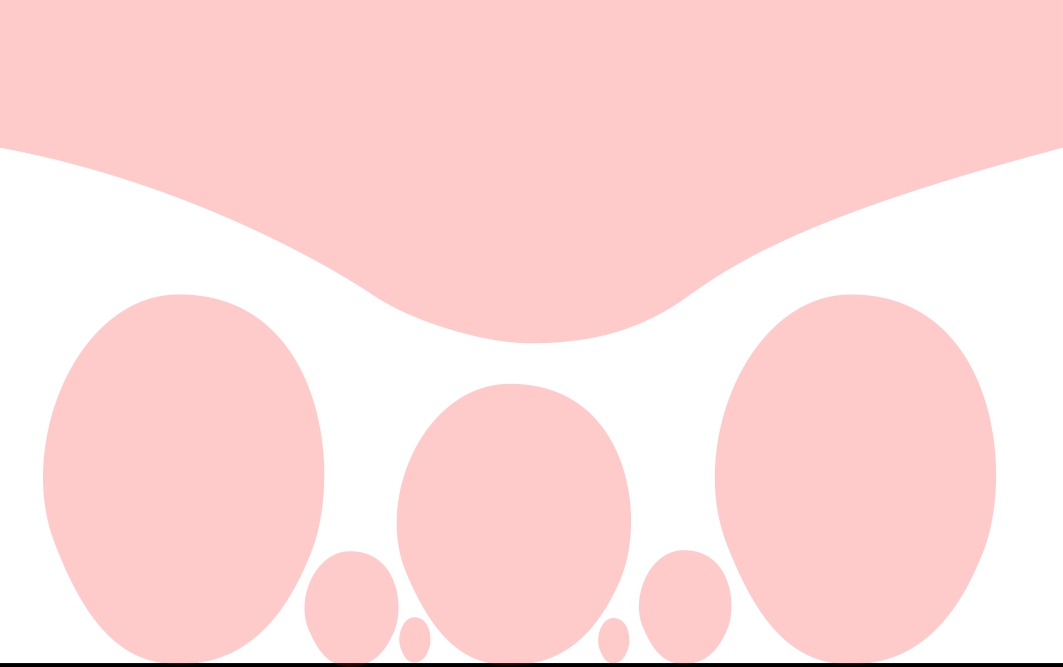}
    \caption{A schematic picture of an unmarked positive leaf, drawn in the upper half plane. For the chamber based at $\infty$, the two ends of its boundary goes to infinity with asymptotics $\asymp\log x$.}
    \label{fig:positive_leaf_schema}
\end{figure}

\subsection{Discussion on related works} A question that naturally arises is under which conditions an isoperiodic fiber is connected. For holomorphic differentials it is known that the isoperiodic fibers in $\Omega\mathcal S_g$ are generically connected up to a few exceptions; see \cite[Theorems 1.2 and 1.3]{CDF}. In \cite{Winsor}, Winsor have studied the connectedness of the isoperiodic foliation for non-principal strata of holomorphic differentials in $\Omega\mathcal M_g$. For meromorphic differentials, an analogous result has recently been obtained in \cite{CD} for differentials with two simple poles and in \cite{ACD} for differentials with at least three simple poles under an additional genericity condition on the periods. It is worth mentioning that in \cite{ACD,CD, CDF} the authors work with strata that have all simple zeros and obtain a classification of the closure of leaves. A study of the isoperiodic foliation has been carried out in \cite{KLS} in the specific case of meromorphic differentials with real absolute periods (\textit{i.e.} real-normalized differentials) and a unique double pole. Theorem~\ref{thm:Unique} of the present paper confirms the counting of leaves realizing a given group of absolute periods in the case of genus one \cite[Theorem 9]{KLS}.

\smallskip

\noindent For strata of meromorphic $1$-forms on the Riemann sphere $\mathbb{CP}^{1}$, the isoperiodic foliation coincides with the \textit{isoresidual fibration} defined by the map sending a differential to the vector of its residues at the (labeled) poles; see \cite{GT1} for details. For strata of differentials on $\mathbb{CP}^{1}$ with exactly two zeros, a systematic study carried out in \cite{CGPT} has shown that generic isoresidual fibers are (possibly nonconnected) complex curves of finite genus (in the multi-scale compactification defined in \cite{BCGGM}). In contrast with the results of the present paper, generally isoresidual fibers in strata of genus zero are defined by algebraic equations (residues at the poles are rational functions of the locations of zeros and poles).

\medskip

\subsection{Organization of the paper} In \S\ref{sec:translation}, we give general background on translation surfaces and their moduli spaces. In particular, we introduce the notion of volume of a representation, the notion of core and the walls-and-chambers decomposition. In \S\ref{sec:kernel_foliation}, we provide the classification of leaves and describe their translation structures: conical singularities and decomposition into chambers. As already alluded above each leaf contains chambers of cylinder type where the cores of the parametrized translation surfaces are cylinders. The geometric description and classification of these chambers are given in \S\ref{sec:chamcyltype}. Proofs of Theorems~\ref{thm:Unique}, \ref{thm:LochNess} and \ref{thm:Veech} are split into the next four sections depending on the type of the isoperiodic leaf as follows:
\begin{itemize}
    \item \textit{Positive leaves}, \textit{i.e.} \(\text{Im}(\,\rho\,)\) is a lattice of $\mathbb{C}$ and \(\rho\) has positive volume, are treated in \S\ref{sec:posleaves};
    \item \textit{Negative leaves}, \textit{i.e.} \(\text{Im}(\,\rho\,)\) is a lattice of $\mathbb{C}$ and \(\rho\) has negative volume, are treated in \S\ref{sec:negleaves};
    \item \textit{Non-arithmetic real leaves}, \textit{i.e.} \(\text{Im}(\,\rho\,)\) is a dense subgroup of \(\mathbb R\), are treated in \S\ref{sec:nonarith};
    \item \textit{Arithmetic real leaves}, \textit{i.e.} \(\text{Im}(\,\rho\,)=a\,\mathbb Z\cong\mathbb Z\) for some $a \in \mathbb{C}^{\ast}$, are treated in \S\ref{sec:arithm}.
\end{itemize}

\smallskip

\noindent Finally, in \S\ref{sec:conf_geom} and the appendix, we discuss in greater detail the conformal geometry of the leaves, describing the large scale asymptotic of the cylinder chambers in them in terms of the Teichm\"uller / hyperbolic metric.

\smallskip

\paragraph{\bf Acknowledgements.} The authors would like to thank Dawei Chen, Alexandre Eremenko, Curtis McMullen and Karl Winsor for valuable remarks and discussions. We also thank Steffen Rhode for giving an illuminating talk on the paper \cite{ILRS}, and Don Marshall for sharing his program which produces true trees. G.F. and Y.Z. acknowledge the support of the Max Planck Institute for Mathematics, where part of the work was done. G.F. was supported by SSP 2026 DFG Priority Program \textit{Geometry at Infinity} and he is a member of GNSAGA. Finally,
research by G.T. is supported by the Beijing Natural Science Foundation (Grant IS23005) and the French National Research Agency under the project TIGerS (ANR-24-CE40-3604), and Y.Z. is partially supported by an AMS-Simons Travel Grant.

\section{Geometry of translation structures}\label{sec:translation}

%\noindent We begin by introducing basic concepts of the main objects of interest. 

\noindent In this section we aim to provide a more geometric description of meromorphic \(1\)-forms on Riemann surfaces via the language of geometric structures, complementing the complex-analytic description.

\subsection{Elliptic curves and strata of differentials} An \textit{elliptic curve} \(E\) is a compact Riemann surface of genus one. By Uniformization Theorem, every elliptic curve is biholomorphic to the quotient space $\C/\Lambda$ where $\Lambda\cong\Z^2$ is a lattice of $\mathbb{C}$. Two lattices, say $\Lambda$ and $\Lambda'$, define the same complex structure if and only if there exists $\alpha \in \mathbb{C}^{\ast}$ such that $\Lambda'=\alpha\Lambda$. Let $\mathcal{M}_{1,1}$ denote the moduli space of elliptic curves with a marked point. This space is well-known to be a one-dimensional complex orbifold with two singularities given by those elliptic curves admitting non-trivial complex automorphisms.

\smallskip

\noindent We define strata \(\Omega\mathcal{M}_1(1,1,-2)\) as the moduli space of pairs \((E,\omega)\) where \(E\) is an elliptic curve and $\omega$ is a meromorphic $1$-form with two unmarked simple zeros and a single double pole. Two pairs, say $(E_1,\omega_1)$ and $(E_2,\omega_2)$, are identified if there is a biholomorphic map \(f\colon E_1\longrightarrow E_2\) such that \(\omega_1=f^*\omega_2\). We also introduce the stratum with marked zeros \(\mathcal H_1(1,1,-2)\) as the double cover of 
\(\Omega\mathcal{M}_1(1,1,-2)\) where the differentials are decorated with the marking of the two simple zeros. We shall label, whenever this is necessary, these zeros as \(\pto B\) and \(\pto W\) -- black and white. For instance we shall make use of this labeling in several pictures along the way.

\begin{rmk}
    Fixing the double pole as the neutral element of the group law of the elliptic curve, a meromorphic $1$-form with two simple zeros and one double pole is characterized by the position of its zeros -- a principal divisor $A$+$B$ on the elliptic curve -- and a scaling factor $\lambda \in \mathbb{C}^{\ast}$. If the two zeros coincide, we obtain a differential with a double pole and Abel-Jacobi relations imply that the double zero is a $2$-torsion point of the elliptic curve.
\end{rmk}

\noindent We also introduce \(\Omega\mathcal M_1(2,-2)\), as the moduli space of pairs \((E,\omega)\) where \(\omega\) has a single zero and a single pole both of order two. The stratum \(\Omega\mathcal M_1(\,-\,,\,-2)\) is defined as \(\Omega\mathcal M_1(2,-2)\,\cup\,\Omega\mathcal M_1(1,1,-2)\) and comprises all meromorphic \(1-\)forms on \(E\) with a double pole. The forgetful map \(\mathfrak F\colon\mathcal H_1(1,1,-2)\longrightarrow \Omega\mathcal M_1(1,1,-2)\) naturally extends to a branched covering map \(\mathcal H_1(\,-\,,-2)\longrightarrow \Omega\mathcal M_1(\,-\,,-2)\) of degree \(2\) with \(\Omega\mathcal M_1(2,-2)\) as the branched locus. Some remarks are in order.

\begin{rmk}
    As a direct consequence the natural \(2-\)fold covering \(\mathfrak {ML}(\,\rho,\,\mu\,)\longrightarrow \mathfrak {L}(\,\rho,\,\mu\,)\) extends to a branched cover of the same degree between their metric completion. We shall make use of this latter covering map to infer the geometric properties of isoperiodic leaves in \(\Omega\mathcal M_1(1,1,-2)\).
\end{rmk}

\begin{rmk}
    Any meromorphic function on $E=\mathbb{C}/\Lambda$ with a unique double pole at the lattice point is of the form $a\wp+b$ with $a \in \mathbb{C}^{\ast}$ and $b \in \mathbb{C}$, where $\wp$ denotes the meromorphic function on $E$ induced from the Weierstarss $\wp$-function associated with the lattice $\Lambda$. Consequently, any meromorphic $1$-form of \(\Omega\mathcal M_1(\,-\,,-2)\) is a linear combination of $\wp dz$ and $dz$. This gives a way of parametrizing such forms, as they are clearly characterized by these coefficients.
\end{rmk}

\subsection{Translation surfaces}\label{ssec:transurfa} We now introduce the geometric counterpart of any pair \((X,\omega)\). As already alluded to in \S\ref{ssec:stratum112}, a \textit{translation surface} is a branched $(\C,\C)$-structure, \textit{i.e.} the datum of a maximal atlas where local charts in $\C$ have the form  $z\longmapsto z^k$, for $k\ge1$, and transition maps given by translations on their overlappings. Any such an atlas defines an underlying complex structure $X$ and the pullbacks of the $1$-form $dz$ on $\C$ via local charts globalize to a holomorphic differential $\omega$ on $X$. Vice versa, a holomorphic differential $\omega$ on a complex structure $X$ defines a singular Euclidean metric with isolated singularities corresponding to the zeros of $\omega$. In a neighborhood of a point $\pto P$ which is not a zero of $\omega$, a local coordinate is defined as
\begin{equation}
    z(\,\pto Q\,)=\int_{\,\pto P}^{\,\pto Q} \omega 
\end{equation} 
in which  $\omega=dz$, and the coordinates of two overlapping neighborhoods differ by a translation $z\mapsto z+c$ for some $c\in\mathbb C$. Around a zero, say $\pto P$ of order $k\ge1$, there exists a local coordinate $z$ such that $\omega=z^kdz$. The point $\pto P$ is also called a \textit{branch point}, since any local chart around it is locally a branched $k+1$ covering totally ramified at $\pto P$ over an open subset of $\C$.

\begin{rmk}[Local geometry around a singularity] 
We have already described the Euclidean geometry around a zero of order \(m\). Shortly, it is locally isometric to the vertex of an Euclidean cone of angle \(2\pi(m+1)\). The geometry around a pole depends not only on its order but also on the value of the residue, measured by evaluating \(\omega\) along any closed path around the singularity. There are canonical models for these geometries and we refer to \cite{Bo, CFG, Ta} for more details.
\end{rmk}

\noindent In the present work we shall consider only double poles with zero residue. Notice that this is far from being restrictive. In fact, for every pair \((E,\omega)\in\Omega\mathcal M_1(1,1,-2)\) with period character \(\chi\), any simple essential loop around the pole in \(\omega\) arises as the commutator of some pair of simple closed curves generating $\textnormal{H}_1(S_{1,1},\mathbb Z)$. Since commutators always have trivial period and the period of a pole equals its residue up to a factor \(2\pi i\) we conclude that the residue around the unique pole of $\omega$ is zero. Referring to the models in \cite{Bo, CFG, Ta} mentioned above, a neighborhood of the double pole is isometric to a neighborhood of infinity in the flat plane. We can finally state the main object of this section.

\begin{defn}[Translation surfaces with poles]\label{tswp}
Let $\omega$ be a meromorphic differential on a compact Riemann surface $\overline{X}\in\mathcal{M}_g$. We define a \emph{translation surface with poles} to be the structure induced by $\omega$ on the surface $X=\overline{X}\setminus\Sigma$, where $\Sigma$ is the set of poles of $\omega$.
\end{defn}

\noindent Given a translation structure $(X,\omega)$ on a surface $S_{g,n}$, local charts globalize to a holomorphic multi-valued function \(F\colon X\longrightarrow \mathbb C\) via analytic continuation. Such a function lifts to a genuine map $\text{dev}\colon\widetilde{S}_{g,n}\longrightarrow \C$ called the \textit{developing map}, where $\widetilde{S}_{g,n}$ is the universal cover of $S_{g,n}$. The developing map satisfies an equivariant property with respect to some representation $\rho\colon\homolz\longrightarrow\C$ called the \textit{holonomy} of the translation structure. The following lemma establishes the relation between holonomy representations and period characters.

\begin{lem}
A representation $\rho\colon\shomolzn\longrightarrow \C$ is the period of some abelian differential $\omega\in\Omega(X)$ with respect to some complex structure $X$ on $S_{g,n}$ if and only if it is the holonomy of the translation structure on $S_{g,n}$ determined by $\omega$.
\end{lem}

\noindent This twofold nature of a representation $\rho$ permits us to tackle our problem by adopting a geometric approach as done in several works aiming to realize representation as period characters. 

\smallskip

\subsection{Action of $\text{GL}^{+}(2,\mathbb{R})$}\label{subsec:glaction}
We next consider the action of \(\text{GL}^+(2,\,\mathbb R)\) on strata defined as follows. For every translation surface \((X,\omega)\) and \(g\in\text{GL}^+(2,\mathbb{R})\), we define \(g\,\cdot\,(X,\omega)\) as the translation surface obtained by post-composing coordinate functions of \((X,\omega)\) with \(g\). A more geometric way to see this action is the following. If $(X,\omega)$ is constructed by identifying parallel sides of a collection of (possibly unbounded) polygons $P_i$ in $\mathbb{C}$, then $g\cdot(X,\omega)$ is obtained by identifying the corresponding sides of the polygons $g(P_i)$, viewing $g$ as a linear function on $\mathbb{C}\cong\mathbb{R}^2$.

\begin{rmk}
    It is easy to see that this induces an action on each stratum, see \cite{Zo} for a general reference on translation surfaces. \cite[\S3.3]{Fil} discusses more specifically the case of strata of meromorphic differentials.
\end{rmk}

\noindent The $\text{GL}^{+}(2,\mathbb{R})$-action as defined above also applies to translation surfaces of infinite type. We recall for the reader's convenience that a translation surface $(X,\omega)$ is:
\begin{itemize}
    \item of \textit{finite type} if $X$ is a compact Riemann surface (from which we may have removed finitely many points) and $\omega$ is of finite area;
    \item of \textit{infinite type} otherwise.
\end{itemize}

\smallskip

\noindent We now recall the following:

\begin{defn}
    The \textit{Veech group} of a translation surface is the group of derivatives of all affine homeomorphisms. Here an \textit{affine homeomorphism} of a translation surface is a homeomorphism whose restrictions in local coordinates are affine maps (\textit{i.e.} elements in \(\text{Aff}(\mathbb{R}^{2})\)). Equivalently, $g\in\text{GL}^{+}(2,\mathbb{R})$ lies in the Veech group of $(X,\omega)$ if and only if $g\cdot (X,\omega)$ is isomorphic (as translation surfaces) to $(X,\omega)$.
\end{defn}

\noindent Although Veech groups of translation surfaces of finite type (thus corresponding to holomorphic $1$-forms) are always discrete subgroups of \(\text{SL}(2,\,\mathbb R)\), Veech groups of translation surfaces of infinite type are much less constrained. See the main result of \cite{RMV}.

\smallskip

\subsection{Volume} Let \(\rho\colon\shomolzoo\longrightarrow \mathbb C\) be a representation and let \(\{\,\alpha,\beta\,\}\) be two generators for \(\shomolzoo\) that represent a pair od simple closed curves on \(S_{1,\,1}\) with intersection number one. We define the \textit{volume} of \(\rho\), denoted by \(\textnormal{vol}(\,\rho\,)\), as the quantity
\begin{equation}
    \textnormal{vol}(\,\rho\,)=\mathfrak{Im}\,\Big(\,\,\overline{\chi(\,\alpha\,)}\,\chi(\,\beta\,)\,\,\Big).
\end{equation}

\noindent It can be checked directly that this does not depend on the choice of $\alpha,\beta$. The following lemma is crucial in its simplicity as it will permit us to make a distinction between positive and negative leaves.

\begin{lem}
    For every \(t\in\mathbb R\) there is a representation \(\rho\colon\shomolzoo\longrightarrow \mathbb C\) such that \(\textnormal{vol}(\,\rho\,)=t\).
\end{lem}

\noindent We do not provide a proof here as it can be directly deduced from \cite[Theorem A]{CFG}, even though a simpler and shorter proof is possible. As we shall see the geometry of an isoperiodic leaf depends strongly on the sign of \(\textnormal{vol}(\,\rho\,)\) which could be positive, negative or zero. See \cite{CF} for further details about volume in the most general setting.

\smallskip

\subsection{Metric features of translation surfaces}\label{ssec:core} We now recall two features of translation surfaces, namely saddle connections and core.

\smallskip

\noindent Given a translation surface \((X,\omega)\in\Omega\mathcal M_{g,n}\), a \textit{saddle connection} is defined as a geodesic segment, in the flat metric induced by \(\omega\) connecting two zeros of \(\omega\), so that no interior points of the segment are singularities. In principle these two zeros do not need to be distinct; we allow them to be the same, in which case we say the saddle connection is \textit{closed}. On a translation surface $(X,\omega)$, every saddle connection represents a class in the relative homology group \(\textnormal{H}_{1}(X,\, Z(\,\omega\,),\,\mathbb Z)\) where \(Z(\omega)\) is the set of zeros of \(\,\omega\,\). Recall that \(\omega\) has poles at the punctures of \(X\).

\smallskip

\noindent The next definition was first introduced in \cite{HKK} and then systematically used in \cite{GT, Ta, Ta1} to study the geometry of translation structures induced by meromorphic differentials. All the constructions and proofs are given in detail in \cite[\S4]{Ta}.

\begin{defn}
A subset, say \(C\), of a translation surface \((X,\omega)\) is said to be \textit{convex} if for any pair of points in \(C\), any geodesic segment joining them is entirely contained in \(C\).
%A subset $C$ of a translation surface $(X,\omega)$ is said to be \emph{convex} if, for any two points $x,y\in C$ joined by a geodesic segment $I$, the entire geodesic segment $I$ is contained in $C$. 
The \textit{core} of \((X,\omega)\)
%, denoted by $\mathcal{C}(X,\omega)$, 
is defined as the convex hull of the zeros of \(\omega\), \textit{i.e.} the smallest closed convex set containing all the zeros.    
\end{defn}

\noindent The following lemma is contained in \cite[Propositions 4.4 and 4.5]{Ta}. We refer to this paper for the proof. 

\begin{lem}\label{lem:retract}
Let \((X,\omega)\) be a translation surface and assume $\omega$ has at least one pole. Then its core is a topological retract of the surface \(X\) (punctured at the poles of $\omega$). Each connected component of the complement of the core is a topological disk containing exactly one pole called \textit{polar domain}. The boundary of each polar domain is formed by finitely many saddle connections and, by convexity, each interior angle of a polar domain is at least $\pi$.
%Let $(X,\omega)$ be a translation surface such that $\omega$ has at least one pole. Then $\mathcal{C}_{(X,\omega)}$ is a topological retract of surface $X^{\ast}$ (punctured at the poles of $\omega$). Each connected component of $X \setminus \mathcal{C}_{(X,\omega)}$ is a topological disk containing exactly one pole called \textit{polar domain}. The boundary of each polar domain is formed by finitely many saddle connections and, by convexity, each interior angle of a polar domain is at least $\pi$.
\end{lem}

\noindent The topological structure of differentials in the stratum $\Omega\mathcal{M}_1(1,1,-2)$ is simple enough to make possible a topological classification of cores, see Figures \ref{fig:emptytype}, \ref{fig:torustype} and \ref{fig:cyltype}. In the following we denote by \(b\) the number of boundary saddle connections of the polar domain and by \(t\) the number of triangles in any geodesic triangulation of the interior of the core. We remark that a boundary saddle connection is counted twice when two boundary edges are identified, see Figure~\ref{fig:cyltype} as a reference picture. \textit{A priori} it is not immediately clear that a geodesic triangulation of the interior of the core exists, but this has been proved in \cite[Lemma 2.2]{Ta2}. The following lemma is a consequence of the Euler characteristic formula on the torus. A formula for any meromorphic $1$-forms on a generic compact Riemann surface can be found in \cite{Ta}.

\begin{lem}\label{lem:coreformula}
For any translation surface of $\Omega\mathcal{M}_1(1,1,-2)$ the following relation holds: \(b+t=6\).
\end{lem} 

\noindent As a consequence of this formula we are able to provide a topological classification of cores, based on the possible values \(b\) (or equivalently \(t\)) may assume.

\begin{prop}\label{prop:coreshapes}
Let \((X,\omega)\) be a translation surface in the stratum \(\Omega\mathcal{M}_1(1,1,-2)\). Then the interior of the core can be:
\begin{enumerate}[label=\normalfont{(\roman*)}]
    \item a torus with a slit (we call this a \emph{torus type});
    \item a cylinder (\emph{cylinder type});
    \item empty (\emph{degenerate type}).
\end{enumerate}
\end{prop}

\begin{proof}
Using Lemma~\ref{lem:coreformula}, we shall investigate possible values of \(b\). First note that since the residue at the double pole is zero, its polar domain cannot be bounded by a unique saddle connection, because its period would be nonzero. It readily follows that \(b\ge2\). We also notice that the case \(b=6\) corresponds to the degenerate case, see Figure \ref{fig:emptytype}. 

\begin{figure}[htp]
    \centering
    \includegraphics[width=0.6\linewidth]{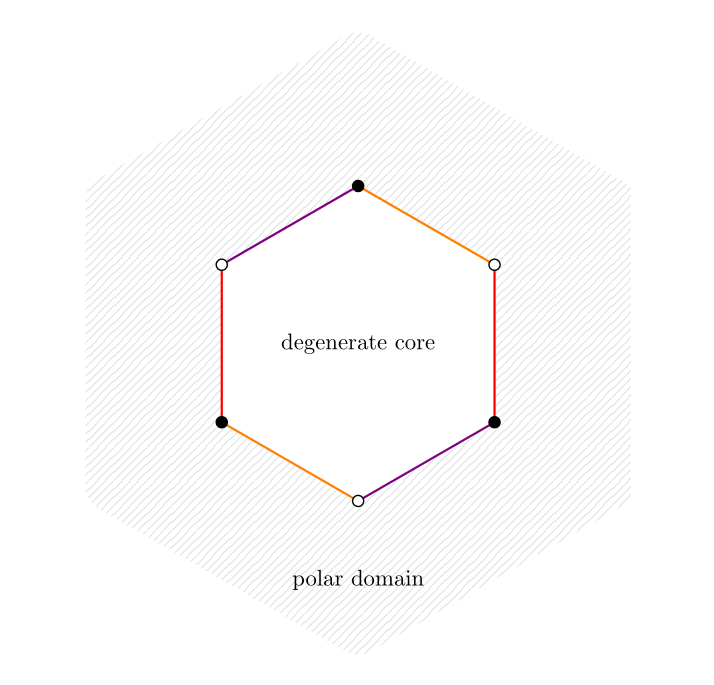}
    \caption{Translation surface in the marked stratum $\mathcal{H}_1(1,1,-2)$ with degenerate core. Such a structure can be realized by gluing sides with same color.}
    \label{fig:emptytype}
\end{figure}

\noindent Thus we only need to consider those cases corresponding to the following possible values of \(t=1,2,3,4\). We shall see only that two out four of these values yield a possible configuration of the core. We begin with the following one.

\smallskip

\textit{Case 1}. We first assume \(t=4\) which means \(b=2\). In this case, the polar domain is a copy of the complex plane with a slit. The two end points of the slit cannot be identified, for otherwise the cone angle is strictly bigger than $4\pi$. The boundary of the core consists of two parallel geodesic segments. If we glue these two segments together, we obtain a flat torus with two marked points. In other words the core is a torus with a slit along a geodesic segment. See Figure \ref{fig:torustype}.

\begin{figure}[htp]
    \centering
    \includegraphics[width=0.6\linewidth]{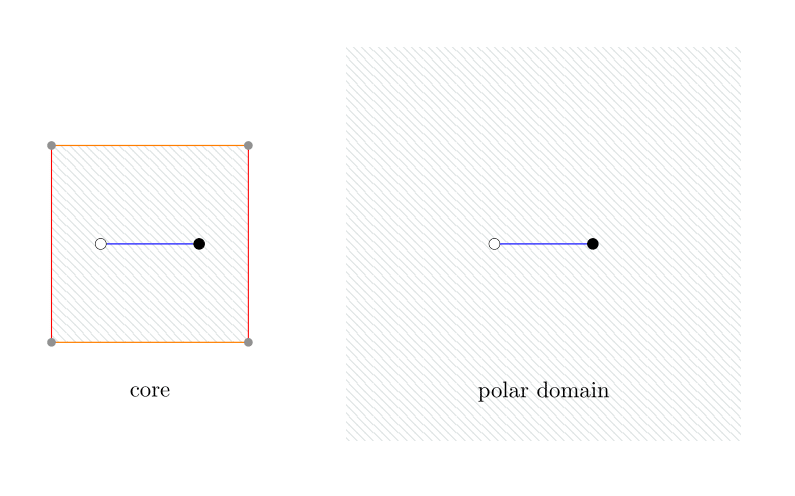}
    \caption{Translation surface in the marked stratum $\mathcal{H}_1(1,1,-2)$ whose core is a slit torus. On the left it is depicted the core and, on the right, it is depicted the domain of the pole. Note this latter is a slit copy of $\mathbb C$ whose point at infinity is the double pole of the resulting structure. In this case $b=2$ because, in the slit plane, there are two saddle connections. }\label{fig:torustype}
\end{figure}

\smallskip

\textit{Case 2}. We next assume \(t=2\), that is \(b=4\). In this case, the polar domain is a complex plane minus a convex quadrilateral $Q$. The interior of the core is either a quadrilateral, whose four sides connects to the polar domain, or a cylinder. The former cannot happen, since otherwise the surface is just a complex plane with four marked points. In the latter case, $Q$ is a parallelogram. Two opposite sides of $\mathbb{C} \setminus Q$ are identified, and the other two sides are connected to the ends of the cylinder. See Figure \ref{fig:cyltype}.

\begin{figure}[htp]
    \centering
    \includegraphics[width=0.6\linewidth]{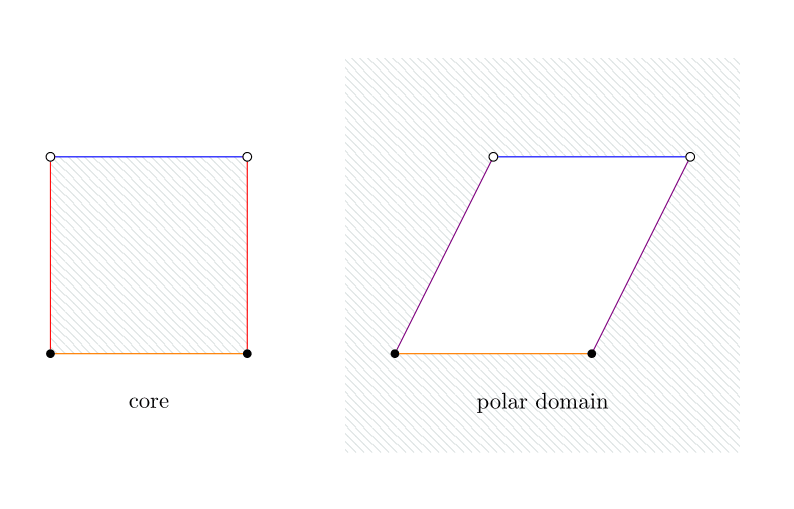}
    \caption{Translation surface in the marked stratum $\mathcal{H}_1(1,1,-2)$ whose core is a cylinder. Such a structure can be realized by gluing sides with same color. On the left it is depicted the core and, on the right, it is depicted the domain of the pole. As above, this latter is a slit copy of $\mathbb C$ whose point at infinity is the double pole of the resulting structure.}\label{fig:cyltype}
\end{figure}

\smallskip

\noindent We are left with the cases $t=1,3$ and we aim to show that these cannot happen. We consider first the following:

\textit{Case 3}. Assume \(t=1\) which means \(b=5\). The polar domain has five sides $A_{1},A_{2},A_{3},A_{4},A_{5}$; two being identified in the surface. These two identified sides cannot be adjacent (otherwise one of the two singularities would be a marked point). Without loss of generality, we assume they are $A_{1}$ and $A_{3}$. Thus, the sum of the angles corresponding to corners $A_{1}A_{2}$ and $A_{2}A_{3}$ is $3\pi$. These corners contribute to the same conical singularity $Z_{1}$. Since the two conical singularities have an angle of $4\pi$, no other corner can contribute to $Z_{1}$ (if there is an angle equal to $\pi$ among the corners $A_{3}A_{4}$, $A_{4}A_{5}$ or $A_{5}A_{1}$, then the triangle forming the interior of the core is degenerate because the slopes of two of its edges coincide). If the three last corners contribute to the other singularity $Z_{2}$, then the three vertices of the triangle forming the interior of the core should contribute to $Z_{2}$. Thus, the two ends of side $A_{2}$ should contribute to $Z_{2}$. This leads to a contradiction.

\smallskip

\textit{Case 4}. We finally assume \(t=b=3\). So the polar domain is a complex plane minus a triangle. The interior of the core is a pentagon, two sides of which are identified. The distribution of three corners in the polar domain among the two singularities must be ``2+1'', so one singularity has contribution strictly bigger than $3\pi$ from the polar domain (the total angle of the three corners is $5\pi$). On the other hand, in the interior of the core, the distribution of the five corners among the two singularities must be ``3+2''. In particular, the contribution to the total angle at either singularity from the core is at least $\pi$. But then one singularity has angle strictly bigger than $4\pi$, which is impossible.
\end{proof}

\noindent The following lemma relates the notions of volume and core. More precisely, it shows that the volume of a representation provides some constraint on the possible type of the core but fails to provide a complete classification. Indeed, there are translation surfaces with a cylindrical core whose volume can be positive, negative, or zero. Referring to Figure \ref{fig:cyltype} above, the volume will be positive (resp.\ negative, zero) if the area of the cylinder is greater than (resp.\ smaller than, equal to) that of the parallelogram removed from the pole's domain. Therefore, the volume provides only partial information about the type of the core.
%does not provide complete information about the core type but it just says what the core cannot be.

\begin{lem}\label{lem:VolConstraints}
Let \((X,\omega)\) be a translation surface in the stratum $\Omega\mathcal{M}_1(1,1,-2)$ with period character $\chi$. If $(X,\omega)$ has a core of \textit{torus type}, then $\textnormal{vol}(\,\chi\,)>0$. If $(X,\omega)$ has a core of degenerate type, then $\textnormal{vol}(\,\chi\,) \leq 0$. In this latter case, the volume is negative if and only if $(X,\omega)$ is obtained by removing a convex hexagon in an infinite plane and identifying pairs of opposite sides.
\end{lem}

\begin{proof}
Let \((X,\omega)\) be a translation surface and let us assume first that its core is of torus type. According to \textit{Case 1} of Proposition \ref{prop:coreshapes} this is a slit torus. We consider a pair, say $(\alpha,\beta)$, of loops disjoint from the slit that form a symplectic basis of $\text{H}_{1}(X,\,\mathbb{Z})$. By gluing the two slides of the slit we obtain a flat torus \(T\). Notice that the same pair $(\alpha,\beta)$ defines a symplectic basis of \(\textnormal{H}_1(T,\,\mathbb Z)\). Since the torus has positive area, it follows that \(\textnormal{vol}(\,\chi\,)>0\).

\smallskip

\noindent We next assume \((X,\omega)\) has a core of degenerate type, meaning that its interior is empty. Then $(X,\omega)$ is obtained by identification of sides in the boundary of the unique polar domain which is a topological disk. It follows from the computation of Euler characteristic that the polar domain has six edges and then the core is formed by three saddle connections (where both sides of each appear in the boundary of the polar domain).

\smallskip

\begin{figure}[htp]
    \centering
        \includegraphics[width=0.75\linewidth]{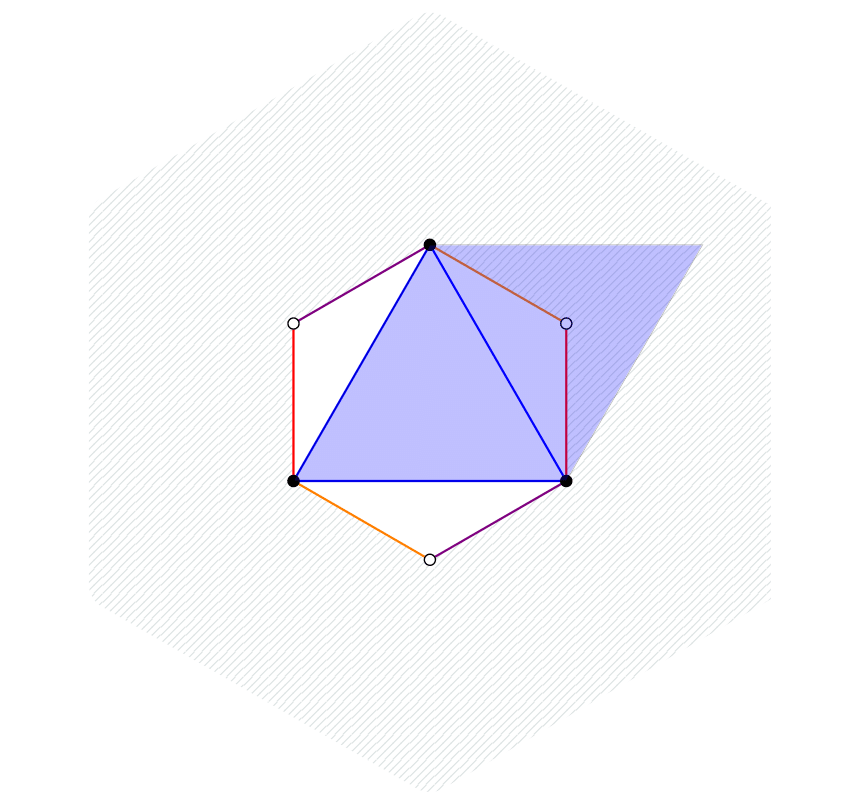}
    \caption{A surface of degenerate type, with a parallelogram of area \(\textnormal{vol}(\,\chi\,)\). It can be easily seen that the area of the blue parallelogram equals the area the hexagon.}
    \label{fig:degenerate}
\end{figure}

\smallskip

\noindent Up to cyclic permutation and relabeling, we can single out four different gluing patterns given by the following words: 
\begin{equation}
    abcabc,\quad aabbcc,\quad aabccb \quad \text{ and }\quad abcacb
\end{equation}

\noindent Gluing pattern corresponding to words $aabbcc$ and $aabccb$ imply the existence of more than two conical singularities. In the case of pattern $abcacb$, four angular sectors of the boundary of the polar domain contribute to one singularity while the two others contribute to the second one. As a consequence of Lemma~\ref{lem:retract} the magnitude of each angular sector in the boundary of the polar domain is at least $\pi$ so the four angles corresponding to the same conical singularity (of total angle $4\pi$ since it is a simple zero) have a magnitude equal to $\pi$. It follows that the distribution of angles is $(\pi,\pi,2\pi,\pi,\pi,2\pi)$. This corresponds to a slit in the infinite plane and some identified sides are then on the same side of the slit. They are identified by a rotation of angle $\pi$ instead of a translation. This is impossible in a translation structure. 

\smallskip

\noindent Therefore, we obtain $(X,\omega)$ by identifying three pairs of opposite sides (gluing pattern corresponding to word $abcabc$). It is obtained from the complement in $\mathbb{C}$ of a topological hexagon $H$ by identification of pairs of opposite sides. Since the total angle of each singularity is $4\pi$ and every angular sector is of magnitude at least $\pi$ (see again Lemma~\ref{lem:retract}), the magnitude of each angular sector belongs to interval $[\pi,2\pi]$. Besides, the total angle of three consecutive angular sectors in the boundary of the polar sector has to be exactly $4\pi$ (by symmetry). It follows that $H$ is a convex hexagon. 

\smallskip

\noindent The three vertices corresponding to the same singularity define a triangle of area $-\frac{1}{2} \textnormal{vol}(\,\chi\,)$, see Figure \ref{fig:degenerate}. Hence it follows that $\textnormal{vol}(\,\chi\,) \neq 0$.
\end{proof}

%\noindent According to Lemma \ref{lem:VolConstraints}, if a Chambers of cylinder type are not discussed in Lemma~\ref{lem:VolConstraints} since we can find such chambers in isoresidual leaves whatever the sign of $Vol_{\chi}$ is (see Section~\ref{sec:chamcyltype}).

\medskip

\section{Isoperiodic foliation of \texorpdfstring{\(\Omega\mathcal M_1(1,1,-2)\)}{M1(1,1,-2)}}\label{sec:kernel_foliation}

\noindent In \S\ref{ssec:overview} we defined isoperiodic fibers as moduli spaces that parametrize translation surfaces with the same signature and period character. In the present section we aim to classify isoperiodic leaves in the stratum \(\Omega\mathcal M_1(1,1,-2)\). Before going through that we wish to provide some more details about isoperiodic fibers in the most general setting for the sake of completeness, see \S\ref{ssec:isofib} and \S\ref{ssec:relper} below. From \S\ref{sub:classification} we shall always assume \(\mu=(1,1,-2)\). 

\subsection{Period map and isoperiodic foliations}\label{ssec:isofib} Besides strata, there is another type of subspaces in $\Omega\mathcal{M}_{g,n}$ which is worth studying, called {\em isoperiodic fibers}. These are defined as the (projections of) non-empty fibers of the so-called \textit{period map}, which associates an abelian differential $\omega$ to its period character, see the map \eqref{permap} below. Unfortunately, for every $n\ge0$, such a map is not well-defined on $\Omega\mathcal M_{g,n}$, see the discussion in \cite[Section \S4.1.1]{Nag} for $n=0$. In order to have a well-defined period map we need to consider a suitable covering space that we denote by $\mathcal S_{g,n}$.

\medskip

\noindent It follows from the classical theory that the orbifold universal cover of $\mathcal M_{g,n}$ is biholomorphic to the Teichm\"uller space $\mathcal T_{g,n}$.
The fiber over each point $X$ consists of all possible identifications of the fundamental group of $X$ with that of a reference topological surface $S_{g,n}$. Under this perspective, the covering group turns out to be isomorphic to the mapping class group $\modul$. In what follows we only need the information of the identification at the homology level, thus we consider the quotient of the Teichm\"uller space $\mathcal T_{g,n}$ by the subgroup $\mathcal I\,(S_{g,n}\,)\subset\modul$ of elements acting trivially on $\shomolzn$. We then define the following space:
\begin{equation}
    \mathcal S_{g,n}=\frac{\mathcal T_{g,n}}{\mathcal I(\,S_{g,n}\,)}.
\end{equation}

\begin{rmk} For $n=0$, the group $\mathcal I\,(S_{g,n}\,)\subset\modul$ is known as the \emph{Torelli group} and the space $\mathcal S_{g,n}$ is thus named as the \emph{Torelli space}. 
\end{rmk}

\noindent Every point in $\mathcal S_{g,n}$ corresponds to an equivalent class of pairs $(X,\, m)$, where $X$ is an unmarked complex structure on $S_{g,n}$ and $m\colon\shomolzn\longrightarrow {\rm H}_1(X,\, \mathbb Z)$ is an identification. Here two pairs $(X,\,m_X)$ and $(Y,\,m_Y)$ are equivalent if there exists a biholomorphic map of marked Riemann surfaces, say $f\colon X \longrightarrow Y$, such that $m_Y=f_*\circ m_X$. The mapping class group $\modul$ acts on $\mathcal S_{g,n}$ by precomposition on the marking. Such an action yields a covering map
\begin{equation}\label{eq:covprojmgn}
    \mathcal S_{g,n}\longrightarrow \mathcal M_{g,n}
\end{equation}

\noindent by construction. Using this projection map, the moduli space of differentials $\Omega\mathcal M_{g,n}\longrightarrow \mathcal M_{g,n}$ pulls back to a bundle $\Omega\mathcal S_{g,n}\longrightarrow \mathcal S_{g,n}$, where $\Omega\mathcal S_{g,n}$ denotes the moduli space of \textit{homologically marked translation surfaces} or \textit{homologically marked differentials}. The period map is then defined as the association
\begin{equation}\label{permap}
    \text{Per}\colon\Omega\mathcal{S}_{g,n}\longrightarrow \text{Hom}\Big(\,\shomolzn,\,\mathbb{C}\,\Big)
\end{equation}
that maps a marked translation surface $(X,\omega,m)$ to its period character.

\medskip

\noindent For holomorphic differentials on compact Riemann surfaces (\textit{i.e.} $n=0$) the image of the period map has been studied by Haupt in \cite{OH} and subsequently rediscovered by Kapovich in \cite{KM2} using Ratner theory. In their work \cite{BJJP}, Bainbridge-Johnson-Judge-Park have provided necessary and sufficient conditions for a representation to be realized in a connected component of a prescribed stratum. Around the same time, in the same spirit of \cite{KM2}, Le Fils has provided in \cite{fils} necessary and sufficient conditions for a representation to arise as the character in a given stratum of holomorphic differentials with an independent and alternative approach. For meromorphic differentials on compact Riemann surfaces, equivalently holomorphic differentials on punctured complex curves (\textit{i.e.} $n\ge1$), the image of the period mapping has been recently determined in \cite{FarGup} (a special case where all zeros and poles are at the punctures), in \cite{CFG} for strata of meromorphic differentials and subsequently extended by \cite{CF} to connected components of strata. We can finally state the main object of study of the present paper.

\begin{defn}\label{def:isofibre}
    An \textit{isoperiodic fiber} is defined as the non-empty preimage of a representation $\chi$ via the period map \eqref{permap}. These fibers clearly project to $\Omega\mathcal M_{g,n}$ and then define a foliation that, with a little abuse of notation, we may still call isoperiodic.
\end{defn}

\noindent Since isoperiodic fibers and strata are both subspaces of $\Omega\mathcal M_{g,n}$, it is natural to inquire about their intersections and this is in fact the main scope of the present paper. The period map \eqref{permap} just defined naturally restricts to a mapping
\begin{equation}\label{eq:respermap}
    \text{Per}(\,\mu\,)\colon\Omega\mathcal{S}_g(\,\mu\,)\longrightarrow \text{Hom}\Big(\,\shomolzn,\,\mathbb{C}\,\Big)
\end{equation}
for every signature $\mu$. A full characterization of the image of $\text{Per}(\,\mu\,)$ can be found in \cite[Theorems B, C, D]{CFG}. In fact, a representation $\chi$ can be realized in a stratum $\Omega\mathcal M_g(\,\mu\,)$ if and only if it can be realized in the marked stratum $\Omega\mathcal{S}_g(\,\mu\,)$, see \cite[Section 1.3]{CF} for more details. As a consequence the isoperiodic foliation on \(\Omega\mathcal M_{g,n}\) restricts to every stratum to an \textit{isoperiodic foliation with fixed signature}. In order to distinguish them from those in \(\Omega\mathcal M_{g,n}\) we introduce the following terminology.

\begin{defn}\label{def:isoleaf}
For a given signature \(\mu\) an \textit{isoperiodic leaf of signature \(\mu\)} %\cm{Yongquan: an isoperiodic leaf of signature $\mu$?}
    is defined as a connected component of the (non-empty) preimage of a representation \(\chi\) via the restricted period map \eqref{eq:respermap}. Each leaf projects to $\Omega\mathcal M_{g,n}$ and then defines a leaf that, with a little abuse of notation, we may still call isoperiodic.
\end{defn}

\noindent Note that the stratification of \(\Omega\mathcal M_{g,n}\) naturally yields a stratification of the isoperiodic fiber in Definition \ref{def:isofibre} into leaves enumerated by partitions of \(2g-2\). In particular, given a representation \(\rho\in\text{Hom}\Big(\,\shomolzn,\,\mathbb{C}\,\Big)\), the space of isoperiodic forms with same signature can be defined as in section \S\ref{ssec:overview}, namely:

\smallskip

\begin{equation}
    \mathfrak {L}(\,\rho,\,\mu\,)=\left\{\,\,(X,\omega)\in\Omega\mathcal M_{g}(\,\mu\,)\,\,\,\Big|\,\,
    \begin{aligned}
        &\,\,\,\rho=f^*\chi, \textnormal{ for some } f\in\textnormal{Homeo}(S_{g,n},\,X)\\
        &\,\,\textnormal{ where } \chi \textnormal{ is the period character of } (X,\omega) 
    \end{aligned}\,\,
    \right\}.
\end{equation}

\smallskip

\noindent We also recall for the reader that in what follows we mainly consider the marked leaf \(\mathfrak{ML}(\,\rho,\,\mu\,)\), which turns out to be easier to navigate. In principle we do not know whether \(\mathfrak{ML}(\,\rho,\,\mu\,)\) is connected or not and hence the term \textit{leaf} may seem inappropriate. In Theorem~\ref{thm:Unique}, the moduli space of differentials that realize a given period character is connected and therefore coincide with an isoperiodic leaf.

\smallskip

\subsection{Period atlas on strata}\label{ssec:relper}
As already alluded to several times above, every pair \((X,\omega)\in\Omega\mathcal{M}_{g,n}\) yields a representation \(\chi\) called the period character. Topologically $X$ can be seen as the connected sum of a closed surface $S_g$ of genus $g$ and a sphere $S_{0,n}$ with $n$ punctures. Thus its homology admits a splitting $\textnormal{H}_1(X,\mathbb{Z})=\textnormal{H}_1(S_g,\mathbb{Z})\,\oplus\,\textnormal{H}_1(S_{0,n},\mathbb{Z})$. Given a symplectic basis $\{\alpha_1,\beta_1,\dots,\alpha_g,\beta_g\}$ of the first part $\textnormal{H}_1(S_g,\,\mathbb Z)$, the integrals of $\omega$ along this collection of curves provide a system of $2g$ local coordinates called \textit{absolute periods}, and the residues of $\omega$ computed around the poles of $\omega$ (which by assumption are punctures of $X$) provide additional \(n-1\) coordinates.

\smallskip

\noindent Generally, absolute periods and residues do not determine a translation surface even locally. In order to have a local determination of a translation surface one has to consider more parameters. These correspond to the so called \textit{relative periods}, which are not encoded by the period character $\chi$. 

\begin{rmk}
    To keep track of relative periods we consider the relative homology group \( \textnormal{H}_1(X, Z(\omega), \mathbb{Z}) \). An independently interesting question is to determine which representations of relative homology arise as relative period characters of some translation structure on the given surface. For holomorphic differentials, this problem was posed by Filip in \cite{Fil} and independently solved by Chen-Faraco \cite{CF2} and Le Fils \cite{fils2}.
\end{rmk}

\noindent For a translation surface $(X,\omega)$, we denote by $Z(\,\omega\,)=\{z_1,\ldots,z_k\}$ the set of zeros of $\omega$. If we choose the last zero $z_k$ as a base-point, the relative periods are then defined as the integrals
\begin{equation}
    \int_{z_i}^{z_{k}} \omega \quad \textnormal{ for } i=1,\dots,k-1,
\end{equation}
along arcs $\gamma_i$ joining $z_i$ and $z_k$. Marking the punctures as \(p_1,\dots,p_n\), we denote by $\delta_i$ a peripheral loop around the puncture \(p_i\). The family of curves $\big\{\,\alpha_{1},\dots,\beta_{g},\gamma_{1},\dots,\gamma_{k-1},\delta_1,\dots,\delta_{n-1}\,\big\}$ provides a basis for the relative homology group $\textnormal{H}_1(X\,,\,Z(\,\omega\,),\,\mathbb Z)$. The period map as defined in \eqref{eq:respermap} extends to a mapping 
\begin{equation}\label{eq:perresmpa}
    \textnormal{Per}(\,\mu\,)\,\times\,\textnormal{Res}(\,\mu\,)\colon\Omega\mathcal S_g(\,\mu\,)\longrightarrow \mathbb{C}^{2g+k+n-2},
\end{equation}
where \(\textnormal{Res}(\,\mu\,)\) is the map that associates to each structure the \(n-\)tuple of its residues at the punctures. Since the image of the mapping  \(\textnormal{Res}(\,\mu\,)\) has been completely determined in \cite{GT}, it readily follows that the image of \(\textnormal{Per}(\,\mu\,)\,\times\textnormal{Res}(\,\mu\,)\) is also completely understood. Explicitly, such a map is defined as
\begin{equation}
    (X,\omega) \longmapsto \left(\,\int_{\alpha_1}\omega\,,\dots,\int_{\beta_g}\omega \,,\int_{\gamma_1}\omega\,,\dots,\int_{\gamma_{k-1}}\omega,\,\int_{\delta_1}\omega\,,\dots,\int_{\delta_{n-1}}\omega\,\right).
\end{equation}
\noindent It can be shown that this map provides the desired system of local coordinates around $(X,\omega)$. On the other hand, a representation \(\rho\) is completely determined by the images of the loops \(\big\{\,\alpha_{1},\dots,\beta_{g},\gamma_{1},\dots,\gamma_{k-1}\,\big\}\). Therefore an isoperiodic leaf \(\mathfrak{L}(\,\rho,\,\mu\,)\) is locally homeomorphic to \(\mathbb C^{k-1}\). This means that two structures in the same leaf only differ by their relative periods and, in fact, one structure may be deformed into another by changing the relative periods and by keeping fixed the absolute periods. Notice that if \(k=2\), as in our case, an isoperiodic leaf is locally homeomorphic to \(\mathbb C\). In what follows we shall show that each marked leaf \(\mathfrak{ML}(\,\rho,\,\mu\,)\) in the marked stratum \(\mathcal H_1(1,1,-2)\) naturally carries a translation structure, c.f.\ \S\ref{ssec:transurfa}.

\medskip

\subsection{Classification of period characters}\label{sub:classification}

We now consider the stratum $\Omega\mathcal{M}_1(1,1,-2)$. We begin with the following:

\begin{lem}\label{lem:tricannotbereal}
Let \((X,\omega)\in\Omega\mathcal M_1(1,1,-2)\) be a translation surface with period character \(\chi\). Then \(\textnormal{Im}(\,\chi\,)\) is either a lattice of \(\mathbb C\) or a nontrivial subgroup of a real line of \(\mathbb C\) up to rotation. 
\end{lem}

\begin{proof}
We first recall that \(\shomolzoo \cong \textnormal{H}_1(\,\textnormal{T},\,\mathbb Z\,)\cong \mathbb Z^2\). Since \(\textnormal{Im}(\,\chi\,)\) is the image of a group isomorphic to \(\mathbb Z^2\) into \(\mathbb{C}\) under a homomorphism, it suffices to show that it cannot be trivial. This follows from \cite[Theorem B]{CFG}.
\end{proof}

\begin{rmk}
A different argument for Lemma \ref{lem:tricannotbereal} goes as follows. Let $(X,\omega)\in \mathcal{H}_1(1,1,-2)$ and let $\chi$ be its period character. If the period of every closed loop is trivial; \textit{i.e.} $\chi$ is the trivial representation, then there is no closed saddle connection and the oriented saddle connections from \(\pto B\) and \(\pto W\) have the same period (recall that \(\pto B\) and \(\pto W\) are the marked zeros of $\omega$). All the saddle connections thus have the same direction, and there is no flat triangle inside the core. Consequently, the surface is just formed by the polar domain, two conical singularities and three saddle connections, see \S\ref{ssec:core}. At each zero, two saddle connections with the same period are separated by an angle which is an integer multiple of $2\pi$. Since zeros are simple the number of saddle connections between them is at most two and this lead to a contradiction.
\end{rmk}

\noindent Based on Lemma~\ref{lem:tricannotbereal} we can thus distinguish four types of isoperiodic leaves as follow. For a given representation \(\rho\) we shall say that a leaf \(\mathfrak{L}(\,\rho,\,\mu\,)\) is
\begin{enumerate}[label=\arabic*.]
    \item \textit{non-arithmetic} if \(\textnormal{Im}(\,\rho\,)\) is dense in a real line of $\mathbb{C}$;
    \smallskip
    \item  \textit{arithmetic} if \(\textnormal{Im}(\,\rho\,)=\mathbb{Z}a\) with $a \in \mathbb{C}^{\ast}$. In this case, we can always find a closed simple loop whose absolute period is zero.
\end{enumerate}

\noindent Notice that in both cases \(\textnormal{vol}(\,\rho\,)=0\). There are other two cases depending on the sign of the volume: a leaf \(\mathfrak{L}(\,\rho,\,\mu\,)\) is

\begin{enumerate}[label=\arabic*.]
    \smallskip
    \item[3.] \textit{positive} if \(\textnormal{Im}(\,\rho\,)\) is a lattice and the image of a symplectic basis of the homology is a direct basis of $\mathbb{C}$ (positive volume);
    \smallskip
    \item[4.] \textit{negative} if \(\textnormal{Im}(\,\rho\,)\) is a lattice and the image of a symplectic basis of the homology is an indirect basis of $\mathbb{C}$ (negative volume).
\end{enumerate}

\noindent By using the natural \(2-\)fold covering \(\mathcal H_1(1,1,-2)\longrightarrow \Omega\mathcal{M}_1(1,1,-2)\) it is natural to extend such a terminology as follows. 

\begin{defn}
    A marked leaf \(\mathfrak{ML}(\,\rho,\mu\,)\) is said to be \textit{positive, negative, arithmetic or non-arithmetic} if the corresponding leaf \(\mathfrak{L}(\,\rho,\,\mu\,)\) in the unmarked stratum is. 
\end{defn}

%\smallskip

\subsection{Isoperiodic leaves in \(g=1\)} One of the advantages of working with the signature \(\mu = (1,1,-2)\) is that all Abelian differentials in the stratum \(\mathcal{H}_1(\,\mu\,)\) have a double pole with zero residue. Hence the map in \eqref{eq:perresmpa} reduces to \(\textnormal{Per}(\,\mu\,)\colon\Omega\mathcal S_1(\,\mu\,)\longrightarrow \mathbb{C}^{2}\cong\textnormal{Hom}\Big(\,\shomolzoo,\,\mathbb{C}\,\Big)\), \textit{i.e.} a representation is uniquely determined by its absolute periods. Action of the symplectic group on the homology induces an action on representations as well.
%We now consider the action of the symplectic group both on the domain and target of \(\textnormal{Per}(\,\mu\,)\).
Recall that for \(g = 1\) the symplectic group coincides with the well-known \(\slz\). Two representations in the same orbit have the same absolute periods and therefore the same images. Although the converse is not generally true for \(g \ge 2\), in our case we can state the following fact.

\begin{prop}\label{prop:iso}
    Two representations with equal non-zero volume \(\rho_1\) and \(\rho_2\) are in the same \(\slz\)-orbit if and only if their images coincides. As a consequence, the space of isoperiodic forms \(\mathfrak L(\,\rho,\mu\,)\) is uniquely determined by the lattice \(\Gamma=\textnormal{Im}(\,\rho\,)\), \textit{i.e.} \(\mathfrak L(\,\rho_1,\mu\,)=\mathfrak L(\,\rho_2,\mu\,)\) if and only if \(\,\textnormal{Im}(\,\rho_1\,)=\textnormal{Im}(\,\rho_2\,)\).
\end{prop}

\begin{rmk}
    For \(g\ge2\), it is possible to find translation surfaces with the same absolute periods and different volumes. As a consequence, these structures do not belong to the same isoperiodic leaf because they have different period characters. Thus the group of periods is generally not sufficient to determine an isoperiodic leaf.
\end{rmk}

\begin{proof}
    One direction is clear: If two representations are in the same \(\slz\)-orbit then their images must coincide and have the same volume. This is in fact true in a more general setting. Let us assume two representations have the same group \(\Gamma\) of absolute periods and assume they both have non-zero volume. According to Lemma \ref{lem:tricannotbereal}, \(\Gamma\) is a lattice and, up to \(\glplus\), we may assume \(\Gamma=\ziz\). Thus we only need to show that if two representations, say \(\rho_1\) and \(\rho_2\), are such that \(\textnormal{Im}(\,\rho_1\,)=\textnormal{Im}(\,\rho_2\,)=\ziz\)  and \(\textnormal{vol}(\,\rho_1\,)=\textnormal{vol}(\,\rho_2\,)\), then they are in the same \(\slz\) orbit. Let \(\{\,\alpha,\,\beta\,\}\) be any basis in homology and suppose
    \begin{equation}
        \rho_1(\,\alpha\,)=a_1+ib_1, \quad \rho_1(\,\beta\,) =c_1+id_1,  \quad \rho_2(\,\alpha\,)=a_2+ib_2, \quad \rho_2(\,\beta\,) =c_2+id_2,
    \end{equation}
    where \(a_i,b_i,c_i,d_i\in\mathbb Z\). Since both matrices
    \begin{equation}
        \begin{pmatrix}
            a_1 & c_1\\
            b_1 & d_1
        \end{pmatrix}, \quad 
        \begin{pmatrix}
            a_2 & c_2\\
            b_2 & d_2
        \end{pmatrix}
    \end{equation}
    are invertible over the integers and then they are related by a matrix in \(\slz\). The result follows.
\end{proof}

\subsection{Isoperiodic leaves are translation surfaces}\label{ssec:isoleaftrans} We focus on isoperiodic leaves \(\mathfrak{ML}(\,\rho,\mu\,)\) in the stratum with marked zeros $\mathcal{H}_1(1,1,-2)$. Recall that zeros are labeled as \(\pto B\) and \(\pto W\). As a consequence of our discussion in Section \S\ref{ssec:relper}, two structures in the same leaf are distinguished by their relative period. The crucial observation here is that there is \textit{no} canonical choice of a saddle connection joining \(\pto B\) and \(\pto W\), or even a consistent choice throughout the leaf. Different saddle connections yield different relative periods and thus provide different local coordinates for the same structure. What we show below is that these coordinates always differ by some constant \(c\in\mathbb C\) and hence they define a translation structure on each isoperiodic leaf.

%This lack is what permits us to claim the existence of a translation structure on a leaf.%\cm{Yongquan: the lack of canonical relative period means we don't have a global period coordinate chart, but does not mean it has no translation structure, at least to me?}\cm{Guillaume: I agree. We should say it differently. Even if there were a canonical relative period, we would have a translation structure corresponding to an exact 1-form.} \textcolor{red}{You're both right, I think I was misunderstood. What I meant it was that by changing the relative period the local coordinate just change with a translation. I change it as I do not want misleading sentences.}

\begin{rmk}
     Any path on \((X,\omega)\) having \(\pto B\) and \(\pto W\) as end-points has at least one geodesic representative in its homotopy class. Since two paths in the same \textit{relative} homotopy class have the same period we can always compute the relative period along the geodesic representative of the given path.
     %\textcolor{red}{Gianluca: Any such class of paths from one conical singularity to the other has a representative which is a chain of saddle connections, \textit{i.e.} geodesic segments, one of which is between the black zero and the white one. Therefore, there is at least one of these relative homotopy class represented by a saddle connection. ... what do we mean?}
\end{rmk}

\smallskip

\noindent Let \((X,\omega)\) be a structure in \(\mathcal H_1(1,1,-2)\) and let \(\gamma_1\) be any path joining \(\pto B\) and \(\pto W\). Its relative period, say \(r_1\)%$r(\,\gamma_1\,)$
, provides a local complex coordinate in the isoperiodic leaf \(\mathfrak{ML}(\,\chi,\,\mu\,)\). Here we use \(\chi\) to signify that we are considering the leaf containing \((X,\omega)\). If we choose any other path, say \(\gamma_2\) with relative period %\(r(\gamma_2)\)
\(r_2\), then their difference %$r(\gamma_1)-r(\gamma_2)$ 
\(r_1-r_2\) is the absolute period of the closed loop \(\alpha\) homotopic to $\gamma_1 \cup \gamma_2^{-1}$. On the other hand the absolute period of \(\alpha\) is fixed all along the isoperiodic leaf. As a consequence, the choice of a path $\gamma$ results in a coordinates system defined up to a constant, up to a translation in $\mathbb{C}$. These coordinates thus define a \textit{translation atlas} in the isoperiodic leaf.

\smallskip

\noindent Let \(\rho\colon\shomolzoo\longrightarrow \mathbb C\) be a representation and assume \(\textnormal{Im}(\,\rho\,)\) is a lattice. Then we can show that
%\subsection{Isoperiodic leaves: cover structure}
%In an isoperiodic leaf where the absolute period group $\Gamma$ is a lattice, 
the period coordinates of the leaf are defined up to a period in \(\textnormal{Im}(\,\rho\,)\).
%$\Gamma$. 
Consequently, we have
\begin{prop}
Suppose \(\textnormal{Im}(\,\rho\,)\) is a lattice in $\mathbb{C}$. Then a translation structure on the leaf \(\mathfrak{ML}(\,\rho,\,\mu\,)\) is induced by a covering of torus $\mathbb{C}/\textnormal{Im}(\,\rho\,)$ ramified over the lattice point.
\end{prop}
\noindent Moreover, the conical singularities on the leaf correspond to translation surfaces obtained by degeneration, \textit{i.e.} shrinking a saddle connection between the two zeros. See later sections for details.

\smallskip

\noindent Because of the action of $\textnormal{GL}^{+}(2,\mathbb{R})$, every generic leaf is conjugated to one where $\Gamma$ is normalized to $\mathbb{Z}[\,i\,]$. In particular, the ramified covering over the square torus endows such isoperiodic leaves with a structure of square-tiled translation surface.

\smallskip

\subsection{Walls-and-chambers structure of isoperiodic leaves}\label{sub:WallsAndChambers} We now turn into the description of leaves as alluded to in Section \S\ref{ssec:isostrucandconfgeo}. We have already seen above that every translation surface \((X,\omega)\in\Omega\mathcal M_1(1,1,-2)\) is essentially determined by its core and we already proved that there are three different kinds of cores. The purpose of the present section is to study the behavior of the core under small deformations. We introduce the following:

\begin{defn}\label{def:transitional}
    A translation surface \((X,\omega)\) is said to be \textit{transitional} if the shape of its core changes under small deformations.
\end{defn}

\noindent The following claim provides a characterization of transitional translation surface. 

\begin{prop}\label{prop:transitionalcharac}
    A translation surface is transitional if there exist two consecutive saddle connections on the boundary of its domain of pole forming an angle of $\pi$. Moreover, the subset of transitional translation surfaces is a possibly disconnected closed subset of co-dimension one in each marked leaf \(\mathfrak{ML}(\,\rho,\,\mu\,)\).
\end{prop}

\noindent We refer to \cite[\S4]{Ta} for the proof and further details. Proposition \ref{prop:transitionalcharac} enables us to decompose every leaf as follows.

\begin{defn}[Wall-chamber decomposition]
    Let \(\mathfrak{ML}(\,\rho,\,\mu\,)\) be a leaf in the marked stratum. A \textit{wall} is defined as a maximal connected subset of transitional translation surfaces. A \textit{chamber} is defined as a connected component of the complement of the union of walls.
\end{defn}

\noindent Notice that each chamber is an open subset and every pair of structures in the same chamber have by definition the same kind of core. The possible shapes have been classified in Proposition~\ref{prop:coreshapes}. Therefore, we decompose each leaf into:

\begin{itemize}
    \item \textit{chambers of torus type};
    \item \textit{chambers of cylinder type};
    \item \textit{chambers of degenerate type}.
\end{itemize}

\noindent We recall from Section \S\ref{ssec:isoleaftrans} that each leaf in \(\mathcal H_1(1,1,-2)\) has a natural translation structure. Such a structure naturally restricts to a translation structure on each chamber. The next step is to show that each chamber is isometric to an open subset of \(\mathbb C\). Moreover, we aim to show that on every leaf \(\mathfrak{ML}(\,\rho,\,\mu\,)\) the conical singularities of its translation structures correspond to translation surfaces in the lower stratum \(\Omega\mathcal{M}_1(2,-2)\).

\smallskip

\subsection{The stratum $\Omega\mathcal{M}_1(2,-2)$}
We recall the walls-and-chambers structure of the stratum $\Omega\mathcal{M}_1(2,-2)$. This was extensively discussed in \cite{Ta1}. For this, note that there is a natural $\mathbb{C}^{\ast}$-action on each stratum by scaling the meromorphic differential. A \textit{projectivized} stratum is the quotient of a stratum by this action. The projectivized stratum $\mathbb{P}\Omega\mathcal{M}_1(2,-2)$ is biholomorphic to the modular curve $X_{1}(2)$. The modular curve $X_{1}(2)$ is a complex curve of genus zero with two cusps and one orbifold point of order $2$. The (projectivized) wall is an arc connecting one cusp to itself, dividing $X_{1}(2)$ into two chambers:
\begin{itemize}
    \item One chamber contains the other cusp of $X_{1}(2)$. This chamber consists of translation surfaces with cylindrical cores, so we call it a chamber of \textit{cylinder type}. Translation surfaces in this chamber can be constructed by gluing a cylinder to a copy of the complex plane along a slit. The cusp contained in this chamber corresponds to an ``infinitely long'' cylinder.
    \item The other chamber contains the orbifold point. This chamber consists of translation surfaces whose core has empty interior, so we call it a chamber of \textit{degenerate type}. Translation surfaces in this chamber can be constructed by removing a parallelogram from a copy of the complex plane, and then identifying opposite sides. The orbifold point corresponds to the case where the parallelogram is a square.
\end{itemize}

\noindent Translation surfaces in the wall can be constructed similarly as degenerate types, by removing a ``degenerate parallelogram'', whose sides are all parallel. 

\subsection{Isoperiodic leaves: conical singularities}\label{ssec:singkerleaves} As already discussed, each isoperiodic leaf \(\mathfrak{ML}(\,\rho,\,\mu\,)\) admits a natural translation atlas that turns the leaf into a translation surface. As such, an isoperiodic leaf may admit conical singularities. We aim to show that these singularities correspond to translation surfaces obtained as limits of degenerating sequences.

\smallskip

\noindent Whenever the length of some saddle connection tends to zero, a translation surface degenerates. Let $(X,\omega)\in \mathcal{H}_1(1,1,-2)$ be a translation surface with period character $\chi$ and let \(\Gamma=\textnormal{Im}(\,\chi\,)\). Any deformation inside an isoperiodic leaf \(\mathfrak{ML}(\,\rho,\,\mu\,)\) keeps the lengths of closed saddle connections fixed. Therefore, the only saddle connections that can shrink are those between the two conical singularities \(\pto B\) and \(\pto W\). 

\smallskip

\noindent If two saddle connections between \(\pto B\) and \(\pto W\) shrink simultaneously, then they must have the same relative period all along the deformation because their difference is a closed curve whose absolute period remains constant under the deformation. Notice that their union is a simple loop, say $\gamma$, of period zero. Consequently, the angles at the two singularities between them is an integer multiple of $2\pi$. Since conical singularities of translation surfaces of $\mathcal{H}_1(1,1,-2)$ have an angle of $4\pi$, at most two saddle connections can shrink simultaneously. In the case $\Gamma$ is not a cyclic group, then this happens only if the homology class of $\gamma$ is trivial. In this case, the degeneration disconnects the surface into flat plane -- the polar domain -- and a flat torus. Such a degeneracy occurs only in the chamber of torus type, see \S\ref{sec:posleaves} for a precise definition. The local model of this degeneration is just a punctured disk. We compactify it by adding a marked point.

\smallskip

\noindent For an arithmetic real leaf the picture is quite different from the one above, as now there are homologically nontrivial loops with period zero. Note that there exists a simple closed loop $\gamma$ so that $\Gamma=\chi(\gamma)\mathbb{Z}$. If two saddle connections shrink simultaneously in such isoperiodic leaves, the torus gets pinched and degenerates to a translation surface of genus zero with a pole of order two and two marked points connected by a saddle connection of period $\chi(\gamma)$. The local model of this degeneration is also a punctured disk (because of the symmetries). Once again, we compactify it by adding a marked point.

\smallskip

\noindent Otherwise, only one saddle connection shrinks. In this case, the limiting translation surface remains a torus. Algebraically, if the two zeros of a differential in the stratum $\mathcal{H}_1(1,1,-2)$ collide, then we get a differential with only one zero in the minimal stratum $\Omega\mathcal{M}_1(2,-2)$. Conversely, on a translation surface with a single zero of order two, we can perform a surgery called \textit{breaking up}, which splits the zero into two simple zeros connected by an arbitrarily small saddle connection, see \cite{Bo,EMZ} for details. This surgery consists of replacing a small disk around the zero of order two, see Figures \ref{fig:breakzer} and \ref{fig:splitlocmodel}.

\begin{figure}[htp]
    \includegraphics[width=0.875\linewidth]{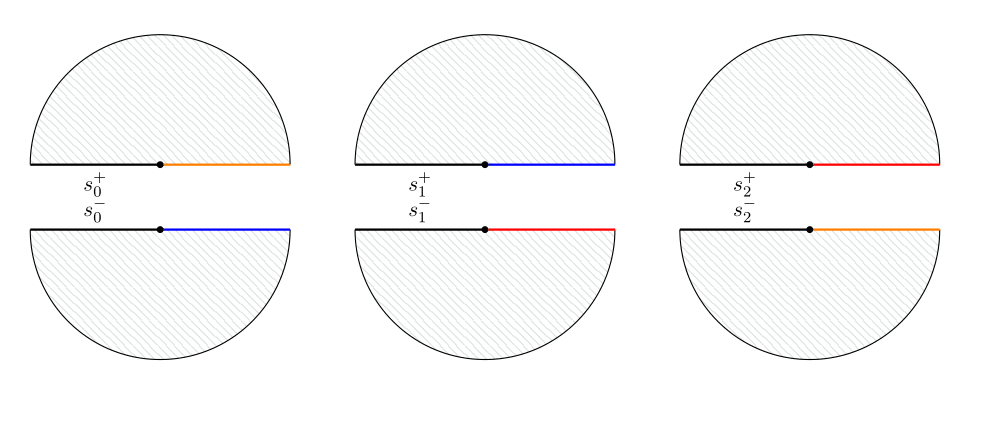}
    \caption{An $\varepsilon$-neighborhood of a zero of order $2$. In this picture $s_i^+$ is identified with $s_i^-$. In the same fashion sides with the same color are identified.}
    %\caption{Breaking up a zero of order two into two zeros of order one (figure courtesy of Corentin Boissy.}
\label{fig:breakzer}
\end{figure}

\begin{figure}[ht]
    \centering
    \includegraphics[width=0.875\linewidth]{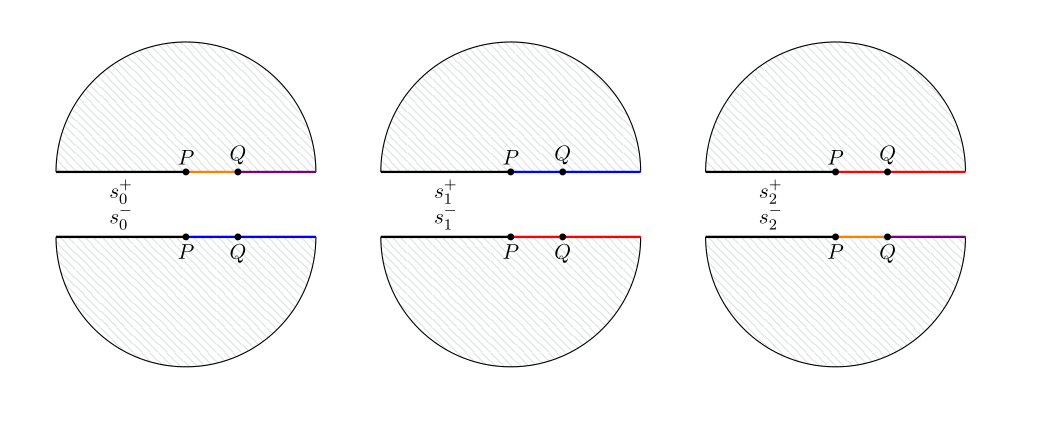}
    \caption{New labelling for breaking up a zero of order $2$ in two zeros of order $1$. Even in this picture $s_i^+$ is identified with $s_i^-$ and sides with the same color are also identified.}
    \label{fig:splitlocmodel}
\end{figure}

\smallskip

\noindent By using the period of the saddle connection as a continuous parameter of an analytic deformation, we can see that every translation surface of $\mathcal{H}_1(1,1,-2)$ that can be obtained as the deformation of some translation surface in the stratum $\Omega\mathcal{M}_1(2,-2)$. The surgery is local and does not affect the absolute periods. What we get is a cover of order three of the pointed disk as a local model for the degeneration. We compactify by adding a conical singularity of angle $6\pi$. The position of the conical singularities of an isoperiodic leaf in $\mathbb{P}\Omega\mathcal{M}_1(2,-2)$ depend on its type.

\begin{prop}
In a real leaf, conical singularities belong to the wall of $\Omega\mathcal{M}_1(2,-2)$. In positive leaves, conical singularities belong to the chamber of cylinder type. Finally, in negative leaves, conical singularities belong to the chamber of degenerate type.
\end{prop}

\begin{proof}
The absolute period group $\Gamma$ of a conical singularity of an isoperiodic leaf is the same as that of the other differentials of the leaf (breaking up zeros is a local surgery). For translation surfaces in a stratum with a single zero, saddle connections are necessarily closed and thus their periods generate $\Gamma$.

\smallskip

\noindent In $\Omega\mathcal{M}_1(2,-2)$, elements of $\Gamma$ are real collinear if and only if the translation surface has a degenerate core formed by a pair of saddle connections with the same direction. This locus coincides with the wall of $\Omega\mathcal{M}_1(2,-2)$.

\smallskip

\noindent For an element in the chamber of cylinder type of $\Omega\mathcal{M}_1(2,-2)$, a symplectic basis of the homology is a pair of oriented loops, say $\alpha$ and $\beta$ with an intersection number of $+1$. We may choose $\alpha$ as a closed geodesic of the cylinder and $\beta$ as a loop crossing the cylinder (and turning around the slit). The image of their periods in $\mathbb{C}$ is clearly a direct basis. Therefore, such conical singularities can only appear in positive leaves. Conversely, the neighborhood of any such differential in $\mathcal{H}_1(1,1,-2)$ (obtained by breaking up) lies on a positive leaf.

\smallskip

\noindent Finally, for an element in the chamber of degenerate type of $\Omega\mathcal{M}_1(2,-2)$, the translation surface obtained from a flat plane with a parallelogram removed by gluing opposite sides (it contains exactly two saddle connections). Let $\alpha$ and $\beta$ be two simple closed loops, each of which crosses one saddle connection and then turns around the parallelogram removed. The period of each loop equals that of the closed saddle connection it does not cross. We check easily that if the two loops intersect positively, their periods in $\mathbb{C}$ are negatively oriented.
\end{proof}

\subsection{Isoperiodic leaves are hyperbolic Riemann surfaces}

For the purpose of proving this fact, we exhibit non-constant holomorphic maps from any isoperiodic leaf to the moduli space $\mathcal{M}_{1,1}$ of elliptic curves. Recall that $(\,-\,,-2\,)$ is used to denote the union of strata of signature $(1,1,-2)$ and $(2,-2)$.

\begin{lem}
For any isoperiodic leaf $\mathfrak L$ in either the marked or unmarked stratum $\Omega\mathcal M_1(\,-\,,-2\,)$,
%= \mathcal{H}(1,1,-2) \cup \mathcal{H}(2,-2)$, 
there is a non-constant holomorphic map $\phi_{\mathfrak L}$ from $\mathfrak L$ to the moduli space $\mathcal{M}_{1,1}$ of elliptic curves.
\end{lem}

\begin{proof}
For each translation surface of $\Omega\mathcal M_1(\,-\,,-2\,)$, we forget the location of the zeros of the differential and mark the double pole in the underlying elliptic curve. In this way, we can define a holomorphic map from $\Omega\mathcal M_1(\,-\,,-2\,)$ to $\mathcal{M}_{1,1}$. Since the period map defines holomorphic charts and the isoperiodic foliation is a holomorphic foliation, for each isoperiodic leaf $\mathfrak{L}=\mathfrak{L}(\,\rho,\mu\,)$, we obtain a holomorphic map $\phi_{\mathfrak L}$ from $\mathfrak L$ to $\mathcal{M}_{1,1}$. If $\phi_{\mathfrak L}$ were constant to some complex structure $\mathbb{C}/\Lambda$, then $\mathfrak L$ would have been contained in the two-dimensional vector space $V$ of meromorphic $1$-forms on $\mathbb{C}/\Lambda$ with at most one double pole at the marked point, generated by $dz$ and $\wp dz$. Since deformations inside $L$ are isoperiodic, then $\mathfrak L$ is an affine complex line of $V$ modeled on the subspace of exact $1$-forms in $V$. Since there is no nontrivial exact $1$-form in $V$, we obtain a contradiction and $\phi_{\mathfrak L}$ is a non-constant holomorphic map.
\end{proof}

\begin{cor}\label{cor:hyp}
Any isoperiodic leaf $\mathfrak{L}(\,\rho,\mu\,)$, either in the marked or unmarked stratum $\Omega\mathcal M_1(\,-\,,-2\,)$, admits a non-constant bounded holomorphic function.
\end{cor}

\begin{proof}
Since $\mathcal{M}_{1,1}$ is biholomorphic to a quotient of the complex disk $\mathbb{D}$, $\phi_{\mathfrak L}$ lifts to a non-constant holomorphic map to $\mathbb{D}$.
\end{proof}

\medskip

\section{Chambers of cylinder   type}\label{sec:chamcyltype}

\noindent In this section we give a detailed description of chambers of cylinder type appearing in every isoperiodic leaf of $\mathcal{H}_1(1,1,-2)$. According to the proof of Proposition~\ref{prop:coreshapes}, a translation surface $(X,\omega)$ of cylinder type has a polar domain with four boundary saddle connections, two of which being glued on the sides of a cylinder (made of two flat triangles), see Figure \ref{fig:cyltype}. The only possible gluing pattern implies that the remaining two sides are glued to each other by a translation. As a consequence the polar domain is the complement of a parallelogram, say \(Q\).

\smallskip

\subsection{Description of the chambers}\label{ssec:chambdes} For readers' convenience, we go through the construction in greater detail, allowing us to provide a description of chambers of cylinder type. Let \(\rho\colon\shomolzoo\longrightarrow \mathbb C\) be a non-trivial representation and let \(\{\,\alpha,\,\beta\,\}\) be two generators that represent a pair of simple closed curves that intersect once (up to homotopy). We shall assume \(\{\,\alpha,\,\beta\,\}\) to be a positively oriented basis, \textit{i.e.} their intersection number is equal to one. Throughout this section we set \(u=\rho(\,\alpha\,)\) and \(v=\rho(\,\beta\,)\) respectively and also \(\Gamma=\textnormal{Im}(\,\rho\,)\). Notice that the complex numbers \(u,v\) cannot both be zero because the representation is assumed to be non-trivial. Moreover, up to replacing \(\{\,\alpha,\,\beta\,\}\) with another suitable pair of generators, we may assume \(u,v\) are both non-zero.

\smallskip

\noindent Notice that a finite cylinder, as a translation surface,  can be constructed by identifying \textit{only one} pair of opposite sides of some parallelogram, say \(P\). In our constructions below it shall be convenient to regard complex numbers as vectors. Keeping this in mind, let \(P\) be \textit{any} parallelogram whose sides are given by \(u\) and another non-zero vector, say \(w\), such that \(\mathfrak{Im}(\,\overline{u}w\,)>0\). We introduce the following terminology.

\begin{defn}
     Let \(C\) be a (finite) cylinder constructed from a parallelogram \(P\) whose sides are given by a pair of vectors \(u,w\in\mathbb C\) such that \(\mathfrak{Im}(\,\overline{u}w\,)>0\). We shall say that \(C\) has period \(u\) (respectively \(w\)), if it is constructed by identifying those sides parallel to \(w\) (respectively \(u\)).
\end{defn}
%are not collinear \cm{Gian: and P has positive oriented area?}. 

%\begin{rmk}\label{rmk:smallpert}
%    Note that any small change in \(w\) still gives a structure of cylinder type. It then follows that these structures form an open subset of the leaf.
%\end{rmk}

\begin{rmk}[Coloring]\label{rmk:coloring}
      In previously shown figures, we have tacitly used a specific coloring system for oriented sides of polygons (see Figures \ref{fig:emptytype}, \ref{fig:torustype} and \ref{fig:cyltype}). We now describe this system more formally due to its importance in what follows. In this work, the sides of the polygons will be colored according to the following criterion:
      \begin{itemize}
          \item two sides are colored the same only if they lie on opposite sides of the polygon(s) to which they belong (not necessarily the same one) and differ by a translation.
          \smallskip
          \item Furthermore, at most one pair of sides share the same color.
      \end{itemize}
      Through this coloring, we can identify sides with the same color and obtain the desired surfaces.
\end{rmk}

\begin{figure}[htp]
    \centering
    \includegraphics[width=0.875\linewidth]{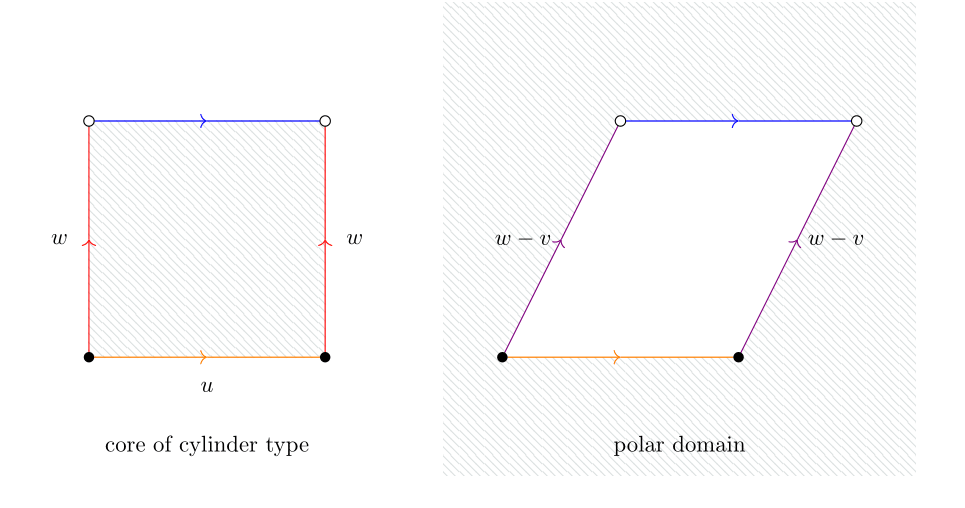}
    \caption{General construction of a surface of cylinder type whose sides are paired and then glued if the have the same color. The two singularities are distinguished by black or white. For an oriented side signified by an arrow, the letter that appears nearby denotes the complex number representing that vector. See also Figure \ref{fig:usepa} for the resulting surface.}
    \label{fig:cylinder_type}
\end{figure}

\noindent We remove from \(\mathbb C\) the interior of a parallelogram \(Q\) whose sides are given by \(u, w-v\). Notice that here we do not have any flexibility here; \(Q\) is completely determined by \(u,v,w\). We then glue \(P\) and \(\mathbb C\setminus \textnormal{int}(\,Q\,)\) by translations as shown in Figure~\ref{fig:cylinder_type}. The resulting structure is a translation surface $(X,\omega)$ in \(\Omega\mathcal M_1(1,1,-2)\). We show it has the period character \(\chi=\rho\). Once the sides of \(P\) labeled with the same color (those colored in red in Figure \ref{fig:cylinder_type}) are identified, we get a cylinder with period \(u\) by construction. This results in a simple closed curve with period \(u\) in $X$. We denote by \(\gamma_1\) the path coming from the identified sides of \(P\). By construction, it has relative period \(w\). Similarly, we denote by \(\gamma_2\) the path coming from the identified sides of \(Q\) with same color. By construction, it has relative period \(w-v\). The closed loop \(\gamma_1\,\cup\,\gamma_2^{-1}\) is thus a simple closed curve with period \(v\). This provides the second desired generator for \(\shomolzoo\). 

\smallskip

\noindent Finally it remains to mark the resulting zeros to turn the structure just realized into a structure in the marked stratum \(\mathcal H_1(1,1,-2)\). In order to distinguish these zeros we first observe that, by construction, there is an oriented saddle connection with relative period \(w\). We then label the starting zero as \(\pto B\) and the other one as \(\pto W\) (black and white). We shall consistently rely on this convention.

\begin{rmk}
    In fact, we will orient every saddle connection joining the two zeros from \(\pto B\) to \(\pto W\). In particular, changing the marking of the zeros reverses the orientation of all such saddle connections. 
\end{rmk}

\smallskip

\noindent It turns out that the choice of \(w\) determines the final structure, and it follows from our construction the existence of some saddle connection joining the two singularities with relative period \(w\). The following proposition gives a classification of chambers of cylinder type, where within each chamber the relative period $w$ serves as a local coordinate.

\begin{prop}\label{prop:cylclass}
Let \(\rho\) be a non-trivial representation and let \(\mathfrak{ML}(\,\rho,\,\mu\,)\) be the isoperiodic fiber in the marked stratum \(\mathcal H_1(\,\mu\,)\), where \(\mu=(1,1,-2)\). Let \(\Gamma=\textnormal{Im}(\,\rho\,)\) and let \(u\in\Gamma\setminus\{\,0\,\}\) be the absolute period of some simple closed curve. The following holds.
%In stratum $\mathcal{H}(1,1,-2)$, we consider the locus formed by translation surfaces realizing a nonzero period character $\chi$ (up to a symplectic change of basis) and a nontrivial absolute period $u \in \Gamma \setminus \lbrace{ 0 \rbrace}$ such that $u$ is the period of a primitive homology class. Then
\begin{itemize}
    \item[1.] If \(\textnormal{vol}(\,\rho\,)\neq0\), or \(\textnormal{vol}(\,\rho\,)=0\) and \(\Gamma\) is dense in a real line in \(\mathbb C\), then there is a unique chamber, denoted by \(\textnormal{C}(\,u\,)\), that parametrizes all translation surfaces whose core is a cylinder of period \(u\). 
    %formed by translation surfaces of cylinder type realizing character $\rho$ and such that the appropriately oriented core geodesic of the cylinder realizes a homology class $\gamma$ of period $\rho(\gamma)=u$;
    \smallskip
    \item[2.] If \(\Gamma=a\,\mathbb{Z}\) for some \(a\neq0\) and hence \(u=ka\)%for some nonzero integer $k$
    , then for every integer \(l\in\mathbb Z\) such that \(0\le l<|\,k\,|\) and $\gcd(\,l,k\,)=1$, there is a unique chamber \(\textnormal{C}(\,k,l\,)\) that parametrizes all translation surfaces whose core is a cylinder of period \(u\).
    %formed by translation surfaces of cylinder type realizing character $\chi$ and such that the appropriately oriented core geodesic of the cylinder realizes a homology class $\gamma$ of period $\chi(\gamma)=u$, and $v=la$ is the period of another homology class $\eta$, where $(\gamma,\eta)$ forms a symplectic basis of the homology group.
\end{itemize}
Moreover, in the translation structure of the isoperiodic leaf containing $\textnormal{C}(\,u\,)$ or $\textnormal{C}(\,k,l\,)$, the chamber is a half-plane defined by inequality \(\mathfrak{Im}(\,\overline{u}\,w\,)>\max(\,0,\textnormal{vol}(\,\rho\,)\,)\).
\end{prop}

\begin{proof}
Let \(u=\rho(\,\alpha\,)\) and \(v=\rho(\,\beta\,)\).
As a consequence of our discussion above, a translation surface of cylinder type is completely determined by the absolute periods \(u,v\) and the relative period \(w\).

\smallskip

\noindent If \(\textnormal{vol}(\,\rho\,)\neq0\), or \(\textnormal{vol}(\,\rho\,)=0\) and \(\Gamma\) is dense in a real line in \(\mathbb C\), then \(\Gamma\) is a rank-2 $\mathbb{Z}$-module. In this case there is exactly one chamber \(\textnormal{C}(\,u\,)\). In fact, if \(v_1,\,v_2\in\Gamma\) are two periods representing simple closed curves such that \(\mathfrak{Im}(\,\overline u v_1\,)=\mathfrak{Im}(\,\overline u v_2\,)\) then \(v_1-v_2\in u\,\mathbb Z\). In other words the volume \(\textnormal{vol}(\,\rho\,)\) determines \(v\) up to an element of \(u\,\mathbb Z\). Note that modifying $v$ up to an element in $u\,\mathbb{Z}$ amounts to a change of basis, so the resulting chamber remains the same. Moreover, \(w\) provides a global parameter in this chamber in the following sense: It is possible to find some open set in \(\mathbb C\) such that for each value of \(w\) in such a set yields a unique translation surface with core of cylinder type and desired absolute and relative periods. We shall determine this maximal open set later together with the arithmetic case.

%We first consider the case of arithmetic groups that turns out to be more subtle.

%Moreover, $v$ is defined up to elements of $u\,\mathbb{Z}$ and is required to satisfy the equality $\mathfrak{Im}(\bar{u}v)=Vol_{\chi}$, because $(u,v)$ is the image of a symplectic basis. This condition fixes $v$ up to $u\,\mathbb{Z}$ when $\Gamma$ is not a cyclic group.

\smallskip

%\noindent On the other hand, when $\Gamma=a\,\mathbb{Z}$ and $u=ak$ for some integer $k$, for each integer $0\le l<|k|$, the chamber $CC_{u,v}$ with $v=al$, up to $u\,\mathbb{Z}$ is distinct from each other. Indeed, for a fixed parameter $z$, the height of the cylinder $z-v$ is distinct, even up to an element $u\mathbb{Z}$.

\noindent We assume next \(\Gamma\) is arithmetic, \textit{i.e.} we suppose \(\Gamma=a\,\mathbb{Z}\) where \(a\in\mathbb C^*\). Thus \(u=ak\) for an integer \(k\). Since \(\langle u,\,v\rangle=a\,\mathbb Z\), we have \(v=al\) where \(l\) is an integer such that \(\gcd(\,k,\,l\,)=1\). We may assume without loss of generality that \(0 \le l <|\,k\,|\) as a direct consequence of the Euclid's division Lemma. Notice that this is allowed because it corresponds to a suitable change of generators for \(\shomolzoo\). For every integer \(l\) that satisfies these properties there exists a unique chamber denoted by \(\textnormal{C}(\,k,\,l\,)\). Once the value of \(w\) is fixed, then there are \(\varphi(\,|\,k\,|\,)\) pairwise non-isometric parallelograms \(Q\) with sides \(u,w-v\), where \(\varphi\) denotes the Euler's totient function. More precisely, each value of \(l\) yields a unique parallelogram and two distinct values determine different parallelograms. In fact, it can be seen that these are pairwise non-isometric because they have diagonals of different lengths; where with diagonal here we just mean any segment that wraps on the cylinder with period \(mu+w\) for some \(m\in\mathbb Z\).

\smallskip

\noindent In our construction the only constraint we have is that the areas of parallelogram \(P\) and \(Q\) must not vanish. These conditions amount to the following two inequalities:
\begin{equation}
    \mathfrak{Im}\big(\,\overline{u}\,w\,\big)>0 \qquad \text{ and } \qquad \mathfrak{Im}(\,\overline{u}\,(w-v) \,)>0.
\end{equation}

\noindent Notice that \(\mathfrak{Im}(\,\overline{u}\,(w-v) \,)>0\) is equivalent to \(\mathfrak{Im}(\,\overline{u}\,w \,)>\textnormal{vol}(\,\rho\,)\). These conditions combined together yields to \(\mathfrak{Im}(\,\overline{u}\,w\,)>\max\big(\,0, \textnormal{vol}(\,\rho\,)\,\big)\). This an open set in \(\mathbb C\) and, in fact, the maximal open subset in which the parameter \(w\) can be taken. Thus translation surfaces with a cylinder of absolute period \(u\) form a half-plane in an isoperiodic fiber realizing character \(\rho\).
\end{proof}

\smallskip

\noindent In the next subsection \S\ref{ssec:algpersp}, we shall provide an algebraic perspective of Proposition \ref{prop:cylclass} for arithmetic representations. For the moment we continue with the description of chambers of cylinder type.
%As a direct consequence we have a description of chambers associated to opposite values periods. Before stating such a description we need to fix some further terminology for our convenience.
Recall that all structures in the present section are realized by gluing a parallelogram \(P\) and a domain arising from the exterior of some other parallelogram in \(\mathbb C\). This domain, seen in the Riemann sphere \(\cp\) is itself a parallelogram with the same sides. In what follow we shall denote this domain as \(\Omega\) (thus \(\,\Omega\,\cup\,Q=\cp\,\)). The following holds.

\begin{cor}\label{cor:oppositechambers}
    Let \(\rho\) be a non-trivial representation and let \(\Gamma=\textnormal{Im}(\,\rho\,)\). Let \(u\in\Gamma\setminus\{\,0\,\}\) be the absolute period of some simple closed curve, Then the chambers \(\textnormal{C}(\,u\,)\) and \(\textnormal{C}(\,-u\,)\) are distinct in \(\mathfrak{ML}(\,\rho,\,\mu\,)\) and project to the same chamber in \(\mathfrak{L}(\,\rho,\,\mu\,)\) via the forgetful map.
\end{cor}

\begin{proof}
        As a direct consequence of Proposition \ref{prop:cylclass}, for every \(u\in\Gamma\setminus\{\,0\,\}\) its corresponding chamber is identified with 
        \begin{equation}
            \textnormal{C}(\,u\,)\cong\Big\{\,w\in\mathbb C\,\,|\,\, \mathfrak{Im}(\overline{u}\,w)>\max\big(\,0, \textnormal{vol}(\,\rho\,)\,\big)\,\,\Big\}
        \end{equation}
        In the same fashion \(\textnormal{C}(\,-u\,)\) is identified with
        \begin{equation}\label{eq:cylchambnegminusone}
            \textnormal{C}(\,-u\,)\cong\Big\{\,w\in\mathbb C\,\,|\,\, \mathfrak{Im}(-\overline{u}\,w)>\max\big(\,0, \textnormal{vol}(\,\rho\,)\,\big)\,\,\Big\};
        \end{equation}
        however a simple manipulation shows that this latter set can be also written as 
        \begin{equation}\label{eq:cylchambnegminustwo}
            \textnormal{C}(\,-u\,)\cong\Big\{\,w\in\mathbb C\,\,|\,\, \mathfrak{Im}(\overline{u}\,w)<\min\big(\,0, -\textnormal{vol}(\,\rho\,)\,\big)\,\,\Big\}.
        \end{equation}
        As a consequence \(\textnormal C(\,u\,)\) and \(\textnormal{C}(\,-u\,)\) are open and disjoint subsets of \(\mathbb C\). We may notice that \(w\in\textnormal{C}(\,u\,)\) if and only if \(-w\in\textnormal{C}(\,-u\,)\). We are going to show that for every value \(w\in\textnormal{C}(\,u\,)\) along with its opposite \(-w\) yield two translation surfaces in the marked stratum \(\mathcal H_1(1,1,-2)\) that are distinguished only by the marking of zeros. For this purpose, let \((X_+,\,\omega_+)\) be the translation surface with relative period \(w\). Recall that this is realized by gluing a parallelogram \(P_+\) with sides \(u,w\) and a parallelogram \(\Omega_+\subset\cp\) with sides \(u, w-v\). In the same fashion, let \((X_-,\,\omega_-)\) be the translation surface with relative period \(-w\). This latter is realized by gluing a parallelogram \(P_-\) with sides \(u,-w\) and a parallelogram \(\Omega_-\subset\cp\) with sides \(u, v-w\). Notice that the following relations hold:
        \begin{equation}
            P_+\,=\,P_- + w \quad \text{ and } \quad \Omega_+\,=\,\Omega_-+w-v.
        \end{equation}
        As a consequence \((X_+,\,\omega_+)\) and \((X_-,\,\omega_-)\) are the same translation surface. It remains to show they have different marking. By construction, there is a saddle connection, say \(\delta_+\), in \((X_+,\,\omega_+)\) with relative period \(w\) and, similarly, there is a saddle connection \(\delta_-\) in \((X_-,\,\omega_-)\) with relative period \(-w\). These saddle connections develop to parallel segments on the complex plane and, up to shift one of them by using a translation, we may assume they overlap. Recall that, according to our convention the starting point is marked with \(\pto B\) and the other is marked with \(\pto W\). Since their relative periods are opposite then these segments point to opposite directions and hence the corresponding structures have different markings as desired. See Figure \ref{fig:markings}.
        
\begin{figure}[htp]
    \centering
    \includegraphics[width=0.75\linewidth]{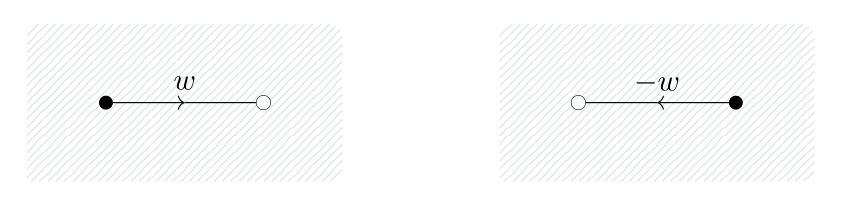}
    \caption{On the left the marking in \((X_+,\,\omega_+)\) and on the right the marking in \((X_-,\,\omega_-)\).}
    \label{fig:markings}
    \end{figure}

        \noindent Thus \(\textnormal C(\,u\,)\) and \(\textnormal{C}(\,-u\,)\) are distinct chambers in \(\mathfrak{ML}(\,\rho,\,\mu\,)\) and both project to the same chamber \(\mathfrak{L}(\,\rho,\,\mu\,)\) via the forgetful map. 
\end{proof}

\begin{rmk}\label{rmk:markdeterminesys}
    The intuitive idea behind the Corollary \ref{cor:oppositechambers} is that our construction also realizes a simple non-separating curve with period \(u\) that separates the zeros. The orientation determines the marking of the zeros. In other words, one could say that the black-marked point is to the right of the curve, while the white-marked point is to the left. Changing the orientation of the curve also changes the marking of the zeros. See Figures \ref{fig:usepa} and \ref{fig:usepa2}.
\end{rmk}

\begin{figure}[htp]
    \centering
    \includegraphics[width=1\linewidth]{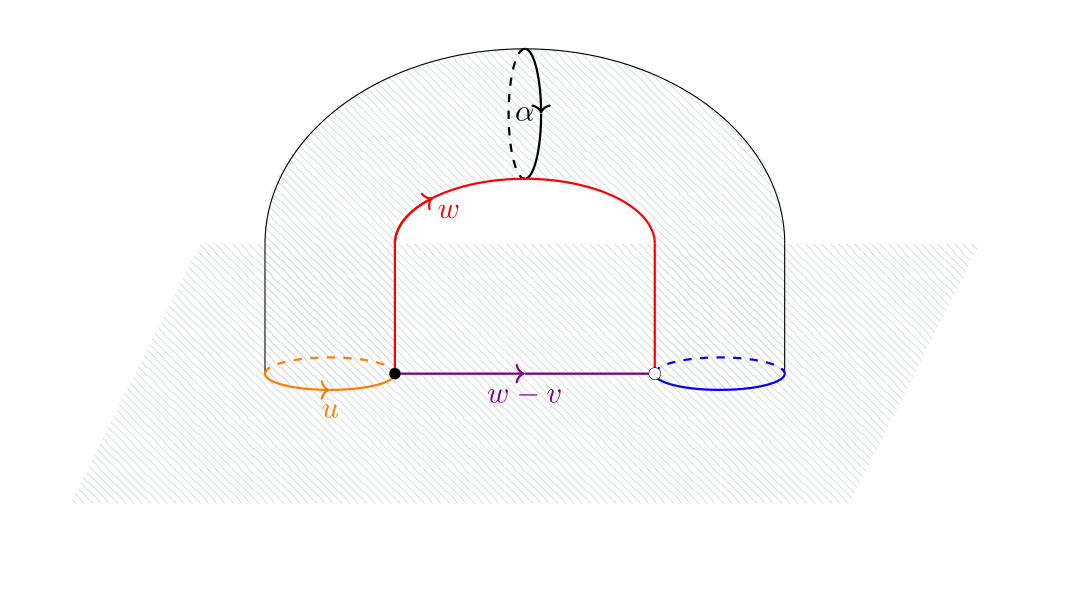}
    \caption{This figure represents the structure obtained by gluing the polygons as described in Figure \ref{fig:cylinder_type}. We can see the existence of a curve labeled with \(\alpha\) that separates the zeros. We then label as \(\pto B\) the zero on its right and with \(\pto W\) the other zero that lies on its left.}
    \label{fig:usepa}
\end{figure}

\begin{figure}[htp]
    \centering
    \includegraphics[width=1\linewidth]{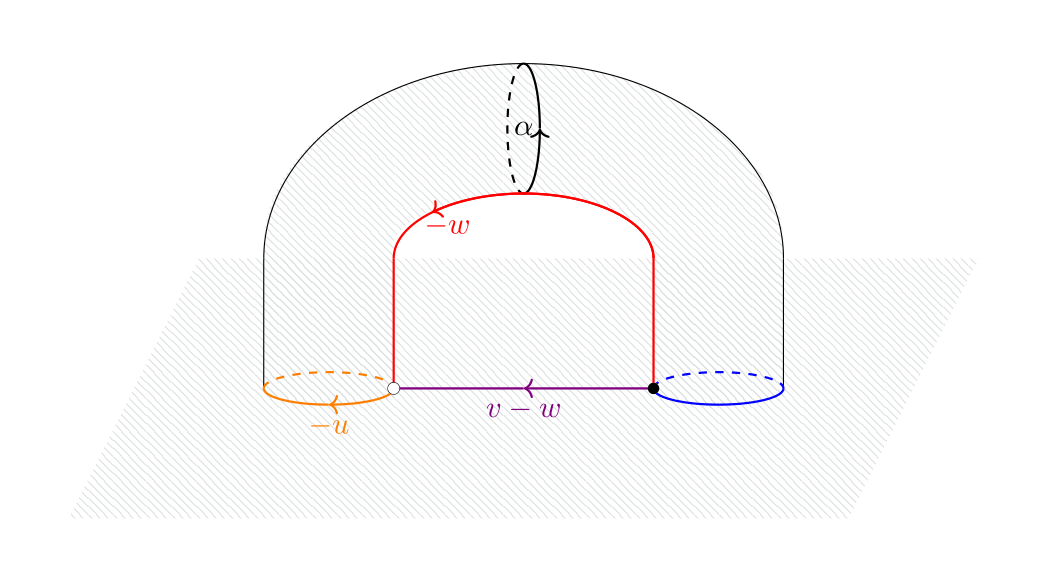}
    \caption{The resulting structure in the chamber \(\textnormal{C}(\,-u\,)\). Even in this case we single out a simple closed non-separating curve with period \(-u\) such that, according to its orientation, \(\pto B\) lies on its right and the \(\pto W\) on its left.}
    \label{fig:usepa2}
\end{figure}

\noindent We now focus again on the arithmetic case. We wish to present a more algebraic argument to explain the aforementioned subtlety that chambers \(\textnormal{C}(\,k,\,l\,)\) and \(\textnormal{C}(\,k,\,m\,)\) are different if \(m-l \notin k\mathbb Z\). 

\medskip

\subsection{An algebraic perspective for arithmetic representations.}\label{ssec:algpersp} Let \(\rho_1,\,\rho_2\colon\shomolzoo\to\mathbb C\) be two representations such that \(\textnormal{Im}(\,\rho_1\,)=\textnormal{Im}(\,\rho_2\,)=a\,\mathbb Z\), where \(a\in\mathbb C^\ast\). For simplicity we assume \(a=1\) since the generic case works \textit{mutatis mutandis}. Fix a set of generators, say \(\{\,\alpha,\,\beta\,\}\) -- recall that we assume this basis to be positively oriented. Without loss of generality we assume the following holds:
\begin{equation}
    \rho_1(\,\alpha\,)\,=\,\rho_2(\,\alpha\,)=k \qquad \rho_1(\,\beta\,)=l \qquad \textnormal{ and } \qquad \rho_2(\,\beta\,)=m. \ 
\end{equation}
Each one of these representations arises as the period character of some translations surface both with core of cylinder type. We show that these representations do not lie in the same \(\spzz\)-orbit unless \(m-l\in k\,\mathbb Z\). In fact, since \(\gcd(\,k,\,l\,)=1\) there exist \(b,d\in\mathbb Z\) such that \(bk\,+\,dl=m\). These coefficients are far from be unique. On the other hand, there exist a solution with \(d=1\) if and only if \(m-l\in k\,\mathbb Z\). Thus the desired conclusion holds. %Now, \(\rho_1\) can be realized as the period character of some translation surface with core of cylinder type, hence it lies in the chamber \(\textnormal{C}(\,k,\,l\,)\). Similarly, \(\rho_2\) can be realized as the period character of some translation surface with core of cylinder type and this latter lies in the chamber \(\textnormal{C}(\,k,\,m\,)\). 
As a consequence, chambers \(\textnormal{C}(\,k,\,l\,)\) and \(\textnormal{C}(\,k,\,m\,)\) must belong to different marked leaves unless \(m-l\in k\,\mathbb Z\) and this case they necessarily coincide -- recall the assumption \(0\le l,m\,< |\,k\,|\).

\medskip

\subsection{Change of coordinates}\label{subsec:changeofcoord}
Before we move on to the description of degeneration of structures towards a wall in a cylinder chamber, it is worth pointing out a natural change of coordinates on each cylinder chamber.

\smallskip

\noindent Observe that the triple of parameters $(u,v,w)$ used in our construction above describes the same structure in $\mathcal{H}_1(1,1,-2)$ as the triple $(u,v+ku,w+ku)$. Indeed, the new parameter amounts to a change of symplectic basis for $\shomolzoo$ and a choice of new saddle connection connecting the two ends of the cylindrical core. While the choice is fixed if we fix the parameter $v$, in later parts it is convenient sometimes to use a different one.

\smallskip

\subsection{Degeneration towards a wall}\label{sec:degtowalls} We finally provide a description of degeneration inside chambers of cylinder type. Recall that inside an isoperiodic leaf \(\mathfrak{ML}(\,\rho,\,\mu\,)\) all structures have the same absolute periods. Therefore, in the light of what has been shown above, a structure with core of cylinder type is thus completely determined by a parameter \(w\in\mathbb C\).

\smallskip

\noindent For clarity, let us assume as at the beginning of \S\ref{ssec:chambdes} that \(\{\,\alpha,\,\beta\,\}\) is a positively oriented basis for \(\shomolzoo\) and that \(u=\rho(\,\alpha\,)\) and \(v=\rho(\,\beta\,)\). According to Proposition \ref{prop:cylclass}, the chamber of cylinder type \(\textnormal{C}(\,u\,)\) is identified with the open set 
\(\big\{\,w\in\mathbb C\,\,|\,\, \mathfrak{Im}(\overline{u}\,w)>\max\big(\,0, \textnormal{vol}(\,\rho\,)\,\big)\, \big\}\). Small perturbations of \(w\) provide structures of the same type, \textit{i.e.} structures whose core is still a cylinder. In the present section we aim to describe how structures degenerate when \(w\) tends toward the boundary \(\partial C(\,u\,)\). Notice that the boundary of \(\textnormal C(\,u\,)\) is identified with \(\big\{\,w\in\mathbb C\,\,|\,\, \mathfrak{Im}(\overline{u}\,w)=\max\big(\,0, \textnormal{vol}(\,\rho\,)\,\big)\, \big\}\) which is a closed subset of \(\mathbb C\) of real codimension one. The main statement of the present section is the following:

\begin{prop}\label{prop:wallchar}
    In the notation above, if \(w\in\partial\textnormal{C}(\,u\,)\) then the resulting structure is transitional. Moreover, \(\partial\textnormal{C}(\,u\,)\) is identified with the wall that separates \(\textnormal{C}(\,u\,)\) from the rest of the isoperiodic leaf \(\mathfrak{ML}(\,\rho,\,\mu\,)\).
\end{prop}

\smallskip

\noindent In the remaining part of this section we investigate how a structure degenerates when \(w\) tends toward the boundary \(\partial\textnormal{C}(\,u\,)\). Proposition \ref{prop:wallchar} will directly follows from this inspection. We first recall how structures with core of cylinder type are realized. Given \(u,v\in\mathbb C\) and the additional parameter \(w\), we first realize a parallelogram \(P\) with sides \(u,w\) and then the domain of the pole \(\Omega\subset\cp\) by removing a parallelogram \(Q\) with sides \(u, w-v\). Then \(P\) and \(\Omega\) are glued as already described several times to get the desired structure. We draw attention to the following crucial points. First, the identity
\begin{equation}\label{eq:volidentity}
    \mathfrak{Im}(\,\overline{u}\,w\,) \,-\, \mathfrak{Im}(\,\overline{u}\,(w-v)\,)\,=\,\textnormal{vol}(\,\rho\,).
\end{equation}
holds. Moreover, \(\textnormal{vol}(\,\rho\,)\) does not change under deformations inside a leaf and hence it remains constant along degenerations towards the boundary \(\partial\textnormal{C}(\,u\,)\). Since both \(\mathfrak{Im}(\,\overline{u}\,w\,)\) and \(\mathfrak{Im}(\,\overline{u}\,(w-v)\,)\) are positive in the chamber, a structure degenerates when at least one of these quantities tends to zero. Thus three mutually disjoint situations may occur when \(w\in\partial\textnormal{C}(\,u\,)\). 
\begin{itemize}
    \item[1.] If \(\textnormal{vol}(\,\rho\,)>0\) then \(\mathfrak{Im}(\,\overline{u}\,(w-v)\,)=0\) and hence \(\Omega\) is the complement of some degenerate parallelogram in \(\mathbb C\) with sides \(u, w-v\) aligned, see Figure \ref{fig:wallcaseone}. We remark that the angle between $u$ and $w-v$ may be $0$ or $\pi$, which affects the relative positions of \(\pto B\) and \(\pto W\) along the degenerate parallelogram.
    \smallskip
    \item[2.] If \(\textnormal{vol}(\,\rho\,)<0\) then \(\mathfrak{Im}(\,\overline{u}\,w\,)=0\). In this case there is no parallelogram \(P\) to glue and \(\Omega\) is the complement of some parallelogram \(Q\) arising from a degenerate hexagon with two pairs of consecutive sides aligned. See Figure \ref{fig:wallcasetwo}. In this case we may observe that \(w\in\mathbb R\,u\).% \(w\in\langle\,u\,\rangle\otimes\mathbb R\).
    \smallskip
    \item[3.] If \(\textnormal{vol}(\,\rho\,)=0\) then \(\mathfrak{Im}(\,\overline{u}\,(w-v)\,)=0\) and \(\mathfrak{Im}(\,\overline{u}\,w\,)=0\) necessarily. In this case \(\Omega\) is the complement of a degenerate hexagon in \(\mathbb C\) with two triples of sides aligned.
\end{itemize}

\smallskip

\begin{figure}[htp]
    \centering
    \includegraphics[width=1\linewidth]{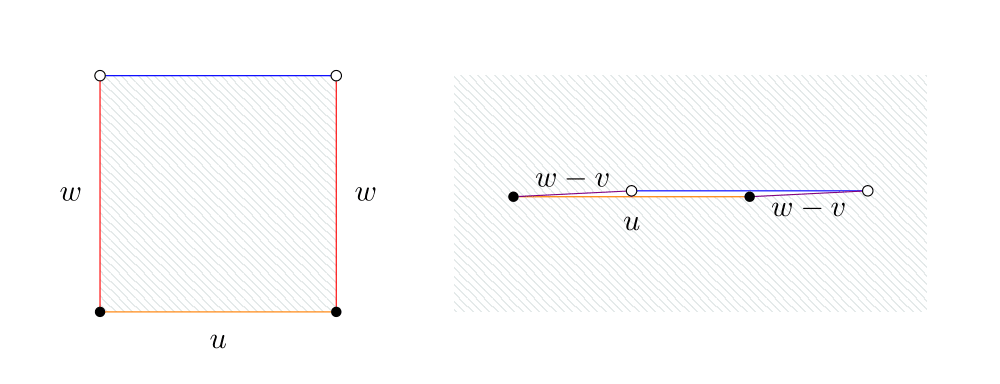}
    \caption{In the right side of the picture we show how the parallelogram \(Q\) degenerates when \(w\) approaches the boundary of \(\textnormal C(\,u\,)\) and the volume is positive. The sides are not perfectly overlapped to make the figure more comprehensible. When \(w-v\) tends to zero, then the domain of the pole is just a slit copy of \(\mathbb C\).}
    \label{fig:wallcaseone}
\end{figure}

\noindent In all cases listed above there are at least two saddle connections aligned, namely they form an angle of magnitude \(\pi\). Therefore all structures with relative period \(w\in\partial\textnormal{C}(\,u\,)\) are transitional as in Definition \ref{def:transitional}.

\smallskip

\noindent It is worth mentioning that some degenerations may yield structures in the minimal stratum \(\Omega\mathcal M_1(2,-2)\). Indeed, this happens when some saddle connection between \(\pto B\) and \(\pto W\) shrinks to a point.
Based on our construction, such a saddle connection may come from the sides of the parallelogram $Q$ with relative period $w-v$, or a geodesic path connecting two ends of the cylindrical core coming from $P$, which may have relative period $w-k\,u$ for any $k\in\mathbb{Z}$.
We now investigate when such a saddle connection can shrink to zero case-by-case according to the list above (hence we use the same enumeration).
\begin{itemize}
    \item[1.] If \(\textnormal{vol}(\,\rho\,)>0\), then \(\Omega\) is the complement of a segment when $w\in\partial C(u)$, while \(P\) is still a parallelogram. Saddle connections on the cylindircal core still have definite size. To get a structure in \(\Omega\mathcal M_1(2,-2)\) we must have $w-v=0$, when one pair of sides of $Q$ shrinks to zero. See Figure \ref{fig:wallcaseone} for an illustration.
    \smallskip
    \item[2.] If \(\textnormal{vol}(\,\rho\,)<0\) then the parallelogram \(P\) degenerates, while \(\Omega\) is the complement of a parallelogram \(Q\) with sides \(u,w-v\). Since saddle connections on \(P\) have periods $w-k\,u$, one of them can shrink to zero if and only if $w=k\,u$ for some $k\in\mathbb{Z}$. See Figure \ref{fig:wallcasetwo} and its caption for \(k=0\).
    \smallskip
    \item[3.] If \(\textnormal{vol}(\,\rho\,)=0\) then both $w=v$ and $w=k\,u$ are possible. In the non-arithmetic case, these correspond to distinct values. In the arithmetic case where $\Gamma=a\mathbb{Z}$, say for the chamber $C(s,l)$, since we choose $u=s a$ and $v=a l$ with $0\le l<|s|$, the values $w=v$ and $w=k\,u, k\in\mathbb{Z}$ are also distinct. For all of the cases, \(\Omega\subset\cp\) is the complement of a degenerate parallelogram in \(\mathbb C\) and the resulting structure lies in \(\Omega\mathcal M_1(2,-2)\).
\end{itemize}

\smallskip

\begin{figure}[htp]
    \centering
    \includegraphics[width=0.875\linewidth]{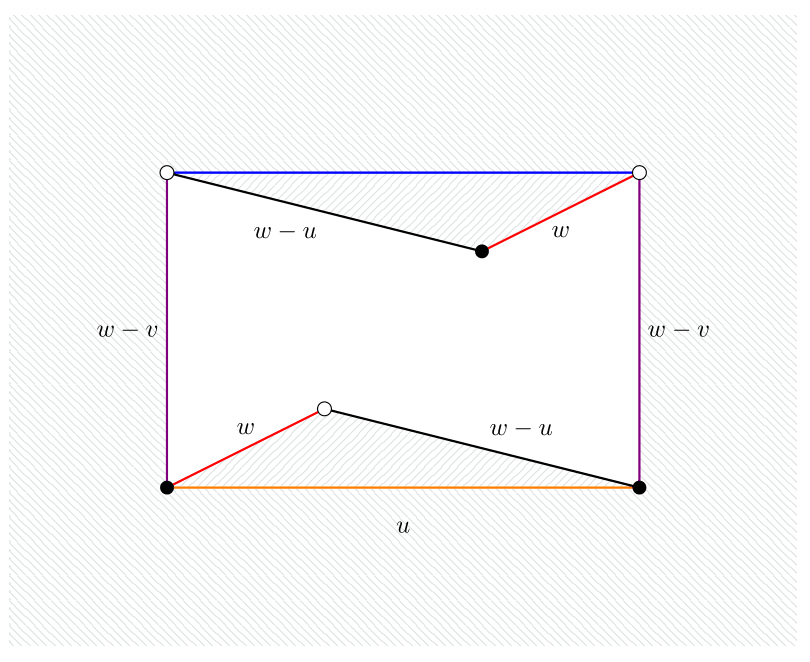}
    \caption{In this picture, we show the realization of translation surfaces with a core of the cylinder type from a different perspective. The two triangles shaded with a different pattern, glued along the black-colored edge, determine the parallelogram \(P\). When \(w\) tends to become real collinear with \(u\), the parallelogram \(P\) degenerates into a segment, and consequently, \(Q\) becomes a degenerate hexagon with two pairs of aligned sides with periods \(w,\,w-u\). Therefore, the structure is transitional. Furthermore, we may notice that if \(w\) tends to zero then \(Q\) degenerates to a parallelogram in \(\mathbb C\) with non-empty interior. Recall that all saddle connections are oriented from \(\pto B\) to \(\pto W\) by our convention.}
    \label{fig:wallcasetwo}
\end{figure}

\noindent This concludes the study of degeneration of structures towards walls and hence the proof of Proposition \ref{prop:wallchar}.

\medskip

%\noindent Finally, we note that the boundary line of chamber $CC_{u}$ or $CC_{u,v}$ (in the direction $u\,\mathbb{R}$) corresponds to the degeneration of:\begin{itemize} \item cylinder $P$ (if $Vol_{\chi}<0$);  \item parallelogram $Q$ (if $Vol_{\chi}>0$); \item both of them (if $Vol_{\chi}=0$). \end{itemize}

\section{Positive leaves}\label{sec:posleaves}
\noindent In this section, we give a complete description of those loci \(\mathfrak{ML}(\,\rho,\,\mu\,)\) where \(\textnormal{vol}(\,\rho\,)\) is positive. We begin with the following observation that will simplify our later arguments. Let \(\rho_1,\,\rho_2\colon\shomolzoo\longrightarrow \mathbb C\) be two representations with positive volume. Then there exists \(A\in\glplus\) such that \(\rho_2\,=\,A\cdot\rho_1\). Such a linear transformation \(A\) determines an homeomorphism, say \(F_A\), of leaves 
    \begin{equation}
        F_A\colon\mathfrak{ML}(\,\rho_1,\,\mu\,)\longrightarrow \mathfrak{ML}(\,\rho_2,\,\mu\,),
    \end{equation}

\noindent and the crucial point is the following:

\begin{prop}\label{prop:posleafpres}
    \(F_A\) preserves the structure of leaves, \textit{i.e.} it maps chambers to chambers and walls to walls. 
\end{prop}

\begin{proof}
    This preliminary result is a consequence of our constructive approach. Recall that structures of torus or cylinder type are obtained by gluing a copy of \( \mathbb{C} \) with a torus or a cylinder, respectively. A structure with a degenerate core, on the other hand, is obtained by removing a suitable hexagon from \( \mathbb{C} \) and gluing the sides of the complement appropriately. In any case, we are dealing with polygons in the plane, and the period character is determined by these polygons. 
    
    \noindent Let us focus on structure of torus type. The torus we glue to the plane is obtained from a parallelogram, say \(P\) with \(\textnormal{area}(\,P\,)>0\), and the periods are determined by this parallelogram. If we deform the given parallelogram with a transformation say, \( A\in\glplus \), we then obtain a new one with positive area equal to \(\det(\,A\,)\,\cdot\,\textnormal{area}(\,P\,)\). Therefore, if a structure \((X,\omega)\) of torus type is realized by gluing a parallelogram \(P\) and a copy of \(\mathbb C\) along a slit, say \(w\), as shown in Figure \ref{fig:torus_type}, then \(A\cdot(X,\omega)\) arises from the gluing of \(\mathbb C\) (simply because \(A\cdot\mathbb C=\mathbb C\)) with \(A\cdot P\) along a slit \(A(\,w\,)\). Thus \(A\cdot(X,\omega)\) is still of torus type. The same argument applies \textit{mutatis mutandis} to structures of cylinder type. As a consequence, \(F_A\) maps chambers to chambers and walls to walls thus preserving the leaf structure.
\end{proof}

\noindent Therefore, in what follows we may assume \(\textnormal{Im}(\,\rho\,)=\mathbb{Z}[\,i\,]\) for simplicity.

\subsection{Realizing structures of torus type}\label{ssec:reltorusstruc}
In \S\ref{ssec:core} we have already given a glimpse of how to realize a structure of torus type, see Figure \ref{fig:torustype}. As in \S\ref{ssec:chambdes}, let \(\rho\colon\shomolzoo\longrightarrow \mathbb C\) be a non-trivial representation and let \(\{\,\alpha,\,\beta\,\}\) be two generators that represent a pair of simple closed curves that intersect only once (up to homotopy). We shall assume \(\{\,\alpha,\,\beta\,\}\) to be a positively oriented basis, \textit{i.e.} their intersection number is equal to one. Throughout this section let \(u,v\in\mathbb C\) denote \(u=\rho(\,\alpha\,)\) and \(v=\rho(\,\beta\,)\) respectively. Notice that %\(u,v\) 
they cannot be both zero because the representation is assumed to be non-trivial. Moreover, since we are assuming \(\textnormal{vol}(\,\rho\,)>0\), we may deduce that \(u,v\) are both non-zero for every suitable set of generators \(\{\,\alpha,\,\beta\,\}\). We emphasizes the fact that we are assuming \(\textnormal{Im}(\,\rho\,)=\mathbb{Z}[\,i\,]\), and hence we suppose \(\mathfrak{Im}(\,\overline u v\,)=1\).

\smallskip

\noindent We now realize a translation surface \((X,\omega)\in\mathcal{H}_1(1,1,-2)\) of torus type and then use this construction to provide a description of chambers of torus type. In short, a surface of torus type is constructed as follows: Take a flat torus and a copy of the complex plane, cut both of them along parallel slits and then glue them together via translations. More precisely, we first realize a parallelogram \(P\) with sides \(u,v\) -- it has positive volume because \(\mathfrak{Im}(\,\overline u v\,)>0\). By gluing the opposite sided with translation we get a flat torus with integral periods and volume one. We now recall the following observation from standard covering theory.

\begin{rmk}
    %In the first place we point out that, up to translations, we may assume parallelogram \(P\) has the bottom left vertex at the origin of \(\mathbb C\). Thus the parallelograms \(z\,+\,P\) for \(z\in\mathbb{Z}[\,i\,]\) form a tiling of \(\mathbb C\) and we may regard \(P\) as a fundamental domain for the action of \(\mathbb{Z}[\,i\,]\). 
    There is a natural covering map \(\pi_P\colon \mathbb C\longrightarrow {\rm T}=\mathbb C\,/\,\mathbb{Z}[\,i\,]\). Let \(\ell\subset\mathbb C\) be any line passing through the origin. %and let \(\delta\) be its projection
    Depending on whether the slope of \(\ell\) is rational, its projection 
    %\(\delta \) 
    may be a closed geodesic curve or a line that wraps on the torus. In the first case, there is a pair \((p,q)\in\ell\,\cap\,\mathbb Z^2\) such that \(\gcd(\,p,\,q\,)=1\) and \(\ell\) has slope \(p/q\). We also remark that such a curve has length equal to \(||\,(\,p,q\,)\,||\).
\end{rmk}

\noindent In its simplicity this remark is what we need to describe chambers of torus type. Let \({\rm T}\) be the torus arising from the parallelogram \(P\) once its sides are identified by translation and mark the point arising from the identification of the vertices as \(\pto B\). Let \(s\) be an oriented geodesic segment starting from \(\pto B\) and mark the other extremal point as \(\pto W\). This segment has a well-defined slope because its lifts via \(\pi_P\) all belong to the same pencil of parallel lines in \(\mathbb C\). We thus have the following dichotomy. If the slope is irrational, then \(s\) can be arbitrarily long. On the other hand, if the slope is rational, its length is bounded above by \(||\,(\,p,q\,)\,||\) for some appropriate vector \((p,q)\in\mathbb Z^2\), since \(s\) is part of a simple closed curve. 

\smallskip 

\begin{figure}[ht!]
    \centering
    \includegraphics[width=0.875\linewidth]{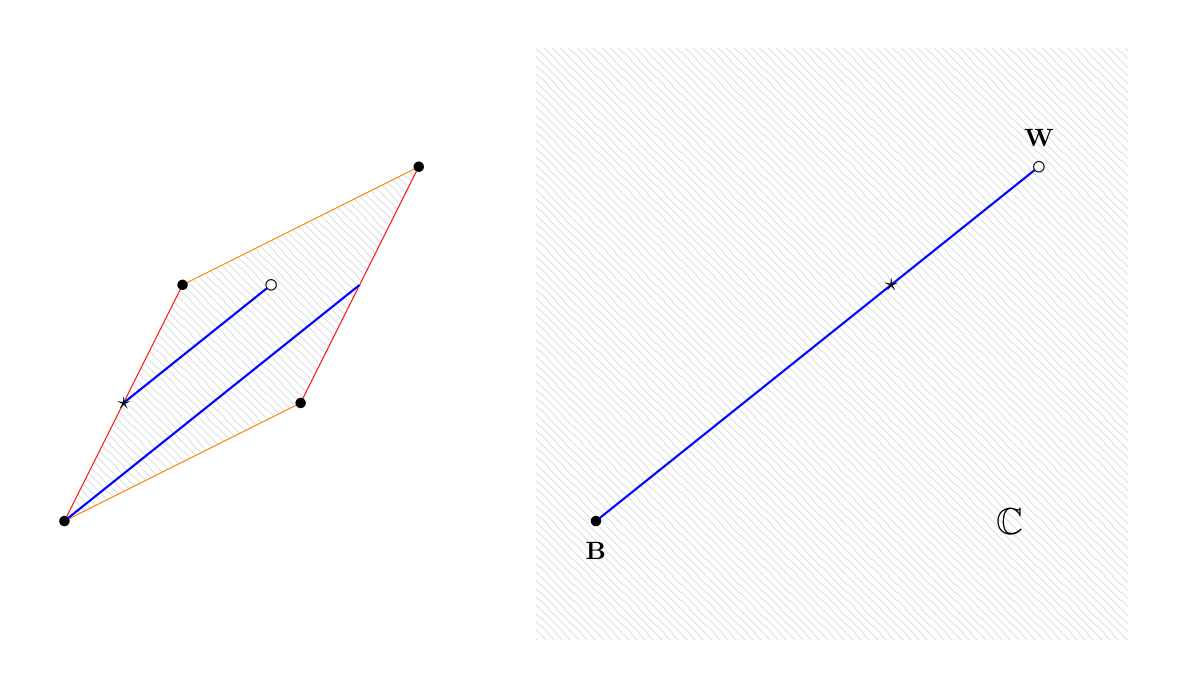}
    \caption{A surface of torus type, with a long slit. In Figure \ref{fig:torustype} we already showed how to realize a structure of torus type. Here we aim to show how this kind of structures can be made when a slit has length greater than the lengths of the sides of the parallelogram. It is then clear that for countably many directions -- those with rational slope -- the lengths must be bounded by the length of the diagonal with the same rational slope.}
    \label{fig:torus_type}
\end{figure}

\noindent We note that \( \rm T \) is itself a translation surface and, as such, admits a developing map, see \S\ref{ssec:transurfa}. The image of \( s \) under the developing map defines an oriented segment, say \( \sigma \), which has the same length, orientation, and slope; see Figure~\ref{fig:torus_type}. Slit both \({\rm T}\) and \(\mathbb C\) along \(s\) and \(\sigma\) respectively and then glue the resulting sides by using translations as follows. According to their orientations, once the segments are slit, we may distinguish a left and a right side. Thus the left side of \(s\) is glued with the right side of \(\sigma\) and, in the same fashion, the right side of \(s\) is glued with the left side of \(\sigma\). The resulting structure is a translation surface, say \((X,\omega)\), in the marked stratum \(\mathcal H_1(1,1,-2)\) with the marking given by the choice of the orientation of \(s\). In fact, as an artifact of the construction, it is not difficult to show that a change in orientation corresponds to a change in the marking. By construction \((X,\omega)\) is a structure of torus type and it has the prescribed absolute periods. Moreover, the zeros of \((X, \omega)\) arise from the identification of the extremal points of \(s\) with those of \(\sigma\), which we still label as \(\pto B\) and \(\pto W\) with a slight abuse of notation. By construction, there are two saddle connections, say \(\delta_1,\,\delta_2\), joining these points, having the same orientation consistent with that of \(s\) — and thus of \(\sigma\) — and we denote the corresponding relative period by \(w\). We then have that,
\begin{equation}
    w\,=\,\int_\sigma dz\,=\,\int_{\delta_i} \omega.
\end{equation}

\noindent In what follows all structures of torus type are realized in this way and we shall refer to this section when needed.

%\smallskip

\subsection{Chamber of torus type}\label{ssec:toruschamb} Consequently, we can determine all the possible values that \( w \) may assume. Indeed, although \( \mathbb{C} \) can be slit along a segment of any length and direction, there are certain obstructions on the possible choices of \( s \) in \( \rm T \). We define \(\mathcal T\) as the following subset of \(\mathbb C\) shown in Figure~\ref{fig:torus_chamber}.
\begin{equation}
    \displaystyle\mathcal{T}\,=\,\mathbb C^{\ast}\,\setminus \,\,\bigcup_{\mathclap{\substack{(p,q)\,\in\,\mathbb Z^2\\ \gcd(\,p,\,q\,)=1}}} \left\{\, t\,(\,p\,+\,iq\,)\,|\, t \ge 1\, \right\},
\end{equation}

%\smallskip

\begin{figure}[ht!]
    \centering
    \includegraphics[width=0.875\linewidth]{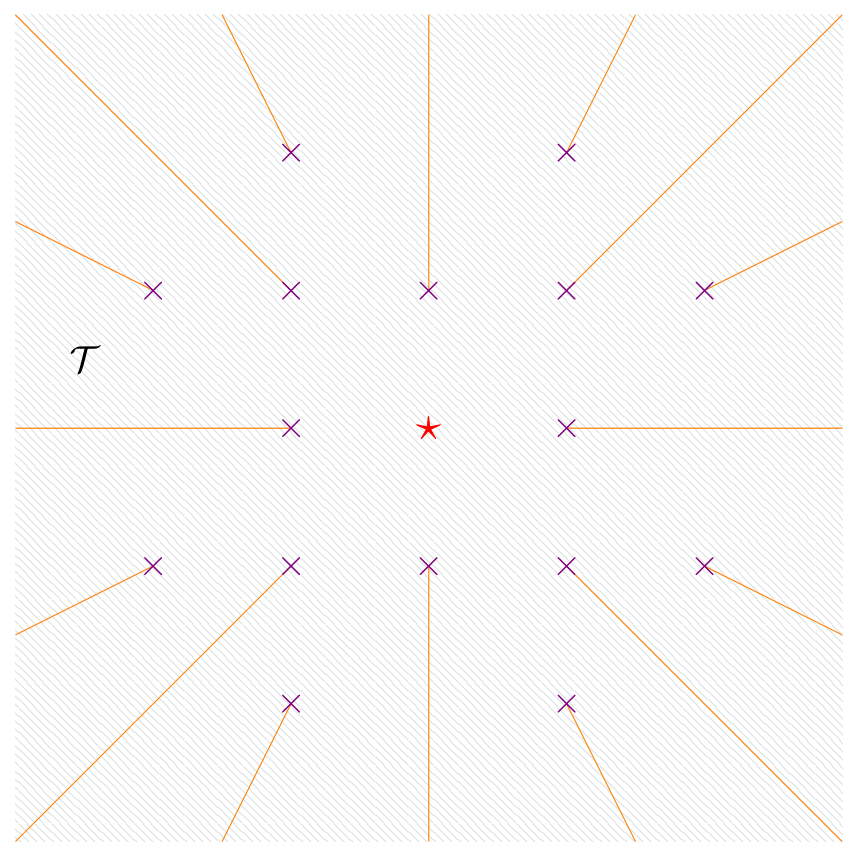}
    \caption{Representation of \(\mathcal T\). The red star denoted the origin and it has a special meaning. Indeed, it is the only point that corresponds to a particular kind of degeneration. All points marked with a violet cross \(\times\) denote structures of cylinder type in the minimal stratum \(\Omega\mathcal M_1(2,-2)\). The reader may compare with symbols used in Figure \ref{fig:isoleaf}. We shall see this is a chamber, indeed a chamber of structures of torus type, and the orange lines are the walls.
    }
\label{fig:torus_chamber}
\end{figure}

\noindent where \(\mathbb C^{\ast}\) denotes \(\mathbb C\setminus \{\,0\,\}\). In what follows we shall say that \(w\in\mathbb C^*\) is a \textit{radial point} if it belongs to the complement of \(\mathcal T\). We will also see in \S\ref{ssec:degtorustowardzero} how the structures deform when we move towards the origin in \(\mathcal T\), by shrinking the slit segment to a point.
%Indeed, so far we have tacitly assumed that the segments were not degenerate and thus had positive length. In what follows we shall say that \(w\in\mathbb C^*\) is a \textit{radial point} if it belongs to the complement of \(\mathcal T\). We will see in the next section, see \S\ref{ssec:degtorustowardzero}, how the structures deform by degenerating a segment until it becomes a point.

\smallskip

\noindent It is clear from \S\ref{ssec:reltorusstruc} that for a given value in \(\mathcal T\) we can always realize a structure of torus type in the marked stratum \(\mathcal H_1(1,1,-2)\). The converse is also true: every structure of torus type can be decomposed into a torus and a copy of the complex plane. The main result of the present section is the following statement.

\begin{prop}\label{prop:toruschamb}
    Let \(\mu=(1,1,-2)\). Then there is a one-to-one correspondence between the following sets:
    \begin{equation}
        \mathfrak T(\,\rho\,) =
        \left\{\,\, \begin{array}{cc}
            \textnormal{Translation surfaces of} \\
            \textnormal{torus type in } \mathfrak{ML}(\,\rho,\,\mu\,)
        \end{array}
        \,\,\right\}\,\,\mathrel{\longleftrightarrowTilde}\,\,\mathcal T.
    \end{equation}
\end{prop}

\smallskip

\begin{proof}
    It is sufficient to show that a structure of torus type with prescribed absolute periods is uniquely determined by the relative period \(w\). Thus, it is sufficient to show that two structures in \(\mathfrak T(\,\rho\,)\) have the same relative period if and only if they are isometric by means of translations in local charts. In the first place we notice that as a consequence of Proposition \ref{prop:coreshapes} every structure in \(\mathfrak{T}(\,\rho\,)\) arises from the gluing of some flat torus \((\,\rm T,\,\xi\,)\) and a copy of \((\mathbb C,\,dz)\) along parallel slits such that 
    \begin{equation} w\,=\,\int_s\,\xi\,=\,\int_{\sigma}\,dz. \end{equation} 
    Since absolute periods are fixed every element in \(\mathfrak T(\,\rho\,)\) defines the same flat torus \((\,\rm T,\,\xi\,)\in\mathcal M_1\), even if \(\rm T\) may arise from different parallelograms on \(\mathbb C\). As a consequence, it is easy to show that two structures in \(\mathfrak T(\,\rho\,)\) define the same translation surface if and only if they have the same relative period.
\end{proof}

\noindent We define \(\mathfrak{T}(\,\rho\,)\) as a \textit{chamber of torus type}, and in light of the result just shown, there exists a unique such chamber for each isoperiodic fiber \(\mathfrak{ML}(\,\rho,\,\mu\,)\). This is open, connected, in fact path-connected, and it has the homotopy type of a circle. Moreover the parameter \(w\) provides a global coordinate in this chamber as for every value \(w\in\mathcal T\) there exists a unique translation surface of torus type in \(\mathfrak T(\,\rho\,)\). In the next sections, we shall study how a sequence of structures degenerates when the value of \(w\) tends to a value in the complement of \(\mathcal{T}\). A sequence of structures of torus type can degenerate in three different ways:  
\begin{enumerate}
    \item[1.] Their cores may degenerate to a cylinder, and hence, in the limit, we obtain a structure of cylinder type in \( \mathcal{H}_1(1,1,-2) \), 
    \item[2.] the zeros collapse into a single zero of higher order, and the resulting surface lies in the minimal stratum \( \Omega\mathcal{M}_1(2,-2) \). 
\end{enumerate}
    Note that in both cases, the underlying topological surface does not change.% but there is only a degeneration of the geometric structures. 
\begin{enumerate}
    \item[3.] The topological surface tends to a nodal curve formed by a torus and a sphere identified at a point.  
\end{enumerate}

\noindent In what follows we shall consider the first two cases together and the third separately. We shall begin with this latter. See Figure \ref{fig:structuresinposleaves}.

\smallskip

\begin{figure}[htp]
    \centering
    \includegraphics[width=1\linewidth]{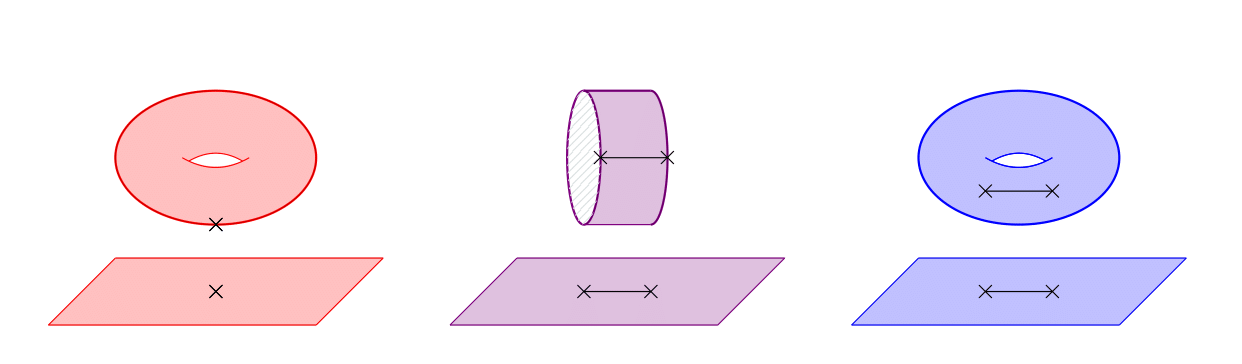}
    \caption{The rightmost picture shows a structure of torus type just described. The others represent possible degenerations toward the origin and walls.}
    \label{fig:structuresinposleaves}
\end{figure}

\smallskip

\subsection{Degeneration towards the origin}\label{ssec:degtorustowardzero}

\noindent As discussed above, every translation surface of torus type arises from the gluing of some torus and a copy of \((\mathbb C,\,dz)\) along parallel slits of period \(w\) and hence of lengths \(|\,w\,|\). In this section, we describe how a given structure degenerates as \( w \to 0 \). Notice that, \emph{a priori}, the limiting structure may depend on the direction along which \( w \) approaches zero. We shall see that this is not the case. We first highlight the following observation. 

\begin{rmk}\label{rmk:starconvex}
    \(\mathcal T\) is star-connected, in fact for all \( w \in \mathcal T \) there is a line segment from \( 0 \) to \( w \) entirely contained in \( \mathcal T \). 
\end{rmk}

\noindent Let \((X_o,\,\omega_o)\) be a translation surface of torus type. According to Proposition \ref{prop:toruschamb}, the structure is obtained by gluing a torus, say \((\,\rm T,\,\xi\,)\) and a copy of \(( \mathbb{C},dz) \) along a slit. Assume \(w\in\mathcal T\) is the relative period of such a  slit and let \(r\subset \mathbb C\) be the segment connecting \(w\) with the origin. Let \(\tau\colon [\,0,\,1\,]\to \mathbb C\) be a parametrization of \(r\) so that \(\tau(\,t\,)=(1-t)w\). Then define \((X_t,\,\omega_t)\) be the translation surface obtained by gluing \((\,\rm T,\,\xi\,)\) and a copy of \(( \mathbb{C},dz) \) along a slit of relative period \(\tau(\,t\,)\) and hence of length \(|\,\tau(\,t\,)\,|\). We thus have a sequence of structures that, as \( t \) approaches one, degenerate to a nodal curve because the slits along which the torus and the plane are glued become shorter and shorter until they degenerate to a point. See Figure \ref{fig:structuresinposleaves}. This nodal curve, in particular, does not depend from which direction \(w\) approaches zero. 

\smallskip

\noindent We finally note that this nodal curve does not belong to \(\mathfrak{ML}(\,\rho,\,\mu\,)\). As a consequence, \(\mathfrak{ML}(\,\rho,\,\mu\,)\) is not complete. To obtain the metric completion of \(\mathfrak{ML}(\,\rho,\,\mu\,)\) we thus need to add this additional point.

%However this structure belongs to its closure 
%of \(\mathfrak{ML}(\,\rho,\,\mu\,)\) 
%and hence to its metric completion once we equipped it with a suitable metric structure.

\medskip

\subsection{Degeneration towards slits}
We next discuss how a given structure degenerates as \(w\) tends to some value of the form \(t\,(\,p+iq\,)\) where \((p,\,q)\in\mathbb Z^2\) is a primitive pair and \(t\ge1\). We shall distinguish two cases as follows:
\begin{itemize}
    \item[\textnormal{i}.] \(w\longrightarrow t\,(\,p+iq\,)\) for \(t>1\); and we shall refer to this case as \textit{degeneration towards radial points}. Otherwise,
    \smallskip
    \item[\textnormal{ii}.] \(w\longrightarrow p+iq\); and we refer to this case as \textit{degeneration towards primitive lattice points}.
\end{itemize}
These cases correspond to those degenerations enumerated with \((\,1\,)\) and \((\,2\,)\) in the list of \S\ref{ssec:toruschamb} respectively. We shall begin with the easier latter case.

\subsubsection{Degeneration towards a primitive lattice point}\label{sssec:deglatticepoint} 
Let \(w_\infty=\,p\,+\,iq\,\) and let \(\textnormal{B}_{\varepsilon}(\,w_\infty\,)\) be an open ball containing no point of the lattice \(\mathbb Z[\,i\,]\) other than \(w_\infty\). Define \(\textnormal{U}(\,\varepsilon, \,w_\infty\,)=\textnormal{B}_\varepsilon(\,w_\infty\,)\,\cap\,\mathcal T\). Notice that this is equivalent to defining \(\textnormal{U}(\,\varepsilon, \,w_\infty\,)\) as \(\textnormal{B}_{\varepsilon}(\,w_\infty\,)\setminus\left\{\, t\,(\,p\,+\,iq\,)\,|\, t \ge 1\, \right\}.\) Let \((\,w_n\,)\subset\textnormal{U}(\,\varepsilon, \,w_\infty\,)\) be a sequence converging to \(w_\infty\) and define \((\,X_n,\,\omega_n\,)\in\mathcal H_1(1,1,-2)\) as the translation surface of torus type realized for the parameter $w_n$ as described in \S\ref{ssec:reltorusstruc}.

\medskip

\noindent We may assume for simplicity that the torus \(\textnormal T\) is realized from the unit square by identifying opposite sides. Recall that the corners are all identified and we label by \(\pto B\) the resulting point in the torus. 

\medskip

\noindent Let \(s_n\subset \textnormal{T}\) be an oriented segment starting from \(\pto B\) and ending at \((\,\mathfrak{Re}(\,w_n\,),\mathfrak{Im}(\,w_n\,)\,)\). Then, as \(n\to\infty\), the segments \(s_n\) converge to a segment \(s_\infty\) with relative period equal to \(w_\infty\) by construction, see Figure \ref{fig:sequenceofsegments}. The extremal point of \(s_n\) other than \(\pto B\) converges to \(\pto B\) and hence the limit is a simple closed curve on \(\textnormal T\). Since the limit does not depend on the direction along which \(w_n\) tends to \(w_\infty\), the limiting structure \((\,X_\infty,\,\omega_\infty\,)\) is well-defined. It remains to determine this structure explicitly.

\begin{figure}[htp]
    \centering
    \includegraphics[width=0.675\linewidth]{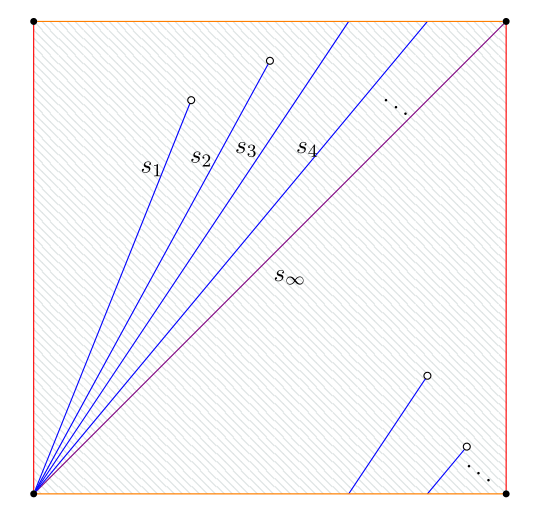}
    \caption{In blue a sequence of segments converging to \(w_\infty=1\,+\,i\). In violet we draw the limit segment. Notice that as \(w_n\rightarrow w_\infty\) the white point tends to the black point.}
    \label{fig:sequenceofsegments}
\end{figure}

\smallskip

\noindent Let \(\alpha=s_\infty\) and let \(\beta\) be another simple closed curve such that \(\{\,\alpha,\,\beta\,\}\) is a positively oriented basis. As usual, we denote by \(v\) the period of \(\beta\). The crucial observation here is that we can realize the torus \(\textnormal T\) by identifying the opposite sides of a parallelogram \(P\) with sides given by the vectors \(w_\infty,\,v\). In fact, \(P\) and the unit square are related by some transformation \(A\in\spzz\) having as columns the integer vectors \((\,p,\,q\,)\) and \((\,x,\,y\,)\) where this latter pair is a solution of the Diophantine equation \(py\,-\,qx=1\). Let \((\,\mathbb C,\,dz\,)\) be a copy of the complex plane slit along a segment parallel to \(w_\infty\). Notice that the prescribed relative period \(w_\infty\) coincides with one of the sides of \(P\). Thus we glue \(P\) and \((\,\mathbb C,\,dz\,)\) as described in Case \((\,1\,)\) of \S\ref{sec:degtowalls} and the resulting structure belongs to the minimal stratum \(\Omega\mathcal M_1(2,-2)\) as claimed.

\medskip

\subsubsection{Degeneration towards a radial point}\label{sssec:difflimits} We now discuss how a structure degenerates when \(w\) approaches a ray \(r_{p\,q}=\left\{\, t\,(\,p\,+\,iq\,)\,|\, t \ge 1\, \right\}\). Notice that \(t=1\) corresponds to the case just considered above, so we assume now $t>1$. It is convenient to orient all rays so that they point towards infinity, which distinguishes a left and a right side. %Let \(w_\infty\,\in r_{p\,q}\) be a radial point and l
Let \((\,w_n\,)\subset \mathcal T\) be a sequence converging to some radial point %\(w_\infty\) 
and define \((\,X_n,\,\omega_n\,)\in\mathcal H_1(1,1,-2)\) as the translation surface of torus type realized for the parameter $w_n$ as described in \S\ref{ssec:reltorusstruc}. Unlike the case above, the limit structure \textit{depends} on the direction of convergence. Indeed, we distinguish two different limits depending on whether the sequence approaches \(r_{p\,q}\) from the left or the right. In this notation the main result of the present section is the following:

\begin{prop}\label{prop:difflimits}
Let \(u=p+iq\in\mathbb Z\,+\,i\mathbb Z \) be a primitive lattice point and let \(r_{p\,q}\) be the ray emanating from \(u\) with slope equal to \(q/p\). Let \((w_n)\subset \mathcal T\) be a sequence converging to a radial point, say \(w_\infty\in r_{p\,q}\), and then define \((\,X_n,\,\omega_n\,)\) as the structure of torus type realized by slitting \(w_n\). Then, \((X_n,\,\omega_n)\) converges to a structure of cylinder type \((\,X_\infty,\,\omega_\infty\,)\). In particular, if \(\arg(\,w_n\,)<\arg(\,u\,)\) then \((\,X_\infty,\,\omega_\infty\,) \,\in\,\partial\textnormal{C}(\,u\,)\); otherwise if \(\arg(\,w_n\,)>\arg(\,u\,)\) then \( (\,X_\infty,\,\omega_\infty\,) \,\in\,\partial\textnormal{C}(\,-u\,) \).
\end{prop}

\smallskip

\noindent The remaining part of this section is devoted to the proof of this statement. Note that no sequence \( (w_n) \subset \mathcal{T} \) converging to a radial point satisfies \( \arg(\,w_n\,) = \arg(\,u\,) \), so the proposition above provides a complete description of the limits of sequences.

\smallskip

\noindent We first introduce a generic situation that we will use to better understand this degeneration. Let \(w\in\mathcal T\) be a value sufficiently close to a ray \(r_{p\,q}\); how close this should be will be clear from the context. Let \(P\) be a parallelogram with sides given by the integral vectors \(u=p+iq\) and \(v=x+iy\) such that \(py\,-\,qx=1\), and define \(C\) to be the cylinder with period \(u\) (\emph{i.e.} \(C\) is obtained from \(P\) by identifying sides parallel to \(v\)). The vertices of \(P\) are then identified in pairs into two marked points; one for each boundary component. We label these points as \(\pto B_1\) and \(\pto B_2\). Note that once the boundaries are identified with a translation, \(\pto B_1\) and \(\pto B_2\) merge into a marked point on the resulting torus that we usually mark with \(\pto B\). We next define \(s\) as the geodesic curve emanating from \(\pto B_i\) with slope \(\mathfrak{Im}(\,w\,)/\mathfrak{Re}(\,w\,)\) and length \(|\,w\,|\). Up to changing the labeling of the points \(\pto B_1\) and \(\pto B_2\), we may assume \(s\) emanates from \(\pto B_1\) if \(\arg(\,w\,)>\arg(\,u\,)\), and from \(\pto B_2\) otherwise; see Figure \ref{fig:segmentswithoppositeslopes}. Notice that since \(w\notin r_{p\,q}\), the arguments cannot be equal. 

\begin{figure}[!ht]
    \centering
    \includegraphics[width=0.825\linewidth]{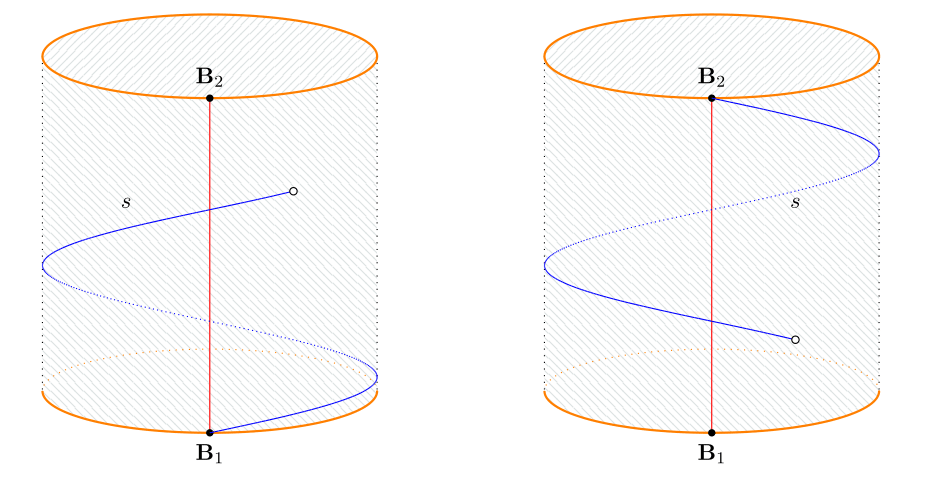}
    \caption{On the left is depicted the case \(\arg(\,w\,)>\arg(\,u\,)\) whereas on the right is depicted the the case \(\arg(\,w\,)<\arg(\,u\,)\).}
    \label{fig:segmentswithoppositeslopes}
\end{figure}

\smallskip

\noindent In both cases, we use the geodesic curve \(s\) to decompose \(P\) into a finite collection of polygons and then rearrange them in a different way. The resulting polygon has the shape of a \textit{tomahawk}, see Figure \ref{fig:cutandpaste}, and provides the same torus \(\rm T\), making it easier to understand how the structure degenerates as \(w\) approaches the ray \(r_{p\,q}\). 

\smallskip
\begin{figure}[htp]
    \centering
    \includegraphics[width=1\linewidth]{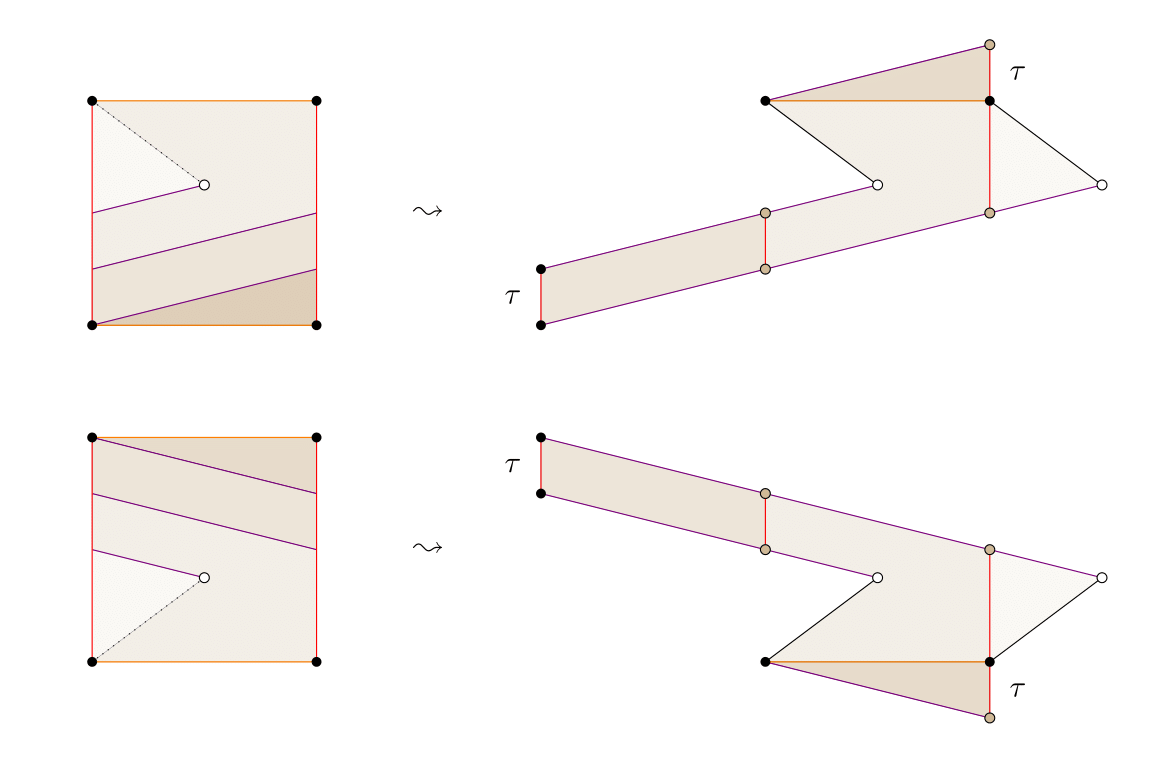}
    \caption{Tomahawks. The picture shows that if \(w\) approaches the ray \(r_{p\,q}\) from different sides then tomahawks are reflected. As a consequence, the degenerated structures turn out to be different.}
    \label{fig:cutandpaste}
\end{figure}

\noindent As we shall see, the two cases are distinguished since the resulting tomahawks are not isometric and define structures that differ not just in terms of marking. The rearrangement also reveals a new invariant for the flat structure on the torus. On both tomahawks we can find a pair of edges labeled with \(\tau\), whose length is proportional to the torsion of the geodesic curve \(s\) that wraps around \(C\), see Figures \ref{fig:segmentswithoppositeslopes} and \ref{fig:cutandpaste}. Notice that, in both cases, the tomahawk has a side with period \(w\). Thus it can be glued by translations on a copy of \((\,\mathbb C,\,dz\,)\) slit along a segment of period \(w\). The resulting translation surface lies in \(\mathcal H_1(1,1,-2)\). In order to understand the degeneration when the parameters \(w\) approaches a ray it is sufficient to see how this latter structure degenerates. 

\smallskip

\noindent Let \((w_n)\subset \mathcal T\) be a sequence converging to some \(w_\infty\in r_{p\,q}\). This sequence yields a sequence of geodesic curves \((s_n)\subset C\) all starting from the same marked point \(\pto B_i\) with slopes tending to \(q/p\). As a consequence, their respective torsions \(\tau_n\) tend to zero, and the geodesic curves \((s_n)\) tend to a curve that overlaps the boundary component of \(C\) containing \(\pto B_i\). 

\begin{rmk}
    Since \(w_\infty\in r_{p\,q}\), we can always write it as \(k\,(\,p+iq\,)+t\,(\,p+iq\,)=(\,k\,+\,t\,)\,u\), where \(k\) is a positive integer and \(0\le t < 1\). Notice that the case \((\,k,t\,)=(1,0)\) is ruled out because \(w_\infty\) is not a primitive lattice point.
\end{rmk}

\noindent For every \(n\) we decompose \(\rm T\) using \(s_n\) as above, creating a sequence of tomahawks. In the limit, as the length of $\tau$ goes to zero, this sequence converges to a parallelogram, say \(P_\infty\), with sides given by \(u\) and \(z\) where \(z=t\,u+ v\) or \(z=t\,u- v\) depending on whether \(\arg(\,w\,)<\arg(\,u\,)\) or \(\arg(\,w\,)>\arg(\,u\,)\) respectively; see bottom and top portions of Figure \ref{fig:cutandpaste} respectively.

\smallskip

\noindent We now describe how to glue this parallelogram to a copy of \((\,\mathbb C,\,dz\,)\). For this purpose let \(\sigma_\infty\) be a segment with period \(w_\infty\) then we may glue the parallelogram \(P_\infty\) as described in \S\ref{sec:degtowalls} for representations with positive volume (the role of \(w\) in \S\ref{sec:degtowalls} here is played by \(z\,\)). More precisely, by distinguishing the two cases just discussed in parallel, we have the following cases:

\begin{itemize}
    \item[1.] \(\arg(\,w\,)<\arg(\,u\,)\). In this case the resulting structure has a saddle connection of period \(z=tu+v\) given by one pair of sides of \(P_\infty\). We set \(\mathfrak u=u\), \(\mathfrak w=z=tu+v\) and \(\mathfrak v=\mathfrak w +\mathfrak u -w_\infty=v+(\,1-k\,)u\). Note that $\mathfrak u$ and $\mathfrak w-\mathfrak v$ are in the same direction. So we then proceed as in \S\ref{sec:degtowalls} with these parameters.
    For later reference, by change of coordinates discussed in \S\ref{subsec:changeofcoord}, this set of parameters gives the same structure as $\mathfrak{u}'=u, \mathfrak{w}'=(k+t-1)u+v, \mathfrak{v}'=v$. 
    \smallskip
    \item[2.] \(\arg(\,w\,)>\arg(\,u\,)\) In this case the resulting structure has a saddle connection of period \(z=tu-v\) given by one pair of sides of \(P_\infty\). In this case we set \(\mathfrak{u,v,w}\) as follows: \(\mathfrak u=-u\), \(\mathfrak w=z=tu-v\) and \(\mathfrak v=\mathfrak w - \mathfrak u-w_\infty=-v-(\,k-1\,)u\). Note that $\mathfrak u$ and $\mathfrak w-\mathfrak v$ are in the opposite direction. Once again we proceed as in \S\ref{sec:degtowalls}. For later reference, by change of coordinates discussed in \S\ref{subsec:changeofcoord}, this set of parameters gives the same structure as $\mathfrak{u}'=-u, \mathfrak{w}'=(k+t-1)u-v, \mathfrak{v}'=-v$. 
\end{itemize}

%\begin{itemize}
%    \item[1.] \(\arg(\,w\,)<\arg(\,u\,)\). In this case the resulting structure has a saddle connection of period \(z=tu+v\) given by one pair of sides of \(P_\infty\). We set \(\mathfrak u=u\), \(\mathfrak w=z=tu+v\) and \(\mathfrak v=\mathfrak w +\mathfrak u -w_\infty=v+(\,1-k\,)u\). It is easy to see that \(\mathfrak u\) and \(\mathfrak w-\mathfrak v\) are parallel. So we can proceed as in \S\ref{sec:degtowalls}.
%    \smallskip
%    \item[2.] \(\arg(\,w\,)>\arg(\,u\,)\) In this case the resulting structure has a saddle connection of period \(z=tu-v\) given by one pair of sides of \(P_\infty\). In this case we set \(\mathfrak{u,v,w}\) as follows: \(\mathfrak u=-u\), \(\mathfrak w=-z=v-tu\) and \(\mathfrak v=\mathfrak w + \mathfrak u-w_\infty=-v-(\,k-1\,)u\). In this case \(\mathfrak u\) is parallel to \(\mathfrak w+\mathfrak v\). Once again we proceed as in \S\ref{sec:degtowalls}.
%\end{itemize}

\smallskip

\noindent Notice that here we used gothic letters in order to avoid confusions with the current notation and thus they should be understood as the corresponding letters in \S\ref{sec:degtowalls}. Two remarks are now in order.

\begin{rmk}\label{rmk:difflimits}
    The distinct choices for \( \mathfrak{u} \) in the two cases can be explained by keeping track of the orientations of the saddle connections. Notably, both the saddle connection with period $\mathfrak w$ and the one with period $\mathfrak w-\mathfrak v$ should point from $\pto B$ to $\pto W$. Consequently, the limiting structure belongs to \( \partial\textnormal{C}(\,u\,) \) if \(\arg(\,w\,)<\arg(\,u\,)\), and belongs to \( \partial\textnormal{C}(\,-u\,) \) if \(\arg(\,w\,)>\arg(\,u\,)\). Thus the limits are different.

    \noindent This can also be seen in the following way.
%It remains to show that the limit depends on which direction \(w\) approaches the ray \(r_{p\,q}\). 
The structures just realized have different relative periods, we have \(z=tu+v\) in one case and \(z=tu-v\) in the other case. We may easily notice that these values do not differ by any element in \(\mathbb Z[\,i\,]\). Moreover, they agree if and only if \(v=0\) which is not allowed.
\end{rmk}

%\begin{rmk}\label{rmk:difflimits}
%    The reason why \( \mathfrak{u} \) is parallel to \( \mathfrak{w+v} \) can be explained as follows. In \S\ref{sec:degtowalls}, we studied how a structure of cylinder type degenerates when the parameter \( w \) approaches a value on the boundary \( \partial\textnormal{C}(\,u\,) \). This clearly holds for every value of \( u \). Therefore, if we were dealing with \(\textnormal{C}(\,-u\,)\), we would obtain the same description with the necessary adjustments. Thus, if we replace \( u \) with its opposite \( -u \) then \( v \) should be replaced with \( -v \), and hence \( w - v \) should be replaced with \( w - (-v) = w + v \). Therefore, this means that the limiting structure belongs to \( \partial\textnormal{C}(\,u\,) \) if \(\arg(\,w\,)<\arg(\,u\,)\) or it belongs to \( \partial\textnormal{C}(\,-u\,) \) if \(\arg(\,w\,)>\arg(\,u\,)\). Thus the limits are different.
%\end{rmk}

\begin{rmk}
    Recall that a chamber of cylinder type \(\textnormal{C}(\,u\,)\) is defined as \( \{\,w\in\mathbb{C} \mid \mathfrak{Im}(\,\overline{u}\,w\,) > 1\,\}\); this is uniquely determined by \(u\), and \(w\) provides a global parameter on the whole chamber that extends to its boundary, defined by the equation \(\mathfrak{Im}(\,\overline{u}\,w\,) = 1\).  Seen as a subspace of \(\mathbb{C}\), the half-plane \(\textnormal{C}(\,u\,)\) is not affected by different choices of \(v\) of the form \(v' = mu + v\), where \(m \in \mathbb{Z}\). In fact, as discussed in \S\ref{subsec:changeofcoord}, different choices of \(v\) yield different global parameters of \(\textnormal{C}(\,u\,)\), say \(w\) and \(w'\), that differ by a translation, \textit{i.e.}, \(w' = w + mu\).
    This justifies our choices of $\mathfrak v$ in the list above.
    %Furthermore, this remark is not specific to the positive volume case but applies to every kind of chamber of cylinder type.
\end{rmk}

\smallskip

\noindent With these remarks we conclude the proof of Proposition \ref{prop:difflimits}.

\smallskip

\begin{figure}[!ht]
    \centering
    \includegraphics[width=1\linewidth]{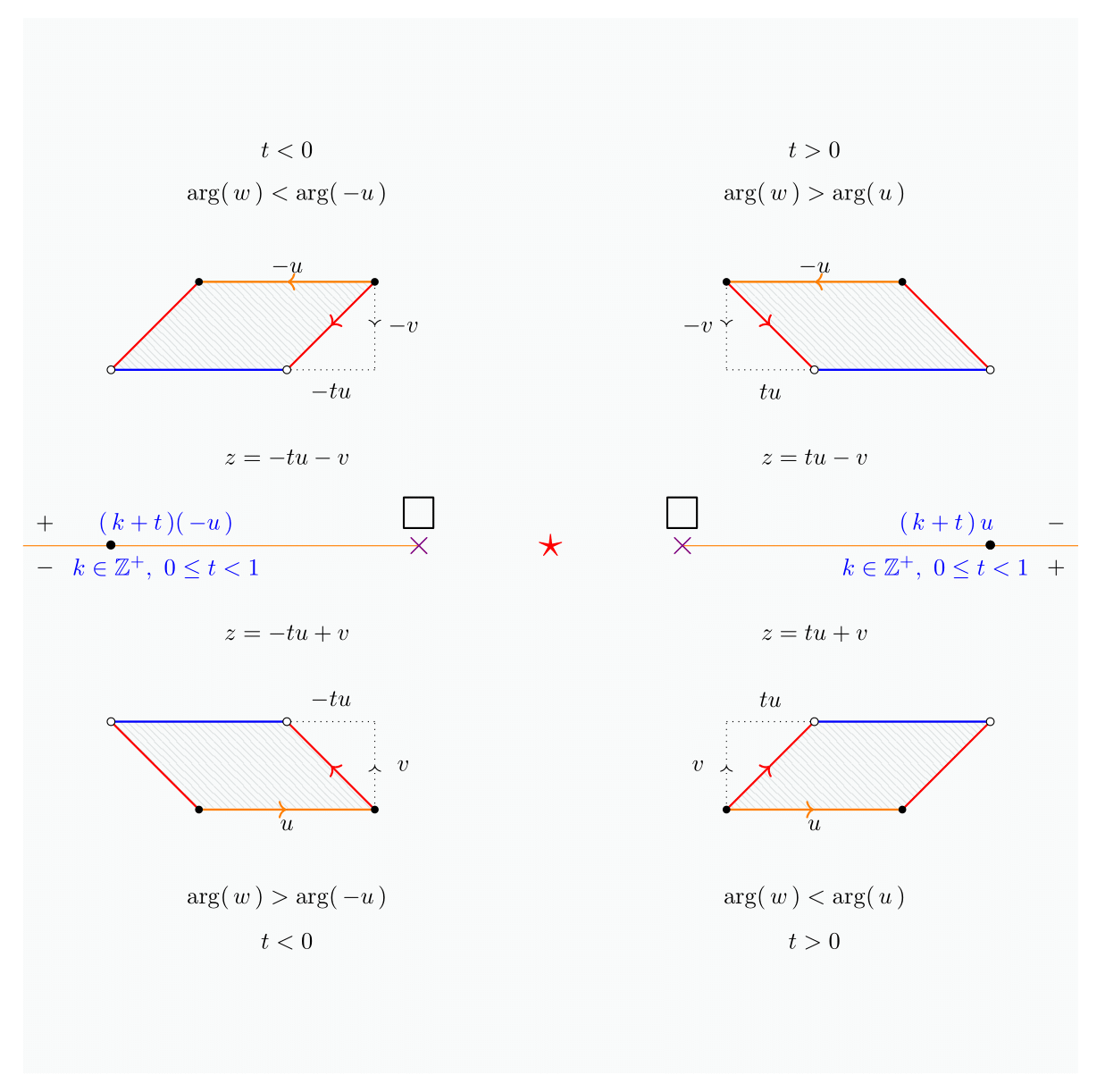}
    \caption{This figure, summarize in a schematically and direct way how a structure of torus type degenerates, as described in Section \ref{sssec:difflimits}. The red star at the center represents the nodal curve described in \ref{sssec:deglatticepoint}. The purple crosses are primitive lattice points, and we have already observed that if a sequence \( w_n \) tends to such values, the associated sequence of structures converges to a translation surface in the minimal stratum \(\Omega\mathcal M_1(2,-2)\) and are labeled with a square. The \( \pm \) signs represent the sides of the orange rays, which in the text are denoted as \( \pm r_{1\,0} \) For typographical reasons, only two opposite rays are represented.}
    \label{fig:degtoradpoints}
\end{figure}

\subsection{Geometry of positive leaves}\label{ssec:genposleaf}
We are now ready to describe the geometry of positive leaves. In the previous section we have given a detailed description of the chamber of torus type, parametrized by the open set \(\mathcal T\subset \mathbb C^*\) and how a sequence of structures degenerates when the parameter \(w\) approaches \(\partial\, \mathcal T\). We have seen that, unless \(w\) tends to a primitive element in \(\mathbb Z[\,i\,]\), the degenerate structure is a translation surface of cylinder type in the marked stratum \(\mathcal H_1(1,1,-2)\). Moreover, if we prolong the deformation beyond the wall we then fall in a chamber of cylinder type just described in \S\ref{sec:chamcyltype}. The whole leaf is thus obtained by gluing chambers of cylinder type to the chamber of torus type. 

\smallskip

\noindent The way of gluing chambers of cylinder type to the chamber of torus type is based on Proposition \ref{prop:difflimits} proved in the previous section. In the same notation as above, let \(u=p+iq\) be a primitive lattice point and let \(r_{p\,q}\) be the ray leaving from \(u\) with slope \(q/p\). We shall denote by \(-r_{p\,q}\) the ray leaving from \(-u\) defined as \(-r_{p\,q}=\{\,t\,(p+iq)\,|\,t\le-1\,\}\). Notice that this ray is symmetric to \(r_{p\,q}\) with respect to the origin. Recall that all rays are oriented so that they point towards the infinity. Hence we can distinguish left and right side. In \S\ref{sssec:difflimits}, we have shown that if a sequence \( (w_n) \) tends to a radial point of a ray \( r_{p\,q} \) from the right, \textit{i.e.} if \( \arg(\,w_n\,) < \arg(\,u\,) \), then the limit belongs to \( \partial\textnormal{C}(\,u\,) \). Conversely, if it tends from the left, \textit{i.e.} if \( \arg(\,w_n\,) > \arg(\,u\,) \), then the limit belongs to \( \partial\textnormal{C}(\,-u\,) \). In the same fashion, if a sequence tends to a radial point of a ray \( -r_{p\,q} \) from the right, \textit{i.e.} if \( \arg(\,w_n\,) < \arg(\,-u\,) \), then the limit belongs to \( \partial\textnormal{C}(\,-u\,) \). Conversely, if it tends from the left, \textit{i.e.} if \( \arg(\,w_n\,) > \arg(\,-u\,) \), then the limit belongs to \( \partial\textnormal{C}(\,u\,) \). Figure \ref{fig:degtoradpoints} summarize these degenerations.

\subsubsection{How to glue chambers}\label{sssec:mainthmposleaf} The way of gluing chambers of cylinder type to the chamber of torus type is now a consequence of the following observations here stated in form of Lemmata. We begin with the following one whose proof is a simple computation.

\begin{lem}\label{lem:boundaychamber}
    Let \(u=p+iq\) be a primitive lattice point and let \(\textnormal{C}(\,u\,)\) be the corresponding chamber of cylinder type. Then its wall \(\partial \textnormal{C}(\,u\,)\) is a straight line with slope equal to \(q/p\).
\end{lem}

\begin{proof}
    Let \(u\in \mathbb Z[\,i\,]\) be a primitive lattice point and let \(\textnormal C(\,u\,)\) the corresponding chamber of cylinder type. Recall that, as a consequence of Proposition \ref{prop:posleafpres}, we assume \(\textnormal{Im}(\,\rho\,)=\mathbb Z[\,i\,]\). Thus, the wall \(\partial\textnormal{C}(\,u\,)\) is defined by
    \begin{equation}
        \partial\textnormal{C}(\,u\,)=\,\Big\{\,w\in\mathbb C\,\,|\,\, \mathfrak{Im}\,\big(\,\overline{u}\,w\,\big)=1 \,\Big\}.
    \end{equation}
    Set \(u=p+iq\) and \(w=x+iy\). A direct computation shows that \(\mathfrak{Im}\big(\,\overline{u}\,w\,\big)=py-qx\). Thus \(\partial\textnormal{C}(\,u\,)\) is the straight line defined by the equation \(-qx+py=1\) which has slope equal to \(q/p\) as desired.
\end{proof}

\noindent Therefore \(\partial\textnormal{C}(\,u\,)\) is an affine subspace of dimension one. Recall that if \(v\in\mathbb Z[\,i\,]\) satisfies \(\mathbb Z[\,i\,]=\mathbb{Z}u\oplus\mathbb{Z}v\), then \(ku+v\in\partial\textnormal{C}(\,u\,)\) for every integer \(k\). In particular we have already observed that \(v\in\partial\textnormal{C}(\,u\,)\), see \S\ref{sec:degtowalls}. We introduce an affine frame on \(\partial\textnormal{C}(\,u\,)\) such that \(v\) has coordinate \(s=0\) and oriented in such a way \(\partial\textnormal{C}(\,u\,)\) bounds the chamber on its left. The following lemma is easy to establish and the proof is left to the reader.

\begin{lem}\label{lem:afframe}
    There exist a unique affine framing on \(\partial\textnormal{C}(\,u\,)\) such that if \(w=su+v\) then its affine coordinate is \(s\). Moreover this framing extends to an affine framing \((s,r)\) on \(\overline{\textnormal{C}(\,u\,)}\) and \(\textnormal{C}(\,u\,)\) corresponds to the open set determined by the condition \(r>0\).
\end{lem}

\noindent In what follows we shall rely on this affine framing. We can now state a third Lemma which is a direct corollary of the constructive proof of Proposition \ref{prop:wallchar}, Proposition \ref{prop:difflimits} and Lemma \ref{lem:afframe} . For our convenience, in the following statement we set \(\lambda=k+t\).

\begin{lem}\label{lem:gluinglem}
    Let \(u\) be a primitive lattice number and let \((w_n)\in\mathcal T\) be a sequence converging to a radial point \(w_\infty\in\,\big\{\,\lambda\,u\,|\,\lambda\ge1 \,\big\}\) from the right, \textit{i.e.} \( \arg(\,w_n\,) < \arg(\,u\,) \). Then the limit structure is a translation surface of cylinder type in \(\partial\textnormal{C}(\,u\,)\) with parameter \(s=\lambda-1\). Similarly, if \((w_n)\in\mathcal T\) be a sequence converging to a radial point \(w_\infty\in\,\big\{\,\lambda\,u\,|\,\lambda\ge1 \,\big\}\) from the left, \textit{i.e.} \( \arg(\,w_n\,) > \arg(\,u\,) \), then the limit structure is a translation surface of cylinder type in \(\partial\textnormal{C}(\,-u\,)\) with parameter \(s=1-\lambda\). In both cases, if \(\lambda=1\) the resulting structure is of cylinder type and lies in the minimal stratum \(\Omega\mathcal M_1(2,-2)\).
\end{lem}

\noindent These Lemmata tell us how to glue chambers of cylinder type to the chamber of torus type. We begin with the following: 

\begin{rmk}\label{rmk:boundaryofchamtor}
    The chamber of torus type \(\mathcal T\) is an open set in \(\mathbb C^*\) whose boundary \(\partial\mathcal T\) is given by an isolated point (that fills the origin) and a collection of countably many rays leaving from primitive lattice points. From every lattice point \(u=p+iq\) leave exactly two rays, say \(r_{p\,q}^{\pm}\), each one of which parametrizes the limits structures obtained from degenerations from the left or right. In what follows we shall agree that \(r_{p\,q}^+\) and \(r_{p\,q}^-\) bounds \(\mathcal T\) on the right and left respectively -- this explains the labels ``\(\pm\)'' in Figure \ref{fig:degtoradpoints}. Notice that on both of these rays leaving from \(u\) there is a point labeled with \(\lambda\ge1\) that represent different structures. We finally observe that the completion of \(\mathcal T\) is not \(\mathbb C\) nor homeomorphic to it.
\end{rmk}

\noindent We then glue chambers of cylinder type to the \textit{completion} of \(\mathcal T\) as follows. Let \(u\) be any primitive lattice point and let \(\textnormal C(\,u\,)\) be the corresponding chamber of cylinder type. As a consequence of Lemma \ref{lem:boundaychamber}, its boundary \(\partial\textnormal C(\,u\,)\) is an affine subspace of \(\mathbb C\), namely a straight line for the Euclidean metric. According to this orientation, the boundary can be seen as the union of two half-rays, say \(\partial\textnormal{C}(\,u\,)^{\pm}\), where ``\(+\)'' labels the half-ray parametrized by non-negative values of \(s\) and ``\(-\)'' the half ray parametrized by non-positive values. Recall that \(v\) is parametrized with zero according to the choice of the affine frame. We glue the half-ray \(r_{p\,q}^+\) with \(\partial\textnormal{C}(\,u\,)^+\) and similarly we glue the half-ray \(-r_{p\,q}^-\) with \(\partial\textnormal{C}(\,u\,)^-\). The identification is made so that:
\begin{itemize}
    \item for every \(s\ge0\), the point of \(\partial\textnormal C(\,u\,)^+\) parametrized by \(s\) is glued to the point of \(\,r_{p\,q}^+\) parametrized by \(\lambda-1\); and
    \smallskip
    \item for every \(s\le 0\), the point of \(\partial\textnormal C(\,u\,)^-\) parametrized by \(s\) is glued to the point of \(-r_{p\,q}^-\) parametrized by \(1-\lambda\).
\end{itemize}

\noindent Notice that gluing the chamber of cylinder type \(\textnormal{C}(\,-u\,)\) to \(\mathcal T\) along their boundary as described above involves the rays \(r_{p\,q}^-\) and \(-r_{p\,q}^+\). By applying the same process for every chamber of cylinder type we obtained a topological surface, say \(\mathfrak S\), that parametrizes all structures of torus type with prescribed absolute periods and their degenerations. In other words, \(\mathfrak S\) turns out to be bigger than the desired isoperiodic fiber because it also parametrizes those structures that lie in the minimal strata \(\Omega\mathcal M_1(2,-2)\) and the nodal curve discussed in \S\ref{ssec:degtorustowardzero}. In the next section we shall study the geometry of this space and determine the desired isoperiodic fiber. Lemma \ref{lem:gluinglem} ensures that after this gluing a continuous deformation inside any chamber can be prolonged beyond any wall in a continuous way. As a consequence, every structure in a chamber of cylinder type is thus connected by a path to a structure into the chamber of torus type. Therefore, the resulting space \(\mathfrak S\) is path-connected and hence connected. 

\smallskip

\subsubsection{Translation structure on positive leaves} We now describe the geometry of positive leaves. Recall that we are assuming that \(\textnormal{Im}(\,\rho\,)=\mathbb Z[\,i\,]\). First, we observe that each chamber of cylinder type, along with its wall, forms a closed half-plane in the complex plane \(\mathbb{C}\). According to Remark \ref{rmk:boundaryofchamtor}, the completion \(\overline{\mathcal{T}}\) cannot be regarded as a polygon nor as a subspace of \(\mathbb C\). However it still carries a translation structure with geodesic boundary and corner points of magnitude \(2\pi\) at the primitive lattice points. Lemma~\ref{lem:boundaychamber} and Lemma~\ref{lem:gluinglem} ensure that these gluings are performed via translations, thereby equipping \(\mathfrak{S}\) with a translation structure, namely it is a Riemann surface equipped with an abelian differential, say \(\omega\). The following holds.

\begin{lem}
    \((\,\mathfrak S,\,\omega\,)\) is a translation surface with infinitely many zeros of order two. Furthermore, a point of \(\mathfrak S\) is a zero of \(\omega\) if and only if it corresponds to a translation surface in the minimal stratum \(\Omega\mathcal M_1(2,-2)\).
\end{lem}

\begin{proof}
    Every point in \(\mathcal T\) is a regular point for the translation structure induced by the embedding \(\mathcal T \hookrightarrow\mathbb C\) and hence a regular point for \((\,\mathfrak S,\,\omega\,)\). These points correspond to translation surfaces of torus type realized by gluing a torus and a copy of the complex plane along a slit and we have already seen that all structures that thus arise belong to \(\mathcal H_1(1,1,-2)\).
    \smallskip
    
    \noindent In the same fashion, given any primitive lattice points \(u\), every point of the chamber of cylinder type \(\textnormal{C}(\,u\,)\) is a regular point for the translation structure determined by the embedding \(\textnormal{C}(\,u\,)\hookrightarrow \mathbb C\) and hence a regular point of \(\mathfrak S\). In \S\ref{sec:chamcyltype} we have seen that all of these points correspond to structures that belong to \(\mathcal H_1(1,1,-2)\). Therefore, the zeros of \(\omega\), if any, must arise from the identification of \(\partial\mathcal T\) with the boundaries of chambers of cylinder type \(\partial\textnormal{C}(\,u\,)\).
    \smallskip
    
    \noindent Let \(u\) be any primitive lattice point. According to Lemma \ref{lem:gluinglem} every point of \(\overline{\mathcal T}\) of the form \(\lambda\,u\), with \(\lambda>1\), is identified with the boundary point of \(\partial\textnormal{C}(\,u\,)\) or \(\partial\textnormal{C}(\,-u\,)\) of the form \((\lambda-1)u+v\) or \(-(\lambda-1)u-v\), where \(v\in\mathbb Z[\,i\,]\) is such that \(\langle\,u,\,v\,\rangle\) is a positively oriented basis for \(\mathbb Z[\,i\,]\). As an artifact of our definitions and constructions they both are regular boundary points and represent the same transitional structure in \(\mathcal H_1(1,1,-2)\). Thus, once identified, they determine a regular point on \(\mathfrak S\).
    \smallskip
    
    \noindent We finally consider the case \(\lambda=1\). We have already seen that \(w=\pm u\in\partial \overline{\mathcal T}\), \(v\in\partial\textnormal{C}(\,u\,)\) and \(-v\in\partial\textnormal{C}(\,-u\,)\) all represent a translation surface in the minimal stratum \(\Omega\mathcal M_1(2,-2)\), see \S\ref{sec:degtowalls} and \S\ref{sssec:deglatticepoint}. In fact, they all represent the same translation surface, see Figure \ref{fig:degtoradpoints} for \(t=0\). These points are all identified according to the gluing process described in \S\ref{sssec:mainthmposleaf}. Both \(\pm v\) are regular boundary points of their respective chambers of cylinder type. On the other hand, both \(\pm u\) are corner points, and their boundary angles each amount to \(2\pi\). Thus, once these points are identified, they form a branch point with angle \(6\pi\), namely a zero of order two for the resulting translation surface.
\end{proof}

\subsubsection{Topology of \(\mathfrak{ML}(\,\rho,\mu\,)\)}\label{sssec:profuniquepos} By design, there is a natural embedding \(\imath\colon\mathcal{T} \cup \{\,0\,\} \xhookrightarrow{\quad} \mathfrak{S}\), which defines a distinguished point that we label as \(o\). This point is a regular point of \(\mathfrak{S}\) and corresponds to the nodal curve arising from the degeneration described in \S\ref{ssec:degtorustowardzero}. The isoperiodic fiber is thus identified with the subset
\begin{equation}
    \mathfrak{ML}(\,\rho,\mu\,)\,=\,\mathfrak S\,\setminus\,\Big(\,\{\,o\,\}\,\cup\,\{\,\textnormal{zeros of}\,\omega\,\}\,\Big).
\end{equation}

\noindent In what follows we shall denote by \(\textnormal{deg}(\,\rho,\,\mu\,)=\Big(\,\{\,o\,\}\,\cup\,\{\,\textnormal{zeros of}\,\omega\,\}\,\Big)\), \textit{i.e.} the subset of degenerated structures arising as limits of degenerating sequences. As a consequence of our construction, there exists a unique isoperiodic leaf realizing a non-trivial representation \(\rho\) satisfying the condition \(\textnormal{vol}(\,\rho\,) > 0\). We also observe that the subset \(\textnormal{deg}(\,\rho,\,\mu\,) \subset \mathfrak{S}\) consists of isolated points. Consequently, the space \(\mathfrak{ML}(\rho, \mu)\) is path-connected and therefore connected and \(\mathfrak S\) is its metric completion. %This proves Theorem~\ref{thm:Unique} in the case of positive leaves.

\smallskip

\subsection{Topology of positive marked leaves} The study of positive leaves developed so far has allowed us to draw conclusions about their geometric properties. Now, we aim to understand their topology. We already know that they are connected surfaces, and here we seek to determine their type. 

\smallskip 

\noindent To understand the topology of the leaves, it is first necessary to recall some key concepts. In \S\ref{sec:chamcyltype}, we observed that the marking of zeros on a structure of cylinder type is equivalent to determining an oriented simple curve, which can be completed into a system of generators in homology, see Remark \ref{rmk:markdeterminesys}. Changing the opposite marking defines the opposite system of generators, namely the same curve with the opposite orientations. It is not had to see that the same observation extends to structures of torus type. 

\begin{lem}\label{lem:marktorus}
    Marking the zeros on a torus-type structure in the stratum \(\Omega\mathcal{M}_1(1,1,-2)\) is equivalent to determining an oriented simple closed curve that separates the zeros, such that one lies on the left and the other on the right with respect to the orientation of the curve.
\end{lem}

\subsubsection{\(\mathfrak S\) is symmetric} In the present section we aim to show the existence of a symmetric of order two. As a consequence of Lemma \ref{lem:marktorus}, if a given marking determines a curve, say \(\alpha\) with period \(u\), the opposite marking determines the curve \(\alpha^{-1}\), which has period \(-u\). Consequently, two opposite points in \(\mathcal{T}\) define structures that differ only by the marking of the zeros. Therefore, the natural involution given by \(w \mapsto -w\) identifies structures with different markings, and the quotient map is nothing more than the restriction of the forgetful map to \(\mathcal{T}\). Let us consider the natural embedding \(\imath\colon\mathcal{T} \cup \{\,0\,\} \xhookrightarrow{\quad} \mathfrak{S}\) defined in \S\ref{sssec:profuniquepos} and the action of \(\mathbb Z_2\) on \(\mathbb C\) given by rotations about the origin. Notice that \(\mathcal T\,\cup\,\{\,0\,\}\) is invariant under this action. Then the following holds. 

\begin{lem}\label{lem:involution}
    Let \(\mathfrak F\colon\mathfrak{ML}(\,\rho, \mu\,)\to \mathfrak{L}(\,\rho, \mu\,)\) be the forgetful map and consider its restriction to \(\imath\,\big(\,\mathcal T\,\cup\,\{\,0\,\}\,\big)\). Then the map \(\imath\,\big(\,\mathcal T\,\cup\,\{\,0\,\}\,\big) \to \mathfrak F\,\circ\,\imath\,\big(\,\mathcal T\,\cup\,\{\,0\,\}\,\big)\) is \(\mathbb Z_2-\)invariant and it naturally extends to a \(\mathbb Z_2-\)invariant map on \(\partial\mathcal T\).
\end{lem}

\noindent We would like to further extend this involution to the whole surface \(\mathfrak S\). Given a primitive lattice point and its opposite, say \(\pm u\), the corresponding chambers of cylinder type \(\textnormal{C}(\,u\,)\) and \(\textnormal{C}(\,-u\,)\) are disjoint seen as open subset of \(\mathbb C\) and symmetric with respect to the origin \(0\in\mathbb C\). In fact, we have already observed that \(w\in\textnormal{C}(\,u\,)\) if and only if \(-w\in\textnormal{C}(\,-u\,)\). Moreover this symmetry extends to their boundaries. As a consequence of Remark \ref{rmk:markdeterminesys} and Lemmata \ref{lem:gluinglem} and \ref{lem:involution} we have the following:

\begin{cor}
    \((\mathfrak S,\,\omega)\) admits a symmetry of order two that preserves the subset \(\textnormal{deg}(\,\rho,\,\mu\,)\). Moreover, the quotient map restricted to \(\mathfrak{ML}(\,\rho, \mu\,)\) yields the forgetful map \(\mathfrak F\colon\mathfrak{ML}(\,\rho, \mu\,)\to \mathfrak{L}(\,\rho, \mu\,)\).
    %The forgetful \(\mathfrak F\colon\mathfrak{ML}(\rho, \mu)\to \mathfrak{L}(\rho, \mu)\) is induced by an involution that preserves the subset \(\big(\,\{\,o\,\}\,\cup\,\{\,\textnormal{zeros of}\,\omega\,\}\,\big)\).
    %Then the map \(\imath\,\big(\,\mathcal T\,\cup\,\{\,0\,\}\,\big) \to \mathfrak F\,\circ\,\imath\,\big(\,\mathcal T\,\cup\,\{\,0\,\}\,\big)\) extends to an involution of \(\mathfrak S\) and preserves the subset \(\big(\,\{\,o\,\}\,\cup\,\{\,\textnormal{zeros of}\,\omega\,\}\,\big)\). In particular, 
\end{cor}

\noindent Theorem \ref{thm:Unique} for leaves of positive volume is now a straightforward consequence of our discussion so far. Since the forgetful map is continuous the following holds.

\begin{cor}
    \(\mathfrak{L}(\,\rho,\mu\,)\) is path-connected and hence connected.
\end{cor}

\subsubsection{Chamber structure of \(\mathfrak{L}(\,\rho,\mu\,)\) and its geometry} The isoperiodic fiber \(\mathfrak{L}(\,\rho,\mu\,)\) also admits a decomposition into chambers and walls as the marked leaf \(\mathfrak{ML}(\,\rho,\mu\,)\). This can be directly realized as follows. We first note that the quotient of the chamber of torus type is a topological disk whose boundary is given by infinitely many slits, each originating and pointing towards infinity at a primitive lattice points with non-negative imaginary part. In other words it is homeomorphic to

\begin{equation}
    \displaystyle\mathcal{HT}\,=\,\left(\,\,\overline{\mathbb H}\,\setminus \,\,\bigcup_{\mathclap{\substack{(p,q)\,\in\,\mathbb Z^2,\,q \ge 0\\ \gcd(\,p,\,q\,)=1}}} \left\{\, t\,(\,p\,+\,iq\,)\,|\, t \ge 1\, \right\}\,\,\right)\,/\,\sim\,\,,
\end{equation}

\noindent where \(\mathbb H\) is the upper-half plane and \(\sim\) is an equivalence relation on \(\partial\mathbb H\) such that two points, say \((s_1,0)\) and \((s_2,0)\), are related by \(\sim\) if and only if they coincide or \(s_1=-s_2\) and \(0\le s_i \le1\). Let us recall our notation in use: \(r_{p\,q}=\{\, t\,(\,p\,+\,iq\,)\,|\, t \ge 1\,\}\) and we denote the sides of the slits with \(r_{p\,q}^\pm\) where, as in Remark \ref{rmk:boundaryofchamtor}, we shall agree that \(r_{p\,q}^+\) and \(r_{p\,q}^-\) bounds \(\mathcal{HT}\) on the right and left respectively. If \(u=p+iq\), we then glue the chamber of cylinder type \(\overline{\textnormal{C}}(\,u\,)\) as done in \S\ref{sssec:mainthmposleaf}. We shall denote by \(\mathfrak D\) the resulting surface. The following holds and it is easy to establish.

\begin{lem}
     \(\mathfrak{D}\) is homeomorphic to an open disk and admits a chamber decomposition.
\end{lem}

\noindent Regarding its complex structure, the surface must be either a complex plane or a complex disk. By Corollary~\ref{cor:hyp}, it follows that the surface is biholomorphic to a disk. We can easily notice that \(\mathfrak D\) naturally carries a half-translation structure determined by a quadratic differential, say \(q\), with a simple pole and infinitely many zeros of order one, that is conical singularities of angle \(3\pi\). By design the following holds.

\begin{cor}
    The forgetful map extends to a branch covering \(\overline{\mathfrak F}\colon\,(\,\mathfrak S,\,\omega\,)\longrightarrow(\,\mathfrak D,\,q\,)\) of degree two, with ramification points at \(\textnormal{deg}(\,\rho,\,\mu\,)\), that preserves the chamber decomposition; \textit{i.e.} it maps chambers to chambers and walls to walls. Moreover the following equation holds: \(\overline{\mathfrak F}^*q=\omega\). 
\end{cor}

\noindent By using this result, combined with the following lemma, we can conclude that a positive leaf is a Loch Ness Monster, \textit{i.e.} a surface of infinite genus with one end. This gives Theorem \ref{thm:LochNess} for positive leaves. Notice that the singularities in the metric completion of the leaf are located at the lattice points in the chamber of torus type and are therefore a discrete infinite set.

\begin{lem}\label{lem:double_cover}
Let $S$ be a topological surface, and let $\mathcal{P}$ denote a topological surface homeomorphic to the plane $\mathbb{C}$. Suppose there exists a branched double cover $\pi:S\to\mathcal{P}$, branched over a infinite, discrete set of points. Then $S$ is a Loch Ness Monster.
\end{lem}

\begin{proof}
First we show that $S$ has one end. Suppose not. Since $\pi$ is a branched double cover and $\mathcal{P}$ is one-ended, $S$ has two ends, each covering the unique end of $\mathcal{P}$. Thus for some large enough simply connected compact set $K\subset\mathcal{P}$, the preimage of $\mathcal{P}-K$ consists of two connected components. But $\mathcal{P}-K$ contains ramification points of $\pi$, so its preimage must be connected, a contradiction.

\smallskip

\noindent Next we show that $S$ has infinite genus. Indeed, take a disk $D$ in $\mathcal{P}$ containing exactly two ramification points. Then $\pi^{-1}(D)$ is connected, and by Riemann-Hurwitz formula, has Euler characteristic $0$. It has either one or two boundary components. If it has one, then by capping off the boundary component we obtain a closed surface with Euler characteristic $1$, which cannot be. So $\pi^{-1}(D)$ has two boundary components, and is therefore a cylinder. There are infinitely many disjoint pairs of ramification points, so we obtain infinitely many handles of $S$. By classification of topological surfaces, $S$ is thus a Loch Ness Monster.
%take any sequence of nested connected open sets with compact boundary $U_1\supset U_2\supset U_3\supset\cdots$ escaping every compact set, the images $\pi(U_1)\supset\pi_1(U_2)\supset\cdots$ gives such a sequence for the plane $\mathcal{P}$.
\end{proof}

\subsection{Veech group} Finally, we can determine the Veech group of a positive leaf seen as a translation surface and hence prove Theorem~\ref{thm:Veech} for positive leaves. 

\begin{prop}\label{prop:Veechpositive}
The Veech group of a positive leaf is a conjugate of $\slz$.
\end{prop}

\begin{proof}
The action of \(\glplus\) on the (marked or unmarked) stratum commutes with the period character. Therefore, the action of an element \(g\in\glplus\) takes one positive leaf with absolute period group $\Gamma$ to another positive leaf whose absolute period group is \(g\Gamma\). Since the translation structure of the leaf is characterized by its absolute period group, the Veech group is exactly the stabilizer of \(\Gamma\) in \(\glplus\), which is a conjugate of $\slz$. This can also be proved by looking at the torus chamber: any element in the Veech group must leave the set of primitive elements in \(\Gamma\) invariant.
\end{proof}

\section{Negative leaves}\label{sec:negleaves}
\noindent In this section, we give a complete description of negative leaves, namely those loci \(\mathfrak{ML}(\,\rho,\,\mu\,)\) where \(\textnormal{vol}(\,\rho\,)\) is negative. Like in \S\ref{sec:posleaves}, we may observe that given two representations, say \(\rho_1,\,\rho_2\colon\shomolzoo\longrightarrow \mathbb C\) with negative volume there exists \(A\in\glplus\) such that \(\rho_2\,=\,A\cdot\rho_1\). Such a linear transformation \(A\) determines an homeomorphism, say \(F_A\), of leaves 
    \begin{equation}
        F_A\colon\mathfrak{ML}(\,\rho_1,\,\mu\,)\longrightarrow \mathfrak{ML}(\,\rho_2,\,\mu\,),
    \end{equation}

\noindent and the crucial point is the following:

\begin{prop}\label{prop:negleafpres}
    \(F_A\) preserves the structure of leaves, \textit{i.e.} it maps chambers to chambers and walls to walls. 
\end{prop}

\noindent The same argument used to show Proposition \ref{prop:posleafpres} works \textit{mutatis mutandis}. Therefore, in what follows we shall assume \(\textnormal{Im}(\,\rho\,)=\mathbb{Z}[\,i\,]\) for simplicity. In \S\ref{sec:chamcyltype} we have described chambers of cylinder type with negative volume. In the present section we deal with chambers of degenerate type. 
%Here, $\Gamma$ is a lattice and the image of a symplectic basis of the homology is an indirect basis of $\mathbb{C}$. Once again, up to $\glplus$-action, we may assume $\Gamma=\mathbb{Z}[\,i\,]$. Chambers of cylinder type are already described in Section \S\ref{sec:chamcyltype} and we have seen that chambers of torus type only exist in positive leaves, see Lemma~\ref{lem:VolConstraints}. In the present section we focus on chambers of degenerate type.

\subsection{Realizing structures of degenerate type}\label{ssec:realdegcorestruc} In \S\ref{ssec:core} we have already given a glimpse of how to realize a structure of degenerate type, see Figure \ref{fig:degenerate}. As in \S\ref{ssec:chambdes}, let \(\rho\colon\shomolzoo\longrightarrow \mathbb C\) be a non-trivial representation and let \(\{\,\alpha,\,\beta\,\}\) be two generators that represent a pair of simple closed curves that intersect only once (up to homotopy). We shall assume \(\{\,\alpha,\,\beta\,\}\) to be a positively oriented basis, \textit{i.e.} their intersection number is equal to one. Throughout this section let \(u,v\in\mathbb C^*\) denote \(u=\rho(\,\alpha\,)\) and \(v=\rho(\,\beta\,)\) respectively. Notice that \(u,v\) cannot be both zero because the representation is assumed to be non-trivial. We emphasizes the fact that we are assuming \(\textnormal{Im}(\,\rho\,)=\mathbb{Z}[\,i\,]\), and hence we suppose \(\mathfrak{Im}(\,\overline u\, v\,)=-1\) (the negative sign here comes from the assumption that the volume is negative).

\smallskip

\noindent We first describe in details how to realize a structure with core of degenerate type for the readers' convenience. We recall in the first place that a translation surface is of degenerate type if it arises from the identification of pairs of opposite sides in the complement of a \textit{convex} hexagon in the infinite plane, see Lemma~\ref{lem:VolConstraints} and the bottom part of Figure~\ref{fig:choiceofw}. In fact, a concave hexagon yields a structure of cylinder type, for instance see Figure \ref{fig:wallcasetwo}. It is worth mentioning that the sequence of sides is invariant under small deformations and \(\glplus\) action. Before continuing we need to introduce the following conventions for our convenience.

\medskip

\textit{Convention.} We recall the following convention that we tacitly used so far. As usual we shall use bold Latin letters to denote elements of \(\mathbb{C}\) as points in the plane and then use italic letters to denote the elements of \(\mathbb{C}\) as complex numbers or vectors. Therefore, \( \pto P + v \) represents the point obtained by translating \( \pto P \) %in the direction determined 
by \( v \).

\medskip

\textit{Convention.} Throughout the present section and whenever it will be convenient we shall denote triangles with symbols \(\Delta\) and \(\nabla\). Although the former is of general use the latter may seem unusual. The reason of this choice will be clear later on along our discussion. In both cases, if \(a,b\,\in\mathbb C\) are two non-zero and non real-collinear complex numbers, we denote by \(\nabla(\,a,\,b\,)\) or \(\Delta(\,a,\,b\,)\) a triangle in \(\mathbb C\) isometric to the unique triangle with vertices at \(\{\,0,\,a,\,b\,\}\). 

\medskip

\noindent Let \(\pto B\in\mathbb C\) be any point and consider \(\pto B\,+\, u\) and \(\pto B\,+\, v\). Since \(u,\,v\) are nonzero and distinct, these points are pairwise distinct. Moreover, since \(\mathfrak{Im}(\,\overline u\, v\,)=-1\neq0\) these points are not aligned and they thus correspond to the vertices of some triangle, say \(\nabla(\,u,\,v\,)\), with signed area equal to \(-1/2\). Let \(w\in\mathbb C\) and then define another collection of points as follows:
\begin{equation}
    \pto W=\pto B\,+\,w, \quad \pto W\,+\,v=\pto B\,+\,(\,v\,+\,w\,), \quad \pto W\,-\,u=\pto B\,+\,(\,-u\,+\,v\,+\,w\,).
\end{equation}

\noindent It is not hard to see that \(\pto W,\,\pto W\,+\,v,\,\pto W\,-\,u\) determine a triangle isometric to \(\nabla(\,u,\,v\,)\). We next define the following polygonal, say \(\mathcal C(\,w\,)\), of segments
\begin{equation}\label{eq:chain}
    \pto B\,\longmapsto\,\pto W\,\longmapsto\,
    \pto B\,+\,u\,\longmapsto\,\pto W\,+\,v\,\longmapsto\,
    \pto B\,+\,v\,\longmapsto\,\pto W\,-\,u\,\longmapsto\,
    \pto B.
\end{equation}

\noindent Notice that for \(w=0\) this chain degenerates to a quadrilateral. In principle, \(\mathcal{C}(\,w\,)\) may be self-intersecting and thus may not necessarily bound any polygon. Moreover, if it bounds a polygon, namely a hexagon, then this may be concave or even a parallelogram with two pairs of aligned sides. Notice that the shape of \(\mathcal{C}(\,w\,)\) depends only on \(w \in \mathbb{C}^*\). In fact, since \(u\) and \(v\) are given, the points \(\pto B\,+\,u\) and \(\pto B\,+\,v\) are uniquely determined by the choice of the starting point \(\pto B\). In the same fashion, \(\pto W, \pto W\,+\,v,\) and \(\pto W\,-\,u\) all depend on the choice of \(\pto B\) and \(w \in \mathbb{C}^*\). However, if we replace \(\pto B\) with any other point, say \(\pto B + z\) for some \(z \in \mathbb{C}\), then the resulting chain starting from \(\pto B\, +\, z\) coincides with \(\mathcal{C}(\,w\,)\, +\, z\), as every point of the chain is shifted by the same translation. Therefore, the shape of the chain, and consequently its property of bounding an embedded polygon, depends only on the choice of \(w\). We shall determine later on necessary and sufficient conditions on \(w\) to ensure that \(\mathcal C(\,w\,)\) bounds an embedded hexagon, say \(H(\,u,\,v,\,w\,)\subset \mathbb C\subset\cp\) with no pairs of sides aligned. In the case \(w\) is any of such values, then \(\Omega=\cp\setminus\,\textnormal{int}\,\Big(\,H(\,u,\,v,\,w\,)\,\Big)\) is a domain of the pole whose boundary is the polygonal \(\mathcal C(\,w\,)\) made of three pairs of parallel sides for the standard Euclidean metric by design. %By design \(H(\,u,\,v,\,w\,)\) has three pairs of parallel sides 
Parallel edges of \(\partial\Omega\) can be glued by using translations and the resulting space is a translation surface of degenerate type in \(\mathcal H_1(1,1,-2)\) with absolute period prescribed by \(\rho\).

\smallskip

\subsection{Embedded hexagons}\label{ssec:embdhex} To understand the chambers of structures with a degenerate core, it is first necessary to determine for which values of \(w\) the chain described above defines a convex hexagon without pairs of aligned sides. We determine explicit conditions and from these we can directly deduce the openness of chambers of degenerate type. We begin with the following standard fact from Euclidean geometry whose proof is left to the reader.

\begin{lem}\label{lem:baselemma}
    Let \(H\subset \mathbb C\) be a hexagon with piecewise geodesic boundary with respect to the standard Euclidean metric. Then \(H\) is convex with no pairs of sides aligned if and only if all interior angles have magnitude less than \(\pi\). 
\end{lem}

\begin{rmk}
    The fact that a chamber of degenerate type is open can be deduced from the following heuristic argument. Suppose that for a given value of \(w\), the chain \(\mathcal{C}(\,w\,)\) described above defines a convex hexagon without pairs of aligned sides. This condition is equivalent to requiring that all interior angles have a measure strictly less than \(\pi\). It can be observed that small perturbations of \(w\) lead to small perturbations of the chain \(\mathcal{C}(\,w\,)\) and, in particular, sufficiently small perturbations produce hexagons whose interior angles all remain less than \(\pi\). Therefore, the choice of \(w\) is an open condition.
\end{rmk}

\begin{figure}[!ht]
    \centering
    \includegraphics[width=0.75\linewidth]{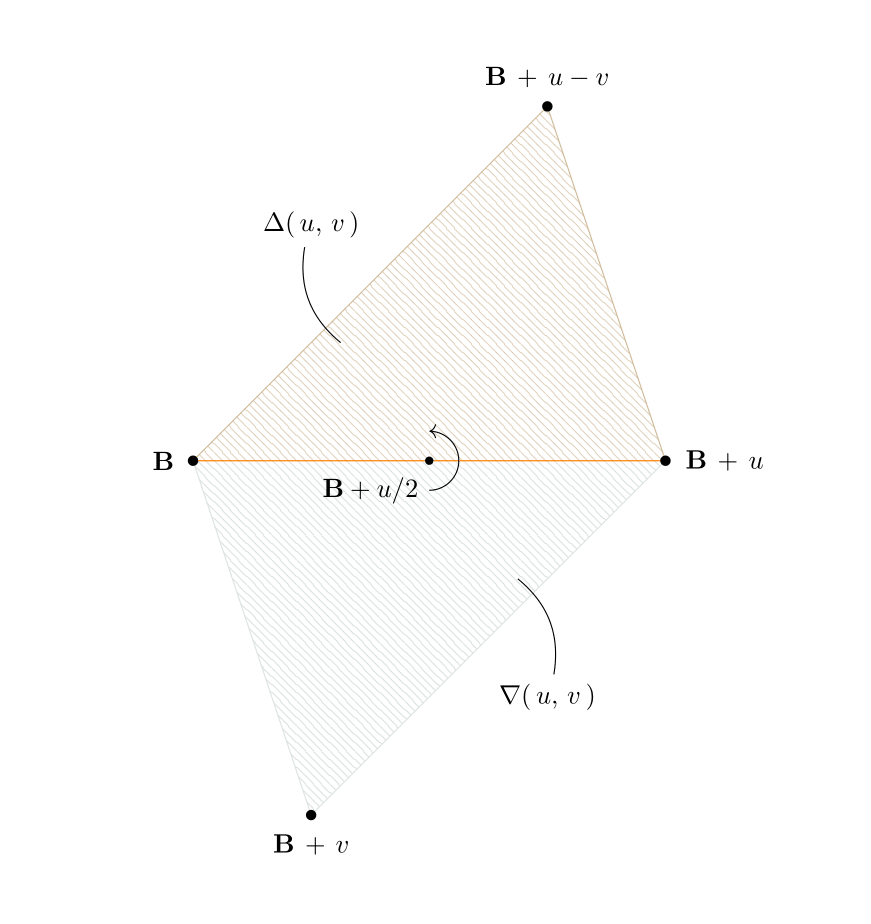}
    \caption{Schematic representation of triangles \(\nabla(\,u,\,v\,)\) and \(\Delta(\,u, v\,)\) along with the mid point of the common edge drawn in orange.}
    \label{fig:triangles}
\end{figure}

\noindent We first recall our notation for the readers' convenience. A pair of complex numbers, say \(u,\,v\in\mathbb{C}^*\), that are not real-collinear determine a triangle, say \(\nabla(\,u,\,v\,)\), in the plane, whose area is equal to one-half of \(\mathfrak{Im}(\overline{u}\,v)\). The choice of \(w\in\mathbb{C}\) uniquely determines a polygonal path as in \eqref{eq:chain}. Both the triangle and the polygonal path do not depend on the choice of the initial point, as mentioned above, and therefore we can assume that it is the origin. We define a triangle \(\Delta(\,u, v\,)\subseteq\mathbb{C}\) with vertices at \(\{\,\pto B,\, \pto B\,+\,u,\,\pto B\,+\,u- v\,\}\). Notice that \(\nabla(\,u,\,v\,)\) and \(\Delta(\,u, v\,)\) have one edge in common and they are related by a rotation of order two about the midpoint of the common edge, see Figure \ref{fig:triangles}. This justify the notation adopted according to our convention.

%\noindent With the notation above and the objects already introduced, we state the following characterization

\begin{prop}\label{prop:polychar}
    Let \(\textnormal{\textbf{W}}\in\mathbb C\) and let \(w\) be the relative period of the oriented edge \(\overline{\textnormal{\textbf{B}}\textnormal{\textbf{W}}}\). Then \(\mathcal C(\,w\,)\), the polygonal determined by \(w\), bounds a convex hexagon if and only if \(\,\textnormal{\textbf{W}} \in\Delta(\,u, v\,)\). Moreover, the hexagon is convex without pairs of sides aligned if and only if \(\,\textnormal{\textbf{W}}\) belongs to the interior of \(\Delta(\,u, v\,)\).
\end{prop}
 
\begin{proof}
    Let \(u,v\in\mathbb C^*\) be any pair of non-real collinear and let \(\nabla(\,u,\,v\,)\) and \(\Delta(\,u,\,v\,)\) be the triangles defined as above. Notice that they are both non-degenerate. Every point \(\pto W\in\Delta(\,u,\,v\,)\) yields a partition of \(\Delta(\,u,\,v\,)\) into three sub-triangles: \(\Delta_1\) with vertices at \(\{\,\pto B,\,\pto B\,+\,u,\,\pto W\,\}\), \(\Delta_2\) with vertices at \(\{\,\pto B\,+\,u,\,\pto B\,+\,u-v,\,\pto W\,\}\), and \(\Delta_3\) with vertices at \(\{\,\pto B,\,\pto W,\,\pto B\,+\, u-v\,\}\), see the top row of Figure \ref{fig:choiceofw}. Notice that if \(\pto W\in\partial\Delta(\,u,\,v\,)\) then some of these triangles are degenerate (in fact, there are at most two non-degenerate triangles, possibly one if \(\pto W\in\{\,\pto B,\,\pto B\,+\,u,\,\pto B\,+\, u-v\,\}\)). We remark that we shall describe in later sections how these cases correspond to possible degenerations.

\begin{figure}[!ht]
    \centering
    \includegraphics[width=0.925\linewidth]{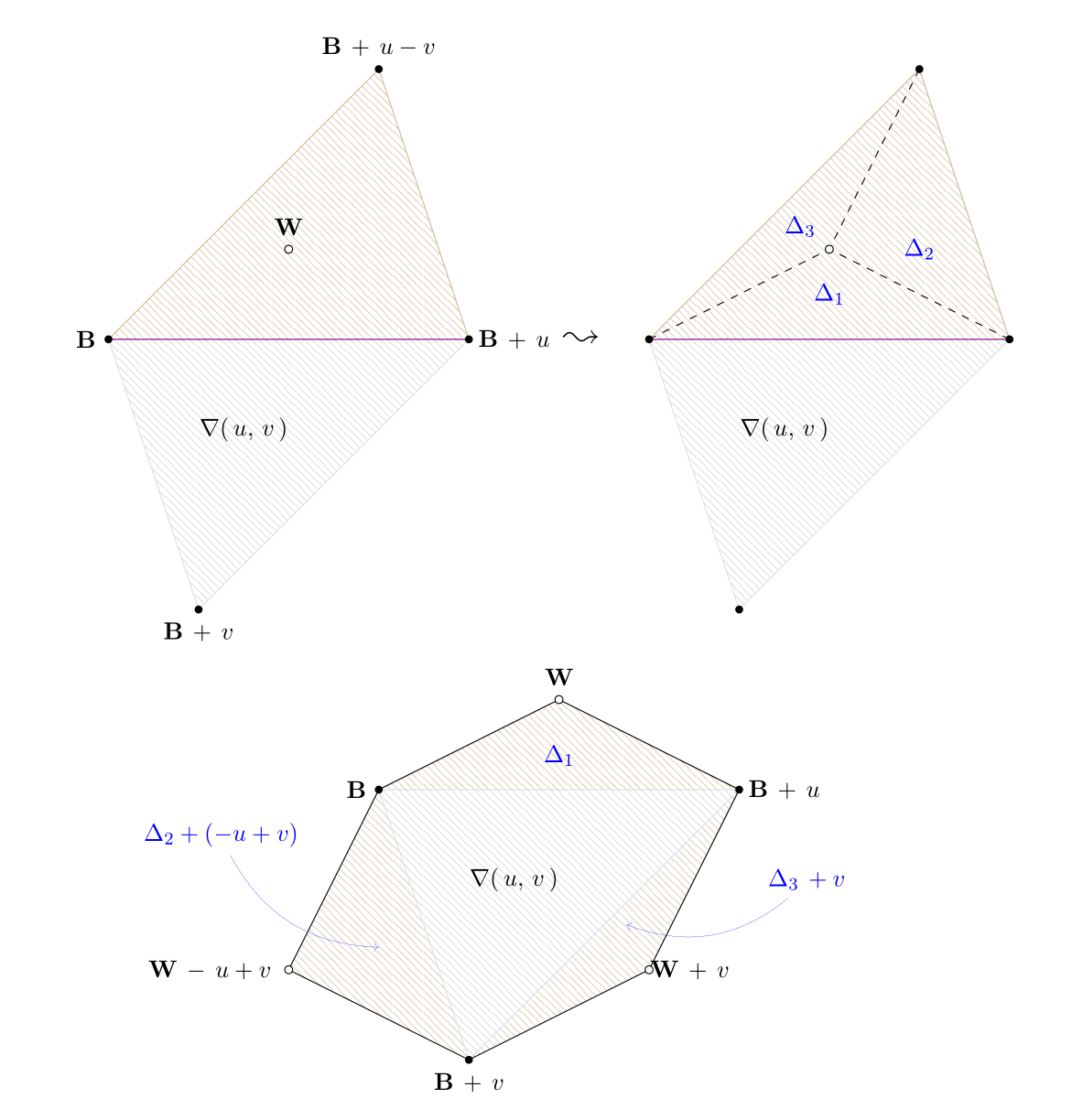}
    \caption{A choice of \( w \) inside the triangle \(\Delta(\,u,\,v\,)\) uniquely determines the polygonal curve \( \mathcal{C}(\,w\,) \). Indeed, any choice of \( \pto W \) within the triangle induces a subdivision of the triangle into three sub-triangles, which, when appropriately rearranged, form a convex hexagon. Moreover, it is easy to deduce how the shape changes if \( \pto W \in \partial \Delta(\,u,\,v\,) \).}
    \label{fig:choiceofw}
\end{figure}
\smallskip

\noindent %Every \(w\) in the interior of \(\Delta(\,u,\,v\,)\) single out three triangles. Since these none of them i say \(\Delta_1,\,\Delta_2,\,\Delta_3\) such that the 
%Every \( \pto W \) in the interior of \( \Delta(\,u,\, v\,) \) determines three sub-triangles, each with interior angles of magnitude less than \( \pi \) because they are all non-degenerate. According to our notation above we denote them as \(\Delta_1,\,\Delta_2,\,\Delta_3\).
The three triangles \(\Delta_1,\,\Delta_2,\,\Delta_3\) as defined above for each \( \pto W \) in the interior of \( \Delta(\,u,\, v\,) \) are non-degenerate.
We then define \(H(\,u,\,v,\,w\,)\) as the hexagon
\begin{equation}                
H(\,u,\,v,\,w\,)\,=\,\nabla(\,u,\,v\,)\,\cup\,\Delta_1\,\cup\,\left(\,\Delta_2\,+\,(\,v-u\,)\,\right)\,\cup\,\left(\,\Delta_3\,+\,v\,\right),
\end{equation}
see the bottom part of Figure \ref{fig:choiceofw}, and it is easy to observe that its boundary is realized by the chain of segments \(\mathcal C(\,w\,)\) as in \eqref{eq:chain}.
% with \(0\)\cm{Yongquan: is 0 here necessary? why not just stick to B?} as the starting point.
Since \(\Delta_1,\,\Delta_2,\,\Delta_3\) are non-degenerate, the angles at \(\pto W,\,\pto W\,+\,v\) and \(\pto W\,-\,u+v\) all have magnitude less than \(\pi\). Since the edges are parallel in pairs we easily conclude that all angles have magnitude less than \(\pi\). Thus \(H(\,u,\,v,\,w\,)\) is a convex hexagon with no pair of sides aligned. If \( \pto W \in \partial\Delta(\,u,\,v\,) \), then at least one of the triangles \( \Delta_i \), with \( i \in \{1,2,3\} \), must be degenerate and, at the same time, at least one must be non-degenerate. There is only one degenerate triangle if and only if \( \pto W \notin \{\,\pto B,\,\pto B\,+\,u,\,\pto B\,+u-v\,\} \), and in this case, \(\mathcal C(\,w\,)\) bounds a convex hexagon with two pairs of aligned sides. Finally, in the case that \( \pto W \in\{\,\pto B,\,\pto B\,+\,u,\,\pto B\,+\,u-v\,\} \), \(\mathcal C(\,w\,)\) degenerates into a polygonal consisting of only four segments and it bounds a parallelogram.

\smallskip

\noindent Conversely, let \(u,v\in\mathbb C^*\) be a pair of non-real collinear complex numbers and let \(w\in\mathbb C\). Let \(\mathcal C(\,w\,)\) be the polygonal of segments defined as in \eqref{eq:chain}. %and we assume without loss of generality that \(\pto b_1=0\)
Clearly, if \( \pto W \in \{\,\pto B,\,\pto B\,+\,u,\,\pto B\,+\,u-v\,\} \), then \( \mathcal C(\,w\,) \) is made up of only four segments and therefore bounds a parallelogram and not a hexagon. Let us assume \( \mathcal C(\,w\,) \) bounds a non-degenerate hexagon, say \(H(\,u,\,v,\,w\,)\).
%, then the polygonal is as follows:\begin{equation}\label{eq:chain2} \pto 0\,\longmapsto\,\pto w\,\longmapsto\,\pto u\,\longmapsto\,\pto w+v\,\longmapsto\,\pto v\,\longmapsto\,\pto w-u+v\,\longmapsto\,\pto 0. \end{equation}
In the hexagon defined by this polygonal, we can identify the triangle \( \nabla(\,u,\,v\,) \) with vertices at \( \{\,\pto B,\,\pto B\,+\,u,\,\pto B\,+\,v\,\} \), which is necessarily non-degenerate because \( u, v \) are not real-collinear. Furthermore, inside \(H(\,u,\,v,\,w\,)\) we can identify three additional triangles \( \Delta'_1,\,\Delta'_2,\,\Delta'_3 \). Up to relabeling the triangles, we can assume that \( \Delta'_1 \) has vertices at \( \{\,\pto B,\,\pto B\,+\,u,\,\pto W\,\} \) and set \( \Delta_1=\Delta'_1 \). Similarly, we can assume that \( \Delta'_2 \) has vertices at \( \{\,\pto B,\,\pto B\,+\,v,\,\pto W\,-\,u+v\,\} \) and set \( \Delta_2=\Delta'_2\,+\,(u-v) \). Finally, \( \Delta'_3 \) has vertices at \( \{\,\pto B\,+\,u,\,\pto B\,+\,v,\,\pto W\,+\,v \,\} \) and we set \( \Delta_3=\Delta'_3\,-\,v \). Using the classical rules of Euclidean geometry, it is not difficult to show that \( \Delta_1,\,\Delta_2,\,\Delta_3 \) define the triangle \( \Delta(\,u,\,v\,) \) and that \( \pto W\in\textnormal{int}\,\big(\,\Delta(\,u,\,v\,) \,\big) \) if and only if no triangle is degenerate, see Figure \ref{fig:choiceofw}. In fact, if any one of the triangles is degenerate, then one of the vertices in \( \{\,\pto W,\,\pto W\,+\,v,\,\pto W\,-\,u+v\,\} \) belongs to \( \partial \nabla(\,u,\,v\,) \), and as a consequence, its translate belongs to \( \partial \Delta(\,u,\,v\,) \).
\end{proof}

\smallskip

\noindent As a straightforward corollary we can derive conditions on \(w\in\mathbb C\) for the hexagon \(H(\,u,\,v,\,w\,)\) to be embedded and non-degenerate. %We define \(\Delta(\,u,v\,)\) as a chamber of degenerate type.

\begin{cor}
    The polygonal \(\mathcal C(\,w\,)\) bounds a convex hexagon if and only if the following conditions hold simultaneously:
    \begin{equation}\label{eq:boundarycond}
        \mathfrak{Im}\,\big(\,\overline u\,w\,\big)\ge0, \qquad \mathfrak{Im}\,\big(\,(\overline{w-v})\,u\,\big)\ge0 \qquad \textnormal{and} \qquad \mathfrak{Im}\,\big(\,\overline{w}\,(u-v)\,\big)\ge0.
    \end{equation}
    Moreover, \(\mathcal C(\,w\,)\) bounds a convex hexagon with no pairs of side aligned if and only if all of the inqualities above are strict.
\end{cor}

\begin{rmk}
    It is fundamental to observe that the construction above is general in the sense that it holds for any pair \(u,v\in\mathbb{C}^*\) that are not real-collinear. Different pairs clearly define different triangles. In our setting, \(u\) and \(v\) are the absolute periods of two simple curves with intersection index one, meaning they are the absolute periods of a positively oriented basis of homology, and different bases give rise to different triangles. We must take all of these into account.
\end{rmk}

%\smallskip

\subsection{Chambers of degenerate type}\label{sec:chambdegtype}
We could define \(\Delta(\,u,v\,)\) as the chamber of degenerate type associated to the pair \(\{\,u,v\,\}\). However, it is not difficult to see that some other pairs could give rise to a chamber parametrizing the same structures. For example, starting with any pair of sides of \(\Delta(\,u,v\,)\), one can go through the construction above to obtain isometric triangles that parametrize the same chamber. We now make this statement precise.
%We may be tempted to define \(\Delta(\,u,v\,)\) as a chamber of degenerate type associated to the pair \(\{\,u,v\,\}\). On the other hand, our choice of \(\Delta(\,u,v\,)\) is purely conventional as the triangles \(\Delta(\,u,v\,)\,+\,v\) and \(\Delta(\,u,v\,)\,-\,u+v\) can be both equally entitled as chambers of degenerate type associated to the pair \((\,u,v\,)\) because they all parametrize the same structures of degenerate type. In fact, it is straightforward to show that \(\pto W \in \Delta(\,u,v\,)\) if and only if \(\pto W\,+\,u\in \Delta(\,u,v\,)\,+\,v\). Similarly, \(\pto W \in \Delta(\,u,v\,)\) if and only if \(\pto W\,-\,u+v\in \Delta(\,u,v\,)\,-\,u+v\). We show that these triangles actually arise from different pairs. We first introduce with the following terminology.

\begin{defn}\label{def:charpair} 
    Let \(\rho\colon\shomolzoo\longrightarrow \mathbb C\) be a representation and let \(\Gamma=\textnormal{Im}(\,\rho\,)\). A \textit{characteristic triple} is a triple \((\,u,v,z\,)\) of absolute period in \(\Gamma\) such that \(u-v-z=0\) and each pair \(\{\,u,v\,\}\), \(\{\,-v,u-v\,\}\) and \(\{\,v-u,-u\,\}\) is a (ordered) pair of absolute periods of a positively oriented basis of homology. 
\end{defn}

\noindent The next claim is quite immediate and it can be deduced from the proof of Proposition \ref{prop:polychar}. As a direct consequence we will derive Proposition \ref{prop:degchamb} below. The following claims hold.

\begin{lem}
    Let \((\,u,v,z\,)\) be a characteristic triple. Then \(\Delta(\,u,v\,)\), \(\Delta(\,-v,u-v\,)\) and \(\Delta(\,v-u,-u\,)\) are all isometric and parametrize the same structures of degenerate type. Moreover, if we choose base points so that \(\nabla(\,u,v\,)\), \(\nabla(\,-v,\,u-v\,)\), and \(\nabla(\,v-u,\,-u\,)\) overlap, then \(\Delta(\,-v,u-v\,)=\Delta(\,u,v\,)\,-\,u+v\) and \(\Delta(\,v-u,-u\,)=\Delta(\,u,v\,)\,+\,v\)
\end{lem}
%\cm{Yongquan: these triangles are based at what points?}

\begin{proof}[Sketch of the proof]
    Let \((\,u,v,z\,)\) be a characteristic triple and let us consider first the pair \((\,u,v\,)\). Then every point \(\pto B\in\C\) yields the triangles \(\nabla(\,u,v\,)\) and \(\Delta(\,u,v\,)\). Let us now consider \((\,-v,u-v\,)\) and set \(\pto B'=\pto B\,+\,v\). Then \(\nabla(\,-v,u-v\,)=\nabla(\,u,v\,)\) and \(\Delta(\,-v,u-v\,)=\Delta(\,u,v\,)\,-\,u+v\). Similarly, we may consider \((\,v-u,-u\,)\) and set \(\pto B''=\pto B\,+\,u\). Then \(\nabla(\,v-u,-u\,)=\nabla(\,u,v\,)\) and \(\Delta(\,v-u,-u\,)=\Delta(\,u,v\,)\,+\,v\). 
\end{proof}

\begin{prop}\label{prop:degchamb}
Let \(\rho\) be a representation with negative volume and let \(\mathfrak{ML}(\,\rho,\,\mu\,)\) be the isoperiodic fiber in the marked stratum \(\mathcal H_1(\,\mu\,)\), where \(\mu=(1,1,-2)\). Then there exists a unique chamber of degenerate type for every characteristic triple \((\,u,v,z\,)\) whose closure can be identified with \(\Delta(\,u,v\,)\). Equivalently, it can be identified with \(\Delta(\,-v,u-v\,)\) or \(\Delta(\,v-u,-u\,)\).
\end{prop}

\begin{defn}\label{def:degchamb}
     We denote a chamber of degenerate type by \(\Delta(\,u,v,u-v\,)\).
\end{defn}

\noindent From now on, we shall use \(\nabla(\,u,\,v,\,u-v\,)\) to denote any of \(\nabla(\,u,v\,)\), \(\nabla(\,-v,\,u-v\,)\), or \(\nabla(\,v-u,\,-u\,)\), as they all represent the same triangle defined using different bases. In contrast to Definition \ref{def:degchamb} we refer to \(\nabla(\,u,\,v,\,u-v\,)\) as an \textit{anti-chamber}, since it is image of a chamber of degenerate type under a rotation of angle \(\pi\); we refer once again to Figure \ref{fig:triangles}. In what follows each of these triangles is regarded as a particular \textit{configuration} of the anti-chamber. In fact, it will sometimes be convenient to refer to a specific configuration determined by a characteristic triple, \textit{c.f.} \S\ref{ssec:jumpchamb}. %Clearly, analogous considerations apply to the other two triangles, with the appropriate modifications.

\smallskip

%\noindent Although we have reduced the study to the case of representations with image in \(\mathbb{Z}[\,i\,]\) and negative volume equal to \(-1\), we have stated the definition and the respective proposition in purely general terms for simplicity, thus avoiding a more complicated statement. 

\subsection{Changing the marking}\label{ssec:chancingthemarking}
The structures just realized naturally come equipped with a marking on their zeros determined by the choice of marking the starting point of the chain \(\mathcal C(\,w\,)\) with \(\pto B\). In the present section we analyze what happens if we had chosen the opposite basis in homology \textit{i.e.} \(\{\,\alpha^{-1},\,\beta^{-1}\,\}\). Clearly, the representation being fixed, this new basis yields a new pair of non-zero complex numbers, that is \(\{\,-u,-v\}\). 

\smallskip

\noindent We thus proceed in the same way: We label the starting point of the chain \eqref{eq:chain} with \(\pto B\) and then determine the parameter space for the point \(\pto W\). It can be shown that a choice of \(w\in\mathbb C\) yields a polygonal \(\mathcal C(\,w\,)\) as in \eqref{eq:chain} that bounds a convex embedded hexagon if and only if \(-w \in \Delta(\,u,v,u-v\,)\). The key observation in this case is that if the oriented segment \(\overline{\textnormal{\textbf{B}}\textnormal{\textbf{W}}}\) has relative period \(w\) in the former construction now it has opposite relative period. As a consequence this latter structure realized by using \(\{\,-u,\,-v\}\) in place of \(\{\,u,\,v\}\) must have a different marking. See Figure \ref{fig:marking}. The following statement holds.

\begin{figure}[!ht]
    \centering
    \includegraphics[width=1\linewidth]{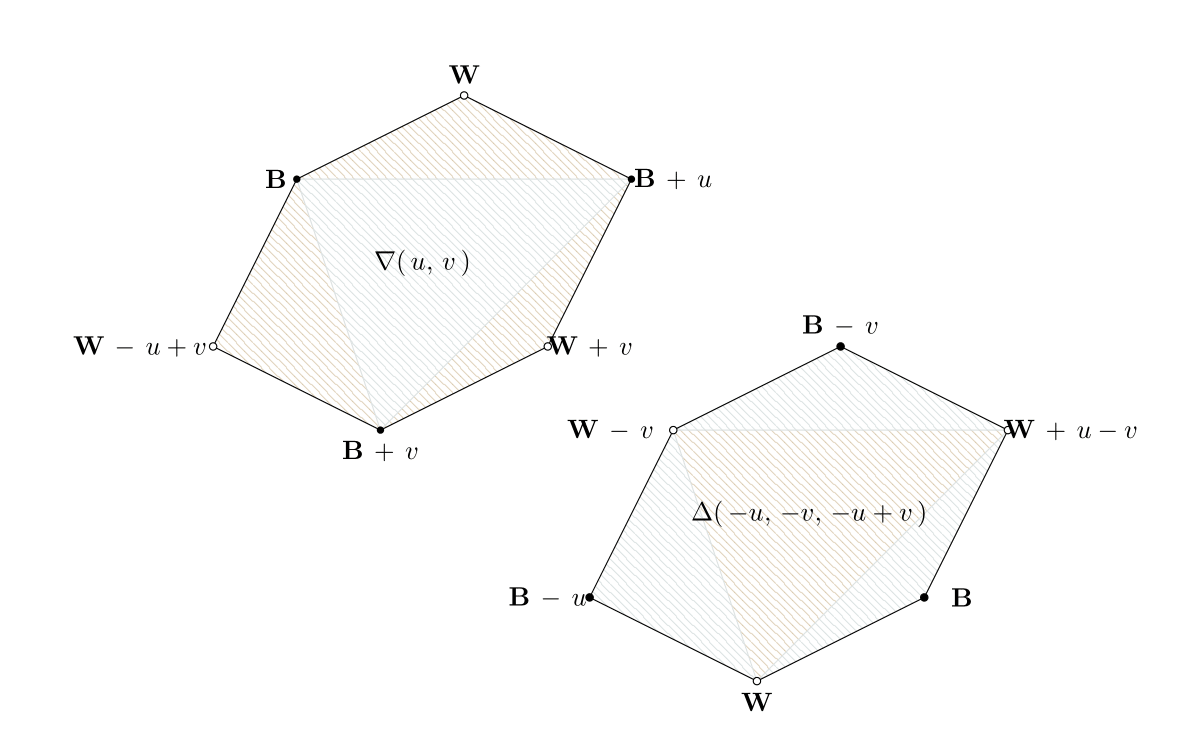}
    \caption{Different markings determine isometric hexagons that do not differ by a translation. Clearly, these hexagon are indistinguishable once we forget the marking. The top left hexagon has been realized by using \(\{\,u,v\,\}\) whereas the bottom right has been realized by using \(\{\,-u,-v\,\}\).}
    \label{fig:marking}
\end{figure}

%\smallskip

%\noindent Let \(\pto W\in\mathbb C\) be any point and let \(\nabla(\,u,\,v\,)\) be consider the triangle with vertices at \(\{\,\pto W,\,\pto W\,-\,u,\,\pto W\,-\,v\,\}\). Any choice of \(w\in\mathbb C\) thus determines a polygonal \(\mathcal C(\,w\,)\) as in \eqref{eq:chain}. Even in this case such a polygonal may be self-intersecting and it may not bound a convex hexagon. We thus want to determine conditions for this polygonal to be embedded.  

%\noindent At this point, we can proceed in two ways: we can either study the case \textit{ex novo} in detail, as done in the previous sections, and deduce conditions on \(w\), or we can rely on the previously studied case. For the sake of brevity and to avoid unnecessary repetition, we opt for the second approach. The key observation in this case is that if the oriented segment \(\overline{\textnormal{\textbf{B}}\textnormal{\textbf{W}}}\) has relative period \(w\) then the oriented segment \(\overline{\textnormal{\textbf{W}}\textnormal{\textbf{B}}}\) has the opposite relative period. The following statement holds. 

\begin{prop}\label{prop:oppmarkingdeg}
    Every choice of a basis in homology yields a well-defined marking. Conversely, the choice of a marking yields a preferred basis in homology. Moreover, opposite markings correspond to opposite bases in homology.
\end{prop}

\begin{proof}
    The first claim is already subsumed in \S\ref{ssec:embdhex}. The representation \(\rho\) being fixed, every basis in homology yields a pair of non-zero complex numbers \(\{\,u,v\,\}\) and the hence a triangle \(\Delta(\,u,v,u-v\,)\). Any choice of a base point \(\pto B\in\C\), and \(w\in\nabla(\,u,v,u-v\,)\) determine a convex hexagon \(H(\,u,v,w\,)\) bounded by a chain as in \eqref{eq:chain}. The corners of this chain being already labeled yields a marking.
    \smallskip

    \noindent Conversely, suppose we are given a convex hexagon whose sides are parallel in pairs. There are thus two ways to label the vertices with the letters \(\pto B\) and \(\pto W\) in an alternating fashion. In fact, there are only two possible markings. The vertices marked with \(\pto B\) in one of the two sequences determine the anti-chamber \(\nabla(\,u,v,u-v\,)\), which in turn determines a basis \(\{\,u, v\,\}\), and hence a basis in homology. One can observe that the two sequences determine isometric anti-chambers that differ by a rotation of \(\pi\). Therefore, different markings produce distinct, opposite bases in homology.
\end{proof}

%that determining the starting point of a saddle connection, having fixed its endpoint, results in determining the end point of a saddle connection with opposite relative. What one can observe is that a point \(\pto B\) determines a convex hexagon, possibly with two pairs of sides aligned, if and only if it belongs to a triangle with vertices at \(\{\pto W,\,\pto W\,+\,v,\, \pto W\,-\,u+v\}\). It can be checked that this is isometric to \(\Delta(\,u,\,v\,)\).

\smallskip

\noindent Therefore, choosing the marking means choosing the vertices of the reference anti-chamber \(\nabla(\,u,\,v,\,u-v\,)\) used to construct the convex hexagon. As can also be seen from the Figure \ref{fig:marking}, the two hexagons are clearly isometric and differ by a translation, but as marked structures they are different because no translation maps the white points to the black ones. As a direct consequence of Proposition \ref{prop:oppmarkingdeg} we have the following:

\begin{cor}\label{cor:degchambproj}
\(\Delta(\,u,\,v,\,u-v\,)\) and \(\Delta(\,-u,\,-v,\,-u+v\,)\) are distinct chambers of degenerate type in the marked leaf \(\mathfrak{ML}(\,\rho,\mu\,)\) and project to the same chamber in the unmarked leaf \(\mathfrak{L}(\,\rho,\mu\,)\) via the forgetful map.
\end{cor}

\smallskip

\subsection{Degeneration towards the boundary of \(\Delta(\,u,\,v,\,u-v\,)\)}
Let \(u,v\in\mathbb C^*\) be two non-$\mathbb{R}$-collinear complex numbers such that \(\mathfrak{Im}(\,\overline uv\,)=-1\) and let \(\Delta(\,u,\,v,\,u-v\,)\subset \mathbb C\) be a chamber of degenerate type as defined in \S\ref{sec:chambdegtype}. In the present section we aim to determine the behavior of a sequence of structures of degenerate type that converges toward the boundary of \(\Delta(\,u,\,v,\,u-v\,)\). The boundary of the triangle is a \(1\)-simplex consisting of three vertices and three edges. We thus divide the study into two cases: in the first, we assume that the sequence converges to a vertex of the simplex, while in the second, we assume that it converges to a point in the interior of an edge. We remark that aspects of these degenerations have already been discussed in the previous section.

\subsubsection{Degeneration towards the vertices of \(\Delta(\,u,\,v,\,u-v\,)\)} Let \((\,\pto W_n\,)\subset \Delta(\,u,\,v,\,u-v\,)\) be a sequence converging to any one of the vertices of \(\Delta(\,u,\,v,\,u-v\,)\) and let \(w_n\) be the relative period of the oriented segment \(\overline{\textnormal{\textbf{B}}\textnormal{\textbf{W}}}_n\). Then three distinct degenerations may happen according to the limit point. For simplicity we list them as follows. See Figure \ref{fig:degcorner} for the second case listed below:
\begin{itemize}
    \item[1.] The limit point is \(\pto B\). In this case \(w_n\longrightarrow 0\) and the polygonal \(\mathcal C(\,w_n\,)\) defined as in \eqref{eq:chain} degenerates to a parallelogram bounded by the chain
    \begin{equation}\label{eq:chaindeg1}
    \pto B\,      \longmapsto\, \pto B\,+\,u\,  \longmapsto\,
    \pto B\,+\,v\,\longmapsto\, \pto B\,-\,u+v\,\longmapsto\,
    \pto B.        
    \end{equation}
    because \(\pto B=\pto W\).
    \smallskip
    \item[2.] The limit point is \(\pto B\,+\,u\). In this case \(w_n\longrightarrow u\) and the polygonal \(\mathcal C(\,w_n\,)\) defined as in \eqref{eq:chain} degenerates to a parallelogram bounded by the chain
    \begin{equation}\label{eq:chaindeg2}
    \pto B\,        \longmapsto\, \pto B\,+\,u\,  \longmapsto\,
    \pto B\,+\,u+v\,\longmapsto\, \pto B\,+\,v\,\longmapsto\,
    \pto B.        
    \end{equation}
    because \(\pto W=\pto B\,+\,u\). 
    \smallskip
    \item[3.] The limit point is \(\pto B\,+\,u-v\). In this latter case \(w_n\longrightarrow u-v\) and the polygonal \(\mathcal C(\,w_n\,)\) defined as in \eqref{eq:chain} degenerates to a parallelogram bounded by the chain
    \begin{equation}\label{eq:chaindeg3}
    \pto B\,      \longmapsto\, \pto B\,+\,u-v\,  \longmapsto\,
    \pto B\,+\,u\,\longmapsto\, \pto B\,+\,v\,    \longmapsto\,
    \pto B.        
    \end{equation}
    because \(\pto W=\pto B\,+\,u-v\). 
    \smallskip
\end{itemize}

\begin{figure}[!ht]
    \centering
    \includegraphics[width=1\linewidth]{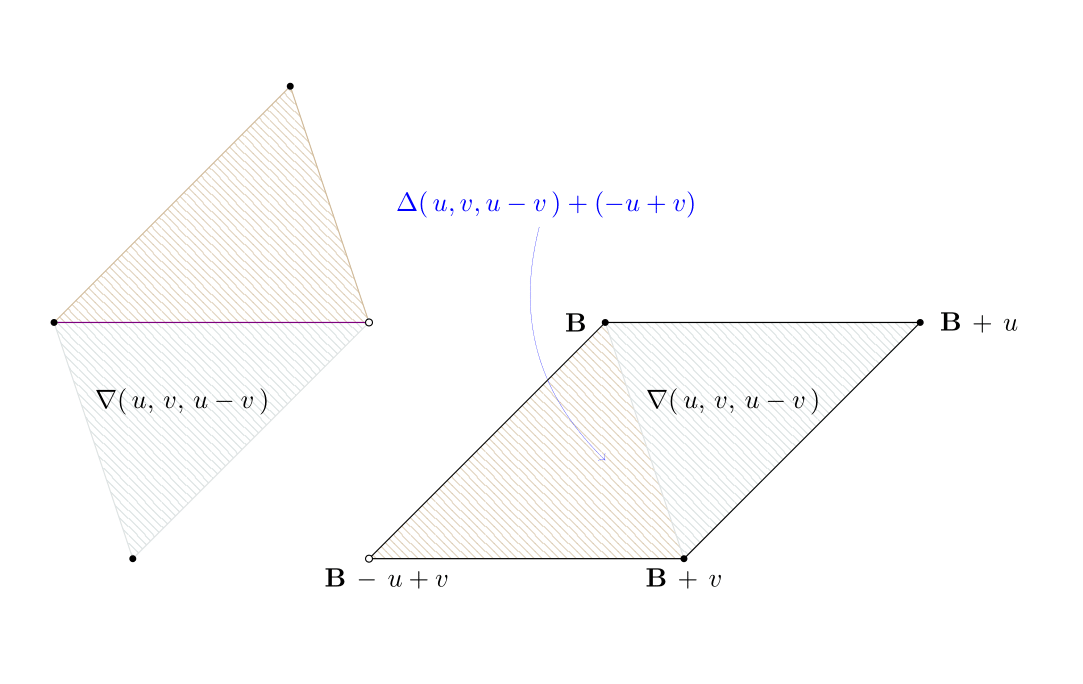}
    \caption{Degeneration in \(\Omega\mathcal M_1(1,1,-2)\).}
    \label{fig:degcorner}
\end{figure}

\smallskip

\noindent In all cases listed above, by gluing opposite sides by using translations, it is easy to see that we get a structure in the minimal stratum \(\Omega\mathcal M_1(2,-2)\) with empty core and hence negative volume.

\smallskip

\begin{figure}[!htp]
    \centering
    \includegraphics[width=1.125\linewidth]{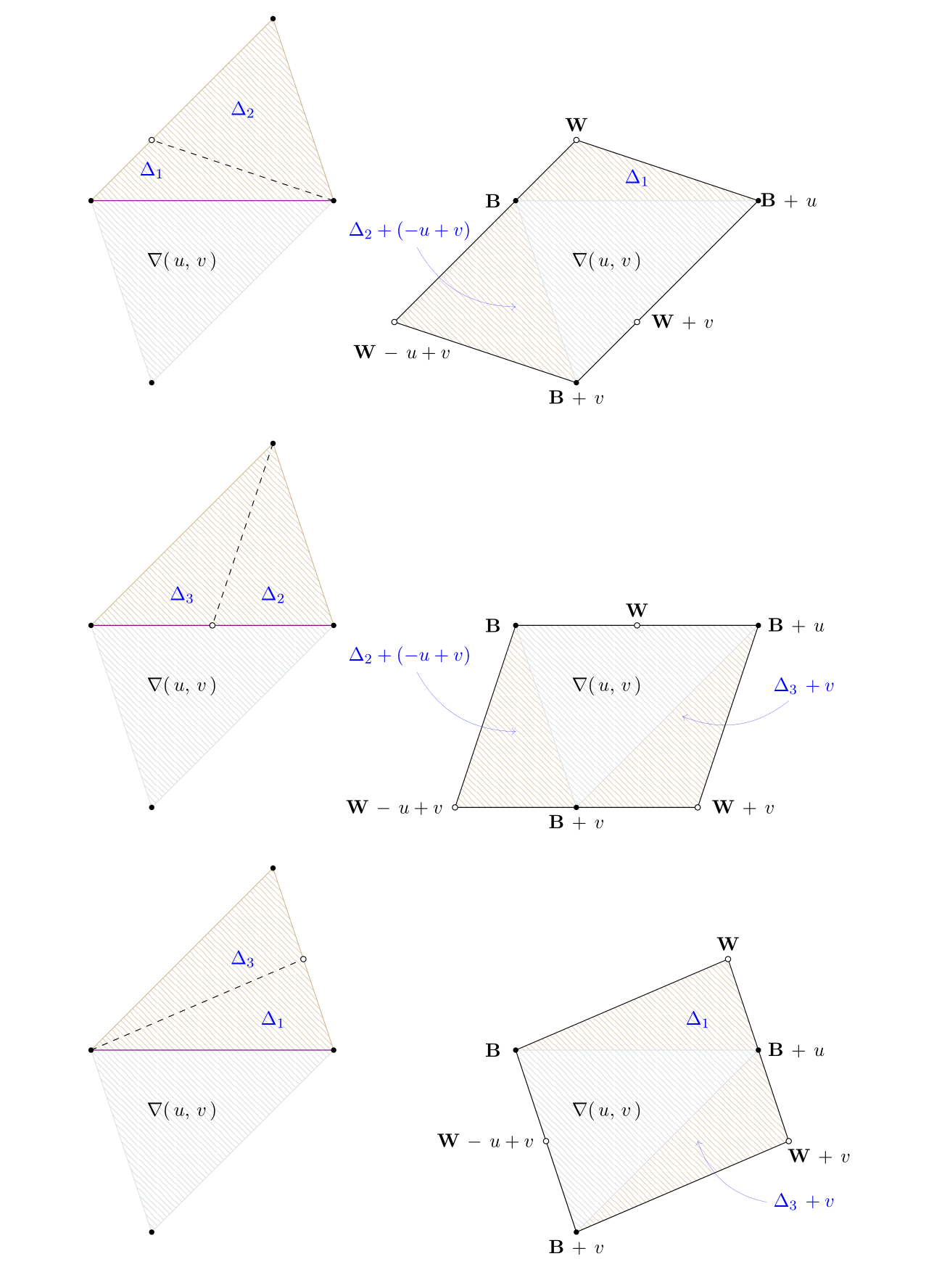}
    \caption{Degenerations towards the boundary of \(\partial\Delta(\,u,\,v\,)\).}
    \label{fig:degboundary}
\end{figure}

\subsubsection{Degeneration towards the edges of \(\Delta(\, u,\,v,\,u-v\,)\)} We next consider degenerations towards to edges of the triangle \(\Delta(\,u,\,v,\,u-v\,)\). We introduce the following angles: We define \(\theta_i\) to be the interior angle of \(\Delta_i\) at \(\pto W\). Clearly \(\theta_1+\theta_2+\theta_3=2\pi\). The following holds.

\begin{lem}\label{lem:anglelem}
    A sequence \((\,\textnormal{\textbf{W}}_n\,)\) converges to the boundary point of \(\Delta(\,u,\,v,\,u-v\,)\) which is not a vertex if and only if exactly one of the angles \(\theta_1,\,\theta_2,\,\theta_3\) converges to \(\pi\).
\end{lem}

\begin{proof}
    It is easy to observe that if \(\pto W\) belongs to \(\partial\Delta(\,u,\,v,\,u-v\,)\) and it is not a vertex then one and only one of the inequalities in \eqref{eq:boundarycond} is an equality. As a consequence two vectors are aligned and they form an angle of magnitude \(\pi\).
\end{proof}

\noindent Whenever exactly one of the angles \(\theta_i\) equals \(\pi\), the polygonal chain \eqref{eq:chain} defines a hexagon with two pairs of aligned sides. We remove the interior of this hexagon and, after gluing the sides appropriately via translations, the resulting structure belongs to the marked stratum \(\mathcal{H}_1(1,1,-2)\). By construction, the domain of the pole features two consecutive saddle connections forming an angle of \(\pi\), as illustrated in Figure~\ref{fig:degboundary}. Thus, the resulting structure is transitional, in the sense of Definition~\ref{def:transitional} and Proposition~\ref{prop:transitionalcharac}, and therefore lies at the boundary between two distinct chambers.

\medskip

\subsection{Beyond the boundary of \(\Delta(\,u,\,v,\,u-v\,)\).}\label{ssec:beyondboundary} We now allow \(\pto W\) to cross the boundary of a chamber of degenerate type \(\Delta(\,u,\,v,\,u-v\,)\) and observe how the corresponding structure further degenerates. In fact, although it is clear from the previous section that converging to the boundary results in a structure on the boundary of a chamber of cylinder type (cf.\ Figure~\ref{fig:wallcasetwo}), it is not immediately apparent to which chamber these structures belong. In this section, we aim to determine which chamber is accessed.

\smallskip

\begin{figure}[!ht]
    \centering
    \includegraphics[width=1\linewidth]{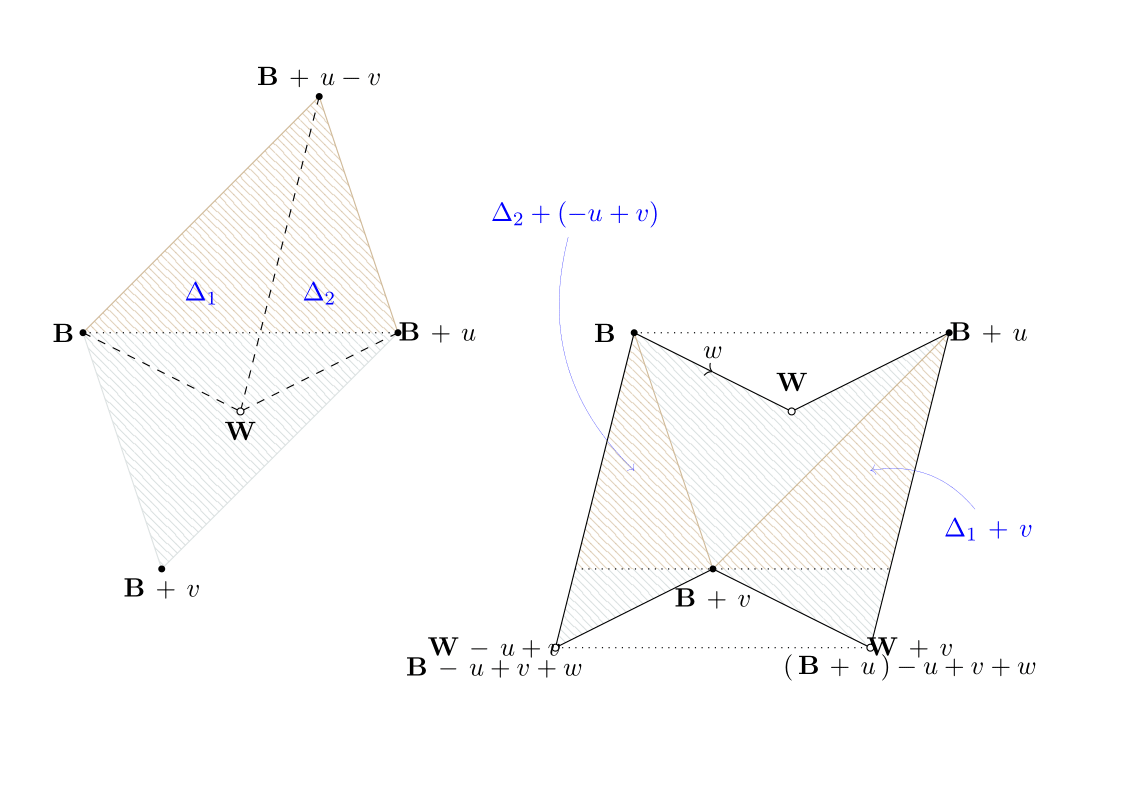}
    \caption{Concave hexagon with two opposite angles greater than \(\pi\) that arises when \(\pto W\) belongs to \(\nabla(\,u,\,v\,)\). On the right, if one cut along the dotted segment connecting $\pto B$ and $\pto B+u$ and the one connecting $\pto W-u+v$ and $\pto W+v$, we obtain two excess triangles that glue to be a parallelogram.}
    \label{fig:concavehex}
\end{figure}

\noindent We consider the case \(\pto W\) crosses the edge with period \(u\), namely the edge of \(\Delta(\,u,\,v,\,u-v\,)\) that joins the points \(\pto B\) and \(\pto B\,+\,u\). The other cases works \textit{mutatis mutandis} and hence can be easily understood. The basic idea is the same as the one used above in the previous sections. The point \( \pto W \) determines a concave quadrilateral, say \(\textnormal Q(\,u,v,w\,)\), bounded by the chain
\begin{equation}
    \pto B\,\longmapsto\,\pto W\,\longmapsto\,\pto B\,+\,u\,\longmapsto\,\pto B\,+\,v\,\longmapsto \pto B.
\end{equation}
The point \(\pto W\) also yields two triangles, say \( \Delta_1 \) and \( \Delta_2 \), which, once translated, together with \( \textnormal Q(\,u,v,w\,) \) form a concave hexagon with two angles greater than \( \pi \) (see Figure~\ref{fig:concavehex}). However, unlike the previous cases, these triangles now overlap with \( \nabla(\,u,\,v,\,u-v\,) \). As we will see, the overlapping region has a well-defined meaning. 

\smallskip

\noindent We shift these \(\Delta_1\) and \(\Delta_2\) as shown in Figure \ref{fig:concavehex} and the resulting shape is a concave hexagon, say \(H(\,u,v,w\,)\) with two opposite angles greater than \(\pi\). We thus consider the unbounded region bounded by the perimeter of \(H(\,u,v,w\,)\) and glue the opposite by using translation as usual to get a translation surface in \(\mathcal H_1(1,1,-2)\) with relative period \(w\) and fixed marking.

\smallskip

\noindent We claim that this gives a structure of cylinder type in the chamber \(\textnormal{C}(\,-u\,)\), with parameter \(w\). Indeed, the resulting structure can also be obtained by removing the interior of a parallelogram, say \(Q\), with sides \(u\) and \(-u+v+w\) (see the outermost outline for the shape on the right in Figure~\ref{fig:concavehex}), and gluing another parallelogram \(P\) with sides \(u\) and \(w\) (obtained by combining the two excess triangles on the right in Figure~\ref{fig:concavehex}). Since \(\mathfrak{Im}(\,\overline{u}\,w\,)<0\), then the resulting structure belongs to \(\textnormal{C}(\,-u\,)\), see \eqref{eq:cylchambnegminusone} and \eqref{eq:cylchambnegminustwo}. Observe that the exterior of the right side in Figure \ref{fig:concavehex} resembles the shadowed area in Figure \ref{fig:wallcasetwo} (which depicts the case \(\mathfrak{Im}(\,\overline{u}\,w\,)>0\,\)).

\smallskip

\begin{figure}[!ht]
    \centering
    \includegraphics[width=1\linewidth]{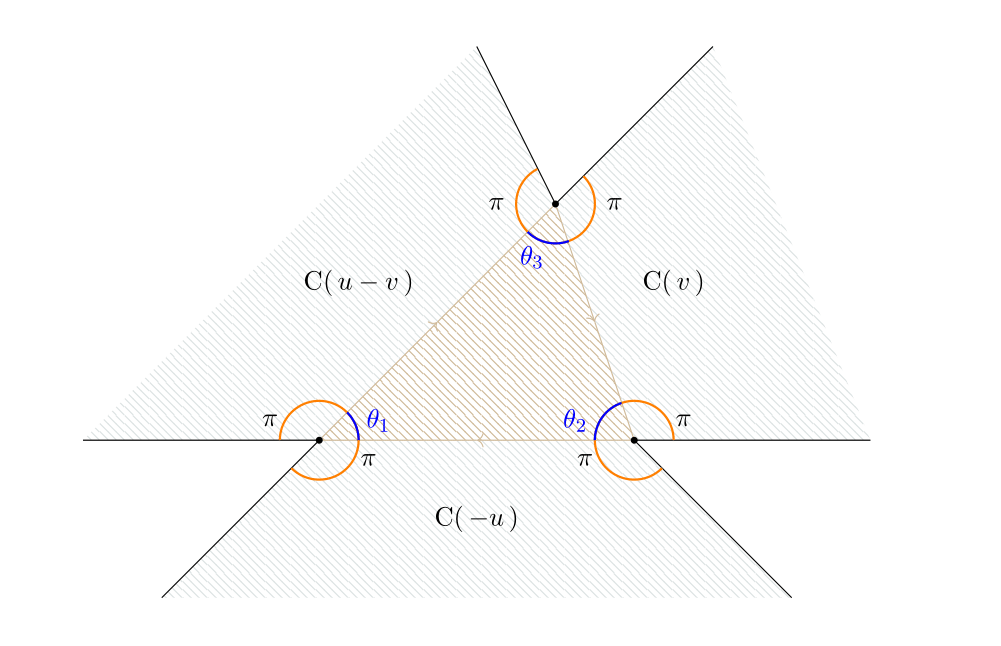}
    \caption{A chamber of degenerate type along with adjacent chambers of cylinder type. In the next section we aim to complete the picture by filling the blank parts.}
    \label{fig:degchambandadj}
\end{figure}

\smallskip

\noindent The same kind of reasoning applies \textit{mutatis mutandis} when \(\pto W\) cross the other boundaries of \(\Delta(\,u,\,v,\,u-v\,)\). In particular it is possible to see that if \(\pto W\) crosses the edge with period \(u-v\) that joins \(\pto B\) and \(\pto B\,+\,u-v\) then the structure degenerates to a cylinder type structure in \(\textnormal{C}(\,u-v\,)\) and, similarly, if \(\pto W\) crosses the edge with period \(v\) namely the edge that joins \(\pto B\,+\,u-v\) and \(\pto B\,+\,u\) then the structure degenerates to a cylinder type structure in \(\textnormal{C}(\,v\,)\). We thus have a description of any chamber of degenerate type with its adjacent chambers of cylinder type. See Figure \ref{fig:degchambandadj}.

\medskip

\subsection{Jumping from one chamber of degenerate type to another}\label{ssec:jumpchamb} In the present section we investigate how the triangles \(\nabla(\,u,\,v,\,u-v\,)\) and \(\Delta(\,u,\,v,\,u-v\,)\) change if we change the symplectic basis and its consequences. In order to abridge the notation, it is more convenient to work with a particular configuration \textit{e.g.} \(\nabla(\,u,v\,)\). Since the same discussion holds for any configuration the following arguments works for every chamber of degenerate type.

\smallskip

\noindent Let us recall our notation from \S\ref{ssec:realdegcorestruc}. We are given a non-trivial representation \(\rho\colon\shomolzoo\longrightarrow \mathbb C\) be a with negative volume equal to \(-1\) and given  two generators that represent a pair of simple closed curves that intersect only once, say \(\{\,\alpha,\,\beta\,\}\), we define \(u=\rho(\,\alpha\,)\) and \(v=\rho(\,\beta\,)\) respectively. Recall that \(u,v\) are non-zero because the representation has non zero volume. 

\subsubsection{The effect of changing basis} Two symplectic basis are always related by an element in \(\slz\). Since this latter is generated by the matrices
\begin{equation}\label{eq:gen}
    T=\begin{pmatrix}
        1 & 1\\
        0 & 1
    \end{pmatrix} \quad \text{ and } \quad
    S=\begin{pmatrix}
        0 & -1\\
        1 & 0
    \end{pmatrix}
\end{equation}
in the following it will be often sufficient to consider basis of the form \(\{\,\alpha,\,\alpha^k\,\beta\,\}\) for any \(k\in\mathbb Z\) because all other cases directly follow, \textit{e.g.} see \S\ref{ssec:adjacency}. In fact the matrix \(S\) in \eqref{eq:gen} just acts as a rotation of order four. This leads us to consider triangles \(\nabla(\,u,\,ku+v\,)\) and thus their respective opposite \(\Delta(\,u,\,ku+v\,)\). It is not hard to see that \(\nabla(\,u,\,ku+v\,)\) is a triangle bounded by the chain
\begin{equation}
    \pto B\,      \longmapsto\, \pto B\,+\,u\,  \longmapsto\,
    \pto B\,+(\,ku+v\,)\longmapsto\, \pto B,       
\end{equation}
that is it has the same basis as \(\nabla(\,u,\,v\,)\) but the vertex \(\pto B\,+\,v\) is now replaced with \(\pto B\,+\,ku+v\). Following the same argument of \S\ref{sec:chambdegtype}, \(\Delta(\,u,\,ku+v\,)\) is also a chamber of degenerate type. See Figure \ref{fig:changeofbasis}.

\begin{figure}[!ht]
    \centering
    \includegraphics[width=1\linewidth]{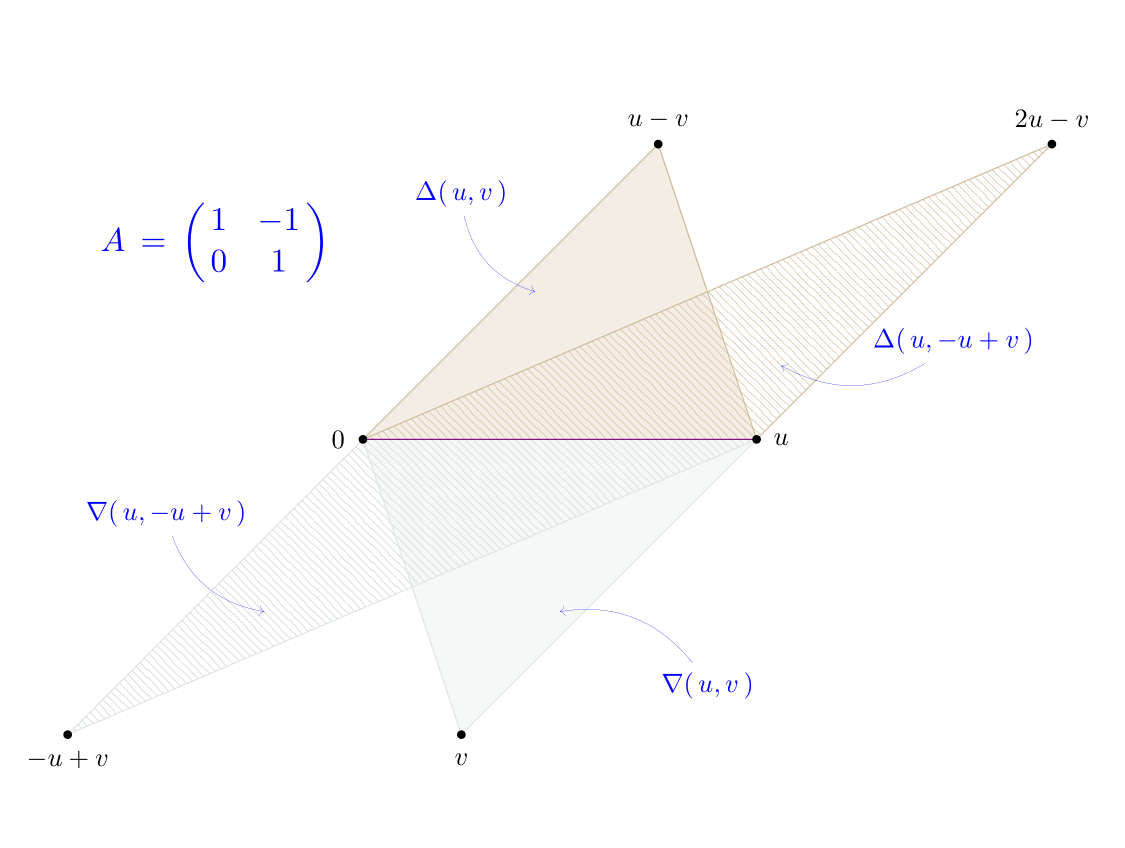}
    \caption{How triangles change under a change of basis in homology. The figure depicts the geometric interpretation of the change of basis given by \(\{\,u,v\,\} \rightarrow \{\,u, -u+v\,\}\) determined by the matrix \(A\). Here the point \(\pto B\) is the origin. Notice that, being \(\{\,u,\,v\,\}\) negatively oriented, the mapping \(A\) shears backswords the lower half-plane.}
    \label{fig:changeofbasis}
\end{figure}

\noindent  We recall for the reader's convenience that \(T\) as in \eqref{eq:gen} acts as a horizontal shear maps on \(\C\) and then two chambers of degenerate type of the form \(\Delta(\,u,\,ku+v\,)\) and \(\Delta(\,u,\,lu+v\,)\), for some \(k,l\in\mathbb Z\), differ by a horizontal shear, \textit{i.e.} \(T^{l-k}\).

\begin{rmk}\label{rmk:notisodegchamb}
    As artifact of our constructions, we may notice that \(\Delta(\,u,v\,)\) and \(\Delta(\,u,\,ku+v\,)\) do \textit{not} parametrize the same structures. In fact \(\Delta(\,u,v\,)\) parametrizes all structures of degenerate type arising by removing the interior of some convex hexagon that contains an embedded copy of \(\nabla(\,u,v\,)\), \textit{e.g.} see Figure \ref{fig:choiceofw}. Similarly, \(\Delta(\,u,\,ku+v\,)\) parametrizes all structures of degenerate type arising by removing the interior of some convex hexagon that contains an embedded copy of \(\nabla(\,u,\,ku+v\,)\). Since \(\nabla(\,u,v\,)\) and \(\nabla(\,u,\,ku+v\,)\) do not differ by a translation unless \(k=0\) -- in fact \(\nabla(\,u,\,ku+v\,)\) and \(\nabla(\,u,\,lu+v\,)\) are not isometric when \(k\neq l\) -- it follows that these chambers parametrize different structures.
\end{rmk}

\noindent We will resume from this point in \S\ref{ssec:adjacency}. In the next two paragraphs, we will consider a generic change of basis.

\smallskip

\subsubsection{Topology} Remark \ref{rmk:notisodegchamb} is not specific to changes of basis such as \(\{\,u,v\,\} \rightarrow \{\,u, ku+v\,\}\); in fact, it holds in the most general setting \(\{\,u,v\,\} \rightarrow \{\,A(\,u\,), A(\,v\,)\,\}\) where \(A\in\slz\). %The only caveat is that two triangles may be isometric if the matrix \(A\) is a rotation. In any case, two triangles differ by a translation only if \(A\) is the identity. 
We thus derive Proposition \ref{prop:degchambpara} that illuminates how negative leaves differ from positive leaves.

\smallskip

\noindent There is a natural action of \(\slz\) on the set of chambers of degenerate type as follows. We first consider a triangle \(\nabla(\,u,\,v\,)\) and an element \(A\in\slz\), the \(\slz\) action is defined as 
\begin{equation}\label{eq:slzact}
    A\cdot\nabla(\,u,\,v\,)\,=\,\nabla(\,A(\,u\,),\,A(\,v\,)\,)\,.
\end{equation}

\noindent We cannot expect this action to be faithful. Recall from \S\ref{sec:chambdegtype} that an anti-chamber has different configurations arising from different pairs determined by a particular characteristic triple; \textit{i.e.} two pairs \(\{\,u,\,v\,\}\) and \(\{\,A(\,u\,),\,A(\,v\,)\,\}\) may define the same anti-chamber for some matrix \(A\). We define \(R=STST\in\slz\) and it can be seen with a direct check that it has order three. Moreover, it is straightforward to check that
\begin{equation}
    \{\,u,\,v\,\} \underset{R}{\longrightarrow} \{\,v-u,\,-u\,\} \underset{R}{\longrightarrow} \{\,-v,\,u-v\,\} \underset{R}{\longrightarrow} \{\,u,\,v\,\}.
\end{equation}

\noindent Since \(\{\,u,v\,\}\), \(\{\,-v,u-v\,\}\) and \(\{\,v-u,-u\,\}\) all arise from the same characteristic triple \((\,u,\,v,\,u-v\,)\), it follows that the \(\slz\) action defined in \eqref{eq:slzact} has a finite kernel of order three generated by \(R\). In other words \(R\) fixes every anti-chamber by changing the its configuration (see final comments in \S\ref{sec:chambdegtype}).

%\begin{rmk}\label{rmk:antirottwo} We recall for the readers' convenience that, for every pair \(\{\,u,v\,\}\), the anti-chambers \(\nabla(\,u,\,v,\,u-v\,)\) and \(\nabla(\,-u,\,-v,\,v-u\,)\) differ by a rotation of order two. \end{rmk}

\smallskip

\noindent We now study the action of \(\slz\) on chambers of degenerate type. As an artifact of our constructions the following lemmata hold.

\begin{lem}\label{lem:degchambactone}
    Let \(\Delta(\,u,v\,)\) be a chamber of degenerate type and let \(\Delta(\,A(\,u\,),A(\,v\,)\,)\) be the chamber of degenerate type obtained by replacing the basis \(\{\,u,v\,\}\) with \(\{\,A(u),A(v)\,\}\) with \(A\in\slz\). Then
    \begin{equation}
        \Delta(\,A(\,u\,),A(\,v\,)\,)\,=\,A\cdot\Delta(\,u,v\,).
    \end{equation}
    In particular \(w\in\Delta(\,u,v\,)\) if and only if \(A(\,w\,)\in\Delta(\,A(\,u\,),A(\,v\,)\,)\). Moreover, \(\Delta(\,A(\,u\,),A(\,v\,)\,)\,\) and \(\,A\cdot\Delta(\,u,v\,)\) differ by a translation if and only if \(A\in\langle\,R\,\rangle\).
\end{lem}
%\noindent As a direct consequence we have the following

\begin{lem}\label{lem:degchambacttwo}
    Let \(\Delta(\,u,v\,)\) be a chamber of degenerate type and let \(H(\,u,v,w\,)\) be the convex hexagon determined by some \(w\in\Delta(\,u,\,v\,)\). For every \(A\in\slz\) the following equality holds:
    \begin{equation}
        H(\,A(\,u\,),A(\,v\,),A(\,w\,)\,)\,=\,A\cdot H(\,u,v,w\,)\,.
    \end{equation}
\end{lem}

\medskip

\begin{figure}[!ht]
    \centering
    \includegraphics[width=1\linewidth]{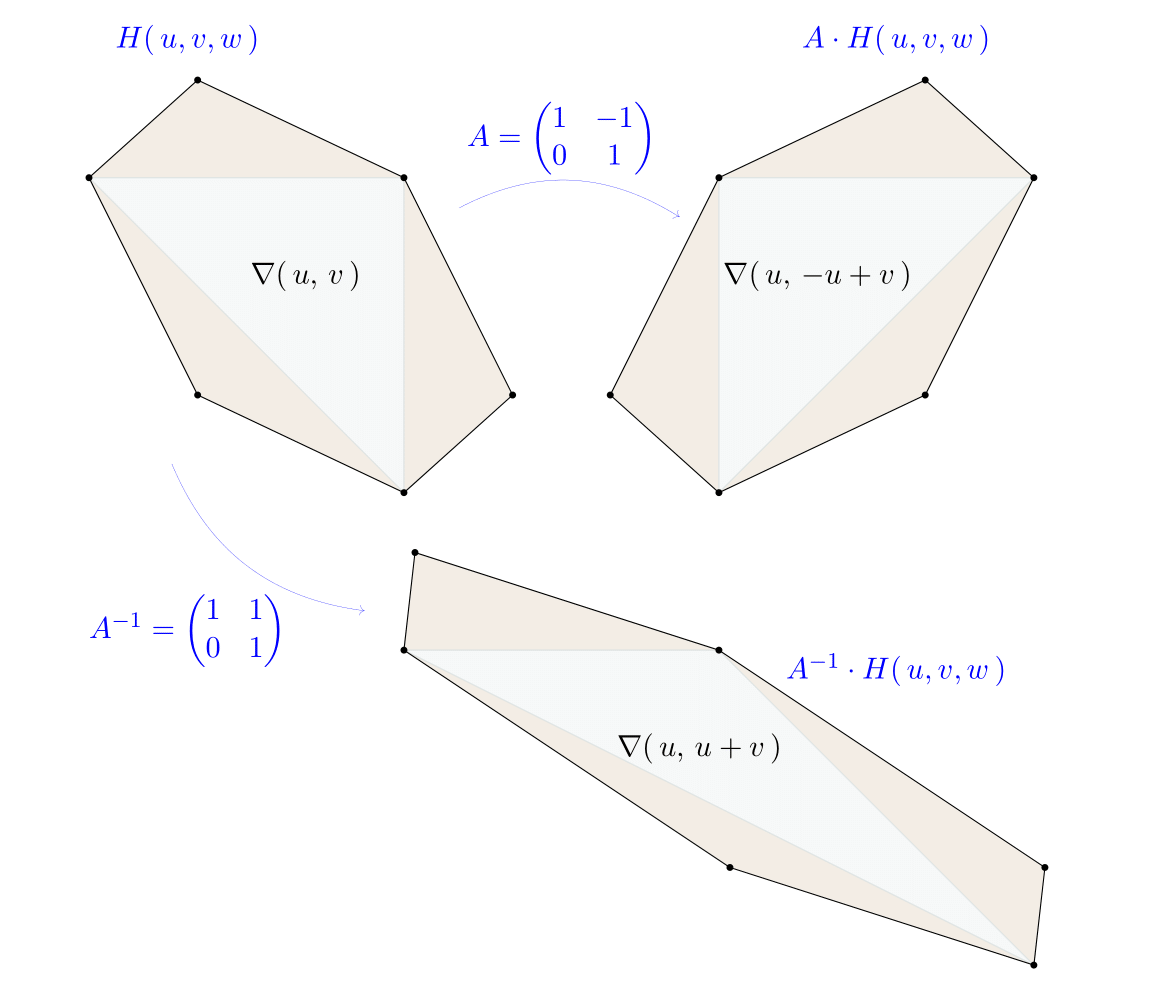}
    \caption{Deformation of the hexagon \(H(\,u,v,w\,)\) under the matrices \(A\) and its inverse.}
    \label{fig:transf}
\end{figure}

\noindent The proofs of Lemmata \ref{lem:degchambactone} and \ref{lem:degchambacttwo} are easy to establish and follow by using simple arguments of Euclidean geometry; see Figure \ref{fig:transf}. As a direct consequence we have the following:

\begin{prop}\label{prop:degchambpara}
    In a negative marked leaf, %\(\mathfrak{ML}(\,\rho,\mu\,)\) 
    the subspace of structures of degenerate type has infinitely many connected components distinguished by the elements of \(\slz\) up to configuration. In other words there is a one-to-one correspondence between the sets 
    \begin{equation}\label{eq:chambcorr}
        \mathfrak{C}\colon\left\{\,\, 
        \begin{array}{cc}
            \textnormal{Chambers of} \\
            \textnormal{degenerate type} 
        \end{array}
    \,\,\right\}\,\,\mathrel{\longleftrightarrowTilde}\,\,\frac{\,\slz\,}{\langle\,R\,\rangle}\,\,.
    \end{equation}
\end{prop}

\smallskip 

\noindent The one-to-one correspondence \(\mathfrak C\) in \eqref{eq:chambcorr} is far from unique. In fact, it depends on the choice of a basis in homology, and the absence of a preferred basis a priori means that no canonical correspondence exists. Proposition \ref{prop:degchambpara} and Corollary \ref{cor:degchambproj} together imply the following:

\begin{cor}
    In a negative unmarked leaf, %\(\mathfrak{ML}(\,\rho,\mu\,)\) 
    the subspace of structures of degenerate type has infinitely many connected components distinguished by the elements of \(\,\pslz\) up to configuration. In other words there is a one-to-one correspondence between the sets 
    \begin{equation}\label{eq:chambcorrtwo}
        \mathfrak{C}\colon\left\{\,\, 
        \begin{array}{cc}
            \textnormal{Chambers of} \\
            \textnormal{degenerate type} 
        \end{array}
    \,\,\right\}\,\,\mathrel{\longleftrightarrowTilde}\,\,\frac{\,\pslz\,}{\langle\,R\,\rangle}\,\,.
    \end{equation}
\end{cor}

\smallskip

\noindent \noindent In summary, this final corollary captures the geometric symmetry and underlying intuition that the anti-chambers \(\nabla(\,u,\,v,\,u-v\,)\) and \(\nabla(\,-u,\,-v,\,v-u\,)\) are related by a rotation of order two.

%The mapping \(\mathfrak C\) is also invariant under the action of \(\slz\) as follows: \begin{equation}\mathfrak C\,\Big(\,A\cdot\,\Delta(\,u,v\,)\, \Big) = A^{-1}\cdot\mathfrak C\,\Big(\,\Delta(\,u,v\,)\, \Big)\end{equation}

\medskip

\subsection{Chambers of degenerate type adjacent to a chamber of cylinder type}\label{ssec:adjacency} In order to describe the geometry and topology of negative leaves it is first convenient to determine which chambers of negative type are adjacent to a given chamber of cylinder type. The main statement of the present section is the following:

\begin{prop}\label{prop:chambadjacency}
    Let \(\{\,u,v\,\}\) be a basis in homology. Then the chamber of cylinder type \(\textnormal{C}(\,-u\,)\) is adjacent to all and only chambers of degenerate type of the form \(\Delta(\,u,\,ku+v,\,(\,1-k\,)u-v\,)\), where \(k\in\mathbb Z\).
\end{prop}

\smallskip

\begin{proof} One direction is a direct consequence of our study so far. Our discussion in \S\ref{ssec:beyondboundary} applies to every chamber of the form \(\Delta(\,u,\,ku+v,\,(\,1-k\,)u-v\,)\), as a consequence every such a chamber is adjacent to the chamber \(\textnormal{C}(\,-u\,)\). We only need to show that no other chamber can be directly reached from \(\textnormal{C}(\,-u\,)\). We first provide a partition of the boundary of the chamber \(\textnormal{C}(\,-u\,)\).

\medskip

\noindent Let \(\{\,u,v\,\}\) be a basis in homology and consider first the chamber of cylinder type \(\textnormal{C}(\,-u\,)\). Recall that its boundary \(\partial\textnormal{C}(\,-u\,)\) is the straight line 
\(\{\,w\in\mathbb C\,\,|\,\, \mathfrak{Im}(\,\overline{u}\,w\,)=0\,\}=\mathbb R\,u\). We introduce on the boundary a coordinate \(t\) such that \(w=tu\). The subset of integral points, \textit{i.e.} \(t\in\mathbb Z\), yields a partition of the boundary into infinitely many segments of the same length equal to \(|\,u\,|\). According to our study in \S\ref{sec:degtowalls}, every integral point parametrizes a structure in the minimal stratum \(\Omega\mathcal M_1(2,-2)\) whereas all other points parametrize transitional structures in \(\Omega\mathcal M_1(1,1,-2)\), see Figure \ref{fig:transstruct}. We introduce the following notation: We define \(e_k\) as the sub-segment of \(\partial\textnormal{C}(\,-u\,)\) bounded by the integral points \(ku\) and \((k+1)u\). 

\smallskip

\begin{figure}[!ht]
    \centering
    \includegraphics[width=1\linewidth]{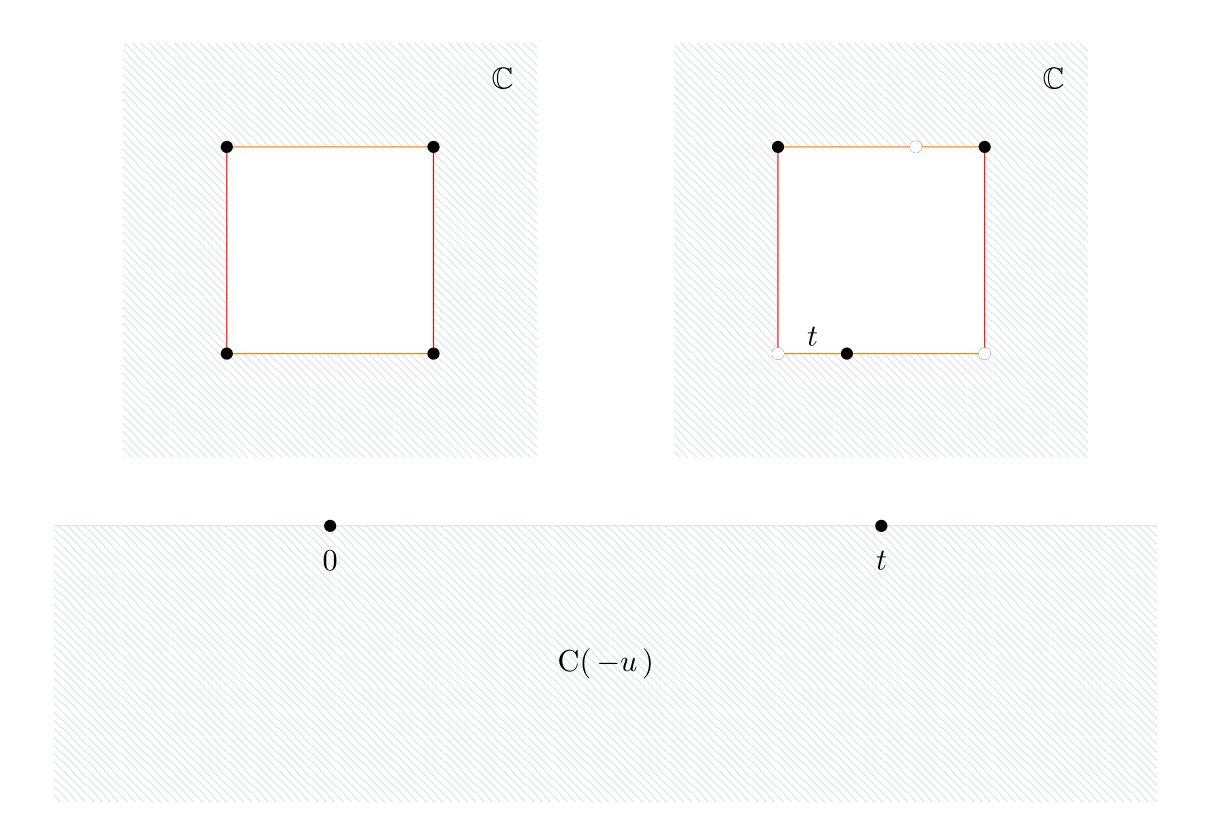}
    \caption{Transitional structures on the boundary of \(\textnormal{C}(\,-u\,)\). At \(t=0\) we have a structure in the minimal stratum. For \(t\in\mathbb R\setminus \mathbb Z\) we have structures in the principal stratum.}
    \label{fig:transstruct}
\end{figure}

\noindent It is sufficient to show that \(\Delta(\,u,\,ku+v,\,(\,1-k\,)u-v\,)\), for \(k\in\mathbb Z\), is adjacent to \(\textnormal{C}(\,-u\,)\) along the edge \(e_k\). In fact, fix any value \(k\in\mathbb Z\). Let \(\pto B\in\C\) be any point, let \(t\in[\,0,1\,]\) be any value and define \(\pto W=\,tu\). According to our notation, we define \(H(\,u,\,kv+u,\,tu\,)\) the degenerate hexagon bounded by the chain 
\begin{equation}
    \pto B\,\longmapsto\,\pto W\,\longmapsto\,
    \pto B\,+\,u\,\longmapsto\,\pto W\,+\,ku\,+\,v\,\longmapsto\,
    \pto B\,+\,ku\,+\,v\,\longmapsto\,\pto W\,-\,u\,\longmapsto\,
    \pto B
\end{equation}

\noindent as shown in the middle of Figure \ref{fig:degboundary}. The parameter \(t\) determines \(\pto W\) uniquely and, clearly, if \(t\in\{\,0,1\,\}\) then either \(\pto W=\pto B\) or \(\pto W=\pto B\,+\,u\) and the hexagon is a genuine quadrilateral. Assume \(t\in\, (\,0,1\,)\). The key observation here is that \(H(\,u,\,ku+v,\,tu\,)\) is the same degenerate hexagon we would obtain provided by Proposition \ref{prop:wallchar} for representations of negative volume in \S\ref{sec:degtowalls}, ( see case \(2\) ). In fact, we can set 
\begin{equation}
    \mathfrak u=-u, \quad \mathfrak v=-(\,ku+v\,) \,\,\text{ and }\,\, \mathfrak w=\,(\,t-1\,)\,u
\end{equation}
where we use gothic letters as in \S\ref{sssec:difflimits} to avoid any confusion with the current notation, see Figure \ref{fig:comparisondeghex}. As a consequence, since this holds for every \(k\in\mathbb Z\), it follows that \(\text{C}(\,-u\,)\) gets the access to all and only chambers of the form \(\Delta(\,u,\,ku+v,\,(\,1-k\,)u-v\,)\) and hence the desire conclusion follows.
\end{proof}

\begin{figure}[!ht]
    \centering
    \includegraphics[width=1\linewidth]{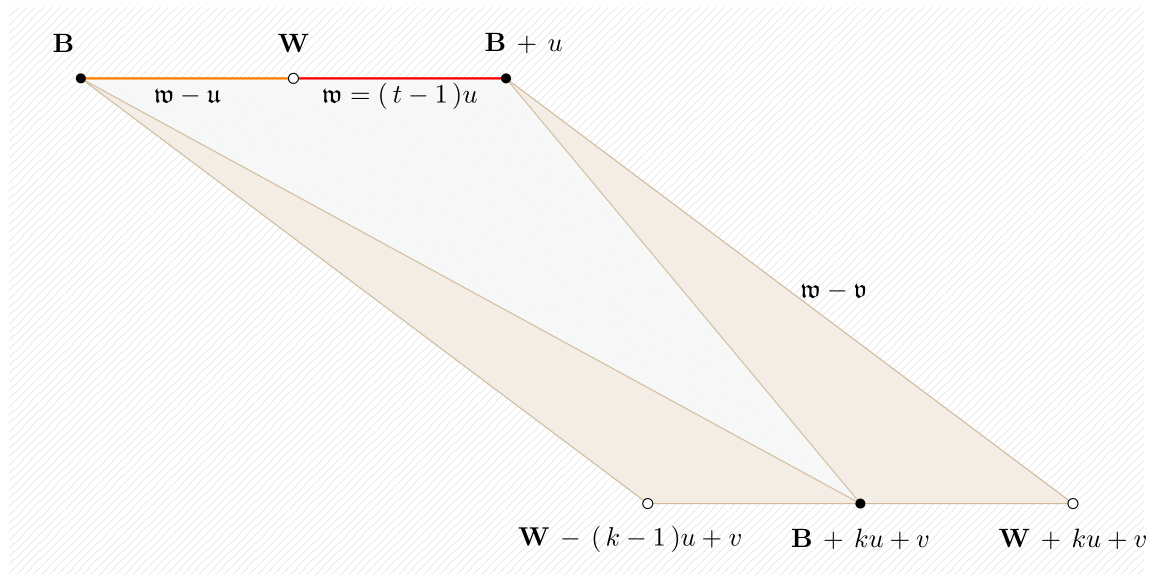}
    \caption{This picture compare the same transitional structure obtained as degeneration of structures of cylinder type and degeneration of structures of degenerate type. Note that we have used both notations to make the comparison easier to understand. Recall that all saddle connections are oriented from \(\pto B\) to \(\pto W\) by convention in the marked stratum.}
    \label{fig:comparisondeghex}
\end{figure}

\noindent Proposition \ref{prop:chambadjacency} combined with Lemma \ref{lem:degchambactone} yield the following straightforward corollary that completes the description of the adjacency of chambers of degenerate type with chambers of cylinder type.

\begin{cor}\label{cor:adjacency}
    Let \(A\in\slz\), then the chamber of cylinder type \(\textnormal{C}(\,-A(\,u\,)\,)\) is adjacent to all and only chambers of degenerate type of the form \(\Delta(\,A(\,u\,),\,A(\,ku+v\,)\,)\).
\end{cor}

%\noindent We now recall that the starting point \(\pto B\) can be any point of the complex plane since our constructions and resulting structures are invariant under translation. Therefore, for 

\smallskip

\subsection{Geometry of negative leaves}\label{ssec:gennegleaf}
We are now ready to describe the geometry of negative leaves. In the present section we combined the results obtained in \S\ref{ssec:beyondboundary} and \S\ref{ssec:adjacency} to provide a full description of the geometry of negative leaves. Unlike what we have done in \S\ref{sec:posleaves}, we first provide a description of unmarked leaves to deduce the geometry of marked leaves which is much easier.

\smallskip

\begin{figure}[!ht]
    \centering
    \includegraphics[width=1\linewidth]{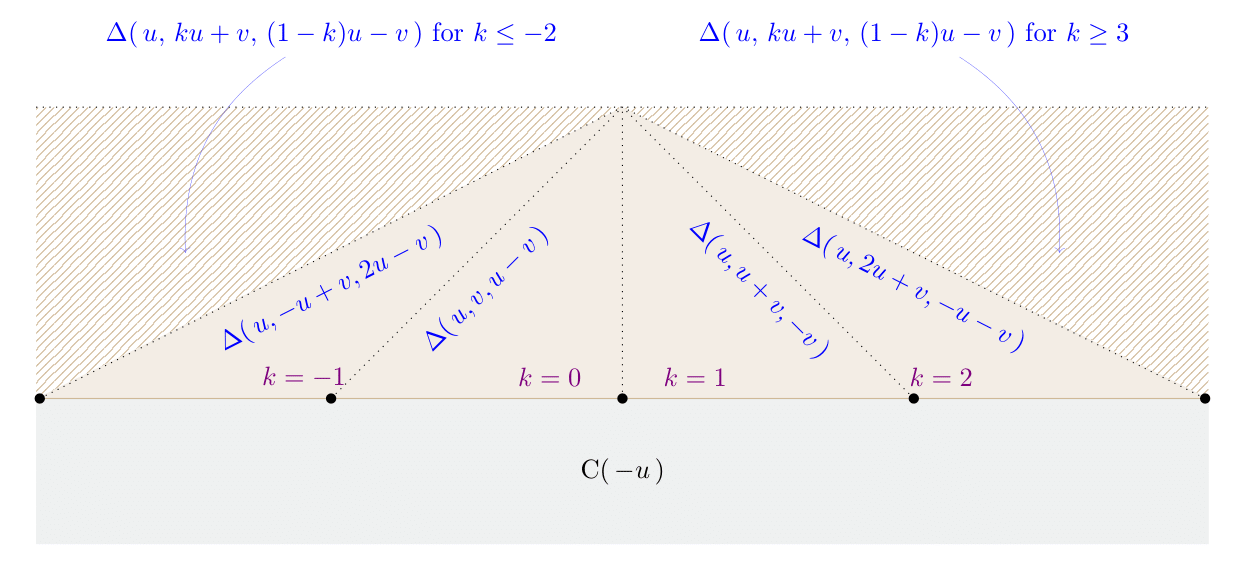}
    \caption{The beige-colored shaded area is entirely filled with chambers of degenerate type. The black-marked points correspond to structures in the minimal stratum. The figure also illustrates that two shear-adjacent chambers have overlapping edges, which are drawn with dotted lines to indicate that they \emph{cannot} be identified.}
    \label{fig:degchambadj}
\end{figure}

\smallskip

\noindent From \S\ref{ssec:adjacency} we already know that a chamber of cylinder type \(\textnormal{C}(\,-u\,)\) is adjacent to all and only chambers of degenerate type of the form \(\Delta(\,u,\,ku+v,\,(1-k)u+v\,)\), see Figure \ref{fig:degchambadj}. We introduce the following terminology.

\begin{defn}
    Two chambers of degenerate type, say \(\Delta_1\) and \(\Delta_2\), are \textit{shear-related} if and only if there is a characteristic triple \((\,u,\,v,\,u-v\,)\) and two integers \(k,\,h\in\mathbb Z\) such that
    \[ \Delta_1\,=\,\Delta(\,u,\,ku+v,\,(1-k)u-v\,)\quad \text{and}\quad \Delta_2\,=\,\Delta(\,u,\,hu+v,\,(1-h)u-v\,). \] In particular, we shall say two chambers are \textit{shear-adjacent} if \(h=k+1\).
\end{defn}

\noindent Two shear-adjacent chambers satisfy a desirable property, as stated in the following lemma, which will be useful later on. The proof is a straightforward application of classical Euclidean geometry and is left to the reader.

\begin{lem}\label{lem:shearadj}
    Two shear-adjacent chambers of degenerate type can be arranged in the Euclidean plane, by means of a translation, to form a Euclidean triangle.
\end{lem}

\noindent The relevant consequence of this lemma is that any two consecutive chambers of degenerate type, adjacent to the same chamber of cylinder type, are shear-adjacent and together form a Euclidean triangle as shown in Figure \ref{fig:degchambadj}. In this presentation two shear-adjacent chambers have an edge in common, but the corresponding walls bordering the two chambers cannot be identified because they parametrize different transitional structures. Moreover, we have already observed in \S\ref{ssec:beyondboundary} that every chamber of degenerate type gets the access to chambers of cylinder type (unless the degeneration passes through a structure in the minimal stratum). 

\smallskip

\noindent In \S\ref{ssec:beyondboundary} and \S\ref{ssec:adjacency} above, we provided local descriptions of a negative leaf. In \S\ref{ssec:beyondboundary}, we showed that a chamber of degenerate type is adjacent to three chambers of cylinder type. In \S\ref{ssec:adjacency}, instead, we showed that a chamber of cylinder type is adjacent to infinitely many chambers of degenerate type, which are obtained by deforming (via shear) a given chamber. Although these descriptions were given for marked structures, it is important to observe that they also hold for unmarked structures. Moreover, by forgetting the marking, it becomes straightforward to determine a collective description that takes into account the ones given previously. The core idea is the following: given a chamber \(\textnormal{C}(\,u\,)\), we glue all the adjacent chambers of degenerate type as shown in Figure \ref{fig:degchambadj}. Then, between any two adjacent chambers, we glue the chamber of cylinder type determined by the common side. By iterating this process indefinitely, at the limit we obtain the a surface \(\mathfrak D\) which we will show is homeomorphic to a plane, see Figure \ref{fig:negleaf}, and parametrizes all structures of degenerate type with prescribed absolute periods and their degenerations. Like in \S\ref{sec:posleaves}, \(\mathfrak D\) turns out to be bigger than the desired isoperiodic fiber because it also parametrizes those structures that lie in the minimal strata \(\Omega\mathcal M_1(2,-2)\). The desired isoperiodic leaf is thus given by
\begin{equation}
    \mathfrak L(\,\rho,\,\mu\,)\,=\,\mathfrak D\setminus \mathfrak L(\,\rho,\,(2,\,-2)\,).
    %\textnormal{deg}(\,\rho,\,\mu\,),
\end{equation}
%where \(\textnormal{deg}(\,\rho,\,\mu\,),\) is defined as the subset of degenerated structures, that is \(\mathfrak L(\,\rho,\,(2,\,-2)\,)\).

\smallskip

\begin{figure}[!ht]
    \centering
    \includegraphics[width=1\linewidth]{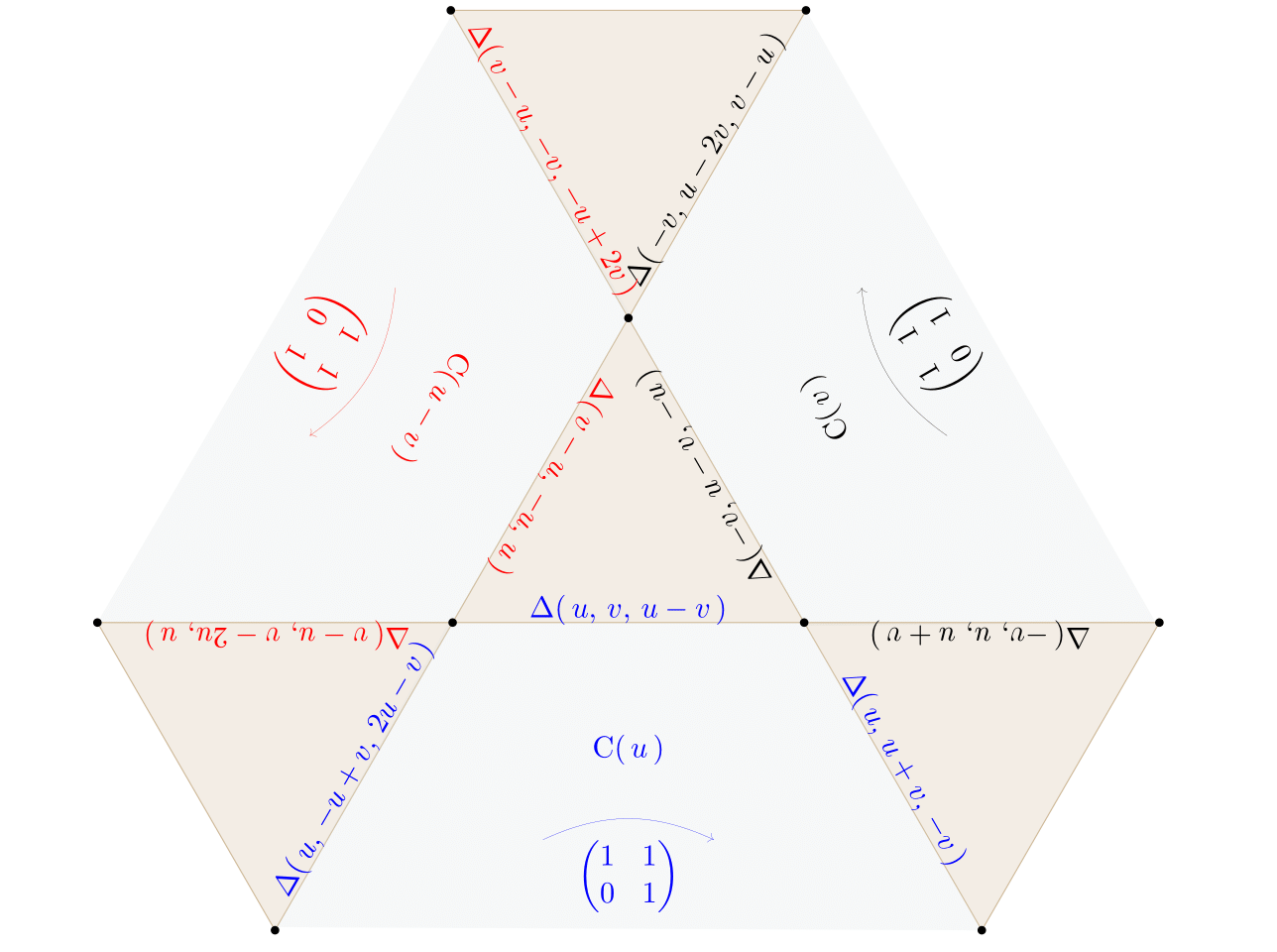}
    \caption{Local structure of an unmarked negative leaf. As above, the pale blue areas are chambers of cylinder type, and the beige areas are chambers of degenerate type. Each chamber of degenerate type is labeled with all the possible configurations that describe it. The different colors show how a chamber changes under the action of the matrix mentioned above, depending on the chosen configuration. %Recall that it is possible to switch from one configuration to another via the matrix \(R = STST\), which has order three. 
    The same pattern repeats by taking any chamber of degenerate type as the central chamber.}
    \label{fig:negleaf}
\end{figure}

\smallskip

\noindent It is straightforward to see that every unmarked negative leaf supports a Euclidean geometry with conical singularities of angle \(3\pi\), corresponding to structures in the minimal stratum. Indeed, recall that given a chamber of cylinder type \(\textnormal{C}(\,u\,)\), the boundary points of the form \(ku\), with \(k \in \mathbb{Z}\), correspond to structures in the minimal stratum. The argument above shows that each such point is adjacent to two chambers of cylinder type, each contributing an angle of \(\pi\), and to two shear-adjacent chambers. As a consequence of Lemma \ref{lem:shearadj}, these latter chambers contribute an additional angle of \(\pi\). Therefore, the points corresponding to structures in the minimal stratum are conical singularities of angle \(3\pi\). The following claim is a consequence of our discussion so far.

\begin{prop}
    Let \(\rho\) be a representation with negative volume. Then the metric completion of \(\mathfrak L(\,\rho,\mu\,)\) is connected and homeomorphic to a disk. Equivalently, every non-trivial representation with negative \(\rho\) is realized by a unique isoperiodic leaf \(\mathfrak L(\,\rho,\mu\,)\).
\end{prop}

\begin{proof}
    We only need to show that \(\mathfrak D\) is homeomorphic to a plane. A simple argument is as follows. In every chamber of degenerate type, we draw three segments between the center of the triangle and the three corners. We denote by \(\mathcal G\) the graph obtained in \(\mathfrak D\). We have already observed above that every conical singularity of \(\mathfrak D\) is incident to exactly two chambers of degenerate type. It follows that \(\mathcal G\) is a connected graph where every conical singularity is of valency two while every centroid is of valency three. Choosing an arbitrary conical singularity as the root, \(\mathcal G\) becomes an infinite binary tree. Each connected component of \(\mathfrak D\setminus \mathcal G\) is a chamber of cylinder type, a half-plane, with a third of triangle, corresponding to a part of a chamber of degenerate type cut out by two medians, glued on each boundary edge. In other words, each connected component of  \(\mathfrak D\setminus \mathcal G\) is a (topological) half-plane. It follows that \(\mathfrak D\) is homeomorphic to a plane.
\end{proof}

\smallskip 

\noindent The geometry and topology of marked leaves \(\mathfrak{ML}(\,\rho,\mu\,)\) is now understood. Let \(\mathfrak F\colon \mathfrak{ML}(\,\rho,\mu\,)\longrightarrow \mathfrak{L}(\,\rho,\mu\,)\) be the forgetful map and recall this is branched at the conical singularities. Lemma \ref{lem:double_cover} applies and hence we have the following:

\begin{cor}
    Let \(\rho\) be a representation with negative volume. Then the metric completion of \(\mathfrak{ML}(\,\rho,\mu\,)\) is connected and homeomorphic to the Loch Ness monster surface.
\end{cor}

\noindent This gives Theorem~\ref{thm:LochNess} for negative leaves. It remains to determine the Veech group of a marked negative leaf seen as a translation surface. It is not hard to see that the argument used to show Proposition \ref{prop:Veechpositive} works \textit{mutatis mutandis} even for negative leaves.  

\begin{prop}\label{prop:veechneg}
The Veech group of a negative leaf is a conjugate of \(\slz\).
\end{prop}

\noindent We conclude the present section with the following:

\begin{rmk}
Despite similar descriptions of positive and negative leaves, it is worth mentioning that positive and negative leaves are generally not isometric to each other.
\end{rmk}

\section{Non-arithmetic real leaves}\label{sec:nonarith}

\noindent In this section, we begin the study of the geometry and topology of isoperiodic leaves associated with representations of zero volume. As already hinted at earlier, we distinguish between arithmetic and non-arithmetic representations of zero volume, depending on whether the image of the representation is discrete or dense in a real line. We begin with non-arithmetic representations and will focus on the arithmetic ones in \S\ref{sec:arithm}.

\smallskip

\noindent In what follows it will be sufficient to consider representations \(\rho\colon\shomolzoo\longrightarrow \C\) such that \(\textnormal{Im}(\,\rho\,)\subset \mathbb R\), by a similar argument as in \S\ref{sec:posleaves} and \ref{sec:negleaves} that allowed us to restrict to representations of unit volume. In this case, any Euclidean rotation induces a homeomorphism of leaves that preserves the cellular structure in chambers and walls. We shall define these leaves as \textit{non-arithmetic real leaves}.

\medskip

 \subsection{Algebraic limit versus geometric limit}\label{ssec:alggeolimit} One of the advantages of distinguishing the non-discrete case from the discrete one is the possibility of determining the geometry and topology of non-arithmetic leaves as the geometric limit of isoperiodic negative leaves. More precisely the leading question of the present section is the following.

\smallskip

 \begin{quote}
     \textit{Let \(\rho_t\) be a sequence of representations with negative volume converging to a non-arithmetic representation \(\rho_\infty\) and let \(\mathfrak{ML}(\,\rho_t,\mu\,)\) be the sequence of negative leaves. Does it converge to \(\mathfrak{ML}(\,\rho_\infty,\mu\,)\) in some sense?}
 \end{quote}

\smallskip

\noindent The scope of this section is therefore to show that non-arithmetic leaves can be interpreted as the degenerative limit of a suitable sequence of negative leaves that preserves the cellular structure in chambers and walls. 

\begin{rmk}
    Notice that it is also possible to provide a direct description of these leaves which turns out to be quite similar to that of negative leaves.
\end{rmk}

\smallskip

\subsection{Contraction flow}\label{ssec:contraction}
In the present subsection, we define a one-dimensional flow on moduli space, which also gives a one-dimensional action on representations. Given a representation with zero volume \(\rho_\infty\), using this flow we will construct a sequence of representations \(\rho_t\) with negative volume converging to \(\rho_\infty\).

\smallskip

\noindent We start with the following notion. Let \(\theta\in\rp\) be any direction such that \(\theta\neq0\). We define the \(\theta-\)\textit{horizontal contraction flow} the monoid \(\{\,g(\,\theta,\,t\,)\,|\,t\ge0\,\}\) of \(\textnormal{GL}^{+}(2,\mathbb{R})\) conjugated to the monoid
    \begin{equation}
       \left\{\,\begin{pmatrix}\,1& \\ &e^{-t}\,\end{pmatrix}\,\big|\,t\ge0\,\right\}
    \end{equation}
contracting exponentially along direction \(\theta\) while preserving the horizontal direction. This then defines a one-dimensional flow on \(\Omega\mathcal M_1(\,-\,,-2)\) and \(\mathcal H_1(\,-\,,-2)\) (cf.\ \S\ref{subsec:glaction}).
This also gives an one-dimensional action on representations by post composition.
\smallskip

%Given a representation with zero volume \(\rho_\infty\), in the present subsection we determine a sequence of representations \(\rho_t\) with negative volume converging to \(\rho_\infty\). 

\noindent Let \(\rho_\infty\colon\shomolzoo\longrightarrow \C\) be a representation with \(\textnormal{Im}(\,\rho_\infty\,)=\mathbb{Z}\,u\,+\,\mathbb{Z}(\lambda\,u\,)\), where \(\lambda\in\mathbb R\).
As mentioned before, we may assume $u\in\mathbb{R}$ and in fact $u=1$. Note that the image of $\rho_\infty$ is dense in $\mathbb{R}$ if and only if $\lambda\notin\mathbb{Q}$.
Let \(\rho_o\colon\shomolzoo\longrightarrow \C\) be a representation with \(\textnormal{Im}(\,\rho_o\,)=\mathbb{Z}\,u\,\oplus\,\mathbb{Z}\,v\,\cong\mathbb Z[\,i\,]\).
%Without loss of generality we assume \(u=1\) so that the image of \(\textnormal{Im}(\,\rho_\infty\,) < \mathbb R\). We observe that \(\textnormal{Im}(\,\rho_\infty\,)\) is a dense subgroup of \(\mathbb R\) if and only if \(\lambda\notin\mathbb Q\). 
Without loss of generality, we may further assume \(v=-i\) so that $\rho_o$ is a representation with negative volume.
%\noindent In order to define the desired sequence we introduce the following notion. Let \(\theta\in\rp\) be any direction such that \(\theta\neq0\). We define the \(\theta-\)\textit{horizontal contraction flow} the monoid \(\{\,g(\,\theta,\,t\,)\,|\,t\ge0\,\}\) of \(\textnormal{GL}^{+}(2,\mathbb{R})\) conjugated to the monoid
%    \begin{equation}
%       \left\{\,\begin{pmatrix}\,1& \\ &e^{-t}\,\end{pmatrix}\,\big|\,t\ge0\,\right\}
%    \end{equation}
%contracting exponentially along direction \(\theta\) while preserving the horizontal direction.
We construct a sequence of representations as follow. Let \(\theta\in\,(\,0,\,\pi\,)\) be the unique value such that \(\lambda=-\cot(\,\theta\,)\), then
\begin{equation}
    \rho_t\,=\,g(\,\theta,\,t\,)\,\rho_o.
\end{equation}

\noindent By design the following properties hold: \(\rho_t(\,u\,)=u\) for every \(t\ge0\). Moreover, a simple argument from trigonometry and linear algebra shows that
 \(\rho_t(\,v\,)\longrightarrow \lambda\,u=\rho_\infty(\,v\,)\) for \(t\longrightarrow \infty\). As a consequence of our argument we have the following:

 \begin{prop}
     Every subgroup \(\Gamma\subset \mathbb R\) arises by contracting the lattice \(\mathbb Z[\,i\,]\) along a non-horizontal direction \(\theta\). Moreover \(\Gamma\) is dense in \(\mathbb R\) if and only if \(\cot(\,\theta\,)\notin\mathbb Q\).
 \end{prop}

\noindent In the next subsection we study how chambers of degenerate type and chambers of cylinder type degenerate under the action of \(\theta-\)horizontal contraction flows. We shall also assume \(\cot(\,\theta\,)\notin\mathbb Q\).

\subsection{Degenerations of chambers}\label{ssec:nonarithchambdeg} To understand how a negative leaf degenerates under the action of the contracting flow, it is sufficient to understand how the individual chambers of cylinder type and degenerate type degenerate. These have well-defined geometric shapes, as studied in previous chapters: half-planes in the first case, and triangles in the second.

\smallskip

\noindent Recall that the matrix $g(\,\theta,\,t\,)$ induces a homeomorphism between leaves $\mathfrak{ML}(\,\rho_o,\,\mu\,)$ and $\mathfrak{ML}(\,\rho_t,\,\mu\,)$ preserving the chamber structure (see Proposition~\ref{prop:negleafpres}). Each chamber of $\mathfrak{ML}(\,\rho_o,\,\mu\,)$ has a global coordinate $w$, and it is easy to see that $w_t:=g(\,\theta,\,t\,).w$ provides a global coordinate on the corresponding chamber of $\mathfrak{ML}(\,\rho_t,\,\mu\,)$. We can use $w_t$ to embed the chambers into $\mathbb{C}$ as half planes, and consider the Gromov-Hausdorff limits of these chambers as subset of $\mathbb{C}$ when $t\to\infty$. In the following, when we talk about the limit of any sequence of geometric objects, we refer to the Gromov-Hausdorff limit.

\smallskip

\noindent Recall that for a primitive element \(z\in\mathbb Z[\,i\,]\) (i.e.\ \(z=au+bv\) where \(a,b\in\mathbb Z\) are coprime), the cylinder chamber $\textnormal{C}(z)$ has boundary $\mathbb{R}\,z$ in the coordinate $w$. Now
    \begin{equation}
        \rho_t(\,z\,)=a\,\rho_t(\,u\,)+b\,\rho_t(\,v\,)\longrightarrow au+\lambda bu\,=\,(\,a\,+\,\lambda\,b\,)\,u \,:=\,z_\infty\quad \text{ for } \quad t\longrightarrow\infty. \qedhere
    \end{equation} 
Thus it is easy to see
%Since the action of the contraction flow is given by affine maps of the plane that preserve the horizontal direction while contracting along another direction, we will deduce the following results. To state the following lemmata, we will refer to the notation introduced in \S\ref{ssec:contraction}.

\begin{lem}\label{lem:nadegone}
    %Let \(\rho\) be a non-arithmetic representation and let \(\mathfrak{ML}(\,\rho,\,\mu\,)\) be the correspondinf leaf. 
    The limit of $g(\,\theta,\,t\,).\textnormal{C}(\,z\,)$ as $t\longrightarrow\infty$ is the chamber $\textnormal{C}(\,z_\infty\,)$. Moreover, for any point $w\in\partial\,\textnormal{C}(\,z\,)$, the structure of degenerate type corresponding to $w_t=g(\,\theta,t\,).w$ limits to the one corresponding to $w_\infty=\lim_{t\to\infty}w_t$ on $\partial\textnormal{C}(\,z_\infty\,)$.
\end{lem}

\noindent Similarly, since each chamber of degenerate type in a negative leaf is a compact triangle, using the same idea above one can show
\begin{lem}\label{lem:nadegtwo}
    The limit of any chamber of degenerate type under the action of $g(\,\theta,\,t\,)$ is a degenerate triangle of zero area with horizontal sides. Moreover, no boundary side shrinks provided \(\cot(\,\theta\,)\notin\mathbb Q\).
\end{lem}

%\begin{proof}[Sketch of proofs of Lemmata \ref{lem:nadegone} and \ref{lem:nadegtwo}]
%    The proofs of these lemmas rely on the same observation, namely how a generic primitive element degenerates. This can be seen from the following straightforward computation. Let \(z\in\mathbb Z[\,i\,]\) be a primitive element, that is \(z=au+bv\) where \(a,b\in\mathbb Z\) are coprime. Then
%    \begin{equation}
        %\rho_t(\,z\,)=a\,\rho_t(\,u\,)+b\,\rho_t(\,v\,)\longrightarrow au+\lambda bu\,=\,(\,a\,+\,\lambda\,b\,)\,u \quad \text{ for } \quad t\longrightarrow\infty. \qedhere
 %   \end{equation} 
%\end{proof}

\noindent Finally, to obtain the gluing pattern of chambers on a non-arithmetic real leaf, note that points on the boundary of chambers of a negative leaf has a well-defined limit under the action of $g(\,\,\theta,\,t\,)$. In particular,
\begin{lem}\label{lem:nadegthree}
    The \(\theta-\)horizontal contraction flow preserves the gluing pattern of chambers for every \(\theta\).
\end{lem}

\noindent As a consequence of Lemma \ref{lem:nadegthree} if two chambers are adjacent during the contraction, they will remain adjacent in the limit. Therefore, the non-arithmetic leaves \(\mathfrak{L}(\,\rho,\,\mu\,)\) of unmarked structures can be constructed in the same way as the negative ones, by gluing together chambers of cylinder type according to a precise criterion. Note that, strictly speaking, there are no chambers of degenerate type, as these are themselves degenerate triangles and hence correspond to segments on the boundary of some chamber of cylinder type. Nevertheless, we include them in the description because these segments encode essential information about the gluing pattern. We can now describe the geometry and topology of an unmarked non-arithmetic real leaf.

%\begin{rmk}
% In the description above, we have repeatedly used the word ``convergence'' without specifying its precise meaning. What can be shown is that, during the contraction, the limiting chamber arises as the limit of a sequence of spaces converging in the Gromov–Hausdorff topology.
%\end{rmk}

\begin{prop}
    Let \(\rho\) be a non-arithmetic representation. Then the metric completion of \(\mathfrak L(\,\rho,\mu\,)\) is connected and homeomorphic to a disk. Equivalently, every non-arithmetic representation \(\rho\) is realized by a unique isoperiodic leaf.
\end{prop}

\begin{proof}
Let \(\rho\) be a non-arithmetic representation and let \(\mathfrak L(\,\rho,\,\mu\,)\) be its isoperiodic leaf with metric completion \(\mathfrak D\). The representation \(\rho\) can be obtained as a limit of the \(\theta-\)horizontal contraction flow applied to a negative representation \(\rho_{o}\). As Theorem~\ref{thm:Unique} has been proved in \S\ref{sec:negleaves}
for negative leaves, \(\rho_{o}\) is realized by a unique leaf \(\mathfrak L(\,\rho_o,\mu\,)\) whose metric completion is a disk \(\mathfrak D_o\). We deduce from Lemma~\ref{lem:nadegthree} that the chambers of \(\mathfrak D_o\) and \(\mathfrak D\) have the same adjacency pattern. Homeomorphisms between individual chambers assemble to a global homeomorphism between \(\mathfrak D_o\) and \(\mathfrak D\) which is therefore connected and homeomorphic to a disk.
\end{proof}

\noindent Even in this case the geometry and topology of marked leaves \(\mathfrak{ML}(\,\rho,\mu\,)\) is now understood. Let us consider the forgetful map once again \(\mathfrak F\colon \mathfrak{ML}(\,\rho,\mu\,)\longrightarrow \mathfrak{L}(\,\rho,\mu\,)\) and recall once again that this is branched at the conical singularities. Lemma \ref{lem:double_cover} applies and hence we have the following:

\begin{cor}
    Let \(\rho\) be a non-aritmethic representation. Then the metric completion of \(\mathfrak{ML}(\,\rho,\mu\,)\) is connected and homeomorphic to the Loch Ness monster surface.
\end{cor}

\noindent This gives Theorem~\ref{thm:LochNess} for non-arithmetic leaves. 

\begin{rmk}
In \cite[Theorem 9]{KLS}, it is proved for each dense subgroup $\Gamma \subseteq \mathbb{R}$ with two generators, there are exactly two connected leaves for which the absolute period group coincides with $\Gamma$. There is no contradiction with Theorem~\ref{thm:Unique} because $\Gamma$ is the image of exactly two period characters, up to change of basis.
\end{rmk}

\subsection{Veech group} We finally determine the Veech group the Veech group of a non-arithmetic real leaf seen as a translation surface. We first need a number-theoretic lemma. 

\subsubsection{Some preliminaries} To set the stage for the lemma, we begin by recalling some standard and well-known results. Let \( D \geq 2 \) be a square-free integer. The ring of integers of the quadratic field \( \mathbb{Q}(\sqrt{D})\, \) is given by \( \mathbb{Z}[\,\gamma\,] \), where
\[
\gamma =
\begin{cases}
\sqrt{D} & \text{if } D \equiv 2,3 \pmod{4}, \\
\displaystyle\frac{1 + \sqrt{D}}{2} & \text{if } D \equiv 1 \quad \pmod{4},
\end{cases}
\]
is the standard generator of the ring. For any element \( x = s + t\sqrt{D} \) with \( s, t \in \mathbb{Q} \), its conjugate is defined by \( \overline{x} = s - t\sqrt{D} \), and its norm is given by \( N(x) = x \overline{x} \). The norm is multiplicative; that is, \( N(xy) = N(x)N(y) \). Recall that \( x \in \mathbb{Z}[\,\gamma\,] \) if and only if \( N(x) \in \mathbb{Z} \), and that \( x \) is a unit in \( \mathbb{Z}[\,\gamma\,] \) if and only if \( N(x) = \pm 1 \). The group of units in \( \mathbb{Z}[\,\gamma\,] \) is generated by \(-1\) and a fundamental unit \( u > 0 \).

\subsubsection{Number-theoretic lemma} Let \(\tau = (\,t,\, l,\, m\,)\in\mathbb Z^3\) such that \(t,\, m \neq 0\) and \(\gcd(\,t,\, l,\, m) = 1\). Consider the \(\mathbb{Z}\)-module defined by \(\Gamma = \langle\,t,\,l+m\,\gamma\,\rangle\subset \mathbb Q(\sqrt{D}\,)\). For any positive real number \(a\), if \(a\Gamma \subseteq \Gamma\), then \(a\) acts as a matrix with respect to the basis \(\{\,t,\, l + m\gamma\,\}\). We denote by \(\det(\,a\,)\) the determinant of this matrix.

\begin{lem}\label{lem:quadratic}
Consider the group
\[ G_{D,\tau}:=\Big\{\,a\in\mathbb{R}_+\,\,\big|\,\,a\,\Gamma\,=\,\Gamma, \,\,\det(a)\,=\,1\,\Big\}.
\]
The following statements hold:
\begin{enumerate}[label=\normalfont{(\arabic*)}]
    \item \( G_{D,\tau} \) is a subgroup of the units in \( \mathbb{Z}[\,\gamma\,] \);
    \smallskip
    \item Given a nontrivial element \( \eta = \alpha + \beta\gamma \) with \( N(\,\eta\,) = 1 \), \( G_{D,\tau} \) contains \(\eta\) if and only if \( m \) divides \( \beta \), and \( t \) divides \( (\beta/m) \cdot N(l + m\gamma) \);
    \smallskip
    \item Let \( k \) be the smallest non-negative integer such that \( u^k \in G_{D,\tau} \). Then \( G_{D,\tau} = \langle \,u^k \,\rangle \).
\end{enumerate}
\end{lem}

\begin{proof}
Set \( N(\,\Gamma\,) := \min\Big\{\, |\,N(x)\,| \,:\, x \in \Gamma,\, x \neq 0 \,\Big\} \). Clearly, \( N(\,\Gamma\,) > 0 \), and for any \( a \in \mathbb{Q}(\,\sqrt{D}\,) \), we have 
\begin{equation}
     N(\,a\,\Gamma\,) = |\,N(a)\,| \cdot N(\,\Gamma\,).
\end{equation}

\noindent This immediately implies that every element in \( G_{D,\tau} \) is a unit.

\medskip

\noindent For the second statement, we first note that
\[
\eta \cdot t = \alpha t + \beta t \gamma = \left(\, \alpha - \frac{\beta\,l}{m} \,\right)t + \frac{\beta\,t}{m}\big(\, l + m\gamma \,\big).
\]
For this expression to lie in \( \Gamma \), we must have \( m \mid \beta t \) and \( m \mid \beta l \). If \( \gcd(\,m,\,\beta\,) \neq |\,m\,| \), then both \( t \) and \( l \) share a nontrivial factor with \( m \), contradicting the assumption \( \gcd(\,t,\,l,\,m\,) = 1 \). Hence, \( m \mid \beta \). Write \( \beta = \beta_0 m \) for some integer \( \beta_0 \).

\medskip

\noindent Next, compute \( \eta \cdot \big(\, l + m\gamma \,\big) \). If \( \gamma = \sqrt{D} \), we have:
\begin{align*}
\eta \cdot \big(\, l + m\gamma \,\big)
&= \alpha l + \beta m \gamma^2 + \big(\, \beta l + \alpha m \,\big)\gamma \\
&= \alpha l + \beta m D + \big(\, \beta l + \alpha m \,\big)\gamma \\
&= -\beta_0\,N(\,l + m\gamma\,) + \big(\, \beta_0 l + \alpha \,\big)\big(\, l + m\gamma \,\big).
\end{align*}

\noindent If instead \( \gamma = \dfrac{1 + \sqrt{D}}{2} \), then:
\begin{align*}
\eta \cdot \big(\, l + m\gamma \,\big)
&= \alpha l + \beta m \left(\, \gamma + \frac{D - 1}{4} \,\right) + \big(\, \beta_0 l + \alpha + \beta \,\big) m\gamma \\
&= \alpha l + \beta_0 m^2 \cdot \frac{D - 1}{4} + \big(\, \beta_0 l + \alpha + \beta \,\big)\big(\, l + m\gamma \,\big) - \big(\, \beta_0 l + \alpha + \beta \,\big)l \\
&= -\beta_0\,N(\,l + m\gamma\,) + \big(\, \beta_0 l + \alpha + \beta \,\big)\big(\, l + m\gamma \,\big).
\end{align*}

\noindent In both cases, for this to lie in \( \Gamma \), we must have \( t \mid \beta_0\,N(\,l + m\gamma\,) \).

\medskip

\noindent Now consider the determinant. When \( \gamma = \sqrt{D} \), we compute:
\begin{align*}
\det(\eta)
&= \big(\, \alpha - \beta_0 l \,\big)\big(\, \beta_0 l + \alpha \,\big) + \beta_0^2\,N(\,l + m\gamma\,) \\
&= \alpha^2 - \beta_0^2 l^2 + \beta_0^2\,N(\,l + m\gamma\,) \\
&= N(\,\eta\,) - \beta_0^2\,N(\,l + m\gamma\,) + \beta_0^2\,N(\,l + m\gamma\,) = 1.
\end{align*}

\noindent If \( \gamma = \dfrac{1 + \sqrt{D}}{2} \), we get:
\begin{align*}
\det(\eta)
&= \big(\, \alpha - \beta_0 l \,\big)\big(\, \beta_0 l + \alpha + \beta \,\big) + \beta_0^2\,N(\,l + m\gamma\,) \\
&= \alpha^2 + \alpha \beta - \beta_0^2 l^2 - \beta \beta_0 l + \beta_0^2\,N(\,l + m\gamma\,) \\
&= N(\,\eta\,) - \beta_0^2\,N(\,l + m\gamma\,) + \beta_0^2\,N(\,l + m\gamma\,) = 1.
\end{align*}

\noindent Thus, when the divisibility conditions are satisfied, we conclude that \( \eta \in G_{D,\tau} \).

\medskip

\noindent Finally, for the third statement, recall that all positive units form a cyclic group generated by the fundamental unit \( u \). Therefore, \( G_{D,\tau} \) is generated by \( u^k \), where \( k \ge 0 \) is the smallest integer such that \( u^k \in G_{D,\tau} \).
\end{proof}

\begin{rmk}
    Note that if \( G_{D,\tau} \) contains a nontrivial element, say \( \eta \), then \( \eta^2 \) is always of norm \( 1 \), so the lemma above covers all cases.
\end{rmk}

\subsubsection{Veech groups} We use Lemma \ref{lem:quadratic} to prove the main result of the present subsection, namely we show the following characterization result. Recall that for a non-arithmetic representation \(\rho\), its image \(\textnormal{Im}(\,\rho\,)=\langle\,u,\,\lambda\,u\,\rangle\) and \(\theta=\arccot(\,\lambda\,)\).

\begin{prop}
The Veech group of a non-arithmetic real leaf is a conjugate of one of the following groups.
\begin{itemize}
    \item In the case \( \Gamma = t\,\mathbb{Z} + (l + m\gamma)\,\mathbb{Z} \), up to scaling, for some triple \( \tau = (\,t,\, l,\, m\,)\in\mathbb Z^3 \) with \( m \neq 0 \) and \( \gcd(\,t,\, l,\, m\,) = 1 \), and \( \gamma \) is the standard generator of the ring of integers in \( \mathbb{Q}(\sqrt{D}\,) \) for some square-free integer \( D \ge 2 \), then
    \[ V_{D,\tau} = \left\{\, \pm\begin{pmatrix} u & b \\ 0 & a \end{pmatrix}\,\, \big|\,\, a, b \in \mathbb{R},\, a > 0,\, u \in G_{D,\tau} \,\right\}. \]
    \smallskip
    \item Otherwise, if the ratios of periods do not all belong to a fixed quadratic number field then \[ V = \left\{\, \pm\begin{pmatrix} 1 & b \\ 0 & a \end{pmatrix} \,\, \big|\,\, a, b \in \mathbb{R},\, a > 0 \,\right\}. \]  
\end{itemize}
\end{prop}

\begin{proof}
Any element in the Veech group of a leaf whose absolute period group is contained in \( \mathbb{R} \) must preserve the real line. Such an element is a matrix of the form 
\[
    \begin{pmatrix} 
    a & b \\ 
    0 & c 
    \end{pmatrix}
\]
where \( a, c \in \mathbb{R}^\ast \) and \( b \in \mathbb{R} \). Moreover, elements of the Veech group must preserve the group \( \mathbb{Z} + \cot(\,\theta\,)\,\mathbb{Z} \). In other words, scaling by \( a \) must preserve the group \( \Gamma \). Unless \( \cot(\,\theta\,) \) satisfies a quadratic equation with integer coefficients, this forces \( a \in \{ \pm 1 \} \). We thus obtain precisely the group \( V \). Conversely, any element in \( V \) leaves each chamber invariant and fixes its boundary, hence lies in the Veech group. 

\smallskip

\noindent In the remaining cases, the ratios of absolute periods belong to a quadratic number field \( \mathbb{Q}(\sqrt{D}\,) \), with \( D \ge 2 \) square-free. Up to scaling of \( \chi \), we may assume \( \Gamma = t\,\mathbb{Z} \,+\, (\,l \,+\, m\gamma\,)\,\mathbb{Z} \). It follows from Lemma~\ref{lem:quadratic} that the coefficient \( a \) must lie in \( G_{D,\tau} \), provided \( a > 0 \). Conversely, any element in \( V_{D,\tau} \) simply permutes the chambers.
\end{proof}

\begin{rmk}
Note that it is possible for \( G_{D,\tau} = \{1\} \); for instance, choosing \( D = m = 3 \), then \( \beta \) is never divisible by \( 3 \), so the second group coincides with the first in such cases. On the other hand, Lemma~\ref{lem:quadratic} provides a criterion to construct parameters for which this group is nontrivial.
\end{rmk}

\bigskip

\section{Arithmetic real leaves}\label{sec:arithm}

\noindent We are left to consider arithmetic representations. In the present section we give a complete description of arithmetic real leaves associated to discrete representations \(\rho\colon\shomolzoo\longrightarrow \C\) such that \(\textnormal{Im}(\,\rho\,)\subset \mathbb R\). In this case, the absolute period group \(\Gamma=\,\lambda\, \mathbb{Z}\) for some nonzero real number \(\,\lambda\,\). Up to the \(\glplus\)-action, we may assume \(\Gamma=\mathbb{Z}\).

\smallskip

\noindent We have discussed chambers of cylinder type in an arithmetic real leaf in \S\ref{sec:chamcyltype}, and it is easy to see that chambers of degenerate type do not exist in such a leaf. Thus the main goal of the present section is to show how chambers of cylinder type are glued together.

\smallskip

\subsection{Chambers of cylinder type} Let \( \rho \) be an arithmetic representation, and let \( k \in\textnormal{Im}(\,\rho\,)=\mathbb{Z} \) be the absolute period of some simple closed curve. In \S\ref{sec:chamcyltype}, we have shown that chambers of cylinder type whose core curve has period $k$ are as many as positive integers that are coprime to \( k \). We have denoted each such chamber by \( \textnormal{C}(\,k,\,m\,) \). Although it was not stated explicitly, Corollary \ref{cor:oppositechambers} also applies in this case, \textit{mutatis mutandis}. In fact, there are two chambers in the marked stratum, each of which is a half plane bounded by \(\mathbb{R}\).

\begin{rmk}
    In the most general setting, if \(\textnormal{Im}(\,\rho\,)=\langle\,u\,\rangle\) with \(u\in\C^*\), then each chamber would be bounded by straight line \(\mathbb R\,u\).
\end{rmk}

\noindent One of them, denoted by \( \textnormal{C}^+(\,k,\,m\,) \), corresponds to the parameters \( u = k \) and \( v = m \), while the other, denoted by \( \textnormal{C}^-(\,k,\,m\,) \), corresponds to \( u = -k \) and \( v = -m \). In the coordinate \( w \), the chamber \( \textnormal{C}^+(\,k,\,m\,) \) is the upper half-plane, while \( \textnormal{C}^-(\,k,\,m\,) \) is the lower half-plane. They are images of each other under the map \( w \mapsto -w \), so we focus on \( \textnormal{C}^+(\,k,\,m\,) \).

\medskip

\subsection{How to realize a translation surface with arithmetic period character}\label{ssec:transarith} In this paragraph, we provide an explicit construction of translation surfaces with arithmetic period \textit{without cylinders}; that is, the structures obtained through this process are transitional. We introduce a specific notation that will be useful in the sequel. %For this purpose, let \(\rho\) be an arithmetic representation let \(\{\,\alpha,\,\beta\,\}\) be a basis in homology and suppose \(\rho(\,\alpha\,)=k\) and \(\rho(\,\beta\,)=m\). Since we are under the assumption \(\textnormal{Im}(\,\rho\,)=\mathbb Z\) then \(\gcd(\,k,\,m\,)=1\).

\smallskip

\noindent To simply our argument, we shall adopt the following notation. Given two positive integers $k,m$ that are coprime, let \( \ell_1,\,\ell_2,\,\ell_3 \ge 0 \) such that \(k=\ell_1\,+\,\ell_2\) and \(m=\ell_2\,+\,\ell_3\). Although \(\ell_i=0\) is admissible for at most two out of three parameters, let us assume for the moment that \(\ell_i\,>\,0\) for every \(i=1,2,3\). Let \( S(\,\ell_1,\,\ell_2,\,\ell_3\,) \) the translation surface of degenerate type constructed as follows. 

\smallskip

\noindent Let \(\textnormal{U}\subset \mathbb C\) be a half-plane with boundary given by a straight line \(r\) parallel to \(\mathbb R\), oriented from the left to the right, so that \(\textnormal{U}\) lies to the \textit{left} of the boundary. Let \(\pto P\in\partial\textnormal{U}\) be any point, define \(\pto Q=\pto P\,+\,\ell_1\) and consider the following chain:
\begin{equation}
    \pto P\, \longmapsto\, \pto Q\,\longmapsto\,\pto P\,+\,k\,\longmapsto\,\pto Q\,+\,m,
\end{equation}
where we are adopting the same convention introduced in \S\ref{sec:negleaves}. Notice that this chain is made by three horizontal segments, say \(e_1^+,\,e_2^+,\,e_3^+\), of lengths \( \ell_1,\,\ell_2,\,\ell_3 \) respectively by design. Denote by \(r_u^-\subset r\) the half-ray on the left of \(\pto P\) and denote by \(r_u^+\subset r\) the half-ray on the right of \(\pto Q\,+\,m\). In the same fashion we define another half-plane as follow. Let \(\textnormal{L}\subset \mathbb C\) be a half-plane with boundary \(r\), again oriented from the left to the right, so that \(\textnormal{L}\) now lies on \textit{right} of the boundary. Let \(\pto P\in\partial\textnormal{L}\) be any point and consider the following chain:
\begin{equation}
    \pto P\, \longmapsto\, \pto Q\,+\,m-k\,\longmapsto\,\pto P\,+\,m\,\longmapsto\,\pto Q\,+\,m.
\end{equation}
Notice that this chain is made by three horizontal segments, say \(e_3^-,\,e_2^-,\,e_1^-\), of lengths \( \ell_3,\,\ell_2,\,\ell_1 \) respectively by design. Like above, we denote by \(r_l^-\subset r\) the half-ray on the left of \(\pto P\) and denote by \(r_l^+\subset r\) the half-ray on the right of \(\pto Q\,+\,m\). These half-planes may be both regarded as non-compact pentagons with one ideal point at the infinity. We glue them as follow: We first identify \(r_u^+\) with \(r_l^+\) and, similarly, we identify \(r_u^-\) with \(r_l^-\). The resulting space is a homeomorphic to a slit plane. We then glue \(e_i^+\) with \(e_i^-\) for \(i=1,2,3\). The resulting space \( S(\,\ell_1,\,\ell_2,\,\ell_3\,) \) is homeomorphic to a once-punctured torus and has arithmetic period character by design; see Figure \ref{fig:transarith}. Since the edges \(e_1^{\pm},\,e_2^{\pm},\,e_3^{\pm}\) are all aligned the resulting surface has no cylinder and thus it belongs to the boundary of some chamber of cylinder type. In fact, in the next paragraph we shall see that these structures arise as deformation of cylinder type structure as realized in \S\ref{sec:chambdegtype}.

\begin{figure}[!ht]
    \centering
    \includegraphics[width=1\linewidth]{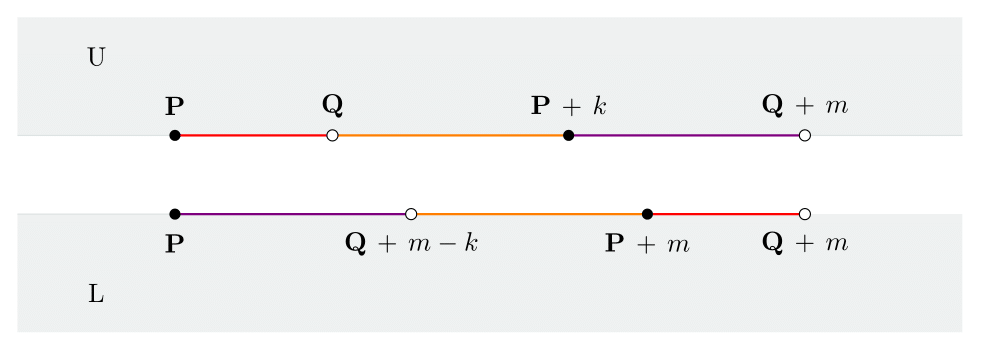}
    \caption{Realizing translation surfaces in \(\mathcal H_1(\,1,1,-2\,)\) with arithmetic period character and no cylinders. }
    \label{fig:transarith}
\end{figure}

\begin{rmk}\label{rmk:changemarkingarit}
    As we can see later, in this construction, coloring of the left-most point $\pto P$ does not determine the orientation. Indeed, the orientation is uniquely determined by the relative period \(w\), that is, by the orientation of the segment \(\overline{\pto B\pto W}\), and not by the relative position of \(\pto B\) and \(\pto W\) on the real line. If we endow \(\mathbb{R}\) with its usual orientation, then for a given marking of the zeros, the segment \(\overline{\pto B\pto W}\) either points to the left or to the right: both may happen on the boundary of the same chamber in the marked leaf.
    %the point \(\pto W\) lies to the right of \(\pto B\) if \(w > 0\), and similarly, the point \(\pto W\) lies to the left of \(\pto B\) if \(w < 0\). Changing the marking means reversing the orientation of the segment \(\overline{\pto B\pto W}\); therefore, under the new marking, the point \(\pto W\) lies to the right of \(\pto B\) if \(w < 0\), and to the left of \(\pto B\) if \(w > 0\).
\end{rmk}

\begin{rmk}
    Notice that if we allow exactly one of \(\ell_i\)'s to be zero, we obtain a surface in the minimal stratum \(\Omega\mathcal{M}_1(2,-2)\). In the case two out of three of \(\ell_i\)'s are zero, then we get a pinched torus.
\end{rmk}

\begin{rmk}
    We also notice that this construction is not specific to arithmetic representations but extends to non-arithmetic once by assuming \(\langle\,\ell_1+\ell_2,\,\ell_2+\ell_3\,\rangle\) is a dense subgroup of \(\mathbb R\).
\end{rmk}

\smallskip

\subsection{Degeneration towards boundary} Consider the chamber $\textnormal{C}^+(k,m)$. We determine how a sequence of structures of cylinder type in this chamber degenerates when it approaches to the boundary of the chamber. We distinguish two cases according to the following paragraphs. In both cases below we rely on our previous notation.

\smallskip

\subsubsection{Case \(m=0\)}\label{sssec:mis0}
We observe then that \(k=1\). First consider the surface with coordinate \(w=n+t+i\varepsilon\) realized as described in \S\ref{ssec:chambdes}, where \(n\ge0\) is an integer, \(0\le t<1\), and \(\varepsilon>0\). Letting \(\varepsilon\to0\), we obtain a sequence of surfaces of cylinder type that converges to a structure of degenerate type, and it is easy to see that such a structure is precisely \(S(\,t,\,1-t,\,n+t\,)\) realized as in \S\ref{ssec:transarith}.

\smallskip

\noindent For \(w=0\), the corresponding structure does not lie in the leaf it self, but its metric completion. It corresponds to a pinched torus: the surface is obtained by identifying two points on a complex plane. This gives the nodal curve in the Deligne-Mumford compactification of \(\mathcal{M}_{1,1}\).

\smallskip

\noindent Translation surfaces corresponding to coordinates with negative real part can be described similarly -- they are, in fact reflections of those with positive real part across a vertical line. Moreover, the nonzero integers on \(\mathbb{R}\) correspond to points in \(\Omega\mathcal{M}_1(2,-2)\).

\smallskip

\subsubsection{Case \(m\neq0\)}
In the chamber of cylinder type \(\textnormal{C}^+(\,k,\,m\,)\), consider the surface with coordinate \(w=t+i\varepsilon\) realized as described in \S\ref{ssec:chambdes}. Since \(u=k\) and \(v=m\) it must follows that \(w-v=t-m+i\varepsilon\). We have four sub-cases to discuss according to the mutual relations between \(m\) and \(t\).
\begin{itemize}
    \item[1.] Assume \(m-k<t<0\). Then \(-k<\mathfrak{Re}(\,w-v\,)<-m\). By taking \(\varepsilon\to0\), since the parallelogram with periods \(u\) and \(w\) degenerates, we obtain the surface \(S(\,|\,t\,|,\,k-m-|\,t\,|,\,m+|\,t\,|\,)\) -- in this case the segment $\overline{\pto B\pto W}$ points to the left.
    %the point \(\pto W\) lies on the left of \(\pto B\).
    \smallskip
    \item[2.] Assume \(0\,<t<\,m\). Then \(-m<\mathfrak{Re}(\,w-v\,)<0\). Letting \(\varepsilon\to0\), we get the surface of degenerate type \(S(\,k+t-m,\,m-t,\,t\,)\) -- in this case the segment $\overline{\pto B\pto W}$ points to the right.
    %in this case \(\pto W\) lies on the right of \(\pto B\).
\end{itemize}

\noindent There next two cases extend the previous ones. More precisely

\begin{itemize}
    \item[3.] Assume \(-(n+1)k+m\,<\,t\,<\,-nk+m\) for some integer \(n\ge1\). Then \(-(n+1)k<\mathfrak{Re}(w-v)<-nk\). Taking \(\varepsilon\to0\), we get the surface of degenerate type \(S(\,|\,t\,|,\,(n+1)k-m-|\,t\,|,\,m+|\,t\,|-nk\,)\) -- in this case the segment $\overline{\pto B\pto W}$ points to the left.
    %in this case the point \(\pto W\) lies on the left of \(\pto B\).
    Finally,
    \smallskip
    \item[4.] Assume \(nk+m\,<\,t\,<\,(n+1)k+m\) for some integer \(n\ge0\). Then \(nk<\mathfrak{Re}(w-v)<(n+1)k\). Taking \(\varepsilon\to0\), we get the surface of degenerate type \(S(\,t-m-nk,\,(n+1)k+m-t,\,t\,)\) -- in this case the segment $\overline{\pto B\pto W}$ points to the right.
    %in this case the point \(\pto W\) lies on the right of \(\pto B\).
\end{itemize}
\smallskip
We notice that the points \(nk+m\) for every integer \(n\) on the real line, together with $0$, correspond to surfaces in the minimal stratum \(\Omega\mathcal{M}_1(2,-2)\). %For \(m=-nk\) we have the same nodal curve mentioned in the former case in \S\ref{sssec:mis0}.

\begin{rmk}
    The description for the boundary of the chamber \(\textnormal{C}^-(\,k,\,m\,)\) can be obtained by taking \(w\mapsto -w\). Note that for surfaces of degenerate type on the boundary, taking \(w\mapsto -w\) yields the same surface but with a different marking of its zeros. See also Remark \ref{rmk:changemarkingarit}.
\end{rmk}

\smallskip

\subsection{Gluing of chambers} We now described how chambers of marked structures are glued together. In later parts, we will use this information to describe local flat geometry of marked leaves (e.g.\ singularities). We will not directly derive from this the global geometry and topology of marked leaves, but rather as in the previous sections, we first describe unmarked leaves, and then marked ones based on those.

\smallskip

\noindent Fix a pair of integers \((\,k,\,m\,)\) with \(0 \leq m < k\) and \(\gcd(\,k, m\,) = 1\). As discussed above, the corresponding chamber of cylinder type \(\textnormal{C}^+(\,k,\,m\,)\) is the upper half-plane bounded by the real line in the coordinate $w$. The boundary real line is divided by the points \(n k + m\) for \(n \in \mathbb{Z}\) and \(0\) into segments. We thus proceed step-by-step as follows.

\begin{itemize}
    \item[\textbf{1.}] First consider the segment \([\,m - k,\,0\,]\) on the boundary of \(\textnormal{C}^+(\,k,\,m\,)\). Let \(p = k - m\), and let \(r\) be the unique integer satisfying
    \[ \frac{2m - k}{k - m} < r \le \frac{m}{k - m},
    \]
    and set \(q = (r + 1)m - rk\). Then the chamber \(\textnormal{C}^-(\,p,\,q\,)\) has a boundary segment 
    \[ -[\, (r + 1)p + q,\, rp + q\,] = [\,m,\,k\,].
    \]
    One can easily check from our discussion in the last section that the corresponding surfaces on this segment match those on \([\,m - k,\,0\,]\) of \(\textnormal{C}^+(\,k,\,m\,)\).
    \smallskip
    \item[\textbf{2.}] Consider \([\, -(\,n+1\,)k + m,\,-nk + m\,]\) for some natural number \(n \ge 1\). Set \(p = (\,n+1\,)k - m\) and \(q = k\). The chamber \(\textnormal{C}^-(\,p,\,q\,)\) has a boundary segment \([\, -q,\,0\,] = [\, -k,\,0\,]\), where the corresponding flat surfaces match those on \([\, -(\,n+1\,)k + m,\,-nk + m\,]\) of \(\textnormal{C}^+(\,k,\,m\,)\).
    \smallskip
    \item[\textbf{3.}] For the segment \([\,0,\,m\,]\) (which does not exist for \(m = 0\)), set \(p = m\), \(q = (\,r + 1\,)m - k\), where \(r\) is the unique integer such that \(r \ge \frac{k}{m} - 1\) and \(r < \frac{k}{m}\). It is easy to see that the flat surfaces on the segment \([\, rk - m,\, (\,r+1\,)k - m\,] = [\,k - m,\,k\,]\) on the boundary of \(\textnormal{C}^-(\,p,\,q\,)\) match those on \([\,0,\,m\,]\) of \(\textnormal{C}^+(\,k,\,m\,)\). In fact, this is the ``inverse'' of the previous case.
    \smallskip
    \item[\textbf{4.}] For the segment \([\,nk + m,\,(n+1)k + m\,]\), where \(n\) is a nonnegative integer, set \(p = (\,n + 1\,)k + m\) and \(q = nk + m\). Then it is easy to see that the flat surfaces on the segment \([\,0,\,p - q\,] = [\,0,\,k\,]\) on the boundary of \(\textnormal{C}^-(\,p,\,q\,)\) match those on \([\,nk + m,\,(n+1)k + m\,]\) of \(\textnormal{C}^+(\,k,\,m\,)\). In fact, this is the ``inverse'' of the first case.
\end{itemize}

\smallskip

\noindent This description of gluing implies the following:

\begin{lem}
For every arithmetic representation \(\rho\) such that \(\textnormal{Im}(\,\rho\,)=\mathbb{Z}\), the marked real leaf \(\mathfrak{ML}(\,\rho,\mu\,)\) is connected. % up to a symplectic change of basis.
\end{lem}

\begin{proof}
Since \(\textnormal{C}^+(\,1,\,0\,)\) is glued to \(\textnormal{C}^-(\,1,\,0\,)\) along the segment \([\,{-}1,\,1\,]\), it suffices to show that, via successive gluings, we can reach any admissible pair \((\,k,\,m\,)\) starting from \((\,1,\,0\,)\). We observe that from any given pair \((\,k,\,m\,)\), it is possible to reach \((\,(\,n+1\,)k - m,\,k\,)\) for any natural number \(n\). We prove this by induction on \(m\). Since we can reach \((\,n+1,\,1\,)\) from \((\,1,\,0\,)\) for any natural number \(n\), the base case \(m = 1\) is satisfied.

\smallskip

\noindent Assume now that the statement holds for all \(m \le m_0\). Consider \(m = m_0 + 1\). Choose \(t \le m_0\) such that \(\gcd(\,t,\,m\,) = 1\). By the inductive hypothesis, the pair \((\,m,\,t\,)\) can be reached by gluing starting from \((\,1,\,0\,)\). Then, for any natural number \(n\), we can also reach the pair \((\,n m + (\,m - t\,),\,m\,)\), which gives all admissible pairs with second index \(m\). Therefore, the statement holds for \(m = m_0 + 1\) as well. By induction, any admissible pair of indices can be obtained via gluing, as desired.
\end{proof}

\subsection{Geometry of arithmetic real leaves} We now provide a description of the flat geometry of marked real leaves. The first result in this sense is the following:

\begin{lem}
The singularities of any arithmetic real leaf in the marked stratum \(\mathcal{H}_1(1,1,-2)\) are cone points of magnitude \(6\pi\).
\end{lem}

\begin{proof}
Note that, by the description of the gluing, every singularity arises from the points \( n\,k + m \) and \( 0 \) on the boundary of \(\textnormal{C}^+(\,k,\,m\,)\), or from \( n\,k - m \) and \( 0 \) on the boundary of \(\textnormal{C}^-(\,k,\,m\,)\). Moreover, it is easy to see that each such point is glued to the point \( 0 \) on the boundary of some chamber. Thus, it suffices to show that six chambers are glued at the point \( 0 \) in each \(\textnormal{C}^+(\,k,\,m\,)\). We can now verify the following, based on the gluing rules:

\begin{itemize}[itemsep=4pt, topsep=4pt]
    \item When \(2m - k > 0\), the point \(0\) on \(\partial \textnormal{C}^+(\,k,\,m\,)\) is glued to the point \({-}m\) on \(\partial \textnormal{C}^-(\,p,\,q\,)\), where \(p = k - m\), and \(q = (\,r+1\,)m - r\,k\) for the unique integer 
    \[ \frac{2m - k}{k - m} < r \le \frac{m}{k - m},
    \]
    as well as to the point \({-}k + m\) on \(\partial \textnormal{C}^+(\,m,\,2m - k\,)\), the point \(0\) on \(\partial \textnormal{C}^-(\,k,\,m\,)\), the point \(m\) on \(\partial \textnormal{C}^+(\,p,\,q\,)\), and the point \(k - m\) on \(\partial \textnormal{C}^-(\,m,\,2m - k\,)\). So, six chambers are glued at \(0\).

    \item When \(2m - k < 0\), the point \(0\) on \(\partial \textnormal{C}^+(\,k,\,m\,)\) is glued to the point \({-}m\) on \(\partial \textnormal{C}^-(\,k - m,\,m\,)\); also to the point \({-}k + m\) on \(\partial \textnormal{C}^+(\,m,\,q\,)\), where \(q = (\,r + 2\,)m - k\) for the unique integer \(r \in [\,k/m - 2,\,k/m - 1\,)\); the point \(0\) on \(\partial \textnormal{C}^-(\,k,\,m\,)\), the point \(m\) on \(\partial \textnormal{C}^+(\,k - m,\,m\,)\), and the point \(k - m\) on \(\partial \textnormal{C}^+(\,m,\,q\,)\). Again, six chambers are glued at \(0\).

    \item When \(k = 2m\), the only possible pair is \((\,k,\,m\,) = (\,2,\,1\,)\). In this case, the point \(0\) on \(\partial \textnormal{C}^+(\,2,\,1\,)\) is glued to both \({-}1\) and \(1\) on \(\partial \textnormal{C}^-(\,1,\,0\,)\), to both \({-}1\) and \(1\) on \(\partial \textnormal{C}^+(\,1,\,0\,)\), and to \(0\) on \(\partial \textnormal{C}^-(\,2,\,1\,)\). Counting the chambers \(\textnormal{C}^\pm(\,1,\,0\,)\) twice, we again obtain six chambers glued at \(0\).
\end{itemize}

\noindent All possibilities have been considered. Since each chamber provides a contribution of \(\pi\), the desired claim follows.
\end{proof}

\smallskip

\noindent Recall that the point \(0\) on the boundary of \(\textnormal{C}^\pm(\,1,\,0\,)\) corresponds to a pinched torus, so it does not lie on the leaf itself, but rather in its metric completion. In this completion, it does not produce a singularity, since the two chambers \(\textnormal{C}^\pm(\,1,\,0\,)\) are glued along the segment \([\,{-}1,\,1\,]\). We refer to the corresponding point in the completion of the leaf as its \emph{center}.

\smallskip

\noindent As in the case of positive and negative leaves, a real arithmetic leaf supports an involution \(\iota\), given by identifying \(\textnormal{C}^\pm(\,k,\,m\,)\) via \(w \mapsto -w\). We have

\begin{lem}
For every arithmetic representation \(\rho\) such that \(\textnormal{Im}(\,\rho\,)=\mathbb{Z}\). The quotient of the completion of the marked leaf \(\mathfrak{ML}(\,\rho,\mu\,)\) by the involution \(\iota\) is a complex disk. The projection is ramified at the cone points and the center.
\end{lem}

\begin{proof}
The only statement that requires a proof is the topology of the quotient. The others follow immediately from the description of gluing and the definition of \(\iota\).

\smallskip

\noindent Under \(\iota\), a chamber \(\textnormal{C}^+(\,k,\,m\,)\) is identified with \(\textnormal{C}^-(\,k,\,m\,)\), so we shall use \(\textnormal{C}(\,k,\,m\,)\) to denote the corresponding chamber in the quotient. Starting with \(\textnormal{C}(\,1,\,0\,)\), we can obtain the whole leaf inductively as follows.
\begin{itemize}
    \item We start by gluing the segment \([\,0,\,1\,]\) to \([\,{-}1,\,0\,]\) on \(\partial\textnormal{C}(\,1,\,0\,)\). This creates a cone angle of \(\pi\) at the center. Let the resulting surface be \(\mathfrak{L}_0\). The boundary of \(\mathfrak{L}_0\) is topologically a line, being the image of \([\,1,\,\infty\,)\) and \((\,{-}\infty,\,-1\,]\) on \(\partial\textnormal{C}(\,1,\,0\,)\), where \(\pm1\) map to the same point \(p_{0\,1}\).
    \smallskip
    \item The segment \([\,{-}1,\,1\,]\) on \(\partial\textnormal{C}(\,2,\,1\,)\) is glued to \(\mathfrak{L}_0\), so that \(0\) on \(\partial\textnormal{C}(\,2,\,1\,)\) is glued to \(p_{0\,1}\). Let this new surface be \(\mathfrak{L}_1\). It has two boundary components, and each boundary component consists of two rays coming from the boundaries of two different chambers meeting at a point. Let these points be \(p_{1\,1}\) and \(p_{1\,2}\) for the two components.
    \smallskip
    \item Inductively, the surface \(\mathfrak{L}_n\) has \(2^n\) boundary components, each with a distinguished point \(p_{n\,i}\). The next surface \(\mathfrak{L}_{n+1}\) is obtained from \(\mathfrak{L}_n\) by gluing the origin of some chamber \(\textnormal{C}(\,k_i,\,m_i\,)\) to \(p_{n\,i}\) for each \(i\). This creates \(2^{n+1}\) boundary components, each with a distinguished point \(p_{n+1\,j}\).
\end{itemize}

\begin{figure}[htp]
    \centering
    \includegraphics[width=1\linewidth]{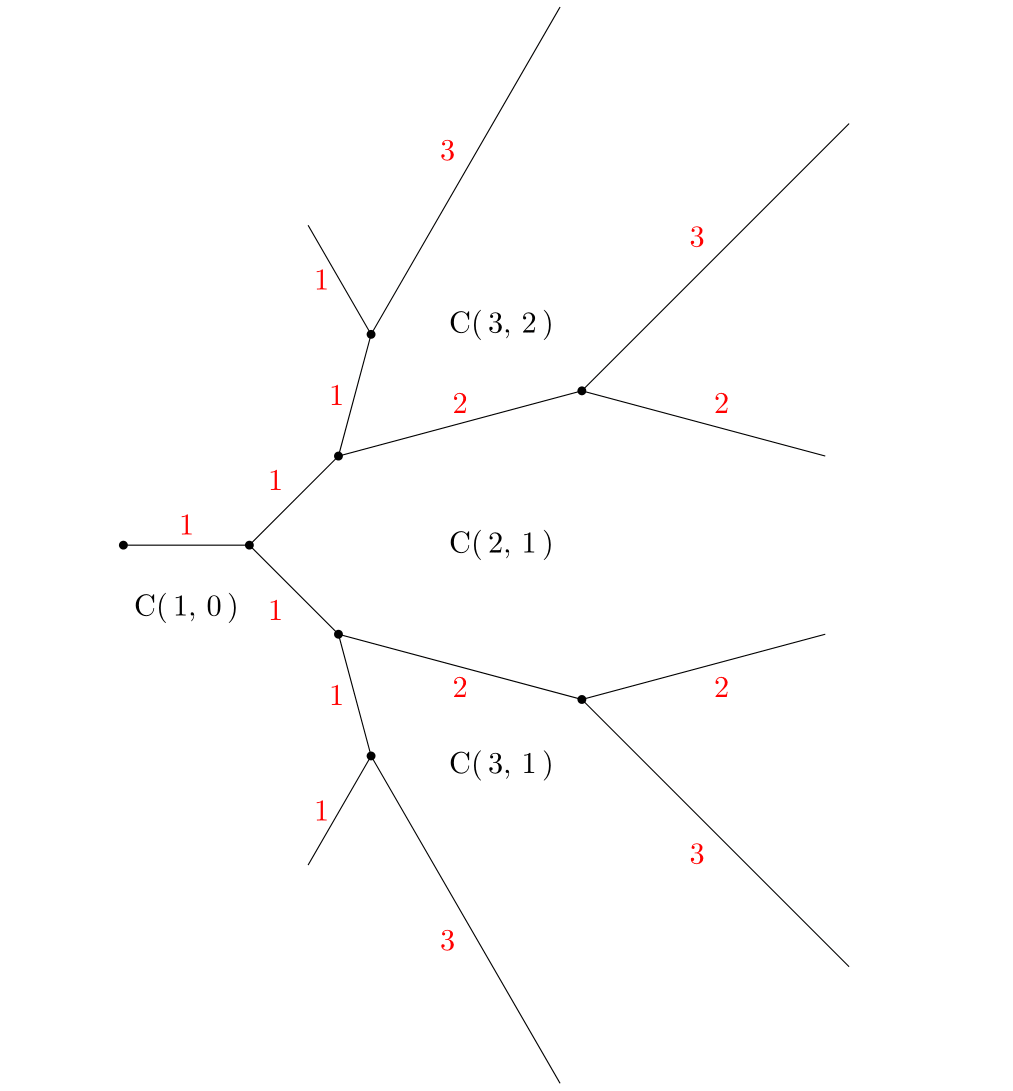}
    \caption{A portion of the arithmetic leaf in the unmarked stratum. Numbers in read are (relative) lengths of the segments. The wall is a binary tree, consisting of surfaces of degenerate type. The root and branching points of the binary tree are the pole and the zeros, respectively, of the quadratic differential defining the flat structure of the leaf}
    \label{fig:binary_tree}
\end{figure}

\noindent It is easy to see that the resulting surface \(\mathfrak{L}(\,\rho,\mu\,) = \bigcup_{n \ge 0} \mathfrak{L}_n\) is topologically a plane. In fact, the set of surfaces of degenerate type in the quotient forms a binary tree rooted at the center; see Figure~\ref{fig:binary_tree}. The branching points of the tree are precisely the zeros of the quadratic differential. If we identify the ends of the tree with a Cantor set, the chambers can be identified with the complementary intervals.

\smallskip

\noindent In terms of complex structures, the completion of the leaf is either a complex plane or a complex disk. It follows from Corollary~\ref{cor:hyp} that the leaf itself must be a punctured disk, so the completion is biholomorphic to a disk.
\end{proof}

\noindent The gluing described above implies that marked arithmetic leaves satisfy the properties of Theorem~\ref{thm:LochNess}. Indeed, we invoke Lemma~\ref{lem:double_cover} once again, we conclude that the marked leaf is a Loch Ness Monster. Moreover, in the quotient by involution, the simple pole of the quadratic differential defining the flat structure comes from the center of the leaf.

\begin{rmk}
    Again, we remark that while this gives the same topology as positive and negative leaves, the translation structure is very different.
\end{rmk}

\subsection{Veech group}

\noindent Finally, we determine the Veech group of an arithmetic real leaf as a translation surface. This concludes the proof of Theorem~\ref{thm:Veech}.

\begin{prop}
The Veech group of an arithmetic real leaf is a conjugate of the following group:
\[ V=\left\{\pm\begin{pmatrix}1&b\\0&a\end{pmatrix}\,\big|\,a,b\in\mathbb{R}, a>0\right\}.\]
\end{prop}

\begin{proof}
Indeed, any element in the Veech group must preserve the real line, and maps $\mathbb{Z}$ to itself. This gives precisely the group $V$. Conversely, any element in $V$ leaves each chamber invariant and its boundary fixed, so is indeed in the Veech group.
\end{proof}

\medskip

\section{Conformal geometry of leaves}\label{sec:conf_geom}

\noindent In this final section, we explore the conformal geometry of the leaves in the unmarked stratum. Our goal is to provide additional details while remaining consistent with the main focus developed in the previous sections. In particular, we will prove the last theorem stated in the introduction, namely Theorem~\ref{thm:confgeoleaves}. Description and drawing of the conformal geometry of leaves date back to McMullen's work in \cite{Mc}.

\subsection{Biholomorphic maps to the Teichm\"uller space}\label{ssec:biholototeich} We begin with explicit map from each leaf to the Teichm\"uller space \(\mathcal{T}_{1,1}\) case-by-case as follows.

\subsubsection{Positive and negative leaves} Let \(\rho\) be any representation with non-zero volume and let \(\mathfrak{L}(\,\rho,\,\mu\,)\) be the associated leaf. We realize an explicit map from any such a leaf to the Teichm\"uller space \( \mathcal{T}_{1,1} \), which serves as a lift of the natural map to the moduli space \( \mathcal{M}_{1,1} \).

\smallskip

\noindent Suppose first that the period group \( \Gamma \) is a lattice in \( \mathbb{C} \). As in sections \ref{sec:posleaves} and \ref{sec:negleaves}, for simplicity we shall assume \( \textnormal{Im}(\,\rho\,)=\Gamma = \mathbb{Z}[\,i\,] \). Given a point \( (X,\,\omega) \in\mathfrak{L(\,\rho,\,\mu\,)}\), there is an isomorphism \( H_1(\,X,\,\mathbb{Z}\,) \cong \mathbb{Z}[\,i\,] \) determined by the integration of \( \omega \). Moreover, if we remain within a positive (resp.\ negative) leaf, this isomorphism preserves (resp.\ reverses) the symplectic pairing. Therefore, each such point determines a \emph{homologically marked} translation surface. In genus one, the space of homologically marked elliptic curves coincides with the Teichm\"uller space, since the Torelli group is trivial.

\smallskip

\noindent The map described above is clearly injective: if two translation surfaces \( (X_1,\,\omega_1) \) and \( (X_2,\,\omega_2) \) define the same marked surface, then \( X_1 = X_2 \), and \( \omega_1 - \omega_2 \) must be an exact holomorphic 1-form. As such, \( \omega_1 = \omega_2 \). The map is also surjective: given any elliptic curve \( \mathbb{C} \big/ (\,\mathbb{Z} + \mathbb{Z}\tau\,) \), with \( \tau \in \mathbb{H} \), we can construct a holomorphic differential with a double pole as
\[
\omega = \big(\,a + b\,\wp_\tau\,\big)\,dz,
\]
where \( \wp_\tau \) denotes the Weierstrass \( \wp \)-function associated to the lattice. Solving the system
\[
\int_0^1 \big(\,a + b\,\wp_\tau\,\big)\,dz = 1 \qquad \text{and} \qquad \int_0^\tau \big(\,a + b\,\wp_\tau\,\big)\,dz = \pm i,
\]
we obtain a differential whose period group is \( \mathbb{Z}[\,i\,] \).

\smallskip

\subsubsection{Non-arithmetic real case}
Essentially the same argument works for the non-arithmetic real case as well: without loss of generality we may assume $\Gamma=\mathbb{Z}+\omega\mathbb{Z}$, where $\omega$ is a irrational number. Replacing $\pm i$ with $\omega$ in the arguments above, we obtain the desired map.

\smallskip

\subsubsection{Arithmetic real case}
Next, consider the case of arithmetic real leaves. Again, for simplicity, we may assume \( \Gamma = \mathbb{Z} \). Recall that the leaf \( \mathfrak{L}(\,\rho,\,\mu\,) \), with period group \( \Gamma \), is biholomorphic to a punctured disk. We construct a covering map from \( \mathcal{T}_{1,1} \) to \( \mathfrak{L}(\,\rho,\,\mu\,) \) as follows.

\smallskip

\noindent For any \( \tau \in \mathbb{H} \), one can solve for \( a, b \in \mathbb{C} \) in the system
\[
\int_0^1 \big(\,a + b\,\wp_\tau\,\big)\,dz = 0 \qquad \text{and} \qquad \int_0^\tau \big(\,a + b\,\wp_\tau\,\big)\,dz = 1,
\]
exactly as in the previous case. This determines a point \( (\,X,\,\omega\,) \) on the leaf, where \( X \cong \mathbb{C} \big/ (\,\mathbb{Z} + \mathbb{Z}\tau\,) \) and \( \omega = a + b\,\wp_\tau \). However, since the horizontal period vanishes, it is easy to verify that the translation surface associated to \( \tau \) is the same as that for \( \tau + 1 \). In particular, the map from \( \mathcal{T}_{1,1} \cong \mathbb{H} \) to \( \mathfrak{L}(\,\rho,\,\mu\,) \) descends to a map
\[
\mathbb{H} \big/ \langle\, z \mapsto z+1 \,\rangle \,\longrightarrow\, \mathfrak{L}(\,\rho,\,\mu\,).
\]

\smallskip

\noindent We claim that this map is bijective. Indeed, an inverse can be constructed as follows. Given \( (\,X,\,\omega\,) \in \mathfrak{L}(\,\rho,\,\mu\,) \), choose a symplectic basis \( \alpha,\,\beta \) of \( H_1(\,X,\,\mathbb{Z}\,) \) such that
\[
\int_\alpha \omega = 0 \qquad \text{and} \qquad \int_\beta \omega = 1.
\]
This induces a marking of \( H_1(\,X,\,\mathbb{Z}\,) \). The choice of \( \beta \) is not unique, as one can always replace \( \beta \) by \( \beta + n\alpha \) for any \( n \in \mathbb{Z} \). Thus, we obtain a well-defined map from \( \mathfrak{L}(\,\rho,\,\mu\,) \) to \( \mathbb{H} \big/ \langle\, z \mapsto z+1 \,\rangle \), which is clearly inverse to the map defined above.

\subsection{The Teichm\"uller geometry of cylinder chambers}
 Recall that the Teichm\"uller space \( \mathcal{T}(\,S\,) \) is endowed with a mapping class group invariant complete metric \( d_{\mathcal T} \), called the \emph{Teichm\"uller metric}. Given two points in Teichm\"uller space \( \mathcal{T}(\,S\,) \) represented by \( f\colon S \to X \) and \( g\colon S \to Y \), the Teichm\"uller distance is given by
\[
d_{\mathcal T}(\,[f],\,[g]\,) := \frac{1}{2} \log \inf \left\{\,K \,\Big|\, h\colon X \to Y \text{ is \( K \)-quasi-conformal and homotopic to } g \circ f^{-1} \,\right\}.
\]

\smallskip

\noindent For \( S = S_{1,1} \), the Teichm\"uller space \( \mathcal{T}_{1,1} = \mathcal{T}(\,S_{1,1}\,) \) is biholomorphic to \( \mathbb{H} \), and it is well-known that the Teichm\"uller metric on \( \mathcal{T}_{1,1} \) agrees with the hyperbolic metric, up to a constant multiple.

\smallskip

\noindent Recall that every cylinder chamber \( \textnormal{C} \) is a half-plane, whose boundary is parallel to \( \mathbb{R}u \) for some complex number \( u \). Let us denote by \( X_t \) the point in Teichm\"uller space \( \mathcal{T}_{1,1} \) corresponding to the surface with period vector \( t u \) on the boundary of \( \textnormal{C} \). Since each leaf is biholomorphic to a quotient of \( \mathcal{T}_{1,1} \cong \mathbb{H} \), we may map the chamber \( \textnormal{C} \) conformally, and hence isometrically, to a region in the upper half-plane \( \mathbb{H} \), such that the limit point at infinity of the chamber is mapped to \( \infty \), and \( X_o \) corresponds to the point \( i \). Let \( \sigma_\textnormal{C}(\,t\,) \) denote the image of \( X_t \) under this map. The main goal of this section is to prove the following:

\begin{thm}\label{thm:teich_geometry}
    For any chamber of cylinder type \(\textnormal{C}\), we have \(\sigma_{\textnormal{C}}(t) \asymp t + i \log t\). More precisely, there exist positive constants \(D, T\) depending only on the leaf containing \(\textnormal{C}\) so that 
    \[ d_{\mathbb{H}}\Big(\,\sigma_{\textnormal{C}}(\,t\,),\, t + i \log t\,\Big) \le D 
    \] when \(|\,t\,| \ge T\).
\end{thm}

\noindent In particular, a chamber of cylinder type \(\textnormal{C}\) does not contain any horodisks, and is also not contained in a \(D\)-neighborhood of some complete Teichm\"uller geodesic for any \(D > 0\). As before, for simplicity, we assume the period group \(\Gamma = \mathbb{Z}[\,i\,]\). The result for other leaves can be obtained via the \(\glplus\)-action.

%Let us first consider the cylinder chamber in the positive leaf with boundary parallel to the real line. In the coordinate described in \S\ref{sec:chamcyltype}, this chamber is identified with the upper half plane.

\smallskip

\noindent First consider the torus obtained as follows. Let \(P\) be the unit square given by \([0,1] \times [0,1]\) and \(Q\) the rectangle given by \([0,1] \times [0,s]\) for some \(s > 0\). As in \S\ref{sec:chamcyltype}, we obtain a torus \(T_s\) by gluing the vertical sides of \(P\), the vertical sides of \(\cp \setminus Q\), and the corresponding horizontal sides of \(P\) and \(\cp \setminus Q\). By horizontal and vertical reflection symmetry, it is easy to see that \(T_s\) corresponds to a point \(i \tau(s)\) on the imaginary axis of \(\mathcal{T}_{1,1} \cong \mathbb{H}\). We can estimate \(\tau(s)\) as follows.

\begin{lem}\label{lem:Teich_asymp}
    The function \(\tau(s)\) satisfies \(\tau(s) \asymp \log(s)\).
\end{lem}

\noindent The lower bound can be obtained from the following result in \cite{DN}:
\begin{thm}\label{thm:DN}
    The modulus of the quadrilateral \(Q^c=\cp\setminus Q\), defined as the extremal length of the path family in \(\cp\setminus Q\) connecting top and bottom sides of \(Q\), satisfies \(\Mod(\,Q^c\,)\sim\frac1\pi\log(s)\).
\end{thm}

\noindent Here \(f(s)\sim g(s)\) means \(\lim f(s)/g(s)=1\), so the \(\frac1\pi\log(s)\) gives the exact asymptotics of \(\Mod(\,Q^c\,)\) as \(s\to\infty\).

\begin{proof}[Proof of Lemma~\ref{lem:Teich_asymp}.]
For the lower bound, note that \(\tau(s)\) is equal to the extremal length of the path family homotopic to the curve obtained by concatenation of a vertical side of \(P\) and a vertical side of \(Q^c\). By additive property of extremal length, we conclude that \(\tau(s)\) is bounded below by the sum of the moduli of \(P\) and \(Q^c\). Since the modulus of \(P\) is 1, the theorem quoted above implies \(\tau(s) \ge 1 + \frac{1}{\pi} \log(s)\).

\smallskip
\noindent For the upper bound, we construct a quasiconformal map \(f : Q^c \to \mathbb{C} - B\big((s + 1/2)i, 1/2\big) - B(-1/2 i, 1/2)\) as follows. The quadrilateral \(Q^c\) is divided into several regions. Outside the red and blue regions in Figure~\ref{fig:qc}, the map \(f\) is identity / translation. In the blue regions, \(f\) is a scaling in the \(x\)-direction.

\begin{figure}[htp]
    \centering
    \includegraphics[width=0.6\linewidth]{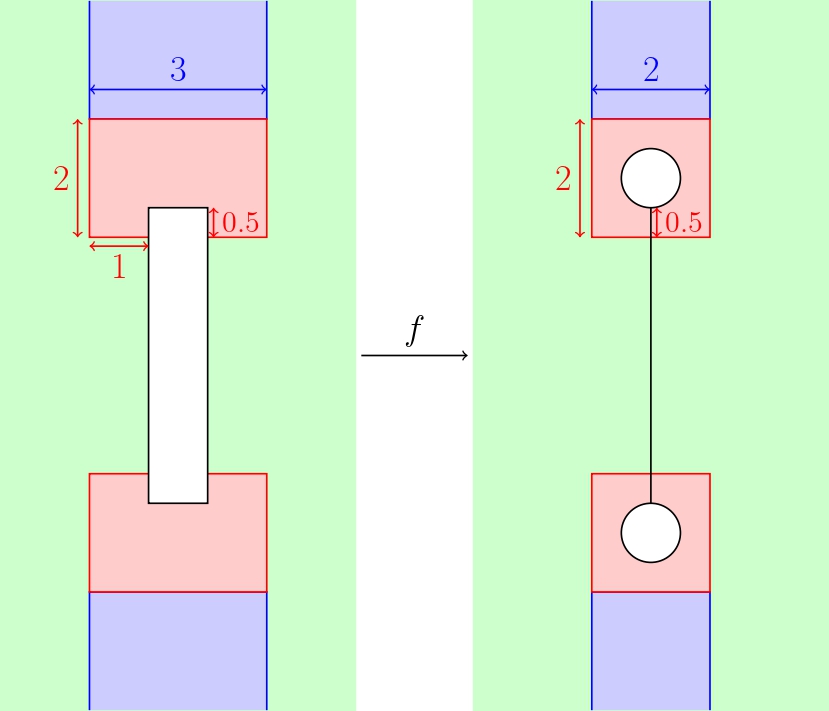}
    \caption{The quasiconformal map}
    \label{fig:qc}
\end{figure}

\noindent It remains to define \(f\) in the two red regions. Over every boundary segment of the upper red region except the top side of \(Q\), \(f\) is either already defined, or can be defined easily as a translation. It also maps the top side of \(Q\) to the circle of radius \(1/2\) centered at \((s + 1/2)i\), via a map to be determined below.

\smallskip
\noindent Consider the map 
\[
g(z) = \frac{\sqrt{s^2 + s} + i z}{\sqrt{s^2 + s} - i z}.
\] 
\smallskip

\noindent This maps the disk \(B((s + 1/2)i, 1/2)\) to \(B(0, r(s))\), and the disk \(B(-(s + 1/2)i, 1/2)\) to the complement of \(B(0, 1 / r(s))\), where \(r(s) = 2s + 1 - 2 \sqrt{s^2 + s}\). Now we can define \(f\) on \([0,1] \times \{s\}\) as follows:
\[
f(t) = g^{-1} \bigl( r(s) e^{2 \pi i (1 - t)} \bigr).
\]
It is easy to check that \(f\) has uniformly bounded derivative on the boundary of the red region, and hence can be extended to the red region as a quasi-conformal map with uniformly bounded dilation \(K\). It can be similarly defined on the lower red region.

\smallskip
\noindent Now \(f\) is constructed so that \(\frac{i}{2\pi} \log (g \circ f)\) maps \(Q^c\) to a rectangle of height \(\frac{1}{\pi} \log (r(s))\) and width 1, which maps the top and bottom sides of \(Q^c\) isometrically to the bottom and top sides of this rectangle. Now we can glue a copy of \(P\) to both rectangles, and obtain a quasi-conformal map from \(T_s\) to a torus with modulus \(\pi^{-1} \log (1 / r(s)) + 1\). By property of quasi-conformal maps, we then have the modulus of \(T_s\), given by \(\tau(s)\), is bounded above by \(K\pi^{-1} \log (1 / r(s)) + K \asymp \log(s)\). This gives the upper bound.
\end{proof}

\noindent Of course, the same argument above also gives a lower bound, but by quoting Theorem~\ref{thm:DN} we obtain the lower bound without an undetermined multiplicative factor.

\smallskip

\noindent First suppose \(\textnormal{C}\) is contained in a positive leaf. Recall that its boundary is parallel to a primitive element \(u \in \mathbb{Z}[\,i\,]\). Let \(X_t\) be the torus corresponding to the point \(t\,u\) on the boundary of the chamber \(\textnormal{C}\). Recall that this is constructed from the complement of a slit of length \((|\,t\,| + 1)|\,u\,|\) in \(\cp\), and a parallelogram \(P_t\) with vertices at \(0\), \(u\), \(t u + v\), and \((t + 1)u + v\), where \((u, v)\) form a positive basis for \(\mathbb{Z}[\,i\,]\). Note that we may choose so that the projection of \(v\) onto \(u\) has length shorter than \(|\,u\,|\). Suppose the corresponding point in the Teichm\"uller space \(\mathcal{T}_{1,1} \cong \mathbb{H}\) is \(\sigma(t)\). We claim

\begin{lem}
    There exists a constant \(D > 0\) independent of the chamber so that 
    \[
    d_{\mathbb{H}}\Big(\,\sigma(\,t\,),\, t + i \log |\,t\,|\,\Big) \le D
    \]
    for all \(|\,t\,|\) large enough.
\end{lem}

\begin{proof}
    It suffices to construct a quasi-conformal map between \(X_t\) and the torus corresponding to \(t+i\log(\,t\,)\).

    \smallskip

    \noindent For simplicity, assume \(t > 0\). The case \(t < 0\) is entirely similar. First note that we can construct a quasi-conformal map \(h_t\) from the complement of the slit to the complement of \(Q = [0,1] \times [0,t]\); see Figure~\ref{fig:qc_shear_draft}. We divide the complement of the slit into pieces by rays perpendicular to or forming \(\pi/4\) degree angles with the slit, as in Figure~\ref{fig:qc_shear_draft}.
    
    \begin{figure}[htp]
        \centering
        \includegraphics[width=0.6\linewidth]{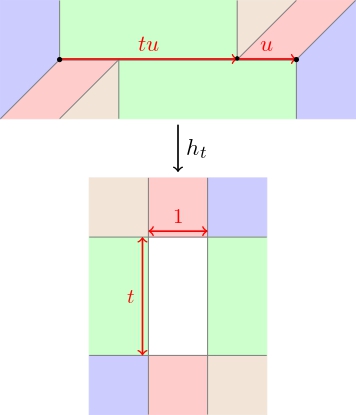}
        \caption{A piecewise affine quasi-conformal map}
        \label{fig:qc_shear_draft}
    \end{figure}
    
 \noindent The complement of \(Q\) can be divided into pieces similarly. Note that we can map each piece to one with the same color by an affine map. The only pieces varying with \(t\) are green ones, but we can map each green piece to a corresponding one by a conformal map (rotation plus scaling). So the dilation of \(h_t\) is independent of \(t\).

\smallskip

\noindent Following the same construction, \(Q^c\) is quasi-conformal to a rectangle with modulus \(\Mod(\,Q^c\,)\asymp \log(t)\). Thus the torus \(X_t\) is quasi-conformal to the torus constructed by gluing \(P_t\) to this rectangle; see Figure~\ref{fig:point_in_teich}. It is easy to see that, up to a rotation and scaling by \(1 / |\,u\,|\), the parallelogram \(P_t\) has lengths labeled as in the figure. This new torus corresponds to a point within bounded distance from \(t + i \log(\,t\,)\), as desired.
\end{proof}

\begin{figure}[htp]
        \centering
        \includegraphics[width=0.4\linewidth]{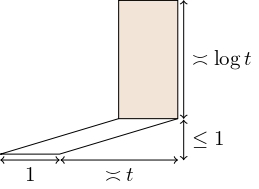}
        \caption{\(X_t\) is quasi-conformal to this torus}
        \label{fig:point_in_teich}
\end{figure}

\noindent Theorem~\ref{thm:teich_geometry} for cylinder chambers in a positive leaf then follows, since \(\glplus\) acts as quasi-conformal maps between leaves, so each group element only causes bounded distortion on Teichm\"uller distance.

\smallskip

\noindent For the negative leaf with period group \(\mathbb{Z}[\,i\,]\), recall that each cylinder chamber is again parametrized by a primitive element \(u\) in \(\mathbb{Z}[\,i\,]\), whose boundary is parallel to \(\mathbb{R}u\). The surface \(X_t\) corresponding to \(t\,u\) is given by the complement of a parallelogram \(Q\) with vertices at \(0,\,u,\, t\,u\,+\,v,\,(\,t\,+\,1\,)\,u\,+\,v\) and identification specified on the left of Figure~\ref{fig:point_in_teich_negative}. Here \((u,\, v)\) again form a positive basis for \(\mathbb{Z}[\,i\,]\) where \(v\) is chosen so that the projection of \(v\) onto \(u\) has length shorter than \(|\,u\,|\).

\begin{figure}[htp]
    \centering
    \includegraphics{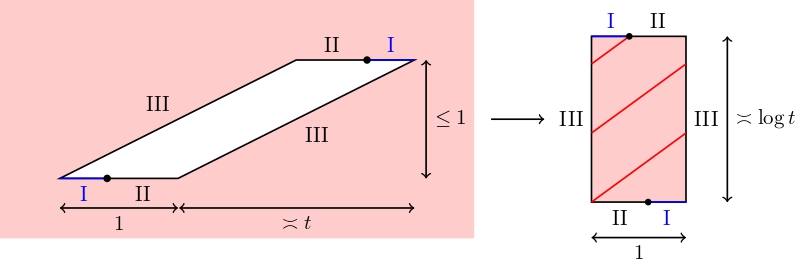}
    \caption{Illustration for negative leaves}
    \label{fig:point_in_teich_negative}
\end{figure}

\smallskip

\noindent As in the case of positive leaves, we can construct an explicit quasi-conformal map from \(X_t\) to a surface obtained from a rectangle of width 1 and height \(\asymp \log t\), with side identification specified on the right of Figure~\ref{fig:point_in_teich_negative}. The dilation of the quasi-conformal map is independent of \(t\). The closed curve with absolute period \(v\) in \(X_t\) is mapped to a curve homotopic to the curve in the rectangle with slope \(\log(t)/t\). Hence the corresponding point in the Teichmüller space is of bounded distance from \(t + i \log(t)\), as in the positive case. This concludes Theorem~\ref{thm:teich_geometry} for cylinder chambers in negative leaves.

\smallskip

\noindent Finally, the same construction works for arithmetic and non-arithmetic real leaves as well.

\subsection{Conjugacy of action of \(\pslz\)}
Recall that \(\slz\) acts as the Veech group of each generic leaf. On the other hand, \(\slz\) acts on \(\mathcal{T}_{1,1}\) by isometries.

\begin{prop}\label{prop:conjugacy}
The natural map defined in the previous section from a generic (unmarked) leaf to \(\mathcal{T}_{1,1}\) conjugates these two actions of \(\slz\) on the boundary at infinity.
\end{prop}

\noindent For simplicity, we focus on positive leaves. The case for negative leaves is entirely similar. Without loss of generality, we may also assume the period group is \(\mathbb{Z}[\,i\,]\). Recall that a cylinder chamber is parametrized by a primitive element \(u = p + iq \in \mathbb{Z}[i]\), up to multiplication by \(-1\). Here \(p, q\) are relatively prime integers. An element \(\gamma \in \slz\) maps the chamber \(\textnormal{C}(\,u\,)\) associated with \(u\) to that associated with \[\gamma \cdot u\, =\, \gamma \cdot \begin{pmatrix} p \\ q \end{pmatrix}.\]

\smallskip

\noindent On the other hand, fixing a base surface \(X \in \mathcal{T}_{1,1}\) and a symplectic basis \((\alpha, \beta)\) for \(H^1(X, \mathbb{Z})\), each point \(p/q \in \mathbb{Q} \cup \{\infty\}\) on \(\partial \mathbb{H}\) corresponds to a pinched hyperbolic surface where the simple closed curve \(p\alpha + q\beta\) is pinched. Moreover, an element \(\gamma \in \slz\) maps \(p/q\) to \[\gamma \cdot p/q = \gamma \cdot \begin{pmatrix} p \\ q \end{pmatrix}.\]

\smallskip

\noindent Thus to prove Proposition~\ref{prop:conjugacy}, it suffices to show that the point at infinity of \(\textnormal{C}(\,u\,)\) is precisely \(p/q\), noting the density of \(\mathbb{Q} \cup \{\infty\}\) in \(\partial \mathbb{H} = \mathbb{R} \cup \{\infty\}\) and the continuity of actions. But the discussions in the previous section imply that the point at infinity of \(\textnormal{C}(\,u\,)\) corresponds to pinching the curve with period \(u\). If we choose \(\alpha, \beta\) to be the curves with periods 1 and \(i\) respectively, then the curve with period \(u\) is given by \(p\alpha + q\beta\). We then have the desired result by the discussions above.

\appendix
\section{Equilateral triangulation, true trees and conformal structure of arithmetic leaves}\label{sec:true_trees}
\noindent In this appendix, we wish to provide another perspective of the conformal structure of arithmetic leaves, inspired by recent work of Ivrii-Lin-Rohde-Sygal \cite{ILRS}.

\medskip

\subsection{Equilateral triangulation and true trees}

A \textit{Bely\u{\i} function} is a holomorphic map \(f\) from a compact Riemann surface \(X\) to \(\cp\) with only 3 critical values (usually taken to be \(0, 1, \infty\)). Equivalently, \(X\) admits an \emph{equilateral triangulation}, by taking the inverse images under \(f\) of the upper half-plane and lower half-plane in \(\mathbb{C}\). Conformally, we may reconstruct \(X\) by gluing identical pieces of equilateral triangles according to the combinatorics given by the triangulation.

\smallskip

\noindent Note that not every Riemann surface admits an equilateral triangulation. A classical result of Bely\u{\i} states that a compact Riemann surface admits an equilateral triangulation if and only if it is algebraic \cite{Bel}. On the other hand, by a recent result of Bishop and Rempe, every non-compact Riemann surface admits an equilateral triangulation \cite{BR}.

\smallskip

\noindent If \(X = \cp\) and the Bely\u{\i} function \(f\) is a polynomial, then it is also called a \textit{Shabat polynomial}. It is easy to see that in this case \(f^{-1}\left(\,[0,1]\,\right)\) is a tree. Such a tree is called a \textit{true tree}, and up to M\"obius transformations, every finite tree in the plane is equivalent to a unique true tree.

\smallskip

\noindent True trees are \emph{conformally balanced}, in the sense that the harmonic measure seen from infinity of every edge on either side is the same. Equivalently, the Riemann mapping from the complement in \(\cp\) of a true tree to \(\cp \setminus \mathbb{D}\), that fixes \(\infty\), maps each side of the edges to a circular arc of equal length.

\smallskip

\noindent Let \(\mathfrak{L}\) be the (unmarked) arithmetic real leaf with period group \(\mathbb{Z}\). Then \(\mathfrak{L}\) admits an equilateral triangulation. Indeed, each chamber is identified with the upper half-plane, which can be tiled by vertical half-strips bounded by vertical rays starting at integers. Mapping the vertex at infinity of each triangle to \(\infty\), we get the corresponding Bely\u{\i} function \(f : \mathfrak{L} \to \mathbb{C}\). The preimage \(f^{-1}([0,1])\) gives precisely the walls of the chambers, shown in Figure~\ref{fig:binary_tree}.

\medskip

\subsection{Approximation by finite trees}

Let \(\mathcal{T}_\infty\) be the trivalent infinite tree with a distinguished vertex \(v\), and \(\mathcal{T}_n\) be the subtree induced by all vertices of distance \(n\) from \(v\). Denote also by \(\mathcal{T}_n\) the conformally balanced form of the tree, normalized so that the Riemann mapping \(\varphi_n: \cp \setminus \mathbb{D} \to \cp\setminus \mathcal{T}_n\) has expansion \(z \mapsto z + O(1/z)\) at infinity. Then the main results in \cite{ILRS} state:
\begin{itemize}
    \item[1.] The trees \(\mathcal{T}_n\) converge in the Hausdorff topology to an infinite trivalent tree union a Jordan curve \(\mathcal{T}_\infty \cup \partial \Omega\);
    \smallskip
    \item[2.] The corresponding Shabat polynomials \(p_n\) converge to a Bely\u{\i} function \(F: \Omega \to \cp\), so that \(\mathcal{T}_\infty = F^{-1}([0,1])\);
    \smallskip
    \item[3.] Each connected component \(\textnormal{C}\) of \(\Omega - \mathcal{T}_\infty\) is a disk, whose boundary meets \(\partial \Omega\) at a unique point and otherwise consists of edges of \(\mathcal{T}_\infty\);
    \smallskip
    \item[4.] More precisely, the domain \(\Omega\) is the so-called \emph{developed deltoid}, studied in \cite{LLMM}.
\end{itemize}
In particular, \(\Omega\) is conformally equivalent to the surface obtained by gluing half planes along the infinite tree \(\mathcal{T}_\infty\), in the spirit of Figure~\ref{fig:binary_tree}.

\smallskip

\noindent Inspired by this, we consider the following way of producing a conformal image of the arithmetic leaf \(\mathfrak{L}(\,\rho,\,\mu\,)\). Instead of looking at the rooted binary tree in Figure~\ref{fig:binary_tree}, we look at a degree 3 cover of it ramified over the root vertex \(v\). This tree is combinatorially the same as \(\mathcal{T}_\infty\), but each edge is assigned a weight according to its length in the period coordinates, as marked in Figure~\ref{fig:binary_tree}. We still denote this tree by \(\mathcal{T}_\infty\). Similarly we have the level \(n\) sub-tree \(\mathcal{T}_n\), and the corresponding weight function on its edges.

\smallskip

\noindent Let \(\mathcal{T}_n^*\) be the tree obtained from \(\mathcal{T}_n\) by adding \(t - 1\) interior vertices to an edge with weight \(t\). In this way, in the corresponding true form, the harmonic measure of an arc between branching points viewed from infinity is proportional to the weight of the corresponding edge in \(\mathcal{T}_n\). We similarly obtain \(\mathcal{T}_\infty^*\). Again, we use the same symbols \(\mathcal{T}_n^*, \mathcal{T}_\infty^*\) to denote the corresponding true trees, normalized so that the Riemann mapping \(\varphi_n^*\colon \cp \setminus \mathbb{D} \to \cp \setminus \mathcal{T}_n^*\) has expansion \(z \mapsto z + O(1/z)\) at infinity. The estimates in \cite{ILRS} can be adapted to prove the following:
\begin{itemize}
    \item A subsequence of the trees \(\mathcal{T}_n^*\) converge in the Hausdorff topology to \(\mathcal{T}_\infty^*\) union a Jordan curve \(\partial \Omega^*\);
    \item The corresponding Shabat polynomials \(p_n^*\) converge to a Bely\u{\i} function \(F^*\colon \Omega^* \to \cp\), so that \(\mathcal{T}_\infty^* = {F^*}^{-1}([0,1])\);
    \item Each connected component \(\textnormal{C}\) of \(\Omega^* - \mathcal{T}_\infty^*\) is a disk, whose boundary meets \(\partial \Omega^*\) at a unique point, and otherwise consists of edges of \(\mathcal{T}_\infty^*\);
    \item In fact \(\Omega^*\) is conformally equivalent to a degree 3 cover of the leaf \(\mathfrak{L}(\,\rho,\,\mu\,)\), with the cusp at the center filled in.
\end{itemize}
See Figure~\ref{fig:true_tree} for the picture of a true tree approximating \(\Omega^*\). The proof of these statements follows the same strategy as \cite{ILRS}, although due to a break in symmetry, many estimates need to be slightly adapted.

\begin{figure}[htp]
    \centering
    \includegraphics[width=0.5\linewidth,angle=90]{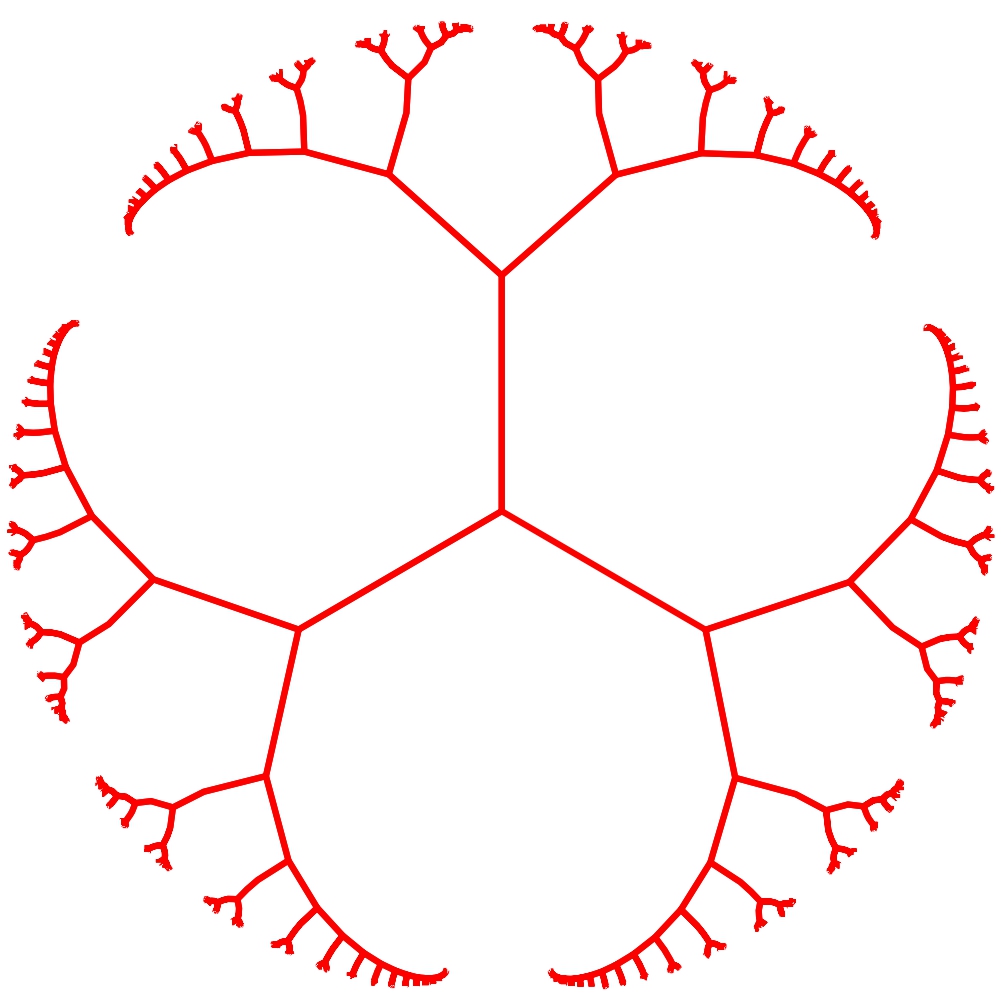}
    \caption{True trees approximate the conformal structure of \(\mathfrak{L}(\,\rho,\,\mu\,)\)}
    \label{fig:true_tree}
\end{figure}

\noindent In \cite{ILRS}, as mentioned above, the domain \(\Omega\) is identified with the developed deltoid, which arises as the mating of the dynamics of the ideal triangle group (an index 2 subgroup of \(\mathrm{PSL}_2(\mathbb{Z})\)) in the interior and the anti-holomorphic map \(z \mapsto \overline{z}^2\) in the exterior (see \cite{LLMM}). They also prove the uniqueness of the Hausdorff limit by showing that every sub-sequential limit gives the same mating. Moreover, because of the action of the ideal triangle group, the Bely\u{\i} function \(F\colon \Omega \to \cp\) composed with the Riemann mapping \(R\colon \mathbb{D} \to \Omega\) is a modular function invariant under the group.

\smallskip

\noindent In our case, however, the different lengths of the edges break the symmetry of the trivalent tree, and there is no natural \(\pslz\)-action (recall that the Veech group of the arithmetic leaf is not \(\pslz\)!). We are not sure what the exterior conformal structure could be, and hence do not claim uniqueness of the Hausdorff limit. On the other hand, it is clear that Figure~\ref{fig:true_tree} is homeomorphic to the picture produced in \cite{ILRS}; see their Figure~1. Could \(\Omega^*\) also arise dynamically? \textit{e.g.}, as the mating of a Fuchsian group and an anti-holomorphic map. We hope to further investigate these questions in the future.

\end{document}